\documentclass[11pt, reqno]{article}
\usepackage[T1]{fontenc}
\usepackage{amsmath,amssymb,amsfonts,amsthm}
\usepackage{color}
\usepackage{xr-hyper}
\usepackage[colorlinks=true,linkcolor=blue,citecolor=blue,urlcolor=blue,pdfborder={0 0 0}]{hyperref}
\usepackage{cleveref}

\usepackage{caption}
\usepackage{thmtools}

\usepackage{mathrsfs}

\crefname{theorem}{Theorem}{Theorems}
\crefname{thm}{Theorem}{Theorems}
\crefname{lemma}{Lemma}{Lemmas}
\crefname{lem}{Lemma}{Lemmas}
\crefname{remark}{Remark}{Remarks}
\crefname{prop}{Proposition}{Propositions}
\crefname{defn}{Definition}{Definitions}
\crefname{corollary}{Corollary}{Corollaries}
\crefname{conjecture}{Conjecture}{Conjectures}
\crefname{question}{Question}{Questions}
\crefname{chapter}{Chapter}{Chapters}
\crefname{section}{Section}{Sections}
\crefname{figure}{Figure}{Figures}
\crefname{example}{Example}{Examples}

\theoremstyle{plain}
\newtheorem{thm}{Theorem}[section]

\newtheorem{lemma}[thm]{Lemma}
\newtheorem{theorem}[thm]{Theorem}

\newtheorem{corollary}[thm]{Corollary}

\newtheorem{prop}[thm]{Proposition}

\theoremstyle{definition}

\theoremstyle{remark}
\newtheorem{remark}[thm]{Remark}

\numberwithin{equation}{section}

\renewcommand{\P}{\mathbb P}
\newcommand{\E}{\mathbb E}

\newcommand{\R}{\mathbb R}
\newcommand{\Z}{\mathbb Z}
\newcommand{\N}{\mathbb N}
\newcommand{\Q}{\mathbb Q}

\newcommand{\cE}{\mathcal E}
\newcommand{\cF}{\mathcal F}
\newcommand{\cG}{\mathcal G}
\newcommand{\cH}{\mathcal H}

\def\P{\mathbb{P}}

\DeclareMathSymbol{\leqslant}{\mathalpha}{AMSa}{"36} 
\DeclareMathSymbol{\geqslant}{\mathalpha}{AMSa}{"3E} 
\DeclareMathSymbol{\eset}{\mathalpha}{AMSb}{"3F}     

\makeatletter
\newcommand{\hathat}[1]{%
\begingroup%
  \let\macc@kerna\z@%
  \let\macc@kernb\z@%
  \let\macc@nucleus\@empty%
  \hat{\mathchoice%
    {\raisebox{.2ex}{\vphantom{\ensuremath{\displaystyle #1}}}}%
    {\raisebox{.2ex}{\vphantom{\ensuremath{\textstyle #1}}}}%
    {\raisebox{.16ex}{\vphantom{\ensuremath{\scriptstyle #1}}}}%
    {\raisebox{.14ex}{\vphantom{\ensuremath{\scriptscriptstyle #1}}}}%
    \smash{\hat{#1}}}%
\endgroup%
}
\makeatother

\renewcommand{\epsilon}{\varepsilon}

\usepackage[margin=1in]{geometry}
\usepackage{tikz}
\usepackage{pgfplots}
\usepackage{extarrows}
\usepackage{titling}
\usepackage{commath}
\usepackage{wasysym}
\usepackage{mathtools}
\usepackage{titlesec}
\newcommand{\eps}{\varepsilon}
\usepackage{bbm}
\usepackage{setspace}
\usepackage{enumitem}
\usepackage{booktabs}

\usepackage[font=footnotesize,labelfont=bf]{caption}

\usepackage[textsize=tiny]{todonotes}

\usepackage{cite}
\newcommand{\papernumber}{III}
\renewcommand{\thesection}{\papernumber.\arabic{section}}

\usepackage{tocloft}
\makeatletter
\pgfdeclareshape{slash underlined}
{
  \inheritsavedanchors[from=rectangle] 
  \inheritanchorborder[from=rectangle]
  \inheritanchor[from=rectangle]{north}
  \inheritanchor[from=rectangle]{north west}
  \inheritanchor[from=rectangle]{north east}
  \inheritanchor[from=rectangle]{center}
  \inheritanchor[from=rectangle]{west}
  \inheritanchor[from=rectangle]{east}
  \inheritanchor[from=rectangle]{mid}
  \inheritanchor[from=rectangle]{mid west}
  \inheritanchor[from=rectangle]{mid east}
  \inheritanchor[from=rectangle]{base}
  \inheritanchor[from=rectangle]{base west}
  \inheritanchor[from=rectangle]{base east}
  \inheritanchor[from=rectangle]{south}
  \inheritanchor[from=rectangle]{south west}
  \inheritanchor[from=rectangle]{south east}
  \inheritanchorborder[from=rectangle]
  \foregroundpath{
    \southwest \pgf@xa=\pgf@x \pgf@ya=\pgf@y
    \northeast \pgf@xb=\pgf@x \pgf@yb=\pgf@y
    \pgf@xc=\pgf@xa
    \advance\pgf@xc by .5\pgf@xb
    \pgf@yc=\pgf@ya
    \advance\pgf@xc by -1.3pt
    \advance\pgf@yc by -1.8pt
    \pgfpathmoveto{\pgfqpoint{\pgf@xc}{\pgf@yc}}
    \advance\pgf@xc by  2.6pt
    \advance\pgf@yc by  3.6pt
    \pgfpathlineto{\pgfqpoint{\pgf@xc}{\pgf@yc}}
    \pgfpathmoveto{\pgfqpoint{\pgf@xa}{\pgf@ya}}
    \pgfpathlineto{\pgfqpoint{\pgf@xb}{\pgf@ya}}
 }
}
\tikzset{nomorepostaction/.code=\let\tikz@postactions\pgfutil@empty}

\newcommand\nxleftrightarrow[2][]{%
  \mathrel{\tikz[baseline=-.7ex] \path node[slash underlined,draw,<->,anchor=south] {\(\scriptstyle #2\)} node[anchor=north] {\(\scriptstyle #1\)};}}
\makeatother

\pretitle{\begin{flushleft}\Large}
\posttitle{\par\end{flushleft}\vskip 0.5em
}
\preauthor{\begin{flushleft}}
\postauthor{
\\
\vspace{0.5em}
\footnotesize{The Division of Physics, Mathematics and Astronomy, California Institute of Technology\\ Email: 
\href{mailto:t.hutchcroft@caltech.edu}{t.hutchcroft@caltech.edu}}
\end{flushleft}}
\predate{\begin{flushleft}}
\postdate{\par\end{flushleft}}
\title{{\bf Critical long-range percolation III: The upper critical dimension}}

\renewenvironment{abstract}
 {\par\noindent\textbf{\abstractname.}\ \ignorespaces}
 {\par\medskip}

\titleformat{\section}
  {\normalfont\large \bf}{\thesection}{0.4em}{}

\author{{\bf Tom Hutchcroft}}

\makeatletter
\newcommand\mytag[2][]{%
  \def\@currentlabel{#2}%
  (#2)\label{#1} 
}
\makeatother

\begin{document}

\date{\small{\today}}

\maketitle

\begin{abstract}In long-range Bernoulli bond percolation on the $d$-dimensional lattice $\Z^d$, each pair of points $x$ and $y$ are connected by an edge with probability $1-\exp(-\beta\|x-y\|^{-d-\alpha})$, where $\alpha>0$ is fixed, $\beta \geq 0$ is the parameter that is varied to induce a phase transition, and $\|\cdot\|$ is a norm. As $d$ and $\alpha$ are varied, the model is conjectured to exhibit eight qualitatively different forms of second-order critical behaviour, with a transition between a mean-field regime and a low-dimensional regime satisfying the  hyperscaling relations when $d=\min\{6,3\alpha\}$, a transition between effectively long- and short-range regimes at a crossover value $\alpha=\alpha_c(d)$, and with various logarithmic corrections to these behaviours occurring at the boundaries between these regimes. 

This is the third of a series\footnote{
An expository account of the series over 12 hours of lectures is available on the YouTube channel of the Fondation Hadamard, see \url{https://www.youtube.com/playlist?list=PLbq-TeAWSXhPQt8MA9_GAdNgtEbvswsfS}.}
 of three papers developing a rigorous and detailed theory of the model's critical behaviour in five of these eight regimes, including all effectively long-range regimes and all effectively high-dimensional  regimes.
In this paper, we analyze the model for $d=3\alpha<6$, when it is at its upper critical dimension and effectively long-range. We prove that 
the \emph{hydrodynamic condition} holds for $d=3\alpha$, which enables us to apply the non-perturbative real-space renormalization group analysis of our first paper to deduce that the critical-dimensional model with $d=3\alpha<6$ has the same superprocess scaling limits as in high effective dimension once appropriate slowly varying corrections to scaling are taken into account. Using these results, we then compute the precise logarithmic corrections to scaling when $d=3\alpha<6$ by analyzing the renormalization group flow to second order.
Our results yield in particular that if $d=3\alpha < 6$ then the critical volume tail is given by
\[
  \P_{\beta_c}(|K|\geq n) \sim \mathrm{const.}\; \frac{(\log n)^{1/4}}{\sqrt{n}}
\]
as $n\to \infty$ while the critical two-point and three-point functions are given by
\[
\P_{\beta_c}(x\leftrightarrow y) \asymp \|x-y\|^{-d+\alpha} \; \text{ and  } \;  \P_{\beta_c}(x\leftrightarrow y \leftrightarrow z) \asymp \sqrt{\frac{\|x-y\|^{-d+\alpha}\|y-z\|^{-d+\alpha}\|z-x\|^{-d+\alpha}}{\log(1+\min\{\|x-y\|,\|y-z\|,\|z-x\|\})}}
\]
for every distinct $x,y,z\in \Z^d$.
 These are the same logarithmic corrections to scaling that appear in \emph{hierarchical} long-range percolation at the critical dimension and are \emph{not} the same conjectured to appear in nearest-neighbour percolation on $\mathbb{Z}^6$. 
\end{abstract}

\newpage
\setcounter{tocdepth}{2}

\tableofcontents

\setstretch{1.1}

\newpage

\section{Introduction}
\label{sec:introduction}

\textbf{Dimension dependence of critical phenomena in percolation.} Critical nearest-neighbour Bernoulli bond percolation on Euclidean lattices is predicted to exhibit qualitatively different behaviour according to whether the dimension $d$ is greater than, less than, or equal to the \textbf{upper critical dimension} $d_c=6$. \emph{Above the critical dimension}, critical percolation is predicted to become ``trivial'', meaning that it becomes equivalent at large scales to  \emph{branching random walk}, a ``Gaussian'' model that can be viewed as a ``non-interacting'' variant of percolation. This is known as \textbf{mean-field} critical behaviour, and is characterised by the simple asymptotic scaling of various important quantities such as
\[
  \P_{p_c}(|K|\geq n) \sim \frac{\mathrm{const.}}{\sqrt{n}}, \qquad \E_{p}|K| \sim \mathrm{const.}\; |p-p_c|^{-1}, \quad \text{ and } \quad \P_{p_c}(x\leftrightarrow y) \sim \mathrm{const.}\; \|x-y\|_2^{-d+2}
\]
as $n\to \infty$, $p\uparrow p_c$, and $\|x-y\|_2\to \infty$ respectively, where we write $K$ for the cluster of the origin. Such mean-field behaviour is now rather well-understood\footnote{The main caveat being that all existing techniques for high-dimensional nearest-neighbour percolation are \emph{perturbative} in the sense that require a ``small parameter'' to work and cannot be used to analyze the model under the minimal assumption that $d>d_c=6$.} rigorously following the seminal work of Hara and Slade \cite{MR1043524} (other important works in this direction include e.g.\ \cite{MR1959796,MR762034,MR2748397,MR1127713}; see \cite{heydenreich2015progress} for an overview), and it has recently been proven that high-dimensional critical percolation even has \emph{the same super-Brownian scaling limit} as critical branching random walk \cite{HutBR_superprocesses}. 
\emph{Below the critical dimension}, the model feels the geometry of the lattice in a non-trivial way and is predicted to have different critical exponents and scaling limits  characterised by the validity of the \emph{hyperscaling relations}; this is reasonably\footnote{Here the main caveat is that the strongest results remain limited to \emph{site percolation on the triangular lattice} \cite{smirnov2001critical2}, with \emph{universality} still a major challenge.} well-understood in two dimensions \cite{smirnov2001critical2,smirnov2001critical,MR879034,camia2024conformal} but remains a major challenge to understand in dimensions $d=3,4,5$. See e.g.\ \cite{MR1431856,MR1716769,LRPpaper2} for further discussion of qualitative distinctions between low and high dimensions. 

\emph{At the critical dimension itself}, it is predicted 
that mean-field critical behaviour \emph{almost} holds, with many quantities of interest expected to have logarithmic corrections to mean-field scaling. (We remark that the basic picture of a transition to mean-field behaviour at an upper critical dimension with logarithmic corrections at the critical dimension itself is common to many other models of statistical mechanics, but that the \emph{value} of the critical dimension is model-dependent.)
Indeed, non-rigorous renormalization group techniques were used by Essam, Gaunt, and Guttmann \cite{essam1978percolation} (see also \cite{ruiz1998logarithmic,amit1976renormalization,harris1975renormalization,gracey2015four,de1981critical}) to predict that e.g.\
\begin{equation}
\label{eq:Z6_predictions_volume}
  \P_{p_c}(|K|\geq n) \sim \mathrm{const.} \; \frac{(\log n)^{2/7}}{\sqrt{n}} \quad \text{ and } \quad \E_{p}|K| \sim \mathrm{const.}\; |p-p_c|^{-1} \left(\log \frac{1}{|p_c-p|}\right)^{2/7}
\end{equation}
as $n\to \infty$ and $p\uparrow p_c$ respectively. 
Combining their predictions (which also include an analysis of the \emph{correlation length}) with standard heuristic scaling theory yields the predicted scaling of the \emph{two-point function}
\begin{equation}
  \P_{p_c}(x\leftrightarrow y) \sim \mathrm{const.}\; \|x-y\|_2^{-d+2} (\log \|x-y\|_2)^{1/21}
  \label{eq:Z6_predictions_two_point}
\end{equation}
as $\|x-y\|\to \infty$ when $d=d_c=6$. It is also conjectured that the critical-dimensional model has the same super-Brownian scaling limit as the high-dimensional model once appropriate logarithmic corrections to scaling are taken into account. 
%
%
%
%
Despite substantial progress on critical-dimensional phenomena in a number of other models including the Ising model \cite{aizenman2019marginal}, the $\varphi^4$ model \cite{aizenman2019marginal,MR3339164,bauerschmidt2014scaling,bauerschmidt2017finite}, continuous time weakly self-avoiding walk \cite{bauerschmidt2017finite,bauerschmidt2015critical}, loop-erased random walk \cite{lawler2020logarithmic,lawler1986gaussian}, and the uniform spanning tree \cite{hutchcroft2023logarithmic,halberstam2024logarithmic}, critical-dimensional percolation remains very poorly understood, and indeed there appears to be an absence of adequate tools to address the problem. One key difficulty is that all existing heuristic approaches do not analyze percolation directly but instead work with the $n$-component $\varphi^3$ model or $q$-state Potts model before taking formal $n\to 0$ or $q\to 1$ limits at the end of the analysis, and it is not yet known how rigorously interpret such arguments for percolation\footnote{Similar obstacles have been overcome for \emph{continuous time weakly self-avoiding walk} (WSAW) and the \emph{arboreal gas}, which have proven to be ``equivalent'' in a certain non-probabilistic sense to \emph{supersymmetric} analogues of the $\varphi^4$ model and \emph{hyperbolic sigma models} respectively \cite{bauerschmidt2021geometry,MR3339164,swan2021superprobability}. The equivalence of WSAW with the supersymmetric $\varphi^4$ model is used as the basis for all of the rigorous work on the model's critical behaviour at its critical dimension \cite{MR3339164,bauerschmidt2015critical,bauerschmidt2017finite}.}.

\medskip

The goal of this paper is to develop a detailed and rigorous theory of critical behaviour at the critical dimensional for \emph{long-range} 
 percolation, 
a Bernoulli percolation model on $\Z^d$ that exhibits many of the same phenomena but is (perhaps surprisingly) more tractable than nearest-neighbour percolation on $\Z^d$.
Our results
establish long-range analogues of many of the central conjectures for six-dimensional nearest-neighbour percolation discussed above, and 
 include both the precise computation of logarithmic corrections to mean-field scaling for several relevant quantities (\cref{thm:critical_dim_moments_main,thm:pointwise_three_point}) and the convergence of the model to an appropriate super-L\'evy scaling limit  once these logarithmic corrections to scaling are taken into account (\cref{cor:superprocess_main_CD}). 

\medskip

This work is the third in a series of three papers on critical long-range percolation, with the first two papers focussing on the high-dimensional \cite{LRPpaper1} and low-dimensional regimes \cite{LRPpaper2} respectively. This series builds both on previous works on long-range percolation \cite{hutchcroft2020power,hutchcroft2022sharp,hutchcroft2024pointwise,baumler2022isoperimetric,MR1896880} and parallels the analysis of \emph{hierarchical percolation} in our earlier work \cite{hutchcroft2022critical} (which we do not assume familiarity with). The present paper is the most technical of the series,
with the critical-dimensional model sharing some features of both the high- and low-dimensional models, 
 and builds on several of the ideas, methods, and results of the first two papers. 
All three papers in the series employ a novel \emph{real-space renormalization group} (RG) \emph{method} that we recall in \cref{subsec:RG_hydro_intro}; this method works directly with percolation and does not pass through any putative connection to $\varphi^3$ or Potts field theories. (A detailed discussion of how this method compares to other approaches to RG is given in \cref{I-subsec:RG}.)
One of the main technical contributions of the paper (\cref{thm:critical_dim_hydro}) is a proof that the critical-dimensional model satisfies the \emph{hydrodynamic condition}, a ``marginal triviality'' condition that we recall in detail in \cref{subsec:RG_hydro_intro} and which was shown in \cref{I-sec:analysis_of_moments,I-sec:superprocesses}
to imply that the ``renormalization group flow'' has mean-field behaviour \emph{to first order} and hence that
 the model has superprocess scaling limits (with undetermined slowly varying corrections to scaling). Once this is done, we determine the precise logarithmic corrections to scaling by computing the ``renormalization group flow'' to \emph{second order} as we explain in detail in \cref{subsec:about_the_proof}. These second-order computations also reveal some interesting new phenomena from the physics perspective that we discuss in \cref{subsec:a_tale_of_two_vertex_factors}.

\medskip

We stress that all of our results are \emph{non-perturbative} in the sense that they do not require any ``small parameter'' to work; this is in contrast to both
to both traditional approaches to \emph{high-dimensional} percolation via the lace expansion \cite{MR1043524,MR2430773} (which has been applied to long-range models in \cite{MR2430773,MR3306002,MR4032873,liu2025high}), the recently introduced lace expansion alternative of Duminil-Copin and Panis \cite{duminil2024alternative,duminil2024alternativeSAW}, and to Bauerschmidt, Brydges, and Slade's RG analysis of the critical-dimensional weakly self-avoiding walk and $\varphi^4$ models \cite{MR3339164,bauerschmidt2014scaling,bauerschmidt2015critical,bauerschmidt2017finite}.
 Our results are also \emph{universal} in the sense that they do not depend on the precise form of the kernel used to define the model and do not rely on any form of ``integrability'' or ``exact solvability''.

 \begin{remark}
This series of papers is entirely devoted to the study of \emph{critical phenomena} in long-range percolation. Other aspects of the model, such as the geometry of the infinite
supercritical cluster, are studied in e.g.\
  \cite{biskup2021arithmetic,ding2023uniqueness,baumler2023distances,benjamini2008long}.
\end{remark}

\subsection{Definition of the model and statement of main results}

\label{subsec:definitions_main_theorems_intro}

Let us now define the model precisely.
Let $d\geq 1$ and let  $J:\Z^d\times \Z^d \to [0,\infty)$ be a symmetric, translation-invariant kernel (meaning that $J(x,y)=J(y,x)=J(0,y-x)$ for every distinct $x,y\in \Z^d$). 
In \textbf{long-range percolation} on $\Z^d$, 
 we define a random graph with vertex set $\Z^d$ by declaring any two distinct vertices $x,y\in \Z^d$ to be connected by an edge $\{x,y\}$ with probability $1-e^{-\beta J(x,y)}$, independently of all other pairs, where $\beta\geq 0$ is a parameter.
We write $\P_\beta$ and $\E_\beta$ for probabilities and expectations taken with respect to the law of the resulting random graph, whose connected components are referred to as \textbf{clusters}.
  The \textbf{critical point} $\beta_c=\beta_c(J)$ is defined by
\[
  \beta_c =\inf\{\beta \geq 0: \P_\beta(\text{an infinite cluster exists})>0\}.
\]
 We will 
 be interested in the case that the kernel $J$ has power-law decay of the form
\begin{equation}
\label{eq:kernel_simple_intro}
  J(x,y) \sim \text{const.}\;\|x-y\|^{-d-\alpha}
\end{equation}
for some $\alpha>0$ and norm $\|\cdot\|$; 
such kernels have a non-trivial phase transition in the sense that $0<\beta_c<\infty$ if and only if $d\geq 2$ and $\alpha>0$ or $d=1$ and $0<\alpha \leq 1$ \cite{newman1986one,schulman1983long}. 
 For technical reasons we will work under the slightly stronger assumption that $J(x,y)= \mathbbm{1}(x\neq y) J(\|x-y\|)$ for some decreasing, differentiable function $J(r)$ satisfying
\begin{equation}
\label{eq:J_derivative_assumption_C}
  |J'(r)| = C(1+\delta_r) r^{-d-\alpha-1}
\end{equation}
for some constant $C>0$ and \textbf{logarithmically integrable error function} $\delta_r$, that is, a measurable function $\delta_r$ satisfying $\delta_r\to 0$ as $r\to \infty$ and $\int_{r_0}^\infty \frac{|\delta_r|}{r}\dif r<\infty$ for some $r_0<\infty$.  As in the rest of the series, 
\textbf{we will assume without loss of generality that that the unit ball in $\|\cdot\|$ has unit Lebesgue measure and that the constant $C$ appearing in \eqref{eq:J_derivative_assumption_C} is equal to $1$, so that \begin{equation}
\label{eq:normalization_conventions}
\tag{$*$}
|J'(r)|\sim r^{-d-\alpha-1} \qquad \text{ and } \qquad |B_r|:=|\{x\in \Z^d:\|x\|\leq r\}|\sim r^d,\end{equation}and will do this throughout the paper.} Any other kernel of the form \eqref{eq:J_derivative_assumption_C} can be re-scaled by a constant to be of this form, which changes the value of $\beta_c$ but not the law of the critical model.

\medskip

As $d$ and $\alpha$ are varied, long-range percolation is predicted to exhibit at least eight qualitatively different forms of second-order critical behaviour, with a transition between an \emph{effectively low-dimensional} and an \emph{effectively high-dimensional} regime occurring as the \textbf{effective dimension} $d_\mathrm{eff}=\max\{d,2d/\alpha\}$ crosses the upper critical dimension $d_c=6$, a transition between an \emph{effectively long-range} (small $\alpha$) and an \emph{effectively short-range} (large $\alpha$) regime at a crossover value $\alpha_c=\alpha_c(d)$, and with further, subtly different behaviours characterised by logarithmic corrections to scaling at the boundaries of these regimes (\cref{fig:cartoon}). As alluded to above, this paper focuses on the case $d=3\alpha<6$ where the model is effectively long-range and \emph{critical dimensional}, exhibiting logarithmic corrections to mean-field scaling, while the other two papers focus on the effectively high-dimensional \cite{LRPpaper1} and effectively long-range low-dimensional \cite{LRPpaper2} regimes respectively. (Despite these focuses, \cite{LRPpaper1,LRPpaper2} both contain results relevant to the critical dimension as discussed below.)

\begin{figure}[t]
\centering
\includegraphics[scale=0.95]{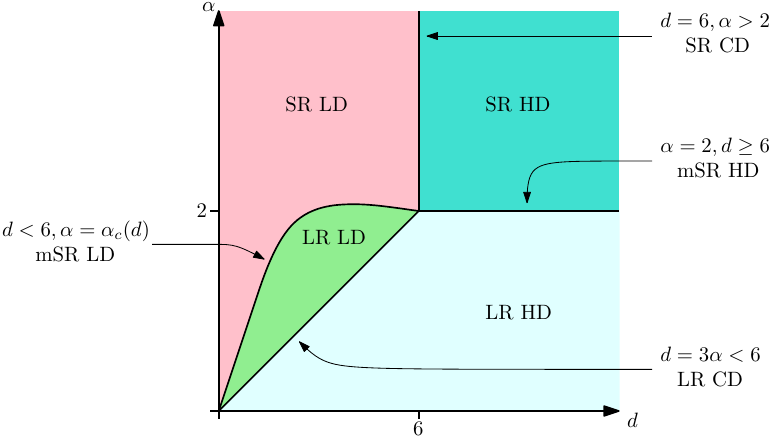}\hspace{1cm}
\caption{Schematic illustration of the different regimes of critical behaviour for long-range percolation. 
 LR, SR, HD, LD, and CD stand for ``Long Range'', ``Short Range'', ``High Dimensional'', ``Low Dimensional'', and ``Critical Dimensional'' respectively, while mSR stands for ``marginally Short Range''. Here we ignore the special behaviours occuring when $d=1$, $\alpha \geq 1$ (where there is either no phase transition when $\alpha>1$ or a discontinuous phase transition when $\alpha=1$ \cite{MR868738,duminil2020long}) to avoid clutter. In this paper we study critical behaviour on the critical line $d=3\alpha<6$ (LR CD), where the model is effectively long-range and at its upper critical dimension. }
\label{fig:cartoon}
\end{figure}

\medskip\noindent \textbf{Asymptotic notation.}
We write $\asymp$, $\preceq$, and $\succeq$ for equalities and inequalities that hold to within positive multiplicative constants depending on $d\geq 1$, $\alpha>0$, and the kernel $J$
but not on any other parameters. If implicit constants depend on an additional parameter (such as the index of a moment), this will be indicated using a subscript. Landau's asymptotic notation is used similarly, so that if $f$ is a non-negative function then ``$f(n)=O(n)$ for every $n\geq 1$'' and ``$f(n) \preceq n$ for every $n\geq 1$'' both mean that there exists a positive constant $C$ such that $f(n)\leq C n$ for every $n\geq 1$. We also write $f(n)=o(g(n))$ to mean that $f(n)/g(n)\to 0$ as $n\to\infty$ and write $f(n)\sim g(n)$ to mean that $f(n)/g(n)\to 1$ as $n\to\infty$. The quantities we represent implicitly using big-$O$ and little-$o$ notation are always taken to be non-negative, and we denote quantities of uncertain sign using $\pm O$ or $\pm o$ as appropriate.

\medskip

\noindent\textbf{Main results.} We now state our main results, beginning with the computation of the logarithmic correction to scaling for the volume tail. Note that the logarithmic correction to scaling appearing here is \emph{not} the same predicted to hold for nearest-neighbour percolation on $\Z^6$ \cite{essam1978percolation} as stated in \eqref{eq:Z6_predictions_volume}, but is the same occurring for \emph{hierarchical percolation} at its critical dimension \cite{hutchcroft2022critical}. 

\begin{theorem}[Cluster volumes in the critical dimension]
\label{thm:critical_dim_moments_main}
If $d=3\alpha<6$ then 
\begin{equation}
\label{eq:critical_dim_volume_tail_main}
\P_{\beta_c}(|K|\geq n) \sim \text{\emph{const.}}\; \frac{(\log n)^{1/4}}{\sqrt{n}}
\end{equation}
as $n\to \infty$.
\end{theorem}

\begin{remark}
Interestingly, the constant prefactor in \cref{thm:critical_dim_moments_main} has an explicit expression in terms of $\beta_c$ and a \emph{universal constant} determined by the scaling limit of the model (see \eqref{eq:volume_tail_explicit}). We do not expect such universality of constant prefactors to occur in high dimensions.
\end{remark}

\begin{remark}
The fact that there are different logarithmic corrections for short- and long-range percolation at the critical dimension is already apparent from the two-point function, which does not have a logarithmic correction in the long-range model with $d=3\alpha<6$ (as was proven for $d=1,2$ in \cite{hutchcroft2024pointwise} and extended to all $1\leq d<6$ in \cref{thm:pointwise_three_point}, below) but is predicted to have a $(\log \|x-y\|)^{1/21}$ correction to scaling in the short-range model as stated in \eqref{eq:Z6_predictions_two_point}. We conjecture that the same logarithmic corrections to scaling predicted in \cite{essam1978percolation} also appear in the long-range model with $d=6$, $\alpha>2$, which is effectively short-ranged and critical dimensional.
\end{remark}

\begin{remark} Let us now explain how the exponent of the logarithmic correction to the volume tail can be obtained heuristically from what is known about the effectively low-dimensional regime.
It is shown in \cref{II-thm:main_low_dim} that the exponent $\delta$ is given in the effectively long-range, low-dimensional regime by $\delta = (d+\alpha)/(d-\alpha)$. This formula for $\delta$ can be obtained heuristically from Sak's prediction \cite{sak1973recursion} for the two-point exponent $2-\eta=\alpha$ via the \emph{hyperscaling relations} as discussed in \cref{II-sec:negligibility_of_mesoscopic_clusters_and_hyperscaling}.
 In the part of this regime in which $\alpha<2$ we have that
 $d_\mathrm{eff}=2d/\alpha$ and can therefore write this equality as
\[\delta = \frac{d+\alpha}{d-\alpha} = \frac{d_\mathrm{eff}+2}{d_\mathrm{eff}-2}.\]
If $d_\mathrm{eff}=6-\eps<6$ we can expand this formula in powers of $\eps$ as
\[
  \delta = \frac{8-\eps}{4-\eps} = 2 + \frac{1}{4}\eps + \frac{1}{16}\eps^2 + \cdots
\]
Thus, \cref{II-thm:main_low_dim,thm:critical_dim_moments_main} are an instance of the general (heuristic) phenomenon that the exponent of the logarithmic correction at the critical dimension $d_c$ is equal to the \emph{left derivative} at $d_c$ of the corresponding critical exponent, viewed as a function of the effective dimension. (We warn the reader that there can be some subtleties in the correct interpretation of this heuristic for each given exponent.)
\end{remark}

\noindent
\textbf{Superprocess scaling limits.}
Our next main result concerns the full scaling limit of the cluster of the origin, considered as a random measure on $\R^d$. The following theorem can be interpreted as stating that the cluster has the same scaling limit as a critical branching random walk with jump kernel proportional to $\|x-y\|^{-d-\alpha}$ once appropriate logarithmic corrections are taken into account: these scaling limits are \textbf{integrated $\alpha$-stable superprocesses}. (Here the word ``integrated'' means that we are considering the clusters as \emph{measures}, without keeping track of their internal structure.) No prior knowledge of superprocesses will be needed to read this paper: we will prove the theorem by applying the sufficient condition for superprocess limits established in \cref{I-thm:superprocess_main_regularly_varying}, and do not work with superprocesses directly other than via the diagrammatic expression for moment asymptotics of \cref{I-thm:scaling_limit_diagrams} that we state here in \eqref{eq:scaling_limit_diagrams} and the associated recurrence relation \eqref{eq:recurrence_from_derivative_LR3}. We direct interested readers to e.g.\ \cite{perkins2002part,dynkin1994introduction,le1999spatial} for an introductory account of the theory of superprocesses and to \cite{slade2002scaling} for a survey of their applications in high-dimensional statistical mechanics. 

\medskip

We write $\delta_x$ for the Dirac delta measure at $x\in \R^d$.

\begin{theorem}[Superprocess limits at the critical dimension]
\label{cor:superprocess_main_CD}
 If $d= 3\alpha<6$ then there exist functions $\zeta(r)\to\infty$ and $\eta(r)\to 0$ such that
\[
\frac{1}{\eta(r)}\P_{\beta_c}\Biggl(|K|\geq \lambda \zeta(r),\; \frac{1}{\zeta(r)}\sum_{x \in K} \delta_{x/r}\in\cdot\Biggr) \to \mathbb{N}\Bigl(\mu(\R^d)\geq \lambda, \mu \in \cdot\Bigr)
\]
weakly as $r\to \infty$ for each fixed $\lambda>0$, where $\mathbb{N}$ denotes the canonical measure of the integrated symmetric $\alpha$-stable L\'evy superprocess excursion associated to the L\'evy measure with density $\frac{\alpha}{d+\alpha}\|x\|^{-d-\alpha}$. Moreover, the normalization factors $\zeta$ and $\eta$ are given asymptotically by
\[
\zeta(r) \sim \text{\emph{const.}}\;r^{2\alpha} (\log r)^{-1/2} 
\quad \text{ and } \qquad
\eta(r) \sim \text{\emph{const.}}\; r^{-\alpha} (\log r)^{1/2} 
\]
as $r\to \infty$.
\end{theorem}

\begin{table}[t]
\centering
\renewcommand{\arraystretch}{1.6} 
\begin{tabular}{|c|c|c|}
\hline
 & \# Large Clusters & Typical size \\ 
\hline
LR HD & $r^{d-3\alpha}$ & $r^{2\alpha}$  \\ 
\hline
mSR HD & $r^{d-6}(\log r)^3$ & $r^4 (\log r)^{-2}$ \\ 
\hline
SR HD & $r^{d-6}$ & $r^4$  \\ 
\hline
LR CD & $\log r$ & $r^{2\alpha} (\log r)^{-1/2}$  \\ 
\hline
LR LD & $O(1)$ & $r^{(d+\alpha)/2}$ \\ 
\hline
\end{tabular}
\caption{Summary of results concerning the ``typical size of a large cluster'' and the ``number of typical large clusters'' in a box of radius $r$ for long-range percolation on $\Z^d$. Precise statements are given in \cref{I-thm:superprocess_main}, \cref{II-thm:hyperscaling} and \cref{cor:superprocess_main_CD}. (See also \eqref{I-eq:number_of_large_clusters} and \eqref{eq:number_of_large_clusters}.) 
The most important qualitative distinction is between the LR LD regime studied in the second paper of this series, where the number of large clusters is $O(1)$, and the other regimes, studied in the first and third papers, where the number of large clusters on scale $r$ diverges as $r\to\infty$. See \cite{MR1431856,borgs1999uniform} for discussions of similar questions in nearest-neighbour percolation.
}
\label{table:large_clusters}
\end{table}

Here, $\N$ is a $\sigma$-finite measure on non-zero finite measures on $\R^d$ giving finite mass to the set $\{\mu:\mu(\R^d)\geq \eps\}$ for every $\eps>0$.
Similar theorems are established for long-range percolation with high effective dimension in \cref{I-thm:superprocess_main} and for high-dimensional nearest-neighbour percolation in our forthcoming joint work with Blanc-Renaudie \cite{HutBR_superprocesses} (see also \cite{MR1773141,van2003convergence,van2013survival}).

\medskip

As discussed in more detail in \cref{I-subsec:statement_of_high_dimensional_results}, one should interpret the scaling factor $\zeta(r)$ as being the ``typical size of a large cluster on scale $r$'' and $\eta(r)$ as being the probability that the origin belongs to such a typical large cluster. As such, the ratio
\begin{equation}
\label{eq:number_of_large_clusters}
  N(r) := \frac{r^d \eta(r)}{\zeta(r)} \sim \mathrm{const.}\;\log r
\end{equation}
can be interpreted heuristically as the ``number of typical large clusters'' on scale $r$. These heuristic interpretations are justified in part by \eqref{eq:two_and_three_point_heuristic} and \eqref{eq:three_point_heuristic2}, below; see also the discussion around \eqref{I-eq:N_and_zeta_moments}. A comparison of the number and size of typical large clusters on scale $r$ across the different regimes treated in this series of papers is given in \Cref{table:large_clusters}.

\begin{remark}
\label{remark:scaling_limit_volume_formulation}
\cref{cor:superprocess_main_CD} can be rewritten in terms of the volume rather than the length scale, and as a statement about conditioned rather than restricted measures, to yield that
\begin{equation}
\P_{\beta_c}\left( \frac{1}{n}\sum_{x \in K} \delta_{x/\zeta^{-1}(n)}\in \cdot \;\Bigg|\; |K| \geq n \right) \to \mathbb{N}\bigl( \mu \in \cdot \;|\; \mu(\R^d) \geq 1 \bigr)
\end{equation}
weakly as $n\to \infty$, where $\mathbb{N}( \mu \in \cdot \,|\, \mu(\R^d) \geq 1 )$ denotes the probability measure defined by $\mathbb{N}\bigl( \mu \in A \;|\; \mu(\R^d) \geq 1 \bigr) = \mathbb{N}\bigl( \mu \in A, \mu(\R^d)\geq 1)/\mathbb{N}(\mu(\R^d)\geq 1)$.
\end{remark}

\noindent \textbf{Two- and three-point functions.} Our final main result establishes \emph{pointwise} estimates on the two-point and three-point connectivity functions that are consistent with \cref{cor:superprocess_main_CD}. The two-point estimate will be established as a special case of \cref{II-thm:CL_Sak} (which applies in the critical dimensional case $d=3\alpha<6$ by \cref{thm:critical_dim_hydro}), while the three-point estimate is specific to the case $d=3\alpha<6$.


\begin{theorem}[Two- and three-point functions in the critical dimension]
\label{thm:pointwise_three_point} If $d=3\alpha<6$ then
\begin{align*}
  \P_{\beta_c}(x\leftrightarrow y) \asymp \|x-y\|^{-d+\alpha} \;\; \text{ and } \;\;
\P_{\beta_c}(x\leftrightarrow y \leftrightarrow z) \asymp \sqrt{\frac{\|x-y\|^{-d+\alpha}\|y-z\|^{-d+\alpha}\|z-x\|^{-d+\alpha}}{\log(1+\min\{\|x-y\|,\|y-z\|,\|z-x\|\})}}
\end{align*}
for every triple of distinct points $x,y,z\in \Z^d$.
\end{theorem}

(Note that the two-point estimate can be deduced from the three-point estimate by taking $\|z-y\|=O(1)$.)
The three-point function estimate established by \cref{thm:pointwise_three_point} shows that the tree-graph inequality \[\P_{\beta_c}(x\leftrightarrow y \leftrightarrow z) \leq \sum_{w\in \Z^d} \P_{\beta_c}(x\leftrightarrow w)\P_{\beta_c}(w\leftrightarrow y)\P_{\beta_c}(w\leftrightarrow z)\] of Aizenman and Newman \cite{MR762034} (which is sharp in high effective dimension as we prove when $d>3\alpha$ and $\alpha<2$ in \cref{remark:tree_graph_vs_Gladkov} and \cref{prop:three_point_lower}) and the Gladkov inequality~\cite{gladkov2024percolation} \[\P_{\beta_c}(x\leftrightarrow y \leftrightarrow z) \preceq \sqrt{\P_{\beta_c}(x\leftrightarrow y)\P_{\beta_c}(y\leftrightarrow z)\P_{\beta_c}(z\leftrightarrow x)}\] (which is sharp in low effective dimension as proven in \cref{II-thm:k_point_S}) are both wasteful by a $\sqrt{\log}$ factor when $d=3\alpha<6$. (See \cref{remark:tree_graph_vs_Gladkov} for an explanation of why the two bounds have the same order when $d=3\alpha<6$ and $\P_{\beta_c}(x\leftrightarrow y) \asymp \|x-y\|^{-d+\alpha}$.)
When all three distances are of the same order $r$, \cref{thm:pointwise_three_point} can be interpreted in terms of the quantities $\zeta(r)$ and $N(r)$ above as
\begin{equation}
  \P_{\beta_c}(x\leftrightarrow y) \asymp \frac{N(r) \zeta(r)}{r^d} \cdot \frac{\zeta(r)}{r^d} \qquad\text{ and }\qquad \P_{\beta_c}(x\leftrightarrow y\leftrightarrow z) \asymp \frac{N(r) \zeta(r)}{r^d} \cdot \left(\frac{\zeta(r)}{r^d}\right)^2;
  \label{eq:two_and_three_point_heuristic}
\end{equation}
in each case the first factor of $r^{-d}N(r) \zeta(r)$ accounts for the probability that $x$ belongs to a large cluster on scale $r$ while each factor of $r^{-d}\zeta(r)$ accounts for the probability that $y$ or $z$ belongs to the \emph{same} large cluster as $x$. On the other hand, if e.g.\ $\|x-y\|$ is much smaller than $\|x-z\|$ then the three point function scales as
\begin{equation}
  \P_{\beta_c}(x\leftrightarrow y\leftrightarrow z) \asymp \frac{N(\|x-z\|) \zeta(\|x-z\|)}{\|x-z\|^d} \cdot \frac{\zeta(\|x-z\|)}{\|x-z\|^d} \cdot \frac{\zeta(\|x-y\|)}{\|x-y\|^d}.
\label{eq:three_point_heuristic2}
\end{equation}
Intuitively, this means that when $x$ is in a typical large cluster on scale $\|x-z\|$ it is likely to also be in a typical large cluster on the much smaller scale $\|x-y\|$: the first factor accounts for the probability $x$ belongs to a typical large cluster on the large scale $\|x-z\|$ while the second and third factors account for the probability that $z$ and $y$ belongs to the same large cluster, where in the third factor we take the density of this cluster at the smaller scale $\|x-y\|$.

\begin{remark}
We note that the \emph{spatially-averaged upper bound} on the critical two-point function
\begin{equation}
\label{eq:spatially_averaged_two_point_upper_intro}
  \frac{1}{r^d} \sum_{x\in [-r,r]^d} \P_{\beta_c}(0\leftrightarrow x) \preceq r^{-d+\alpha}
\end{equation}
was proven to hold for all $0<\alpha<d$ in \cite{hutchcroft2022sharp} and plays an important role in our analysis. The pointwise estimate $\P_{\beta_c}(x\leftrightarrow y)\asymp \|x-y\|^{-d+\alpha}$ was shown to hold whenever $\alpha<1$ in \cite{hutchcroft2024pointwise} (see also \cite{baumler2022isoperimetric}) and extended in \cref{II-thm:CL_Sak} to the entire ``effectively long-range'' regime as defined by \eqref{II-CL} (see \cref{II-def:CL}); these theorems apply in the case $d=3\alpha<6$ by \cref{thm:critical_dim_hydro}, below, as noted in \cref{II-cor:HD_CL}.
\end{remark}

\begin{remark}
For nearest-neighbour percolation on $\Z^6$, the predictions of \cite{essam1978percolation} are consistent with the asymptotics
\begin{equation}
  N(r) \sim \mathrm{const}.\;\log r \qquad \text{ and } \qquad \zeta(r) \sim  \mathrm{const}.\;  (\log r)^{-10/21}r^4
\end{equation}
for the number and size of typical large clusters on scale $r$. Following the same heuristic computations as in \eqref{eq:two_and_three_point_heuristic} and \eqref{eq:three_point_heuristic2} this leads to the predicted three-point function scaling
\begin{equation}
  \P_{p_c}(x\leftrightarrow y \leftrightarrow z) \asymp \frac{(\log [1+d_\mathrm{max}(x,y,z)])^{1/21}}{
(\log[1+ d_\mathrm{min}(x,y,z)])^{10/21}
  d_\mathrm{max}(x,y,z)^{4} d_\mathrm{min}(x,y,z)^{2} }
\end{equation}
where $d_\mathrm{min}(x,y,z)$ and $d_\mathrm{max}(x,y,z)$ are the minimal and maximal distances between the three distinct points $x,y,z$ respectively.
\end{remark}

\subsection{The RG framework and the hydrodynamic condition}

\label{subsec:RG_hydro_intro}

We now recall the basics of our renormalization group framework and explain how it is used to prove our main theorems.
Given a kernel $J$ as above and a parameter $r>0$, we write $J_r$ for the \textbf{cut-off kernel}
\[
J_r(x,y) = \mathbbm{1}(x\neq y, \|x-y\|\leq r) \int_{\|x-y\|}^r |J'(s)| \dif s
\]
and write $\P_{\beta,r},\E_{\beta,r}$ for probabilities and expectations taken with respect to the law of long-range percolation on $\Z^d$ with kernel $J_r$ at parameter $\beta$. We will always write $\beta_c$ for the critical parameter associated to the original kernel $J$, so that the cut-off measure $\P_{\beta_c,r}$ is subcritical.
(Indeed, for every $r<\infty$ there exists $\eps=\eps(r)>0$ such that $J_r(x,y)\leq (1-\eps)J(x,y)$ for every $x\neq y$, so that $\P_{\beta_c,r}$ is stochastically dominated by the subcritical measure $\P_{(1-\eps)\beta_c}$; see the proof of \cref{lem:moment_poly_upper_bound} for a quantitative version of this argument.)

\medskip

At its core, our method aims to compute asymptotics of observables such as $\E_{\beta_c,r}|K|$, $\E_{\beta_c,r}|K|^2$, and $\E_{\beta_c,r}\sum_{x\in K}\|x\|_2^2$ as functions of the cut-off parameter $r$ before using these asymptotics to deduce properties of the original measure $\P_{\beta_c}$ via a Tauberian analysis. This is done primarily by finding \emph{asymptotic ODEs} satisfied by these quantities. For example, in \cref{I-sec:analysis_of_moments,I-sec:superprocesses} we prove that in high effective dimensions the $p$th volume moment satisfies the ``mean-field'' asymptotic ODE
\begin{equation}
\label{eq:volume_moments_ODE_intro}
  \frac{d}{dr}\E_{\beta_c,r}|K|^p \sim \beta_c r^{-\alpha-1} \sum_{\ell=0}^{p-1}\binom{p}{\ell} \E_{\beta_c,r}|K|^{\ell+1}\E_{\beta_c,r}|K|^{p-\ell}
\end{equation}
as $r\to \infty$, while the spatially-weighted moment $\E_{\beta_c,r}\sum_{x\in K}\|x\|_2^2$ satisfies the asymptotic ODE
\begin{equation}
\frac{d}{dr} \E_{\beta_c,r} \left[\sum_{x\in K} \|x\|_2^2 \right]
\sim  \frac{2 \alpha}{r} \E_{\beta_c,r} \left[\sum_{x\in K} \|x\|_2^2 \right] + \frac{1}{r}\left(\frac{\alpha^2}{\beta_c} \int_{B} \|y\|_2^2 \dif y\right) r^{2+\alpha},
\label{eq:x12_derivative_asymptotic2_intro}
\end{equation}
as $r\to \infty$, where $B=\{x\in \R^d:\|x\|\leq 1\}$. (The precise constant prefactors appearing here depend on our normalization convention \eqref{eq:normalization_conventions}.)  See \cref{I-subsec:definitions} for further discussion and motivation. These asymptotic ODEs along with their extensions to arbitrary spatially-weighted moments (see \cref{I-thm:scaling_limit_diagrams}) are shown in \cref{I-sec:volume_tail,I-sec:superprocesses} to fully determine the volume tail asymptotics and superprocess scaling limits of the model up to slowly varying corrections to scaling, which are shown to be asymptotically constant in high effective dimensions as a consequence of the implicit multiplicative errors in \eqref{eq:volume_moments_ODE_intro} and \eqref{eq:x12_derivative_asymptotic2_intro} being logarithmically integrable in this case.

\medskip

In \cref{I-sec:analysis_of_moments,I-sec:superprocesses}, it is also shown that the mean-field asymptotic ODEs \eqref{eq:volume_moments_ODE_intro} and \eqref{eq:x12_derivative_asymptotic2_intro}, along with their extensions to arbitrary spatially-weighted moments, continue to hold in the critical-dimensional case $d=3\alpha<6$ subject to the \emph{hydrodynamic condition}, a ``marginal triviality'' condition originally introduced for hierarchical models in \cite{hutchcroft2022critical} whose definition we now recall.
For each $r\geq 1$, we define the \textbf{edian}
\[
  M_r =\min\left\{n\geq 0: \P_{\beta_c,r}\left(\max_{x\in \Z^d} |K_x \cap B_r| \geq n\right)\leq e^{-1}\right\},
\]
which measures the typical size of the largest cluster in the ball $B_r$ under the cut-off measure $\P_{\beta_c,r}$. 
As explained in \cref{I-sub:previous_results_on_long_range_percolation}, it follows from the results of \cite{hutchcroft2022sharp} that $M_r$ always satisfies
 \begin{equation}
 \label{eq:pre_Hydro}
 M_r=O(r^{(d+\alpha)/2})
 \end{equation} as $r\to \infty$ when $0<\alpha<d$ (and also vacuously when $\alpha\geq d$). 
 We say that the \textbf{hydrodynamic condition} \eqref{Hydro} holds if this bound admits a strict improvement
 \begin{equation}
 \label{Hydro}
 \tag{Hydro}
 M_r=o(r^{(d+\alpha)/2})
 \end{equation}
  as $r\to\infty$. This terminology was inspired by the notion of \emph{hydrodynamic limits}, where trajectories of certain Markov chains (e.g.\ interacting particle systems) converge to deterministic dynamical systems described by ODEs or PDEs \cite{demasi2006mathematical}; see \cite[Section 4]{hutchcroft2022critical} for a discussion of how the hydrodynamic condition implies the ``asymptotically deterministic'' (i.e., concentrated) geometry of large clusters in a large box in the hierarchical setting. The fact that the edian $M_r$ is a \emph{useful} quantity to study owes largely to the \emph{universal tightness theorem} of \cite[Theorem 2.2]{hutchcroft2020power} as discussed in \cref{I-subsec:correlation_inequalities_and_the_universal_tightness_theorem}.

\medskip

A core technical result of this paper verifies that the hydrodynamic condition holds in the critical case $d=3\alpha$, making the theorems \cref{I-thm:critical_dim_moments_main_slowly_varying,I-thm:superprocess_main_regularly_varying} apply unconditionally in this case. (These theorems do \emph{not} determine the logarithmic corrections to scaling, which require an additional argument described in detail below. Indeed, such logarithmic corrections are not present when $d=3\alpha>6$ and the model is effectively \emph{short-range high-dimensional}.)

\begin{theorem}
\label{thm:critical_dim_hydro}
If $d=3\alpha$ then the hydrodynamic condition holds.
\end{theorem}

This theorem should be compared to the marginal triviality theorem for the four-dimensional Ising and $\varphi^4$ models due to Aizenman and Duminil-Copin \cite{aizenman2019marginal}, which was extended to long-range Ising and $\varphi^4$ models by Panis \cite{panis2023triviality} and which shows that these models have Gaussian scaling limits in their critical dimension (see also \cite{MR1882398}); \cref{I-thm:superprocess_main_regularly_varying,thm:critical_dim_hydro} establish a similar theorem for long-range percolation (with ``Gaussian'' interpreted to mean ``superprocess'') via a rather different technique that we summarise in \cref{subsec:about_the_proof}.

\begin{remark}
The proof of \cref{thm:critical_dim_hydro} proceeds by contradiction, and is \emph{ineffective} in the sense that it does not yield any explicit bound of the form $M_r=o(r^{(d+\alpha)/2})$.
We conjecture that if $d=3\alpha<6$ then
\[
  M_r \sim \text{const.} \frac{\log \log r}{\sqrt{\log r}} r^{2\alpha}
\]
as $r\to \infty$, so that the largest cluster is larger than the typical large cluster size $\zeta(r)$ by a factor of order $\log N(r) \asymp \log \log r$. Despite its non-quantitative nature, \cref{thm:critical_dim_hydro} will play an important role in our eventual computation of the logarithmic corrections to scaling at the critical dimension as we explain below. 
\end{remark}

In addition to the consequences we have already mentioned, \cref{thm:critical_dim_hydro} (together with \cref{I-prop:radius_of_gyration,I-prop:displacement_moments}, which apply when $d=3\alpha$ as a consequence of \cref{thm:critical_dim_hydro}) also implies that the model satisfies the \textbf{correlation length condition for effectively long-range critical behaviour} \eqref{II-CL} of \cref{II-def:CL} when $d=3\alpha<6$ (a fact recorded as part of \cref{II-cor:HD_CL}).
As a consequence, \cref{thm:critical_dim_hydro} allows us to apply several of the results of the second paper (that treat the entire effectively long-range regime) along the critical line $d=3\alpha<6$; these results include the slightly subcritical scaling relations of Theorems \ref{II-thm:scaling_relations}, \ref{II-thm:Fisher_relation}, and \ref{II-thm:subcritical_volume} and the pointwise two-point function estimate of \cref{II-thm:CL_Sak}. In particular, the two-point function estimate of \cref{thm:pointwise_three_point} is an immediate consequence of \cref{thm:critical_dim_hydro,II-cor:HD_CL,II-thm:CL_Sak} leaving us only to prove the corresponding three-point estimate. (It also follows from \cref{thm:critical_dim_hydro} and \cref{I-prop:radius_of_gyration_d=3alpha>6} that the model is \emph{not} effectively long-range when $d=3\alpha\geq 6$.)


\medskip

\noindent
\textbf{The geometric significance of the hydrodynamic condition.} 
The significance of the bound \eqref{eq:pre_Hydro} and the hydrodynamic condition \eqref{Hydro} can be understood in part via the following calculation. Let $r\gg 1$ be some large scale and consider the standard monotone coupling of the two cut-off measures $\P_{\beta_c,r}$ and $\P_{\beta_c,2r}$.
 If $A_1$ and $A_2$ are disjoint subsets of the ball $B_r$ then the expected number of edges between $A_1$ and $A_2$ that are open in the larger configuration but not in the smaller configuration is of order $|A_1||A_2|r^{-d-\alpha}$, so that the probability such an open edge exists is high when $|A_1|,|A_2| \gg r^{(d+\alpha)/2}$, low when $|A_1|,|A_2| \ll r^{(d+\alpha)/2}$, and bounded away from $0$ and $1$ when $|A_1|$ and $|A_2|$ are both of order $r^{(d+\alpha)/2}$. Thus, the bound $M_r =O(r^{(d+\alpha)/2})$ implies that the large clusters at scale $r$ cannot have a \emph{high} probability of merging via the direct addition of a single edge when we pass from scale $r$ to scale $2r$,
  while the hydrodynamic condition ensures moreover that this merging event occurs with \emph{low probability}.
In \cref{II-thm:max_cluster_size_LD} we proved that $M_r \asymp r^{(d+\alpha)/2}$ in the effectively long-range \emph{low-dimensional} (LR LD) regime, so that the hydrodynamic condition does \emph{not} hold in this regime, whereas \cref{I-cor:HD_hydro,thm:critical_dim_hydro} show that the hydrodynamic condition does hold in the effectively high-dimensional (HD) and effectively long-range critical-dimensional (LR CD) regimes.
Geometrically, this means that (for effectively long-range models) 
large critical clusters have a good probability to directly merge with each other via the addition of a single long edge on each scale in low dimensions, 
while in high dimensions and at the critical dimension different clusters interact with each other only very weakly and mean-field asymptotic ODEs such as \eqref{eq:volume_moments_ODE_intro} and \eqref{eq:x12_derivative_asymptotic2_intro} are valid. 
This change in behaviour is related to the validity of the \emph{hyperscaling relations} in low effective dimensions as discussed in \cref{II-sec:negligibility_of_mesoscopic_clusters_and_hyperscaling} and to the fact that the ``number of typical large clusters'' $N(r)$ is bounded in the effectively low-dimensional regime but diverges in the effectively high- and critical-dimensional regimes (\Cref{table:large_clusters}). 






\subsection{About the proof}
\label{subsec:about_the_proof}

\textbf{About the proof of the hydrodynamic condition.}
We now very briefly summarise the proof of \cref{thm:critical_dim_hydro}.
 In \cref{sec:the_hydrodynamic_condition_holds_critical_dim}, we prove that if $d=3\alpha$ and the hydrodynamic condition does \emph{not} hold then this forces the model to have \emph{exactly mean-field critical behaviour} in various senses, with \emph{no logarithmic corrections to scaling}. In particular, this means that the expectation $\E_{\beta_c,r}|K_x||K_y|\mathbbm{1}(x\nleftrightarrow y)$ is of the same order as the product $\E_{\beta_c,r}|K_x|\E_{\beta_c,r}|K_y|$ when $x\neq y$ are fixed and $r\to \infty$. On the other hand, we also argue that, when the hydrodynamic condition does not hold, the primary contribution to $\E_{\beta_c,r}|K_x|$ must come from clusters with size of order $r^{2\alpha}=r^{(d+\alpha)/2}$ for every $r\geq 1$, exactly as in the low-dimensional case (\cref{II-thm:max_cluster_size_LD}); this is proven using a variation on the argument used to establish the validity of the \emph{hyperscaling relations} for low-dimensional models in \cref{II-sec:negligibility_of_mesoscopic_clusters_and_hyperscaling}.
Thus, under the fictitious assumption that the hydrodynamic condition does not hold, we prove that the model has critical behaviour that is in some senses \emph{both exactly mean-field and exactly low-dimensional} (in the sense that the hyperscaling relations hold), with no logarithmic corrections to either behaviour. In \cref{sec:the_hydrodynamic_condition_ii_reaching_a_contradiction} we leverage the tension between these two properties to obtain a contradiction. More precisely, we  use these properties to prove that $\E_{\beta_c,r}|K_x||K_y|\mathbbm{1}(x\nleftrightarrow y)$ must in fact be much smaller than the product $\E_{\beta_c,r}|K_x|\E_{\beta_c,r}|K_y|$ for large $r$ since, roughly speaking, the two large clusters have a constant probability to be ``glued together'' on each of a large collection of intermediate scales due to the ``low-dimensional'' scaling of the largest cluster size. This contradicts the mean-field estimate $\E_{\beta_c,r}|K_x||K_y|\mathbbm{1}(x\nleftrightarrow y) \asymp \E_{\beta_c,r}|K_x|\E_{\beta_c,r}|K_y|$ and completes the proof. 

One significant difficulty 
 is that (at the beginning  of the proof) we do not yet know that the model with $d=3\alpha<6$ is effectively long-range, so that e.g.\ the susceptibility $\E_{\beta_c,r}|K|$ could be much larger than  $r^\alpha$. 
To obtain some technical leverage when dealing with this issue, we will in fact work with the measures $\E_{(1-\lambda r^{-\alpha})\beta_c,r}$ rather than $\E_{\beta_c,r}$ in the argument sketched above, where $\lambda$ is a small positive constant.
A related difficulty is that to implement our gluing argument in \cref{sec:the_hydrodynamic_condition_ii_reaching_a_contradiction} we need that the main contribution to $\E_{\beta_c,r}|K|$ is from clusters \emph{whose intersection with $B_{Cr}$} has size of order $r^{2\alpha}$, rather than just having this size globally, which is a significantly more subtle estimate to obtain.
(These difficulties did not arise in the hierarchical case, where our proof of the hydrodynamic condition was much simpler \cite[Section 6.1]{hutchcroft2022critical}. These parts of the proof can also be simplified significantly in the cases $d=1,2$, where the model with $\alpha=d/3$ has $\alpha<1$ and is therefore effectively long-range by \cref{II-thm:alpha<1_CL}.) 
A further key technical difficulty throughout the proof is that the failure of the hydrodynamic condition means only that $M_r \succeq r^{(d+\alpha)/2}$ on an \emph{unbounded set of scales}, rather than on \emph{all} scales; a significant amount of effort is spent ruling out pathologies relating to possible oscillatory behaviour or irregular growth of $M_r$ and various other important quantities as $r\to \infty$ (under the fictitious assumption that the hydrodynamic condition does not hold).


\medskip

\noindent
\textbf{On the computation of logarithmic corrections to scaling.} 
We now give a brief overview of how our main theorems \cref{thm:critical_dim_moments_main,cor:superprocess_main_CD} are proven once \cref{thm:critical_dim_hydro} is established.
The relevant sections, \cref{sec:the_three_point_function,sec:logarithmic_corrections_at_the_critical_dimension}, can be read independently of the proof of the hydrodynamic condition in \cref{sec:the_hydrodynamic_condition_holds_critical_dim,sec:the_hydrodynamic_condition_ii_reaching_a_contradiction}.
The analysis of \cref{I-sec:analysis_of_moments,I-sec:superprocesses} shows that the mean-field asymptotic ODEs of \eqref{eq:volume_moments_ODE_intro} and \eqref{eq:x12_derivative_asymptotic2_intro}, along with their extensions to arbitrary spatially-weighted moments, are valid when $d=3\alpha$ and the hydrodynamic condition holds. In contrast to the effectively high-dimensional case, however, it is possible that the errors in these asymptotic ODEs are fairly large, so that slowly-varying corrections to scaling may be present in their solutions.
On the other hand, it is shown in the same analysis that if $d=3\alpha<6$ then these slowly varying corrections to scaling, if present, are entirely captured by the asymptotics of the \emph{second moment} $\E_{\beta_c,r}|K|^2$, with all further slowly varying corrections to scaling determined by those of the second moment in a simple way.
More precisely, it follows from \cref{thm:critical_dim_hydro} and Theorem \ref{I-thm:critical_dim_moments_main_slowly_varying} that if $d=3\alpha$ then
\begin{equation}
\label{eq:critical_dim_moments_main}
\E_{\beta_c,r} |K|^p \sim (2p-3)!! \left(\frac{\E_{\beta_c,r}|K|^2}{\E_{\beta_c,r}|K|}\right)^{p-1}\E_{\beta_c,r}|K| \sim  (2p-3)!! A^{p-1}_r \frac{\alpha}{\beta_c} r^{(2p-1)\alpha} 
\end{equation}
for each integer $p\geq 1$ as $r\to\infty$ 
and that
\begin{equation}
\label{eq:critical_dim_volume_tail_main_intro}
\P_{\beta_c}(|K|\geq n) \sim \frac{\alpha}{\beta_c} \sqrt{\frac{2}{\pi A_n^*}} \cdot \frac{1}{\sqrt{n}}
\end{equation}
as $n\to \infty$, where  $A_r$ and $A_n^*$ are bounded, slowly varying functions such that $r\mapsto A_r r^{2\alpha}$ and $n\mapsto (A_n^*)^{-1/2\alpha} n^{1/2\alpha}$ are inverses to each other. 
(The constant $\alpha/\beta_c$ appearing here arises from our normalization convention~\eqref{eq:normalization_conventions}. The double factorial $(2p-3)!!$ is defined to be the product of all odd integers between $1$ and $2p-3$, with $(-1)!!$ defined to be $1$.)
Moreover, it follows from \cref{thm:critical_dim_hydro} and \cref{I-thm:superprocess_main_regularly_varying} that if $d=3\alpha<6$ then a super-L\'evy scaling limit theorem holds as in \cref{cor:superprocess_main_CD} but with the scaling factors $\zeta(r)$ and $\eta(r)$ given by
 \[
  \eta(r) = 4 \cdot \frac{(\E_{\beta_c,r}|K|)^2}{\E_{\beta_c,r}|K|^2}
   \qquad \text{ and } \qquad \zeta(r) = \frac{1}{4} \cdot \frac{\E_{\beta_c,r}|K|^2}{\E_{\beta_c,r}|K|}.
\]
In \cref{sec:the_three_point_function} we prove that the three-point function can also be expressed up-to-constants in terms of the slowly varying function $A_r$ using a variation on the same argument used to study the two-point function in \cref{II-subsec:the_correlation_length_condition_and_the_two_point_function}.

\medskip

In light of these results, to prove \cref{thm:critical_dim_moments_main,cor:superprocess_main_CD} given \cref{thm:critical_dim_hydro} together with \cref{I-thm:critical_dim_moments_main_slowly_varying,I-thm:superprocess_main_regularly_varying} it suffices to prove that there exists a positive constant $A$ such that
\begin{equation}
\label{eq:second_moment_log_intro}
  \E_{\beta_c,r}|K|^2 \sim A\frac{r^{2\alpha}}{\sqrt{\log r}} \E_{\beta_c,r}|K|
\end{equation}
as $r\to \infty$ when $d=3\alpha<6$.
%
%
%
%
%
Indeed, the inverse of a function asymptotic to $A (\log r)^{1/2} r^{2\alpha}$ is asymptotic to $A^{-1/2\alpha} (2\alpha)^{-1/4\alpha} (\log n)^{1/4\alpha} n^{1/2\alpha}$, so that if we write this inverse as $(A_n^*)^{-1/2\alpha} n^{1/2\alpha}$ then $A_n^* \sim \sqrt{2\alpha} A (\log n)^{1/2}$ and we obtain from \eqref{eq:critical_dim_volume_tail_main_intro} and \eqref{eq:second_moment_log_intro} that
\begin{equation}
  \P_{\beta_c}(|K|\geq n) \sim \frac{\alpha}{\beta_c} \sqrt{\frac{2}{\pi A \sqrt{2\alpha}}}\cdot\frac{(\log n)^{1/4}}{\sqrt{n}}
\label{eq:volume_tail_exact_constant_intro}
\end{equation}
as $n\to \infty$ as desired.
We compute the logarithmic correction to scaling of \eqref{eq:second_moment_log_intro} by proving that when $d\geq 3\alpha$ and $\alpha<2$, the derivative of the second moment can be expanded to second order (modulo additional logarithmically integrable error terms)  as
\begin{equation}
\label{eq:second_order_correction_ODE_intro}
  \frac{d}{dr}\E_{\beta_c,r} |K|^2 = \frac{3 \alpha}{r} \left( 1-(C\pm o(1)) \frac{(\E_{\beta_c,r} |K|^2)^2}{r^{d}(\E_{\beta_c,r}|K|)^3}  \right)\E_{\beta_c,r} |K|^2,
\end{equation}
where $C>0$ is a constant arising in a simple, explicit way from the superprocess scaling limit of the model (see \cref{prop:critical_dim_second_order,remark:explicit_constants}). 
When $d=3\alpha<6$ the error term $(\E_{\beta_c,r} |K|^2)^2/(r^{d}(\E_{\beta_c,r}|K|)^3)$ is slowly varying (i.e., the $r^{6\alpha}$ in the numerator cancels with the $r^{d+3\alpha}$ in the denominator) and this asymptotic ODE forces $\E_{\beta_c,r} |K|^2$ to have a $(\log r)^{-1/2}$ correction to power-law scaling as in \eqref{eq:second_moment_log_intro} by elementary analysis (see \cref{lem:ODE_log_correction}). 

\medskip

It is an interesting feature of our proof that we first establish the full superprocess scaling limit of the model via \cref{thm:critical_dim_hydro,I-thm:superprocess_main_regularly_varying,I-thm:scaling_limit_diagrams,I-cor:scaling_limit_cut_off}, with unspecified slowly varying corrections to scaling, before using our knowledge of this scaling limit to compute what these corrections to scaling actually are. Roughly speaking, this approach is viable because our scaling limits are determined by \emph{first-order} asymptotic ODEs whereas computing the logarithmic corrections requires us to analyze these ODEs to \emph{second order}.

\begin{remark}
It would be interesting to apply \cref{thm:critical_dim_hydro} to give a non-perturbative analysis of the model in the cases $d=6$, $\alpha=2$ (where it is effectively high-dimensional and marginally short-ranged) and $d=3\alpha>6$ (where it is effectively short-ranged and high-dimensional). Mean-field behaviour has been established under perturbative assumptions  in both regimes via the lace expansion (which was implemented for $d>6$ in \cite{MR2430773,MR3306002} and for $d=6$, $\alpha=2$ in \cite{MR4032873}), and the outputs of this lace expansion analysis were used to establish the full scaling limit of the model in \cref{I-thm:superprocess_main}.
\end{remark}



\section{Equivalent characterisations of mean-field critical behaviour}
\label{sec:equivalent_characterisations_of_mean_field_critical_behaviour}

In this section we prove that several different characterisations of mean-field critical behaviour are equivalent to each other. Moreover, the equivalence theorems we prove hold not just asymptotically but on a scale-by-scale basis. The results of this section hold for percolation on arbitrary transitive weighted graphs, and will be used in our analysis of $d=3\alpha$ long-range percolation under the fictitious assumption that the hydrodynamic condition does \emph{not} hold in \cref{sec:the_hydrodynamic_condition_holds_critical_dim,sec:the_hydrodynamic_condition_ii_reaching_a_contradiction}. Some of the estimates established in this section will also be used to prove negligibility of error terms when computing logarithmic corrections to scaling in \cref{sec:logarithmic_corrections_at_the_critical_dimension} (see in particular \cref{lem:spatial_geodesic}).

\medskip

We recall that a \textbf{weighted graph} is a triple $G=(V,E,J)$ consisting of a countable graph and a weight function $J:E\to [0,\infty)$ such that the total weight of edges emanating from each vertex is finite. Bernoulli-$\beta$ percolation on a weighted graph is defined by taking each edge $e$ to be open with probability $1-e^{-\beta J(e)}$, independently of all other edges. A graph automorphism is a weighted graph automorphism if it preserves the weights, and a weighted graph is said to be transitive if any vertex can be mapped to any other vertex by an automorphism. (In particular, for each $r\geq 0$ the cut-off kernel $J_r$ gives $\Z^d$ the structure of a transitive weighted graph.) Given a transitive weighted graph $G$ we will write $o$ for an arbitrarily chosen root vertex of $G$ and write $|J|$ for the total weight of edges emanating from $o$. 

\subsection{Relations between the $\beta$ derivative and the second moment}
\label{sub:relations_between_the_beta_derivative_and_the_second_moment}

In this section we explain how two complementary inequalities, originally proven in \cite{1901.10363} and \cite{durrett1985thermodynamic} respectively, imply that the $\beta$-derivative $\frac{d}{d\beta}\E_\beta|K|$ takes its maximal order of $(\E_\beta|K|)^2$ \emph{if and only if} $\E_\beta|K|^2$  takes its maximal order of  $(\E_\beta|K|)^3$. These inequalities, which hold for arbitrary transitive weighted graphs, will play a crucial role in our analysis of the $d=3\alpha$ under the fictitious assumption that the hydrodynamic condition does \emph{not} hold.

\medskip

Before proceeding, let us briefly recall why these two quantities have the maximal orders we have just mentioned. Regarding the second moment, the \emph{tree-graph inequalities} of Aizenman and Newman \cite{MR762034} imply that
\begin{equation}
  \E_\beta|K|^p \leq (2p-3)!! (\E_\beta|K|)^{2p-1}
\label{eq:tree_graph_pth_moment}
\end{equation}
for percolation on any transitive weighted graph, for every $\beta\geq 0$ and integer $p\geq 1$, yielding in particular that 
\begin{equation}
\label{eq:tree_graph_2nd_moment}
\E_\beta|K|^2\leq (\E_\beta|K|)^3.\end{equation}
Meanwhile, for the derivative, we have by Russo's formula that
\begin{multline}
  \frac{d}{d\beta} \E_\beta|K| = \sum_{x\in V} J(e)\P_\beta(\text{$e$ a closed pivotal for $o\leftrightarrow x$}) \\= \sum J(e) \frac{e^{-\beta J(e)}}{1-e^{-\beta J(e)}} \sum_{x\in V} J(e)\P_\beta(\text{$e$ an open pivotal for $o\leftrightarrow x$})
\end{multline}
where we recall that an edge $e$ is said to be a \textbf{closed pivotal} for an increasing event $A$ if the percolation configuration $\omega$ satisfies $\omega \notin A$ and $\omega \cup\{e\} \in A$ (similarly, $e$ is said to be an \textbf{open pivotal} if $\omega \in A$ and $\omega \setminus \{e\}\notin A$). If $e$ is a closed pivotal for the event $\{o\leftrightarrow x\}$ then it admits an orientation $e=(e^-,e^+)$ such that the disjoint occurence $\{o \leftrightarrow e^-\}\circ\{e^+ \leftrightarrow x\}$ holds, and writing $E^\rightarrow$ for the set of oriented edges we obtain from the BK inequality that
\begin{equation}
\label{eq:derivative_upper_general}
  \frac{d}{d\beta}\E_\beta|K| \leq \sum_{x\in V} \sum_{e \in E^\rightarrow} J(e)\P_\beta(\{o \leftrightarrow e^-\}\circ\{e^+ \leftrightarrow x\}) = |J| (\E_\beta|K|)^2.
\end{equation}
The two bounds \eqref{eq:tree_graph_2nd_moment} and \eqref{eq:derivative_upper_general} are predicted to be of the correct order as $\beta\uparrow \beta_c$ in high dimensions but not in low dimensions. In particular, a matching lower bound in \eqref{eq:derivative_upper_general} implies the mean-field divergence of the susceptibility $\E_\beta|K| \asymp |\beta-\beta_c|^{-1}$ as $\beta\uparrow \beta_c$. (Moreover, the proof of \cref{prop:finitary_mean_field} can be used to show that a matching lower bound in \eqref{eq:tree_graph_2nd_moment} implies the mean-field volume tail estimate $\P_{\beta_c}(|K|\geq n) \asymp n^{-1/2}$, which is shown to imply other exponents take their mean-field values in \cite{1901.10363}.)

\medskip

We now prove that each of the inequalities \eqref{eq:tree_graph_2nd_moment} and \eqref{eq:derivative_upper_general} admits a matching lower bound of the same order if and only if the other one does.
We begin by stating the following differential inequality, which was essentially proven in \cite{1901.10363} and made explicit in \cite[Lemma 6.3]{hutchcroft2022critical}. This inequality holds for arbitrary transitive weighted graphs and is a consequence of the OSSS inequality \cite{o2005every,MR3898174}. 

\begin{lemma}\label{lem:second_moment_differential_inequality}
Let $G=(V,E,J)$ be a transitive weighted graph, let $o$ be a vertex of $G$ and let $K$ be the cluster of $o$ in Bernoulli-$\beta$ percolation. Then
\[
\frac{d}{d\beta} \E_\beta |K| \geq \min_{e\in E} \left[\frac{J_e}{e^{\beta J_e}-1}\right] 
 \left(\frac{\E_\beta|K|^2}{4\E_\beta|K|}-\frac{1}{2}\E_\beta|K|+\frac{1}{4}\right)
\]
for every $\beta<\beta_c$.
\end{lemma}

(In \cite{hutchcroft2022critical} this estimate was stated in terms of Dini derivatives. When $\beta<\beta_c$ the susceptibility is differentiable and an ordinary derivative can be used.)
Roughly speaking, this inequality states that if $\E_\beta|K|^2$ takes its maximal order of $(\E_\beta|K|)^3$ then the derivative $\frac{d}{d\beta}\E_\beta|K|$ also takes its maximal possible order of $(\E_\beta|K|)^2$. 
We will also make us of a complementary inequality due to Durrett and Nguyen \cite[Section 5]{durrett1985thermodynamic}, which states that
\[
  \frac{d}{d\beta} \E_\beta |K| \preceq_\beta \sqrt{\E_\beta |K| \E_\beta |K|^2}.
\]
(Their inequality was stated for homogeneous percolation models but extends easily to weighted graphs as we explain below.)
We will prove the following generalization of this inequality to \emph{truncated} moments, which will play an important role in \cref{sub:fictitious_negligibility_of_mesoscopic_clusters}.

\begin{lemma}
\label{lem:Durrett_Nguyen}
Let $G=(V,E,J)$ be a weighted graph, let $o$ be a vertex of $G$ and let $K$ be the cluster of $o$ in Bernoulli-$\beta$ percolation. Then
\[
  \frac{d}{d\beta}\E_\beta \min\{|K|,m\} \leq 
  \sqrt{
  \beta^{-1}
  |J^*| \E_\beta[\min\{|K|,m\}] \E_\beta[\min\{|K|,m\}^2]
  }
\]
for every $\beta<\beta_c$ and $m>0$, where we define $|J^*|$ to be the supremal total weight of edges emanating from a vertex of $G$. If $G$ is transitive and $\beta<\beta_c$ then the inequality also holds with $m=\infty$.
\end{lemma}

When $G$ is transitive and $m<\infty$ we have $|J^*|=|J|$ and can apply the bound $\min\{|K|,m\}^2\leq m \min\{|K|,m\}$ to obtain from this lemma that
\begin{equation}
\label{eq:truncated_susceptibility_derivative_upper_Durrett_Nguyen}
  \frac{d}{d\beta}\E_\beta \min\{|K|,m\} \leq 
  \beta^{-1/2}
|J|^{1/2} m^{1/2} \E\min\{|K|,m\}.
\end{equation}
The proof of \cref{lem:Durrett_Nguyen} is similar to that of \cite{durrett1985thermodynamic} except that we use martingale techniques to bound the second moment of the cluster fluctuation instead of relating it to the second derivative of the magnetization as was done in \cite{durrett1985thermodynamic}. 

\begin{remark}
\label{remark:truncated_susceptibility_simple_diff_ineq}
The proof of \cref{II-lem:truncated_expectation_differential_inequality} also establishes the related inequality
\[\frac{d}{d\beta}\E_\beta \min\{|K|,m\} \leq |J| (\E_\beta \min\{|K|,m\})^2\]
via a much simpler argument than that used for \cref{lem:Durrett_Nguyen}. The inequality of
 \eqref{eq:truncated_susceptibility_derivative_upper_Durrett_Nguyen} is a significant improvement to this inequality when $m \ll (\E_\beta \min\{|K|,m\})^2$.
\end{remark}

\begin{proof}[Proof of \cref{lem:Durrett_Nguyen}] We focus on the case that $G$ is finite and $m<\infty$, the extensions to infinite $G$ and $m$ following by standard limit arguments. (In particular, the sharpness of the phase transition implies that $\frac{d}{d\beta}\E_\beta |K| = \lim_{m\to \infty} \frac{d}{d\beta}\E_\beta \min\{|K|,m\}$ when $G$ is transitive and $\beta<\beta_c$.) We may assume without loss of generality that $m$ is an integer. As in e.g.\ the proof of \cite[Theorem 1.6]{1808.08940}, we can  explore the cluster $K$ one edge at a time in such a way that if $T$ denotes the total number of edges touching $K$, $E_{n}$ denotes the (random, unoriented) edge whose status is queried at the $n$th step of the exploration for each $n\geq 0$, and $\cF_{n}$ denotes the $\sigma$-algebra generated by the first $n$ steps of the exploration, then $\{E_i : 1 \leq i \leq T\}=E(K)$ is the set of all edges with at least one endpoint in $K$ and
$\P_\beta(E_{n+1}=1 \mid \cF_n) = 1-e^{-\beta J_{E_{n+1}}}$
whenever $n<T$. (Indeed, we can define such an exploration process by fixing an enumeration $\{e_1,e_2,\ldots\}$ of the edge set and, at each step, taking $E_{n+1}$ to be minimal with respect to this enumeration among those edges that are incident to the part of the cluster that has been explored but have not already been queried.) We define $T_m$ to be the stopping time for this process where we first discover either that $|K|\geq m$ or that $|K|<m$, where the latter can only occur at the time $T$ when we find all edges in the boundary of the explored region to be closed.

\medskip

We can expand $\E_\beta\min\{|K|,m\}$ as a sum over the set $\mathscr{E}_m$ of triples $(W,O,C)$ that can possibly occur as the sets of vertices  revealed to belong to the cluster, open edges revealed to belong to the cluster, and closed edges  revealed to be adjacent to the cluster up to the stopping time $T_m$:
\[
  \E_\beta \min\{|K|,m\} = \sum_{(W,O,C)\in \mathscr{E}_m} |W| \prod_{e\in O} (1-e^{-\beta J(e)})\prod_{e\in C} e^{-\beta J(e)}.
\]
Here we have used that if the triple $(W,O,C)$ belongs to $\mathscr{E}_m$ (i.e., can possibly be the output of the exploration algorithm) then it is output by the exploration algorithm if and only if every edge of $O$ is open and every edge of $C$ is closed.
Differentiating term-by-term, we obtain that
\begin{align*}
  \frac{d}{d\beta}\E_\beta\min\{|K|,m\} &= \sum_{(W,O,C)\in \mathscr{E}_m} |W| \prod_{e\in O} (1-e^{-\beta J(e)})\prod_{e\in C} e^{-\beta J(e)} \left[ \sum_{e\in O} \frac{J(e)e^{-\beta J(e)}}{1-e^{-\beta J(e)}}-\sum_{e\in C} J(e) \right] \\&= \E_\beta \left[\min\{|K|,m\} Z_{T_m}\right],
  \end{align*}
  where $Z_n$ is the \textbf{fluctuation martingale}, defined by
  \[
    Z_n := \sum_{i=0}^n \left[\frac{J(E_i)e^{-\beta J(E_i)}}{1-e^{-\beta J(E_i)}} \mathbbm{1}(E_i\text{ open}) - J(E_i)\mathbbm{1}(E_i \text{ closed})\right],
  \]
  which is indeed a martingale with respect to the filtration $\mathcal{F}$.
Using Cauchy-Schwarz we obtain that
\begin{equation}
\label{eq:truncated_derivative_martingale1}
  \frac{d}{d\beta}\E_\beta\min\{|K|,m\} \leq \sqrt{\E_\beta Z_{T_m}^2 \E_\beta\min\{|K|,m\}^2}.
\end{equation}
Now, we have by the orthogonality of martingale increments that
\begin{align*}
  \E_\beta Z_{T_m}^2 &= \sum_{n=0}^\infty \E_\beta \left[ \left(\frac{J(E_i)e^{-\beta J(E_i)}}{1-e^{-\beta J(E_i)}} \mathbbm{1}(E_i\text{ open}) - J(E_i)\mathbbm{1}(E_i \text{ closed})\right)^2 \mathbbm{1}(i<T_m)\right]
  \\
  &=\sum_{e} \P(e \text{ revealed by time $T_m$})
   \E_\beta \left[ \left(\frac{J(e)e^{-\beta J(e)}}{1-e^{-\beta J(e)}} \mathbbm{1}(e\text{ open}) - J(e)\mathbbm{1}(e \text{ closed})\right)^2 \right]
   \\
   &= \sum_{e} \P(e \text{ revealed by time $T_m$})
   \left[\left(\frac{J(e)e^{-\beta J(e)}}{1-e^{-\beta J(e)}}\right)^2(1-e^{-\beta J(e)})  +  J(e)^2 e^{-\beta J(e)} \right]
   \\
    &= \sum_{e} \P(e \text{ revealed by time $T_m$})
  \frac{J(e)^2e^{-\beta J(e)}}{1-e^{-\beta J(e)}},
\end{align*}
where the second equality follows using linearity of expectation and the fact that conditioning on the exploration process choosing to query the edge $e$ at step $i$ does not affect the probability that $e$ is open. 
Using the elementary inequality $\frac{x^2e^{-\beta x}}{(1-e^{-\beta x})^2} \leq \frac{1}{\beta^2}$, which holds for all $x>0$ since this function is decreasing in $x$ and converges to $\beta^{-2}$ as $x\downarrow 0$, we obtain that
\begin{multline}
  \E_\beta Z_{T_m}^2 \leq \frac{1}{\beta^2}\sum_{e} \P(e \text{ revealed by time $T_m$})
  (1-e^{-\beta J(e)}) \\\leq \frac{1}{\beta} \sum_{e} \P(e \text{ revealed by time $T_m$}) J(e)
   \leq \frac{|J^*|}{\beta} \E\min\{|K|,m\},
   \label{eq:truncated_derivative_martingale2}
\end{multline}
where the last inequality follows by linearity of expectation. 
Substituting the estimate \eqref{eq:truncated_derivative_martingale2} into \eqref{eq:truncated_derivative_martingale1} concludes the proof.
\end{proof}

\subsection{Chemical distances}

Recall that the graph distance on a cluster in the percolation configuration is referred to as the \textbf{chemical distance}; we denote such distances by $d_\mathrm{chem}(\cdot,\cdot)$.
In this section we prove that the mean-field scaling of the second moment and susceptibility derivative
\[
  \E_\beta|K|^2 \asymp (\E_\beta|K|)^3 \qquad \text{ and } \qquad \frac{d}{d\beta} \E_\beta|K| \asymp (\E_\beta|K|)^2,
\]
which are equivalent by \cref{lem:second_moment_differential_inequality,lem:Durrett_Nguyen}, are also equivalent to the mean-field behaviour of \emph{chemical distances} between typical points in a large cluster
\[
  \E_\beta \sum_{x\in K} d_\mathrm{chem}(o,x) \asymp (\E_\beta|K|)^2.
\]
We begin by noting some simple relationships between these three quantities and the susceptibility, which provide one side of the equivalence claimed above.

\begin{lemma}
\label{lem:chemical_easy}
Let $G=(V,E,J)$ be a transitive weighted graph. The following inequalities hold for every $0\leq \beta<\beta_c$.
\begin{enumerate}
\item $\E_\beta \sum_{x\in K} d_\mathrm{chem}(o,x)\leq (\E_\beta|K|)^2$.
\item $\E_\beta|K|^2\leq \E_\beta|K|\E_\beta \sum_{x\in K} (1+d_\mathrm{chem}(o,x))$.
\item $\frac{d}{d\beta}\E_\beta|K| \leq \sup_e \frac{J(e)e^{-\beta J(e)}}{1-e^{-\beta J(e)}}\E_\beta \sum_{x\in K} d_\mathrm{chem}(o,x)$. 
\end{enumerate}
\end{lemma}

We will keep the proof of this lemma brief since it is not used in the proofs of our main theorems.

\begin{proof}[Proof of \cref{lem:chemical_easy}]
The first inequality follows immediately from the BK inequality since $d_\mathrm{chem}(o,x)$ is bounded by the number of points $y$ such that the disjoint occurence $\{0\leftrightarrow y\} \circ \{y\leftrightarrow x\}$ occurs. The second inequality follows by a small modification of the proof of the tree-graph inequality: Suppose that $o$ is connected to both $x$ and $y$, let $\gamma$ be a geodesic from $o$ to $x$ in this cluster, and let $\gamma'$ be any open simple path from $o$ to $y$. By considering the last place $\gamma'$ intersects $\gamma$, we deduce that there exists a point $w$ such that the events $\{w$ lies on a geodesic from $o$ to $x\}$ occurs disjointly from the event $\{w\leftrightarrow y\}$. Taking a union bound over $w$ and applying Reimer's inequality yields the claim. For the third inequality, we consider the quantity $d_\mathrm{piv}(o,x)$ defined to be the number of open pivotals (i.e., bridges) on the path between $o$ and $x$, which is trivially at most $d_\mathrm{chem}(o,x)$. We have by Russo's formula that
\begin{multline*}
  \frac{d}{d\beta}\E_\beta|K| = \E_\beta\left[ \sum_{x\in K} \sum_{e \text{ open pivotal}} J(e) \frac{e^{-\beta J(e)}}{1-e^{-\beta J(e)}}\right] \\\leq \sup_e \left[\frac{J(e)e^{-\beta J(e)}}{1-e^{-\beta J(e)}}\right]\E_\beta \sum_{x\in K} d_\mathrm{piv}(o,x) 
  \leq \sup_e \left[\frac{J(e)e^{-\beta J(e)}}{1-e^{-\beta J(e)}} \right]\E_\beta \sum_{x\in K} d_\mathrm{chem}(o,x)
\end{multline*}
as claimed.
\end{proof}

We next prove the following less obvious fact, which states that if $\E_{\beta} \sum_{x\in K} d_\mathrm{chem}(o,x)$ takes its maximum order of $(\E_\beta|K|)^2$ then $\frac{d}{d\beta}\E_{\beta}|K|$ also takes its maximum order of $(\E_\beta|K|)^2$. This lemma will be used to show that certain error terms are negligible during our computations of logarithmic corrections to scaling for long-range percolation with $d=3\alpha<6$ in \cref{sec:logarithmic_corrections_at_the_critical_dimension} (see in particular \cref{lem:spatial_geodesic,lem:triple_interaction_negligible1,lem:triple_interaction_negligible3}).

\begin{lemma}
\label{lem:chemical_distances_sublinear}
The implication
\[
\Biggl[\E_{\beta} \sum_{x\in K} d_\mathrm{chem}(o,x) \geq \eps (\E_\beta |K|)^2 \Biggr]\Rightarrow 
\Biggl[\frac{d}{d\beta}\E_\beta |K| \geq \frac{|J| e^{- 64 \beta |J|\eps^{-3}}}{1-e^{- 64 \beta |J|\eps^{-3}}} (\E_\beta |K|)^2 \Biggr]
\]
holds for percolation on any transitive weighted graph, any $\beta< \beta_c$, and any $\eps >0$.
\end{lemma}

Note that this bound is quantitatively much weaker than the other estimates of this section. The bound has not been optimized.

\begin{proof}[Proof of \cref{lem:chemical_distances_sublinear}]
We first observe that the proof of the first inequality of \cref{lem:chemical_easy} also yields that
\begin{equation}
  \E_\beta \sum_{x\in K} d_\mathrm{chem}(o,x)^2 \leq 2 (\E_\beta|K|)^{3}.
   \label{eq:chemical_second_moment}
\end{equation}
Indeed, $d_\mathrm{chem}(o,x)^2$ is bounded by twice the number of ordered pairs of (possibly identical) vertices $(a,b)$ such that the disjoint occurence $\{o\leftrightarrow a\} \circ \{a\leftrightarrow b\}\circ\{b\leftrightarrow x\}$ occurs, where the factor of $2$ accounts for the two different possible orderings of the points, so that
the claim follows easily from the BK inequality as before.

\medskip

Now suppose that $\E_{\beta} \sum_{x\in K} d_\mathrm{chem}(o,x) \geq \eps (\E_\beta |K|)^2$ for some $\eps>0$. We can rewrite this inequality together with \eqref{eq:chemical_second_moment} as
\[
  \hat \E_\beta d_\mathrm{chem}(o,X) \geq \eps \E_\beta |K| \qquad \text{ and } \qquad \hat \E_\beta d_\mathrm{chem}(o,X)^2 \leq 2(\E_\beta |K|)^2,
\]
where $\hat \E_\beta$ denotes the size-biased measure and $X$ denotes a uniform random element of $K$.
%
Applying the Paley-Zygmund inequality, we deduce that 
\begin{equation}
\label{eq:chemical_Payley_Zygmund}
  \hat\P_\beta\left(d_\mathrm{chem}(o,X) \geq \frac{\eps}{2}\E_\beta|K|\right) \geq \frac{(\hat \E_\beta d_\mathrm{chem}(o,X))^2}{4 \hat \E_\beta d_\mathrm{chem}(o,X)^2} \geq \frac{\eps^2}{8}.
\end{equation}
Let $E^\rightarrow$ denote the set of oriented edges of $G$ (where we think of each edge as corresponding to a pair of oriented edges, one with each orientation and write $e^-$ and $e^+$ for the start and endpoint of the oriented edge $e$) and for each vertex $x$ consider the random variable $W(x)$ defined by
\[
  W(x) = \sum_{e \in E^\rightarrow} J(e) \mathbbm{1}(\{0\leftrightarrow e^- \text{ off $e$\}}\circ \{e^+\leftrightarrow x \text{ off $e$}\}).
\]
which
counts the total weight of edges that would belong to a simple open path from $0$ to $x$ if they were open. It follows from the BK inequality as in the derivation of \eqref{eq:derivative_upper_general} that
\[
  \E_\beta \sum_{x\in V} W(x) \leq |J| (\E_\beta|K|)^2 \qquad \text{ and hence that } \qquad
  \hat\E_\beta W(X) \leq |J| \E_\beta|K|.
\]
Using Markov's inequality, it follows from this and \eqref{eq:chemical_Payley_Zygmund}
%
 that
\[
  \hat\P_\beta\left(d_\mathrm{chem}(o,X) \geq \frac{\eps}{2}\E_\beta|K| \text{ and } W(X) \leq 16 |J|\eps^{-2}\E_\beta|K| \right) \geq \frac{\eps^2}{16}.
\]
Let $\mathscr{A}=\mathscr{A}(\eps)$ be the event whose probability we have just estimated. For each $k\geq 0$ and $x\in V$, let $\mathscr{W}_k(x)$ be the set of oriented edges $e$ with $e^-$ having chemical distance exactly $k-1$ from $o$, $e^+$ have chemical distance at least $k$ from $o$ (this distance being infinite if $e^+$ is not connected to $o$), and with $e^+$ connected to $x$ by a path that does not visit the chemical ball of radius $k-1$ around $o$. We also define $W_k(x)$ to be the total weight of the edges of $\mathscr{W}_k(x)$, so that $\sum_k W_k(x) \leq W(x)$ for every $x\in  V$.
%
%
Let $R=\lceil\frac{\eps}{2}\E_\beta|K|\rceil$.
On the event $\mathscr{A}$ we must have that 
that 
\[\sum_{k=1}^R W_k(X) \leq W(X) \leq 16|J| \eps^{-2} \E_\beta|K| \leq 32 |J| \eps^{-3} R\] and hence that there exists at least $R/2$ values of $k$ between $1$ and $R$ such that $W_k(X) \leq 64|J| \eps^{-3}$. This implies in particular that
\begin{multline}
 \sum_{k=1}^R \hat\P_\beta\left(d_\mathrm{chem}(0,X) \geq k \text{ and } W_k(X) \leq 64 |J|\eps^{-3} \right) \\\geq 
 \sum_{k=1}^R \hat\P_\beta\left(d_\mathrm{chem}(0,X) \geq R \text{ and } W_k(X) \leq 64 |J|\eps^{-3} \right) \geq \frac{\eps^2}{32} R,
\end{multline}
which we can rewrite in terms of the original measure as
\[
 \sum_{k=1}^R \E_\beta\left[\#\left\{x\in K:d_\mathrm{chem}(0,x) \geq k \text{ and } W_k(x) \leq 64 |J|\eps^{-3} \right\}\right] \geq \frac{\eps^2}{32} R \cdot \E_\beta|K|.
\]
Now, for each $x\in V$ and $k\geq 1$ we define $P_k(x)=J(e)\frac{e^{-\beta J(e)}}{1-e^{-\beta J(e)}}$ if $d_\mathrm{chem}(o,x)\geq k$ and $e$ is the unique \emph{open} edge in $\mathscr{W}_k(x)$ and define $P_k(x)=0$ if $d_\mathrm{chem}(o,x)\leq k-1$ or $\mathscr{W}_k(x)$ contains either no open edges or strictly more than one open edge. For each $k\geq 1$, let $\mathcal{F}_k(x)$ be the sigma-algebra generated by first exploring the chemical ball of radius $k-1$ around $o$ and then exploring the open cluster that can be reached from $x$ without using any edges contained in or incident to this ball. On the event that $x$ does not belong to the chemical ball of radius $k-1$, the set $\mathscr{W}_k(x)$ is precisely the set of edges that have one endpoint in the revealed chemical ball and the other endpoint in the cluster of $x$ outside of this chemical ball. As such, we can compute that
\[
\E[P_k(x) \mid \mathcal{F}_k(x)] = \sum_{e\in \mathscr{W}_k(x)} J(e) \frac{e^{-\beta J(e)}}{1-e^{-\beta J(e)}} \cdot \frac{1-e^{-\beta J(e)}}{e^{-\beta J(e)}} e^{-\beta W_k(x)} =  W_k(x) e^{-\beta W_k(x)}
\]
and
\[
\P( x\in K \mid \mathcal{F}_k(x)) = 1- e^{-\beta W_k(x)}
\]
on the event that $x$ does not belong to the chemical ball of radius $k-1$. Using the expression for the derivative in terms of open pivotals, we obtain that
\begin{align*}
 \frac{d}{d\beta}\E_\beta|K| &\geq \sum_{k=1}^\infty \E_\beta \left[\sum_{x\in K} P_k(x) \mathbbm{1}(d_\mathrm{chem}(o,x)\geq k)\right]
 \\&\geq \sum_{k=1}^R \E_\beta\sum_{x\in K} P_k(x) \mathbbm{1}\left(d_\mathrm{chem}(0,x) \geq k \text{ and } W_k(x) \leq 64 |J|\eps^{-3} \right)
 \\&\geq
 \frac{64 |J|\eps^{-3} e^{- 64 \beta |J|\eps^{-3}}}{1-e^{- 64 \beta |J|\eps^{-3}}} \sum_{k=1}^R \E_\beta\#\left\{x\in K:d_\mathrm{chem}(0,x) \geq k \text{ and } W_k(x) \leq 64 |J|\eps^{-3} \right\}  \\&\geq 
  \frac{|J| e^{- 64 \beta |J|\eps^{-3}}}{1-e^{- 64 \beta |J|\eps^{-3}}} (\E_\beta|K|)^2
\end{align*}
as claimed. \qedhere

\end{proof}


\section{The hydrodynamic condition I: Fictitious mean-field asymptotics}
\label{sec:the_hydrodynamic_condition_holds_critical_dim}

This is the first of two sections dedicated to proving \cref{thm:critical_dim_hydro}.
We now give an overview of how this proof will proceed. 
Rather than prove \cref{thm:critical_dim_hydro} directly, we will instead study a slightly modified problem in which we vary $\beta$ simultaneously along with $r$. Specifically, for each $0\leq \lambda \leq 1$ and $r\geq 1$ we define
\[
\beta(r,\lambda):=(1-\lambda r^{-\alpha})\beta_c \qquad \text{ and } \qquad \tilde \P_{r,\lambda}:=\P_{\beta(r,\lambda),r},
\]
where we think of $\lambda$ as a small constant. (Note that we also allow $\lambda=0$ in this definition, which recovers our usual cut-off critical measures $\tilde \P_{r,0}=\P_{\beta_c,r}$.) As will become apparent later, the particular scaling $\beta(r,\lambda)=(1-\lambda r^{-\alpha})\beta_c$ is chosen so that the two contributions to the total derivative of the degree of the origin (from varying $r$ and varying $\beta$) are of the same order when $\lambda$ is a positive constant, with the contribution from varying $r$ larger than the contribution from varying $\beta$ by a large constant when $\lambda$ is a small positive constant.
Given $\lambda\geq 0$, we say that the \textbf{$\lambda$-modified hydrodynamic condition} holds if 
\[
\tilde M_{r,\lambda} :=M_{\beta(r,\lambda),r} = o(r^{(d+\alpha)/2})  \qquad \text{as $r\to \infty$}.
\]
 We say that the \textbf{modified hydrodynamic condition} holds if the $\lambda$-modified hydrodynamic condition holds for every $0<\lambda\leq 1$; since $\tilde M_{r,\lambda}$ is decreasing in $\lambda$ this is equivalent to the $\lambda$-modified triangle condition holding for a set of $\lambda$ including arbitrarily small positive elements. (Note also that the usual hyrdodynamic condition is equivalent to the $0$-modified hydrodynamic condition.) We will spend the majority of this section and \cref{sec:the_hydrodynamic_condition_ii_reaching_a_contradiction} proving the following proposition, using it to deduce \cref{thm:critical_dim_hydro} at the very end of \cref{sec:the_hydrodynamic_condition_ii_reaching_a_contradiction}.

\begin{prop}
\label{prop:critical_dim_modified_hydro}
If $d=3\alpha$ then the modified hydrodynamic condition holds.
\end{prop}

The rest of the section is outlined as follows. For the purpose of this overview we will state all the intermediate results at the critical dimension $d=3\alpha$ only, although many of them hold more generally as will be made clear in the body of the section.
\begin{enumerate}
    \item In \cref{subsec:critical_dim_initial_regularity} we prove that if $\lambda>0$ then the susceptibility $\tilde \E_{r,\lambda}|K|$ is of order at most $\lambda^{-1}r^\alpha$ on an unbounded set of $r$. It is in this proof that we benefit most from working with the measures $\tilde \E_{r,\lambda}$ rather than $\E_{\beta_c,r}$. To prove this estimate, we first note that the results of \cite{hutchcroft2020power,1901.10363} imply that the moments $\E_{\beta_c,r}|K|^p$ satisfy polynomial upper bounds of the form $r^{O(1)}$, so that a doubling estimate of the form $\tilde \E_{4r,\lambda}|K|^p \leq C_p \E_{r,\lambda}|K|^p$ must hold on a positive density set of scales (which may depend on $\lambda$). We then argue that on any scale in which an appropriate collection of four functions are all doubling, we must either have that $\tilde \E_{r,\lambda}|K|^2 \succeq (\tilde \E_{r,\lambda}|K|)^3$ or that $\tilde \E_{r,\lambda}|K|^3 \preceq r^{4\alpha} \tilde \E_{r,\lambda}|K|$.  In either case we can prove a differential inequality which yields an upper bound on $\tilde \E_{r,\lambda}|K|$ on that scale. (In the case that the second moment is large, the bound we get is on the $\beta$-derivative rather than the $r$-derivative of $\E_{\beta,r}$, which is why we need to work with the measure $\tilde \E_{r,\lambda}$ whose derivative includes contributions from both the $r$ and $\beta$ derivatives of $\E_{\beta,r}$.)
    \item In \cref{sub:fictitious_beta_derivative_and_susceptibility_estimates_on_all_scales}
    we improve the conclusions of \cref{subsec:critical_dim_initial_regularity} under the fictitious assumption that the modified hydrodynamic condition does \emph{not} hold, showing that the susceptibility $\tilde \E_{r,\lambda}|K|$, second moment $\tilde \E_{r,\lambda}|K|^2$, and $\beta$ derivative  of the susceptibility at $(\beta(r,\lambda),r)$, denoted $\tilde D_{r,\lambda}$, are of order $r^\alpha$, $r^{3\alpha}$, and $r^{2\alpha}$ respectively on \emph{every} scale when $\lambda$ is a sufficiently small positive constant. To do this, we first use the results of \cref{sub:relations_between_the_beta_derivative_and_the_second_moment,subsec:critical_dim_initial_regularity} together with the assumption that the modified hydrodynamic condition does not hold to find an unbounded set of scales on which the $\beta$ derivative is of the same order as the square of the susceptibility. We then
     use a submultiplicativity-type property of the $\beta$-derivative to deduce a lower bound of order $r^{2\alpha}$ on the $\beta$ derivative at \emph{all} scales. Using the relations between the second moment and the $\beta$-derivative established in \cref{sub:relations_between_the_beta_derivative_and_the_second_moment}, this lets us return to the proofs of \cref{subsec:critical_dim_initial_regularity} and improve the conclusions of the argument to give bounds at every scale as desired. The bounds we prove in this section are also strong enough to yield the ($\lambda=0$) bound $\E_{\beta_c,r}\min\{|K|,r^{2\alpha}\} \preceq r^\alpha$ under the fictitious assumption that the modified hydrodynamic condition does not hold.
    \item In \cref{sub:fictitious_negligibility_of_mesoscopic_clusters} we adapt the methods of \cref{II-sec:negligibility_of_mesoscopic_clusters_and_hyperscaling} (which were used to prove the validity of hyperscaling theory in low dimensions) to prove an analogous ``negligibility of mesoscopic clusters'' theorem under the fictirious assumption that the modified hydrodynamic condition does not hold. Concretely, we prove that clusters of size much smaller than $r^{2\alpha}$ do not contribute significantly to the susceptibility $\E_{\beta_c,r}|K|$. Using the results of \cref{I-sec:volume_tail}, we deduce (under our fictitious assumption) that the volume tail has \emph{exact mean-field asymptotics}
    \[
      \P_{\beta_c}(|K|\geq n) \asymp n^{-1/2},
    \]
    with no slowly-varying corrections to scaling.
\end{enumerate}

We summarise the main conclusions of this section in the following proposition.

\begin{prop}[Fictitious mean-field critical behaviour]
Suppose that $d=3\alpha$ and that the modified hydrodynamic condition does \emph{not} hold. There exists $\lambda_0>0$ such that the following hold. 
\begin{enumerate}
\item $r^\alpha \preceq \tilde \E_{r,\lambda}|K| \preceq \lambda^{-1} r^\alpha$ for every $r\geq 1$ and $0<\lambda \leq \lambda_0$.
\item $\lambda r^{3\alpha} \preceq \tilde \E_{r,\lambda}|K|^2 \preceq \lambda^{-3} r^{3\alpha}$ for every $r\geq 1$ and $0<\lambda \leq \lambda_0$.
\item $r^{2\alpha} \preceq \tilde D_{r,\lambda} \preceq \lambda^{-2} r^{2\alpha}$ for every $r\geq 1$ and $0<\lambda \leq \lambda_0$.
\item $\P_{\beta_c}(|K|\geq n) \asymp n^{-1/2}$ for every $n\geq 1$.
\item $\E_{\beta_c}\min\{|K|,r^{2\alpha}\} \asymp r^\alpha$.
\end{enumerate}
Moreover, for every $\eps>0$ there exists $\delta>0$ such that $\E_{\beta_c}\min\{|K|,\delta r^{2\alpha}\} \leq \eps r^\alpha$ for every $r\geq 1$.
\end{prop}

It is presumably possible to deduce stronger versions of mean-field critical behaviour from these estimates with a little further work. These estimates are strong enough to reach our desired contradiction in \cref{sec:the_hydrodynamic_condition_ii_reaching_a_contradiction}, so there is no need to refine them further.

\begin{remark}
We conjecture that the $r$ and $\beta$ derivatives can always be compared via
\[
  \frac{\partial}{\partial r} \E_{\beta,r}|K| \succeq r^{-\alpha-1}\frac{\partial}{\partial \beta} \E_{\beta,r}|K|
\]
for $r\geq 1$ and $\beta_c/2\leq \beta \leq \beta_c$. If true, this inequality would allow us to significantly simplify this section by working with the measures $\E_{\beta_c,r}$ rather than $\tilde \E_{r,\lambda}$ throughout.
\end{remark}

\subsection{Initial regularity estimates}
\label{subsec:critical_dim_initial_regularity}

The goal of this section is to prove the following lemma, which guarantees the existence of infinitely many good scales at which the first moment $\tilde \E_{r,\lambda}|K|$ is not much larger than it should be.
It is in the proof of this lemma that we will benefit from working with the measures $\tilde\E_{r,\lambda}:=\E_{\beta(r,\lambda),r}$ rather than our usual cut-off measure $\E_{\beta_c,r}$.

\begin{lemma}
\label{lem:good_scales}
Suppose that $d=3\alpha$. There exists a constant $C$ such that for each $0<\lambda\leq 1$ the estimate
$\tilde \E_{r,\lambda}|K|\leq C\lambda^{-1} r^\alpha$ holds
for an unbounded set of $r\in [1,\infty)$.
\end{lemma}

The proof of this lemma will establish a number of other conditions that can be imposed on this good set of scales, which will be useful when strengthening the lemma in \cref{sub:fictitious_beta_derivative_and_susceptibility_estimates_on_all_scales}. Although we are interested primarily in the case $d=3\alpha$, we will prove several intermediate statements under weaker assumptions on $\alpha$ to clarify the structure of the proofs.

\medskip

We begin by making note of the following consequence of our earlier works \cite{hutchcroft2020power,hutchcroft2022sharp}, which give much looser control of the same quantity.

\begin{lemma}[Polynomial growth]
\label{lem:moment_poly_upper_bound}
If $\alpha<d$ then for each integer $p\geq 1$ there exists a constant $C_p$ such that
\[
\E_{\beta_c,r}|K|^p \leq C_p r^{C_p}
\]
for every $r\geq 1$.
\end{lemma}

\begin{proof}[Proof of \cref{lem:moment_poly_upper_bound}]
It suffices to prove the claim in the case $p=1$, the corresponding bound on higher moments following from the tree-graph inequality \eqref{eq:tree_graph_pth_moment}. To this end, first note that there exists a positive constant $c$ such that
\begin{multline*}
\beta_c J_r(x,y) = \beta_c \mathbbm{1}(\|x-y\|\leq r) \left(J(\|x-y\|)-\int_r^\infty |J'(s)| \dif s\right) 
\\\leq \beta_c J(x,y) \left(1-\frac{1}{\max_{x\neq y} J(\|x-y\|)}\int_r^\infty |J'(s)| \dif s\right) \leq (1-c r^{-c}) \beta_c J(x,y)
\end{multline*}
for every pair of distinct $x,y\in \Z^d$ and $r\geq 1$, so that
\[
\E_{\beta_c,r}|K| \leq \E_{(1-cr^{-c})\beta_c}|K|
\]
for every $r\geq 1$.
Thus, to prove the claim it suffices to prove that there exists a constant $C$ such that 
$\E_\beta |K| \leq C (\beta_c-\beta)^{-C}$ for every $\beta<\beta_c$. This bound follows from the main result of \cite{hutchcroft2020power} (which provides a power-law bound on the volume tail at criticality) together with \cite[Theorem 1.1]{1901.10363} (which transforms critical volume-tail bounds into bounds on the slightly subcritical susceptibility).
\end{proof}

Our next aim is to show that if $\tilde \E_{r,\lambda}|K|$ is much larger than $\lambda^{-1} r^\alpha$ at all large scales $r$ then the derivative $\tilde \E_{r,\lambda}|K|$ must be too large too often, contradicting the polynomial growth bound of \cref{lem:moment_poly_upper_bound}.
Given a subset $A$ of $\N$, we define the \textbf{lower density} of $A$ to be $\liminf_{N\to\infty} \frac{1}{N} |A \cap \{1,\ldots,N\}|$. Note that if $f:(0,\infty)\to (0,\infty)$ is an increasing function that has upper polynomial growth in the sense that $f(r) \leq Cr^C$ for some constant $C$ and all $r\geq 1$ then
\[
A(\lambda):=\{k \in \N : f(4r) \leq \lambda f(r) \text{ for all $r\in [2^k,2^{k+2}]$}\}
\]
has lower density at least $1- \frac{4C \log 2}{\log \lambda}$ for each $\lambda \geq 1$. Indeed, if $k\in \N$ does \emph{not} belong to this set then we must have that $f(2^{k+4}) \geq \lambda f(2^k)$, so that
\begin{multline*}
f(2^{4n}) f(2^{4n+1})f(2^{4n+2})
f(2^{4n+3}) =  f(4) f(5) f(6) f(7) \prod_{k=1}^{n-1} \frac{f(2^{4k+4})}{f(2^{4k})}
\frac{f(2^{4k+5})}{f(2^{4k+1})}
\frac{f(2^{4k+6})}{f(2^{4k+2})}
\frac{f(2^{4k+7})}{f(2^{4k+3})}
\\ \geq f(4) f(5)f(6)f(7) \lambda^{|\{4,\ldots,4n-1\}\setminus A(\lambda)|},
\end{multline*}
and the claim follows since the left hand side is of order at most $2^{16nC}$ as $n\to\infty$. 
This is the main way in which we will apply \cref{lem:moment_poly_upper_bound}. (If we knew \emph{a priori} that $\tilde \E_{r,\lambda}|K|$ satisfied a doubling inequality of the form  $\tilde \E_{2r,\lambda}|K| \leq C  \tilde \E_{r,\lambda}|K|$ on \emph{all} scales then this would greatly simplify the remainder of our analysis!)

\medskip

We next make note of the following simple lower bound on the $r$ derivative of $\E_{\beta,r} |K|$ in terms of the truncated moment $\E_{\beta,r} \min\{|K|,m\}$.

\begin{lemma}
\label{lem:truncated_expectation_diff_ineq_at_doubling_scales}
If $\alpha<d$ then there exists a positive constant $C$ such that
\begin{multline*}
\frac{\partial}{\partial r} \E_{\beta,r} |K| 
\geq \beta |J'(r)|  \E_{\beta,r}\left[ \sum_{y\in B_r} |K| \min\{|K_y|,m\} \mathbbm{1}(0 \nleftrightarrow y) \right]\\
\geq \beta |B_r||J'(r)| \E_{\beta,r} |K| \left(\E \min\{|K|,m\} - 
C r^{-(d-\alpha)/2}m \right)
\end{multline*}
for every $0<\beta\leq \beta_c$ and $r,m\geq 1$.
\end{lemma}

We will apply this lemma primarily with $d=3\alpha$ and $m=\delta r^{2\alpha}$, in which case the inequality reads
\begin{equation*}
\frac{\partial}{\partial r} \E_{\beta,r} |K| 
\geq \beta |B_r||J'(r)| \E_{\beta,r} |K| \left(\E \min\{|K|,\delta r^{2\alpha}\} - 
C \delta r^\alpha \right)
\end{equation*}
for every $0<\beta\leq \beta_c$, $r\geq 1$ and $\delta>0$; this choice of threshold in the truncated moment is chosen to force the negative term in the derivative associated to the event $\{0\leftrightarrow y\}$ to be small, as we will justify after the proof of \cref{lem:truncated_expectation_diff_ineq_at_doubling_scales}.

\medskip

The proof of this lemma will apply \cref{I-cor:Gladkov_moments}, a corollary of the Gladkov inequality \cite{gladkov2024percolation}, applying to percolation on any weighted graph $G=(V,E,J)$ and stating that
\begin{equation}
  \E_\beta\left[|K_x| |K_x \cap W|\right] \leq 2^{3/2} \E_\beta |K| \sqrt{ |W| \E_\beta |K_x \cap W|}
  \label{eq:Gladkov_moments_restate}
\end{equation}
for every $\beta \geq 0$, $x\in V$, and $W \subseteq V$ finite.

\begin{proof}[Proof of \cref{lem:truncated_expectation_diff_ineq_at_doubling_scales}]
Fix $r,m\geq 1$ and $0<\beta\leq \beta_c$. The first inequality
\begin{multline*}
\frac{\partial}{\partial r} \E_{\beta,r} |K| = \beta |J'(r)|  \E_{\beta,r}\left[ \sum_{y\in B_r} |K||K_y| \mathbbm{1}(0 \nleftrightarrow y) \right] \\
\geq \beta |J'(r)|  \E_{\beta,r}\left[ \sum_{y\in B_r} |K| \min\{|K_y|,m\} \mathbbm{1}(0 \nleftrightarrow y) \right]
\end{multline*}
holds trivially from the definitions. We can use the Harris-FKG inequality to bound the expression appearing on the last line by
\begin{align*}
\frac{\partial}{\partial r} \E_{\beta,r} |K| 
&\geq \beta |J'(r)|  \E_{\beta,r}\left[ \sum_{y\in B_r} |K| \min\{|K_y|,m\}  \right]
-\beta |J'(r)|  \E_{\beta,r}\left[ \sum_{y\in B_r} |K| \min\{|K_y|,m\} \mathbbm{1}(0\leftrightarrow y) \right]
\\
&\geq \beta |J'(r)| |B_r| \E_{\beta,r} |K| \E_{\beta,r} \min\{|K|,m\} - 
\beta |J'(r)| \E_{\beta,r} \left[|K| \min\{|K|,m\} |K\cap B_r|\right]
\\
&\geq \beta |J'(r)| |B_r| \E_{\beta,r} |K| \E_{\beta,r} \min\{|K|,m\} - 
\beta |J'(r)| m \E_{\beta,r} \left[|K| |K\cap B_r|\right].
\end{align*}
 It follows from \cref{I-thm:two_point_spatial_average_upper} and the Gladkov inequality applied as in \cref{I-cor:Gladkov_moments} (restated here as \eqref{eq:Gladkov_moments_restate}) that 
\[
\E_{\beta,r} \left[|K| |K\cap B_r|\right] \preceq \E_{\beta,r}|K| \sqrt{|B_r|\E_{\beta_c} |K\cap B_r| }\preceq r^{(d+\alpha)/2} \E_{\beta,r}|K|
\]
so that there exists a positive constant $C$ such that
\begin{align*}
\frac{\partial}{\partial r} \E_{\beta,r} |K| 
&\geq \beta |B_r||J'(r)| \E_{\beta,r} |K| \left(\E \min\{|K|,m\} - 
C  |B_r|^{-1}r^{(d+\alpha)/2} m \right).
\end{align*}
This trivially implies the claim since $|B_r|\succeq r^d$ for every $r\geq 1$.
\end{proof}

In order to apply \cref{lem:truncated_expectation_diff_ineq_at_doubling_scales}, we will need to have conditions under which the $C\delta r^\alpha$ term can safely be absorbed into the $\E\min\{|K|,\delta r^{(d+\alpha)/2}\}$ term. This will be done using the following proposition, which is a finitary, off-critical version of the mean-field lower bound on the critical magnetization of Aizenman and Barsky \cite{aizenman1987sharpness}.

\begin{prop}
\label{prop:finitary_mean_field}
Consider percolation on any transitive weighted graph. The bound
\[
\E_\beta \min\{|K|,m\} \geq \frac{1}{2} \min\left\{ \E_\beta|K|, \sqrt{\frac{m}{2}}\right\}
\]
holds for every $\beta\geq 0$ and $m\geq 1$.
\end{prop}

While it is possible to prove an inequality of this form using the same method as \cite{aizenman1987sharpness} (which is based on the analysis of certain partial differential inequalities), we will instead give a shorter proof based on the tree-graph inequality. (The fact that the mean-field lower bound on the magnetization is a simple consequence of the tree-graph inequality appears to be a new observation.)

\begin{proof}[Proof of \cref{prop:finitary_mean_field}]
Write $\hat \E_\beta$ for the size-biased measure, so that $\hat \E_\beta|K| = \E_\beta|K|^2/\E_\beta|K| \leq (\E_\beta|K|)^2$ by the tree-graph inequality. We have by Markov's inequality that
\[
  \E_\beta \left[|K| \mathbbm{1}(|K| \geq 2 \hat \E_\beta|K| )\right]=\E_\beta|K| \hat\P_\beta (|K| \geq 2 \hat \E_\beta|K| ) \leq \frac{1}{2}\E_\beta|K|
\]
and hence that
\[
  \E_\beta \left[\min\{|K|, 2 (\E_\beta|K|)^2 \}\right] \geq \E_\beta \left[\min\{|K|, 2 \hat \E_\beta|K| \}\right] \geq \E_\beta \left[|K| \mathbbm{1}(|K| \leq 2 \hat \E_\beta|K| )\right] \geq \frac{1}{2}\E_\beta |K|
\]
for every $\beta\geq 0$, where we used the tree-graph inequality in the first inequality. It follows by monotonicity that
\[
  \E_{\beta} \left[\min\{|K|, 2 (\E_{\beta'}|K|)^2 \}\right] \geq \frac{1}{2}\E_{\beta'} |K|
\]
for every $0\leq \beta'\leq \beta$. Applying the intermediate value theorem, we deduce that for each $\beta\geq 0$ and  $2 \leq m \leq 2(\E_\beta|K|)^2$ there exists $0\leq \beta' \leq \beta$ such that $2 (\E_{\beta'}|K|)^2=m$, so that
\[
  \E_{\beta} \left[\min\{|K|, m \}\right] \geq \sqrt{\frac{m}{8}}
\]
for every $2 \leq m \leq 2(\E_\beta|K|)^2$ as required. The condition that $m\geq 2$ was included to make sure that $2(\E_0|K|)^2=2\leq m$ so that our application of the intermediate value theorem was valid, but the same inequality also trivially holds when $m=1$.
\end{proof}

We next make note of the following simple estimate, which generalizes Corollary \ref{I-cor:mean_lower_bound} and holds without any assumptions on $d$ or $\alpha$. Recall that we write $|J|=\sum_{x\in \Z^d}J(0,x)$.

\begin{lemma}
The estimate
\label{lem:tilde_susceptibility_lower}
\[\liminf_{r\to \infty} r^{-\alpha}\tilde \E_{r,\lambda}|K| \geq \frac{\alpha}{\beta_c(1+\alpha \lambda |J|)} \]
holds for all $0\leq \lambda \leq 1$.
\end{lemma}

The proof will make use of \cref{I-lem:BK_disjoint_clusters_covariance}, a simple consequence of the BK inequality stating that if $G=(V,E,J)$ is a weighted graph then
the inequality
\begin{equation}
\label{eq:BK_disjoint_clusters_covariance}
\E_\beta\left[F(K_x)G(K_y)\mathbbm{1}(x\nleftrightarrow y)\right] \leq \E_\beta \left[F(K_x) \E_\beta G(K_y)\right]
\end{equation}
holds for every $x,y\in V$ and every pair of increasing non-negative functions $F$ and $G$ and the inequalities 
\begin{equation}
\label{eq:BK_disjoint_clusters_covariance2}
\E_\beta\left[F(K_x)G(K_y)\mathbbm{1}(x\nleftrightarrow y)\right] \geq \E_\beta \left[F(K_x)\right] \E_\beta \left[G(K_y)\right] - \E_\beta\left[ F(K_x)G(K_y) \mathbbm{1}(x\leftrightarrow y)\right]
\end{equation}
and
\begin{equation}
\label{eq:BK_disjoint_clusters_covariance3}
  0\leq \E_\beta [F(K_x)G(K_y)] -  \E_\beta \left[F(K_x)\right] \E_\beta \left[G(K_y)\right] \leq \E_\beta\left[F(K_x)G(K_y)\mathbbm{1}(x\leftrightarrow y)\right]
\end{equation}
hold for every $\beta\geq0$, $x,y\in V$, and pair of increasing non-negative functions $F$ and $G$ such that $\E_\beta F(K_x)$ and $\E_\beta G(K_x)$ are finite.

\begin{proof}[Proof of \cref{lem:tilde_susceptibility_lower}]
We have by our usual argument using Russo's formula and \cref{I-lem:BK_disjoint_clusters_covariance} (restated here as \eqref{eq:BK_disjoint_clusters_covariance}) that
\[
 \frac{\partial}{\partial r}\E_{\beta,r}|K| \leq \beta |J'(r)| |B_r| (\E_{\beta,r}|K|)^2  \qquad \text{ and } \qquad \frac{\partial}{\partial \beta}\E_{\beta,r}|K| \leq |J| (\E_{\beta,r}|K|)^2
\]
for every $\beta,r\geq 0$. It follows by the chain rule that
\begin{equation}
\label{eq:lambda_susceptibility_derivative_upper}
  \frac{d}{dr} \tilde \E_{r,\lambda}|K| \leq \beta_c (|J'(r)| |B_r| + \alpha \lambda r^{-\alpha-1}|J|)(\E_{\beta,r}|K|)^2,
\end{equation}
and since $\lim_{r\to \infty}\tilde \E_{r,\lambda}|K|=\E_{\beta_c}|K|=\infty$, it follows that
\[
\frac{1}{\tilde \E_{r,\lambda}|K|} = \int_r^\infty \frac{\frac{d}{ds}\tilde \E_{s,\lambda}|K|}{(\tilde \E_{s,\lambda}|K|)^2} \dif s \leq \int_r^\infty \beta_c (|J'(s)| |B_s| + \alpha \lambda s^{-\alpha-1}|J|)\dif s \sim \frac{\beta_c(1+ \alpha \lambda |J|)}{\alpha}r^{-\alpha}
\]
as claimed.
\end{proof}

We next apply \cref{lem:truncated_expectation_diff_ineq_at_doubling_scales,prop:finitary_mean_field,lem:tilde_susceptibility_lower} to prove the following \emph{reverse doubling} property of the map $r\mapsto \tilde \E_{r,\lambda}|K|$.

\begin{lemma}
\label{lem:truncated_expectation_diff_ineq_at_doubling_scales2}
If $\alpha\leq d/3$ then there exist positive constants $c_1$ and $c_2$ such that
\begin{align*}
\frac{\partial}{\partial r} \log \tilde \E_{r,\lambda} |K| 
&\geq c_1 r^{-\alpha-1}  \tilde \E_{r,\lambda} \min\{|K|, c_1 r^{(d+\alpha)/2}\} \geq c_2 r^{-1}
\end{align*}
for every $r\geq 1$ and $0\leq \lambda \leq 1$. Consequently, 
there exists a positive constant $c_3$ such that
\[
\tilde \E_{2r,\lambda}|K| \geq (1+c_3) \tilde \E_{r,\lambda}|K|
\]
for every $0\leq \lambda \leq 1$ and $r\geq 1$.
\end{lemma}

\begin{proof}[Proof of \cref{lem:truncated_expectation_diff_ineq_at_doubling_scales2}]
The assumption that $\alpha\leq d/3$ implies that $(d+\alpha)/2\geq 2\alpha$.
We have by \cref{lem:tilde_susceptibility_lower} that there exists a positive constant $c_1$ such that
$\tilde \E_{r,\lambda} |K| \geq c_1 r^{\alpha}$
for every $r\geq 1$ and $0\leq \lambda \leq 1$. Thus, it follows from \cref{prop:finitary_mean_field} that there exists a positive constant $c_2$ such that
\[
\tilde \E_{r,\lambda} \min\{|K|, \delta r^{(d+\alpha)/2}\} \geq \tilde \E_{r,\lambda} \min\{|K|, \delta r^{2\alpha}\} \geq c_2 \delta^{1/2} r^\alpha
\]
for every $0\leq \lambda \leq 1$, $0<\delta \leq 1$, and $r\geq 1$ such that $\delta r^{2\alpha} \geq 1$. Letting $C_1$ be the constant from \cref{lem:truncated_expectation_diff_ineq_at_doubling_scales}, it follows that if $\delta$ is taken to be a maximal such that
$c_2\delta^{1/2}\geq 2C_1 \delta$ then 
\begin{align*}
\frac{\partial}{\partial r} \tilde\E_{r,\lambda} |K| 
&\geq \frac{1}{2} \beta_c |B_r||J'(r)| \E_{\beta,r} |K| \left(c_2 \delta^{1/2} r^\alpha - 
C_1 \delta r^{-(d-\alpha)/2}r^{(d+\alpha)/2} \right)
\\
&\succeq  r^{-\alpha-1} \tilde \E_{r,\lambda} |K| 
\tilde \E_{r,\lambda} \min\{|K|, \delta r^{(d+\alpha)/2}\} \succeq r^{-1} \tilde \E_{r,\lambda}|K|
\end{align*} 
for every $0\leq \lambda \leq 1$ and every $r\geq 1$ sufficiently large that $\delta r^{(d+\alpha)/2} \geq 1$ and $\beta(r,\lambda)\geq \beta_c/2$. (Here we used the fact that $\E_{\beta,r}|K|$ is increasing in $\beta$ to bound the $r$ derivative of $\tilde \E_{r,\lambda}|K|$ from below by the $r$ derivative of $\E_{\beta,r}|K|$ at the relevant value of $\beta$.) This inequality can be made to hold for all $r\geq 1$ by decreasing the implicit constant if necessary. The claimed reverse doubling property of $\tilde \E_{r,\lambda}|K|$ follows by integrating this logarithmic derivative estimate between $r$ and $2r$.
\end{proof}

The first part of \cref{lem:truncated_expectation_diff_ineq_at_doubling_scales2} also has the following immediate consequence, which gives us control of the truncated susceptibility on every scale on which $\tilde \E_{r,\lambda}|K|$ is doubling.

\begin{corollary}
\label{cor:truncated_expectation_bound_at_doubling_scales}
If $\alpha<d$ then there exist positive constants $c$ and $C$ such that
\[
\int_r^R s^{-\alpha-1} \tilde\E_{s,\lambda}\min\{|K|,c s^{(d+\alpha)/2}\} \dif s \leq C \log \frac{\tilde\E_{R,\lambda}|K|}{\tilde\E_{r,\lambda}|K|}
\]
for every $R\geq r \geq 1$ and $0\leq \lambda \leq 1$.
\end{corollary}

\begin{proof}[Proof of \cref{cor:truncated_expectation_bound_at_doubling_scales}]
Since $\E_{\beta,r}$ is increasing in $\beta$ and $r$, it follows from \cref{lem:truncated_expectation_diff_ineq_at_doubling_scales2} that 
there exists a positive constant $c$ such that
\[
 \frac{\partial}{\partial r} \log \tilde \E_{r,\lambda} |K| \geq c r^{-\alpha-1} \tilde \E_{r,\lambda} \min\left\{|K|, cr^{(d+\alpha)/2}\right\},
\]
and the claim follows by integrating this inequality.
\end{proof}

In order to prove \cref{lem:good_scales}, we want to turn our control of the truncated susceptibility $\tilde \E_{r,\lambda}\min\{|K|,m\}$ into control of $\tilde \E_{r,\lambda}|K|$. The main step to achieve this will be carried out in the following lemma, whose proof shows that the susceptibility has at least one of two helpful properties on each scale in which a certain finite collection of polynomial-growth quantities are doubling.

\begin{lemma}
\label{lem:good_scales_intermediate}
If $\alpha \leq d/3$ then there exist positive constants $c$ and $C$ such that
the set 
\begin{multline*}\{k\in \N: \tilde \E_{r,\lambda}|K|^2 \geq c (\tilde\E_{r,\lambda}|K|)^3 \text{ for all $r\in [2^k,2^{k+2}]$}\}\\ \cup \{k\in\N: 
\tilde \E_{r,\lambda}|K|^3 \leq C r^{d+\alpha}\tilde\E_{r,\lambda}|K| \text{ for all $r\in [2^k,2^{k+2}]$}
\}
\end{multline*}
has lower density at least $c$ for every $0\leq \lambda \leq 1$.
\end{lemma}

The proof of this lemma will use the generalized tree-graph inequality of \cref{I-lem:generalized_tree_graph} which applies to arbitrary transitive weighted graphs and states that
\begin{equation}
\E_\beta|K|^p \leq (2p-3)!! (\E_\beta|K|^2)^{k} (\E_\beta|K|)^{2p-1-3k}
\label{eq:generalized_tree_graph}
\end{equation}
for every integer $p\geq2$, $0\leq k \leq \lceil \frac{p-1}{2}\rceil$, and $\beta>0$.

\begin{proof}[Proof of \cref{lem:good_scales_intermediate}]
Let $0<c_1\leq 1$ be as in \cref{lem:truncated_expectation_diff_ineq_at_doubling_scales2}. To lighten notation we write $\E_r=\tilde \E_{r,\lambda}$.
For each $r\geq 1$, we define 
\[R=R(r)= \left(\frac{2}{c_1} \sqrt{\E_r|K|\E_r|K|^2}\right)^{2/(d+\alpha)}.\]
We have trivially that
\begin{equation}
\label{eq:trivial_R_restriction}
\E_r |K| \leq \E_r \min\{|K|,2 \sqrt{\E_r|K|\E_r|K|^2}\} + \E_r \left[|K|\mathbbm{1}\left(|K| \geq 2 \sqrt{\E_r|K| \E_r|K|^2} \right)\right].
\end{equation}
By a rearrangement of the tree-graph inequality \eqref{eq:tree_graph_pth_moment} we also have that $\E_r|K|^2/\E_r|K| \leq \sqrt{\E_r|K|\E_r|K|^2}$. (The quantity $R(r)$ is more convenient to use than the arguably more natural quantity $(\E_r|K|^2/\E_r|K|)^{2/(d+\alpha)}$ because it is monotone in $r$ and has the reverse doubling property as a consequence of the reverse doubling property of the susceptibility.) As such, we can apply Markov's inequality to the the size-biased measure $\hat \E_r$ as in the proof of \cref{prop:finitary_mean_field} to obtain that 
\begin{multline*}
\E_r \left[|K|\mathbbm{1}\left(|K| \geq 2 \sqrt{\E_r|K| \E_r|K|^2} \right)\right] \leq \E_r \left[|K|\mathbbm{1}\left(|K| \geq 2 \frac{\E_r|K|^2}{\E_r|K|} \right)\right] 
\\ =
\E_r|K| \hat \P_r(|K|\geq 2 \hat \E_r |K|) \leq \frac{1}{2}\E_r|K|,
\end{multline*}
so that we can rearrange \eqref{eq:trivial_R_restriction} to obtain that
\begin{equation}
\label{eq:less_trivial_R_restriction}
\E_r |K| \leq 2 \E_r \min\left\{|K|,2 \sqrt{\E_r|K|\E_r|K|^2}\right\} = 2 \E_r \min\{|K|,c_1R^{(d+\alpha)/2}\}
\end{equation}
for every $r\geq 1$.

Now, we have by \cref{lem:moment_poly_upper_bound} that $\E_{r}|K|$, $\E_r|K|^2$, and $R(r)$, all of which are increasing in $r$, each admit bounds of the form $r^{O(1)}$. As such, the quantity $\E_{R(r)}|K|$ is also increasing and admits a bound of the same form. In particular, there must exist positive constants $c_1$ and $C_1$ (which do not depend on $\lambda$) such that if we define $\mathscr{K}=\mathscr{K}_\lambda \subseteq \N$ to be the set of integers $k\geq 1$ such that
\[
\E_{4r}|K| \leq C_1  \E_{r}|K|, \quad R(4r) \leq C_1 R(r), \quad 
\E_{R(4r)} |K| \leq C_1  \E_{R(r)} |K|, \quad \text{ and } \quad \E_{4r}|K|^3 \leq C_1 \E_{r}|K|^3
\] 
for every $r\in [2^k,2^{k+2}]$ then $\mathscr{K}$ has lower density at least $c_1$.
Since $\E_{r}|K|^2$ is monotone in $r$, it follows from \cref{lem:truncated_expectation_diff_ineq_at_doubling_scales2} that there exists a positive constant $c_2$ such that
\begin{equation}
\label{eq:R_reverse_doubling}
\frac{R(2r)}{R(r)} \geq \left(\frac{\E_{2r}|K|}{\E_r|K|}\right)^{1/(d+\alpha)} \geq (1+c_2) 
\end{equation}
for every $r\geq 1$ and hence that 
\[
\E_{(1+c_2)R(r)} |K| \leq  \E_{R(2r)} |K| \leq C_1  \E_{R(r)} |K|
\]
for every $k\in \mathscr{K}$ and $r\in [2^k,2^{k+2}]$. It follows from this and \cref{cor:truncated_expectation_bound_at_doubling_scales} that if $k\in \mathscr{K}$ and $r\in [2^k,2^{k+2}]$ then
\begin{multline}
\label{eq:good_scale_initial}
 \E_{R(r)} \min\{|K|,c_1 R(r)^{(d+\alpha)/2}\} \preceq R(r)^{\alpha}\int_{R(r)}^{(1+c_2)R(r)} s^{-\alpha-1}  \E_{s} \min\{|K|,c_1 s^{(d+\alpha)/2}\} \dif s
\\ \preceq R(r)^\alpha \log \frac{ \E_{(1+c_2)R(r)} |K|}{ \E_{R(r)} |K|}\preceq  R(r)^\alpha.
\end{multline}
Let $\mathscr{K}_1=\{k\in \mathscr{K} :2^{k+2}\leq R(2^{k}) \}$ and $\mathscr{K}_2= \mathscr{K}\setminus \mathscr{K}_1 = \{k\in \mathscr{K} :2^{k+2}> R(2^{k}) \}$. If $k\in \mathscr{K}_1$  then $r\leq R(r)$ for all $r\in [2^k,2^{k+2}]$, so that we can apply \eqref{eq:less_trivial_R_restriction} and \eqref{eq:good_scale_initial} to bound 
\[
 \E_{r} |K| \leq 2  \E_{r} \min\{|K| , c_1 R^{(d+\alpha)/2}\} \leq 2  \E_{R(r)} \min\{|K| , c_1 R(r)^{(d+\alpha)/2}\} \preceq R(r)^\alpha
\]
for every $k\in \mathscr{K}_1$ and $r\in [2^k,2^{k+2}]$.
Recalling the definition of $R(r)$, this inequality can be rearranged to yield that
\begin{equation}
\label{eq:K1_property}
 \E_{r} |K|^2 \succeq ( \E_{r} |K|)^{d/\alpha} \succeq ( \E_{r} |K|)^3 \qquad \text{for every $k\in \mathscr{K}_1$ and  $r\in [2^k,2^{k+2}]$.}
\end{equation}
On the other hand, if $k\in \mathscr{K}_2$ then $R(r)\leq R(2^{k+2})\leq C_1 2^{k+2} \leq 4C_1 r$ for every $r\in [2^k,2^{k+2}]$ and we have directly from the definition of $R(r)$ that
\[
\sqrt{\E_{r}|K|\E_{r}|K|^2} \preceq r^{(d+\alpha)/2}
\]
for every $k\in \mathscr{K}_2$ and $r\in [2^k,2^{k+2}]$. It follows from this and the generalized tree-graph inequality of Lemma \ref{I-lem:generalized_tree_graph} (restated here as \eqref{eq:generalized_tree_graph}) that
\begin{equation}
\label{eq:K2_property}
\E_{r}|K|^3 \preceq (\sqrt{\E_{r}|K|\E_{r}|K|^2})^2 \E_{r}|K| \preceq r^{d+\alpha}\E_{r}|K| \quad \text{for every $k\in \mathscr{K}_2$ and $r\in [2^k,2^{k+2}]$.}
\end{equation}
 This completes the proof.
\end{proof}

We will use the sets $\mathscr{K}_\lambda = \mathscr{K}_{\lambda,1} \cup  \mathscr{K}_{\lambda,2}$ defined in this proof again in the proof of \cref{lem:all_scales_good_lambda_positive} when we prove that the estimate $\tilde \E_{r,\lambda}|K| \preceq \lambda^{-C} r^\alpha$ holds on \emph{every} scale under the fictitious assumption that the modified hydrodynamic condition does not hold.

\begin{proof}[Proof of \cref{lem:good_scales}]
Let $\mathscr{K}_\lambda$ be as in the proof of \cref{lem:good_scales}. We have as usual (i.e., by the same calculation leading to \eqref{I-eq:barE11_def}) that
\[
\frac{\partial}{\partial r} \E_{\beta,r} |K| \geq \beta |J'(r)| |B_r| (\E_{\beta,r}|K|)^2 - \beta |J'(r)| \E_{\beta,r} |K|^3,
\]
while \cref{lem:second_moment_differential_inequality} gives that 
\[
\frac{\partial}{\partial \beta} \E_{\beta,r} |K| \succeq \left(\frac{\E_{\beta,r}|K|^2}{4\E_{\beta,r}|K|}-\frac{1}{2} \E_{\beta,r}|K|\right)
\]
when $\beta_c/2\leq \beta \leq \beta_c$.
Putting these two inequalities together yields by the chain rule that
\begin{multline*}
 \frac{\partial}{\partial r} 
\tilde \E_{r,\lambda} |K| = \left(\frac{\partial}{\partial r} 
 \E_{\beta,r} |K| \right)\Biggr|_{\beta=\beta(r,\lambda)} + \lambda \beta_c r^{-\alpha-1} \left(\frac{\partial}{\partial \beta} 
 \E_{\beta,r} |K| \right)\Biggr|_{\beta=\beta(r,\lambda)}
\\ \succeq r^{-\alpha-1}  \max\left\{
\lambda  \left(\frac{\tilde \E_{r,\lambda}|K|^2}{4\tilde \E_{r,\lambda}|K|}-\frac{1}{2}\tilde \E_{r,\lambda}|K|\right), (\tilde\E_{r,\lambda}|K|)^2 - r^{-d}\tilde\E_{r,\lambda} |K|^3\right\}.
\end{multline*}
It follows that
  if $k \in \mathscr{K}_\lambda$ then
\[
\frac{\partial}{\partial r} 
\tilde \E_{r,\lambda} |K|  \succeq \lambda r^{-\alpha-1} (\tilde \E_{r,\lambda}|K|)^2 
\]
for every $r\in [2^k,2^{k+2}]$. (Note that if $k\in \mathscr{K}_{\lambda,2}$ then the factor of $\lambda$ is not needed in this lower bound.) Integrating this differential inequality between $r$ and $2^{k+2}$ as in 
the proof of \cref{lem:tilde_susceptibility_lower} (see also
Lemma \ref{I-lem:f'=f^2})
yields that 
\begin{equation*}
\frac{1}{\tilde \E_{r,\lambda}|K|} = 
\frac{1}{\tilde \E_{2^{k+2},\lambda}|K|}+
\int_r^{2^{k+2}} \frac{\frac{d}{ds}\tilde \E_{s,\lambda}|K|}{(\tilde \E_{s,\lambda}|K|)^2} \dif s 
\\\succeq \lambda \int_r^{2^{k+2}} s^{-\alpha-1} \dif s \succeq \lambda r^{-\alpha}
\end{equation*}
for every $k\in \mathscr{K}_\lambda$ and $r\in [2^k,2^{k+1}]$, which implies the claim.
\end{proof}

The proof of \cref{lem:good_scales} also shows that if the set $\mathscr{K}_{\lambda,2}$ is unbounded then the stronger estimate $\tilde \E_{r,\lambda}|K| \preceq r^\alpha$, with no $\lambda^{-1}$ factor, holds on an unbounded set of scales. In this case one can obtain easily that the stronger ($\lambda=0$) bound $\E_{\beta_c,r}|K| \preceq r^\alpha$ also holds on an unbounded set of scales. We now record the contrapositive of this fact in the following corollary. 

\begin{corollary}
\label{cor:bad_case_large_susceptibility_lambda}
Suppose that $\alpha=d/3$ and for each $0\leq \lambda\leq 1$ let the sets $\mathscr{K}_{1,\lambda}$ and $\mathscr{K}_{2,\lambda}$ be as in the proof of \cref{lem:good_scales_intermediate}. If $\liminf_{r\to \infty} r^{-\alpha}\E_{\beta_c,r}|K|=\infty$ then \[
  \liminf_{r\to \infty} r^{-\alpha}\tilde \E_{r,\lambda}|K| \geq \frac{1}{\lambda \beta_c |J|}
\] 
for every $\lambda>0$ and there exists $\lambda_0>0$ such that if $0\leq \lambda \leq \lambda_0$ then the set $\mathscr{K}_{\lambda,2}$ is finite. 
\end{corollary}

\begin{proof}[Proof of \cref{cor:bad_case_large_susceptibility_lambda}]
The first claim follows from the differential inequality
\begin{equation}
\label{eq:beta_derivative_simpler_upper}
\frac{\partial}{\partial \beta} \E_{\beta,r}|K| = \sum_{x\in \Z^d} J_r(0,x) \E_{\beta,r}\left[|K||K_x|\mathbbm{1}(0\nleftrightarrow x)\right] 
\\\leq
|J_r| (\E_{\beta,r}|K|)^2
 \leq |J| (\E_{\beta,r}|K|)^2,
\end{equation}
which can be integrated to yield that
\begin{align*}
  \frac{1}{\E_{\beta_c,r}|K|} &= \frac{1}{\E_{\beta_c(r,\lambda),r}|K|} - \int_{\beta_c(r,\lambda)}^{\beta_c} \frac{\frac{\partial}{\partial \beta} \E_{\beta,r}|K|}{\E_{\beta,r}|K|} \dif r
 \geq 
\frac{1}{\E_{\beta_c(r,\lambda),r}|K|} - 
\lambda \beta_c |J| r^{-\alpha};
\end{align*}
this inequality implies that if $\liminf_{r\to \infty} r^{-\alpha}\tilde \E_{r,\lambda}|K| < (\lambda \beta_c |J|)^{-1}$ for some $\lambda>0$ then  \[\liminf_{r\to \infty} r^{-\alpha}\E_{\beta_c,r}|K|<\infty.\] As we saw in the previous proof, if $k \in \mathscr{K}_{\lambda,2}$ then 
\[
  \frac{\partial}{\partial r} \tilde \E_{r,\lambda}|K| \succeq r^{-\alpha-1}(\tilde \E_{r,\lambda}|K|)^2
\]
for every $r\in [2^k,2^{k+2}]$, and integrating this differential inequality between $r$ and $2^{k+2}$ as in the proof of \cref{lem:good_scales} (see also
Lemma \ref{I-lem:f'=f^2}) yields that $\tilde \E_{r,\lambda} |K| \preceq r^\alpha$
for every $k\in \mathscr{K}_{\lambda,2}$ and $r\in [2^k,2^{k+1}]$. Moreover, it follows from the first part of the corollary that if $\lambda$ is sufficiently small and $\mathscr{K}_{\lambda,2}$ is infinite then $\E_{\beta_c,r}|K| \preceq r^\alpha$ for an unbounded set of $r$ as claimed.
\end{proof}

\subsection{Fictitious $\beta$ derivative and susceptibility estimates on all scales}
\label{sub:fictitious_beta_derivative_and_susceptibility_estimates_on_all_scales}

In this section we show that the fictitious assumption that the modified hydrodynamic condition does not hold when $d=3\alpha$ 
 implies a lower bound on the $\beta$-derivative of the susceptibility at \emph{all} scales. We then use this estimate to improve the susceptibility estimate of \cref{lem:good_scales} to hold on all scales, again under the fictitious assumption that the modified hydrodynamic condition does not hold.
We define 
\[
  D_r = \frac{\partial}{\partial \beta} \E_{\beta,r}|K| \Bigr|_{\beta=\beta_c} \qquad \text{ and } \qquad \tilde D_{r,\lambda} = \frac{\partial}{\partial \beta} \E_{\beta,r}|K| \Bigr|_{\beta=\beta(r,\lambda)}
\]
to be the $\beta$ derivatives of the susceptibility evaluated at $\beta_c$ or $\beta(r,\lambda):=(1-\lambda r^{-\alpha})\beta_c$ as appropriate, so that $D_r=\tilde D_{r,0}$.

\begin{prop}
\label{prop:fictitious_derivative_and_susceptibility_best}
Suppose that $d=3\alpha$ and that the modified hydrodynamic condition does \emph{not} hold. There exists $0<\lambda_0\leq 1$ such that if $0<\lambda\leq\lambda_0$ then 
\begin{align*} \tilde \E_{r,\lambda}|K| \preceq \lambda^{-1} r^\alpha
,  \qquad
\tilde D_{r,\lambda} \succeq 
r^{2\alpha}
, \qquad \text{and} \qquad
\tilde \E_{r,\lambda} |K|^2  \succeq \lambda r^{3\alpha} 
\end{align*}
for all $r\geq 1$.
\end{prop}

One of the most interesting features of this proposition is that it begins with a (false) assumption on the behaviour of the model on some unbounded, but possibly very sparse, collection of scales, and deduces consequences about the behaviour of the model on \emph{every} scale. Moreover, the bounds we establish on $\tilde D_{r,\lambda}$ and $\tilde \E_{r,\lambda}|K|^2$ are suggestive of mean-field critical behaviour since (leaving aside $\lambda$ factors) they are as large as possible given the scaling of the first moment as discussed in \cref{sec:equivalent_characterisations_of_mean_field_critical_behaviour}.

\medskip

We begin by showing that if the modified hydrodynamic condition does \emph{not} hold, then the good scales that we showed the existence of in the previous subsection, in which $\tilde \E_{r,\lambda}|K|=O(r^\alpha)$, must coincide infinitely often with scales in which $\tilde M_{r,\lambda}$ is large. We phrase the lemma in such a way that it also applies to the case $\lambda=0$ under the assumption that $\liminf_{r\to\infty}r^{-\alpha}\tilde \E_{\beta_c,r}|K|<\infty$.

\begin{lemma}
\label{lem:good_scales_with_large_M}
Suppose that $d=3\alpha$. If the modified hydrodynamic condition does \emph{not} hold then there exist positive constants $c$ and $\lambda_0\leq1$ such that 
if $0\leq \lambda \leq \lambda_0$ is such that the limit infimum $\liminf_{r\to\infty}r^{-\alpha}\tilde \E_{r,\lambda}|K|=:A_\lambda \frac{\alpha}{\beta_c}$ is finite then there exists
 an unbounded set of scales $\mathscr{R}$ such that 
\[\tilde \E_{r,\lambda}|K| \leq 4A_\lambda \frac{\alpha}{\beta_c} r^\alpha \qquad \text{and} \qquad \tilde M_{r,\lambda} \geq  c A_\lambda^{-3} r^{2\alpha}\]
for every $r\in \mathscr{R}$. 
\end{lemma}


\begin{proof}[Proof of \cref{lem:good_scales_with_large_M}]
Since the modified hydrodynamic condition is assumed not to hold, there exists $0<\lambda_0\leq 1$ such that $c_1=\limsup_{r\to \infty}r^{-2\alpha}\tilde M_{r,\lambda_0}$ is positive. Using \cref{lem:tilde_susceptibility_lower}, we may assume that $\lambda_0$ is sufficiently small that $A_\lambda \geq 3/4$ for every $0\leq \lambda \leq \lambda_0$. By monotonicity, we have that 
\[\limsup_{r\to \infty}r^{-2\alpha}\tilde M_{r,\lambda} \geq c_1\] for every $0\leq \lambda \leq \lambda_0$.
Fix $0\leq \lambda \leq \lambda_0$ such that $\liminf_{r\to\infty}r^{-\alpha}\tilde \E_{r,\lambda}|K|=A_\lambda \frac{\alpha}{\beta_c}<\infty$.
By assumption, there exists an unbounded set of scales $\mathscr{R}_1$ such that $\tilde \E_{r,\lambda}|K| \leq 2A_\lambda \frac{\alpha}{\beta_c}r^\alpha$ for every $r\in \mathscr{R}_1$. 
 We now recall the estimate \eqref{I-eq:E1_from_M}, a consequence of the universal tightness theorem \cite{hutchcroft2020power} and
 the two-point upper bound of \cite{hutchcroft2022sharp} (restated here as \eqref{eq:spatially_averaged_two_point_upper_intro}) stating that
  there exists a constant $C_2 \geq 1$ such that
\begin{equation}
\label{eq:E1_from_M_restate_tilde}
  \tilde \E_{r,\lambda}[|K|^2|K\cap B_r|] \leq C_2 \left(\tilde \E_{r,\lambda}|K|\right)^{7/2} r^{\alpha/2} \tilde M_{r,\lambda}^{1/2}.
\end{equation}
Let $a_\lambda>0$ be defined by
\[
  a_\lambda = \frac{\min\{c_1,1\}}{2} \cdot \frac{1}{4C_2^2}\left(\frac{4\alpha A_\lambda}{\beta_c}\right)^{-3} =: C_3 A_\lambda^{-3},\]
  where the constant $C_3$ is defined to make this last equality hold, so that
  \[
   a_\lambda \leq \frac{c_1}{2} \qquad \text{ and } \qquad 1-\left(\frac{4\alpha A_\lambda}{\beta_c}\right)^{3/2}C_2 a_\lambda^{1/2} \geq \frac{1}{2}.
\]
For each $r$ let $r_0=r_0(r)\leq r$ be supremal such that $\tilde M_{R(r),\lambda} \geq a_\lambda$; this value $r_0(r)$ diverges to infinity as $r\to\infty$ since $a_\lambda < \liminf_{r\to \infty} r^{-2\alpha} \tilde M_{r,\lambda}$. It suffices to prove that if $r$ is a sufficiently large element of $\mathscr{R}_1$ then $\tilde \E_{r_0,\lambda}|K| \leq 4A_\lambda\frac{\alpha}{\beta_c} r_0^{\alpha}$. (Here it is not a problem that $\tilde M_{r,\lambda}$ does not depend continuously on $r$; it suffices that $\tilde \E_{r,\lambda}|K|$ is continuous.) Suppose not, let $r_1\in [r_0,r]$ be maximal such that $\tilde \E_{r_1,\lambda}|K| \geq 4A_\lambda \frac{\alpha}{\beta_c} r_1^{\alpha}$ and let $r_2\in [r_1,r]$ be minimal such that $\tilde \E_{r_2,\lambda}|K| \leq 3A_\lambda \frac{\alpha}{\beta_c} r_2^{\alpha}$. Thus, for every $s\in [r_1,r_2]$ we have that $\tilde M_s \leq a_\lambda s^{2\alpha}$ and $3 A_\lambda \frac{\alpha}{\beta_c} s^\alpha \leq \tilde \E_{s,\lambda}|K| \leq 4A_\lambda \frac{\alpha}{\beta_c}s^\alpha$. Thus, we have by \eqref{eq:E1_from_M_restate_tilde} that 
\begin{align*}
\frac{\partial}{\partial s} \tilde \E_{s,\lambda}|K| &\geq \beta_c |B_r| |J'(r)| \left(1-\frac{\tilde \E_{s,\lambda}[|K|^2|K\cap B_r|]}{|B_r| (\tilde \E_{s,\lambda}|K|)^2}\right) (\tilde \E_{s,\lambda}|K|)^2
\\ &\geq (1-o(1)) \left(1- \left(\frac{4A_\lambda\alpha}{\beta_c}\right)^{3/2} C_2  a_\lambda^{1/2} \right) 3A_\lambda \alpha s^{-1} \tilde \E_{s,\lambda}|K| \geq (1-o(1))\frac{3}{2} A_\lambda \alpha s^{-1} \tilde \E_{s,\lambda}|K|
\end{align*}
for every $s\in [r_1,r_2]$. If $r$ is sufficiently large then the coefficient $(3A_\lambda/2-o(1))$ is strictly larger than $1$, so that 
 $s^{-\alpha}\tilde \E_{s,\lambda}|K|$ is increasing on the interval $[r_1,r_2]$. This is a contradiction since  $\tilde \E_{r_1,\lambda}|K| \geq 4A_\lambda \frac{\alpha}{\beta_c}r_1^\alpha$ and $\tilde \E_{r_2,\lambda}|K| \leq 3A_\lambda \frac{\alpha}{\beta_c} r_2^\alpha$.
\end{proof}

We next prove the following relationship between the $\beta$ derivative at two different scales, which does not require any assumptions on $d$ and $\alpha$.

\begin{lemma} 
\label{lem:derivative_stability}
The estimate
\[\tilde D_{r,\lambda}
  \preceq s^{-\alpha} (\tilde \E_{r,\lambda}|K|)^2 + s^{-2\alpha} \tilde D_{s,\lambda} (\tilde\E_{r,\lambda}|K|)^2\]
  holds for every $0\leq \lambda \leq 1$ and $r\geq s \geq 1$.
\end{lemma}

We will primarily apply this inequality with $r$ much larger than $s$, in which case the first term on the right hand side will be negligibly small compared with the second.

\begin{proof}[Proof of \cref{lem:derivative_stability}]
As usual, we can use Russo's formula and the mass-transport principle to write
\[
  \tilde D_{r,\lambda} = \tilde \E_{r,\lambda}\left[ \sum_{x\in K} \sum_{y\in \Z^d} J_r(x,y) \mathbbm{1}(x\nleftrightarrow y)|K_y|\right] = \sum_{y\in \Z^d} J_r(0,y) \tilde \E_{r,\lambda}\left[\mathbbm{1}(0\nleftrightarrow y) |K||K_y|\right].
\]
Let $s\leq r$ and consider the standard monotone coupling of the measures $\tilde \E_{s,\lambda}$ and $\tilde \E_{r,\lambda}$, writing $\tilde \P$ for probabilities taken with respect to this coupling. Write $K^s$, $K^r$, $K^s_y$, and $K^r_y$ for the clusters of the origin and of $y$ in the two configurations, noting that if $K^r\neq K^r_y$ then $K^s\neq K^s_y$. For each $x,z\in \Z^d$, conditioning on $K^s$ and $K^s_y$ and taking a union bound over the possible choices of the first edges that does not belong to $K^s$ or $K^s_y$ on the open paths from $0$ to $x$ and $y$ to $z$ yields that
\begin{align}
&\tilde \P(0\nxleftrightarrow{r} y, 0 \xleftrightarrow{r} x, y \xleftrightarrow{r} z \mid K^s,K^s_y)
  \nonumber\\&\hspace{0.7cm}\leq \mathbbm{1}(0\nxleftrightarrow{s} y) \Bigl(\mathbbm{1}(0 \xleftrightarrow{s} x)+  \sum_{a\in K^s} \sum_{b\in \Z^d} [\beta(r,\lambda) J_r(a,b)-\beta(s,\lambda)J_s(a,b))]\tilde \P_{r,\lambda}(0\leftrightarrow x)\Bigr)
  \nonumber\\
   &\hspace{2.1cm}\cdot\Bigl(\mathbbm{1}(y \xleftrightarrow{s} z)+  \sum_{u\in K^s_y} \sum_{v\in \Z^d} [\beta(r,\lambda) J_r(u,v)-\beta(s,\lambda)J_s(u,v))]\tilde \P_{r,\lambda}(y\leftrightarrow z)\Bigr),
  \label{eq:beta_derivative_stability}
\end{align}
where we write $\{0 \xleftrightarrow{r} x\}$ and $\{0 \xleftrightarrow{s} x\}$ to denote the events that $x\in K^r$ and $x\in K^s$ respectively and where $\beta(r,\lambda)J_r(a,b)-\beta(s,\lambda)J_s(a,b)$ is an upper bound on the probability that the edge $\{a,b\}$ is open in the larger configuration and closed in the smaller configuration. 
Since
\begin{multline*}
  \sum_{x\in \Z^d} \Bigl[\beta(r,\lambda) J_r(a,b)-\beta(s,\lambda)J_s(a,b)\Bigr] 
  \\\leq \beta_c  \sum_{x\in \Z^d} \Bigl[J_r(a,b)-J_s(a,b)\Bigr] + [\beta(r,\lambda)-\beta(s,\lambda)] \sum_{x\in \Z^d} J(0,x) \preceq s^{-\alpha},
\end{multline*}
we can sum \eqref{eq:beta_derivative_stability} over $x$ and $z$ and take expectations to obtain that
\begin{multline*}
  \tilde \E_{r,\lambda}[\mathbbm{1}(0\nleftrightarrow y)|K||K_y|] \preceq (1+s^{-\alpha} \tilde \E_{r,\lambda}|K|)^2  \tilde \E_{s,\lambda}[\mathbbm{1}(0\nleftrightarrow y)|K||K_y|] 
  \\\preceq (s^{-\alpha} \tilde \E_{r,\lambda}|K|)^2\tilde \E_{s,\lambda}[\mathbbm{1}(0\nleftrightarrow y)|K||K_y|] 
\end{multline*}
for every $r\geq s\geq 1$ and $0\leq \lambda \leq 1$, 
where we used \cref{lem:tilde_susceptibility_lower} to absorb the constant $1$ into the term $s^{-\alpha} \tilde \E_{r,\lambda}|K|$.
It follows from this that
\begin{align*}
  \tilde D_{r,\lambda} &= \sum_{y\in \Z^d} [J_r(0,y)-J_s(0,y)] \tilde \E_{r,\lambda}\left[\mathbbm{1}(0\nleftrightarrow y) |K||K_y|\right] + \sum_{y\in \Z^d} J_s(0,y) \tilde \E_{r,\lambda}\left[\mathbbm{1}(0\nleftrightarrow y) |K||K_y|\right]\\
  &\preceq s^{-\alpha} (\tilde \E_{r,\lambda}|K|)^2 + \tilde D_{s,\lambda} (s^{-\alpha}\tilde\E_{r,\lambda}|K|)^2
\end{align*}
as claimed.
\end{proof}

We next lower bound the $\beta$-derivative in terms of the susceptibility under the fictitious assumption that the modified hydrodynamic condition does not hold.

\begin{lemma}
\label{lem:fictitious_derivative_lower_bound}
 Suppose that $d=3\alpha$ and that the modified hydrodynamic condition does \emph{not} hold. There exist positive constants $c$ and $\lambda_0$ such that 
  \[
    \liminf_{r\to \infty} r^{-2\alpha} \tilde D_{r,\lambda} \geq c \left(\liminf_{r\to\infty} \frac{\beta_c}{\alpha}r^{-\alpha}\tilde \E_{r,\lambda}|K|\right)^{-12}
  \]
  for every $0\leq \lambda \leq \lambda_0$.
\end{lemma}

Note that (as of our current knowledge level) the right hand side may be zero when $\lambda=0$, in which case the inequality holds vacuously; \cref{lem:good_scales} implies that it has order at least $\lambda^{-12}$ when $\lambda$ is small and positive.
The proof of this lemma will apply \cref{I-lem:moments_bounded_below_by_M}, which states that
\begin{equation}
\label{eq:moments_bounded_below_by_M_restate}
\E_{\beta,r}|K|^p \geq \E_{\beta,r}|K \cap B_{2r}|^p \geq (M_{\beta,r}-1)^p \P_{\beta,r} \left(|K\cap B_{2r}|\geq M_{\beta,r}-1\right) \geq \frac{(M_{\beta,r}-1)^{p+1}}{e |B_r|}
\end{equation}
for every $\beta,r\geq 0$ and $p\geq 1$.

\begin{proof}[Proof of \cref{lem:fictitious_derivative_lower_bound}] 
Suppose that the modified hydrodynamic condition does not hold. Let $\lambda_0$ and $c_1$ be the constants from \cref{lem:good_scales_with_large_M} and suppose that $0\leq \lambda \leq \lambda_0$ is such that $A_\lambda:=\liminf_{r\to\infty} \frac{\beta_c}{\alpha}r^{-\alpha}\tilde \E_{r,\lambda}|K|<\infty$.
It is an immediate consequence of \cref{lem:good_scales_with_large_M} and \cref{I-lem:moments_bounded_below_by_M} (restated here as \eqref{eq:moments_bounded_below_by_M_restate}) that there exist an unbounded set $\mathscr{R}$ of scales $r$ such that
\[
  \tilde \E_{r,\lambda} |K| \leq 4 A_\lambda r^\alpha \qquad \text{and} \qquad \tilde \E_{r,\lambda}|K|^2 \geq c_2 r^{-d}\tilde M_{r,\lambda}^3 \geq c_3 A_\lambda^{-9} r^{3\alpha}
\]
where $c_2$ and $c_3$ are positive constants independent of the choice of $\lambda$.
It follows from this and \cref{lem:second_moment_differential_inequality} that 
if $r\in \mathscr{R}$ then
\[
\frac{\tilde D_{r,\lambda}}{(\tilde \E_{r,\lambda}|K|)^2} \succeq \frac{\tilde \E_{r,\lambda} |K|^2}{(\tilde \E_{r,\lambda} |K|)^3}  \succeq A_\lambda^{-12}.
\]
Rearranging the estimate of \cref{lem:derivative_stability}, it follows that if $s \leq r$ and $r\in \mathscr{R}$ then
\[
 \tilde D_{s,\lambda} + s^\alpha \succeq \frac{\tilde D_{r,\lambda}}{(\tilde \E_{r,\lambda}|K|)^2} s^{2\alpha} \succeq A_\lambda^{-12} s^{2\alpha}.
\]
The claim follows since $\mathscr{R}$ is unbounded.
\end{proof}

We now apply \cref{lem:fictitious_derivative_lower_bound} to improve the conclusions of \cref{lem:good_scales} to a pointwise estimate (albeit with worse dependence of constants on $\lambda$) under our fictitious assumption.

\begin{lemma}
\label{lem:all_scales_good_lambda_positive}
Suppose that $d=3\alpha$ and that the modified hydrodynamic condition does \emph{not} hold. There exists $0<\lambda_0\leq 1$ and a constant $C$ such that 
\[\limsup_{r\to \infty}r^{-\alpha}\tilde \E_{r,\lambda}|K| \leq C \lambda^{-C}\]
for every $0<\lambda \leq \lambda_0$.
\end{lemma}

\begin{proof}[Proof of \cref{lem:all_scales_good_lambda_positive}]
Recall from the proof of \cref{lem:good_scales} that $\mathscr{K}=\mathscr{K}_\lambda$ was defined to be the set of integers $k\geq 1$ such that
\[
\E_{4r}|K| \leq C_1  \E_{r}|K|, \quad R(4r) \leq C_1 R(r), \quad 
\E_{R(4r)} |K| \leq C_1  \E_{R(r)} |K|, \quad \text{ and } \quad \E_{4r}|K|^3 \leq C_1 \E_{r}|K|^3
\] 
for every $r\in [2^k,2^{k+2}]$,
 where $R(r)$ was defined by
\[R(r)= \left(\frac{2}{c_1} \sqrt{\E_r|K|\E_r|K|^2}\right)^{1/2\alpha}\]
with $c_1>0$ the constant from \cref{lem:truncated_expectation_diff_ineq_at_doubling_scales2} and where
 $C_1$ was chosen to be sufficiently large to make this set have positive lower density. By increasing $C_1$ if necessary, we may assume that $C_1 \geq 4^{8(5\alpha+2)}$. 
  We saw in the proof of \cref{lem:good_scales} that 
\begin{equation}
\label{eq:first_moment_estimate_on_good_scale}
  \tilde \E_{r,\lambda}|K| \preceq \lambda^{-1} r^\alpha 
\end{equation}
for every $k\in \mathscr{K}$ and $r\in [2^k,2^{k+1}]$. On the other hand, we also have by \cref{lem:tilde_susceptibility_lower} that
$\tilde \E_{r,\lambda}|K| \succeq r^\alpha$
for every $r\geq 1$. In light of the derivative lower bound of \cref{lem:fictitious_derivative_lower_bound}, the Durrett-Nguyen bound (\cref{lem:Durrett_Nguyen}), and the tree-graph inequality we have that
\begin{equation}
\label{eq:second_moment_estimate_on_good_scale}
   \lambda^{25} r^{3\alpha} \preceq \frac{\tilde D_{r,\lambda}^2}{\tilde\E_{r,\lambda}|K|} \preceq \tilde \E_{r,\lambda}|K|^2 \leq (\tilde \E_{r,\lambda}|K|)^3 \preceq  \lambda^{-3} r^{3\alpha}
\end{equation}
for every $k\in \mathscr{K}$ and $r\in [2^k,2^{k+1}]$. It follows from this that
\begin{equation}
\label{eq:R_estimate_on_good_scale}
\lambda^{13/2\alpha} r \preceq R(r) \preceq \lambda^{-1/\alpha} r  
\end{equation}
and
\begin{equation}
\label{eq:third_moment_estimate_on_good_scale}
 \lambda^{51} r^{5\alpha} \preceq \frac{(\E_{r,\lambda}|K|^2)^2}{\E_{r,\lambda}|K|} \leq \tilde \E_{r,\lambda}|K|^3 \preceq (\tilde \E_{r,\lambda}|K|)^5 \preceq \lambda^{-5} r^{5\alpha}
\end{equation}
for every $k\in \mathscr{K}$ and $r\in [2^k,2^{k+1}]$, where we used H\"older's inequality and the tree-graph inequality in our bounds on $\tilde \E_{r,\lambda}|K|^3$. We want to use this to deduce that the gaps between elements of $\mathscr{K}$ are bounded by a multiple of $\log (1/\lambda)$. 

\medskip

Suppose that $k_1$, $k_2$ are consecutive elements of $\mathscr{K}$ with $k_2-k_1$ large. By \eqref{eq:R_estimate_on_good_scale} and the reverse doubling property \eqref{eq:R_reverse_doubling} of $R(r)$, there exists a constant $C_2$ such that if we define $\ell_0=C_2 \log (2/\lambda)$ then $R(2^{k_2-\ell_0}) \leq 2^{k_2}$.
  Writing $r_1=2^{k_1}$, $r_2 =2^{k_2}$, and $r_2' = 2^{k_2-\ell_0}$ it follows from the definition of $\mathscr{K}$ that there exists a constant $C_3$ such that
\begin{multline*}
\max\left\{\frac{\E_{r_2}|K|}{\E_{r_1}|K|},\frac{R(r_2)}{R(r_1)},\frac{\E_{r_2}|K|}{\E_{R(r_1)}|K|},\frac{\E_{r_2}|K|^3}{\E_{r_1}|K|^3}\right\} \\\geq
  \max\left\{\frac{\E_{r_2'}|K|}{\E_{r_1}|K|},\frac{R(r_2')}{R(r_1)},\frac{\E_{R(r_2')}|K|}{\E_{R(r_1)}|K|},\frac{\E_{r_2'}|K|^3}{\E_{r_1}|K|^3}\right\} \geq C_1^{(k_2-k_1-\ell_0)/8} 
  \\\succeq \lambda^{C_3}C_1^{(k_2-k_1)/8} \succeq \lambda^{C_3} \left(\frac{r_2}{r_1}\right)^{5\alpha+2},
\end{multline*}
On the other hand, the estimates \eqref{eq:first_moment_estimate_on_good_scale}, \eqref{eq:second_moment_estimate_on_good_scale}, \eqref{eq:R_estimate_on_good_scale}, and \eqref{eq:third_moment_estimate_on_good_scale} above along with \cref{lem:tilde_susceptibility_lower} allow us to bound
\begin{multline*}
  \max\left\{\frac{\E_{r_2}|K|}{\E_{r_1}|K|},\frac{R(r_2)}{R(r_1)},\frac{\E_{r_2}|K|}{\E_{R(r_1)}|K|},\frac{\E_{r_2}|K|^3}{\E_{r_1}|K|^3}\right\} 
  \\\preceq \max\left\{ \lambda^{-1} \left(\frac{r_2}{r_1}\right)^\alpha,\lambda^{-15/2\alpha}\left(\frac{r_2}{r_1}\right)^1,\lambda^{-14} \left(\frac{r_2}{r_1}\right)^\alpha, \lambda^{-56} \left(\frac{r_2}{r_1}\right)^{5\alpha}\right\} \\\preceq \lambda^{-\max\{56,15/2\alpha\}}\left(\frac{r_2}{r_1}\right)^{5\alpha+1}.
\end{multline*}
Rearranging, we obtain that
\[
\frac{r_2}{r_1} \preceq \lambda^{-C_3-\max\{56,15/2\alpha\}} =: \lambda^{-C_4}.
\]
It follows that if $r$ is any sufficiently large parameter then there exists $k\in \mathscr{K}$ with $2^k \preceq \lambda^{-C_4}r$ and hence that $\tilde\E_{r,\lambda}|K| \leq \tilde\E_{2^k,\lambda}|K| \preceq \lambda^{-1}(2^k)^\alpha \preceq \lambda^{-1-C_4 \alpha} r^\alpha$. 
\end{proof}

The bounds we have proven so far are already sufficient to get bounds of the correct order on the truncated susceptibility $\E_{\beta_c,r}\min\{|K|,r^{2\alpha}\}$ (equivalently, bounds on $\tilde \E_{r,\lambda}\min\{|K|,r^{2\alpha}\}$ that are uniform in the choice of $\lambda>0$).

\begin{lemma}
\label{lem:truncated_susceptibility_sharp_all_scales}
Suppose that $d=3\alpha$ and that the modified hydrodynamic condition does \emph{not} hold. Then
\[
  \E_{\beta_c,r}\min\{|K|,r^{2\alpha}\} \preceq r^{\alpha}
\]
for every $r\geq 1$.
\end{lemma}

\begin{proof}[Proof of \cref{lem:truncated_susceptibility_sharp_all_scales}]
We have by \cref{cor:truncated_expectation_bound_at_doubling_scales} and \cref{lem:all_scales_good_lambda_positive} that there exist positive constants $c$ and $C$ such that
\begin{multline*}
\left(\min_{s\in [r,R]} s^{-\alpha} \tilde \E_{s,\lambda}\min\{|K|,cs^{2\alpha}\} \right) \log \frac{R}{r} \leq  \int_r^R s^{-\alpha-1} \tilde \E_{s,\lambda}\min\{|K|,cs^{2\alpha}\} \dif s \\ \preceq \log \frac{\tilde \E_{R,\lambda}|K|} {\tilde \E_{r,\lambda}|K|}
\preceq 
  \log \frac{C\lambda^{-C} R^\alpha}{r^\alpha} \preceq \log \frac{R}{r} + \log \frac{1}{\lambda}
\end{multline*}
for all $0<\lambda \leq \lambda_0$ and $R\geq r \geq 1$. Taking $R=2r$, it follows by monotonicity that 
\begin{equation}
\label{eq:truncated_expectation_log_lambda}
  \tilde \E_{r,\lambda}\min\{|K|,cr^{2\alpha}\} \preceq \left(\log \frac{2}{\lambda}\right) r^\alpha
\end{equation}
for every $0<\lambda \leq \lambda_0$ and $r\geq 1$. 
We now want to completely remove the dependence on $\lambda$ from this bound.
As mentioned in \cref{remark:truncated_susceptibility_simple_diff_ineq}, the proof of Lemma~\ref{II-lem:truncated_expectation_differential_inequality} also yields the differential inequality
\begin{equation}
\label{eq:beta_derivative_simple_upper}
  \frac{\partial}{\partial \beta} \E_{\beta,r}\min\{|K|,m\} \leq |J| (\E_{\beta,r}\min\{|K|,m\})^2.
\end{equation}
Since the $O(\log 2/\lambda)$ coefficient appearing on the right hand side of \eqref{eq:truncated_expectation_log_lambda} is much smaller than $\lambda^{-1}$, it follows from the same analysis as in the proof of \cref{cor:bad_case_large_susceptibility_lambda} that if $\lambda$ is sufficiently small then we can integrate this differential inequality 
 from $\lambda$ to $0$ to obtain that $\E_{\beta_c,r}\min\{|K|,cr^{2\alpha}\} \preceq r^\alpha$. We conclude by noting that $\min\{|K|,r^{2\alpha}\}\leq c^{-1}\min\{|K|,cr^{2\alpha}\}$ and hence that  $\E_{\beta_c,r}\min\{|K|,r^{2\alpha}\}\leq c^{-1}\E_{\beta_c,r}\min\{|K|,cr^{2\alpha}\}$.
\end{proof}

We now prove \cref{prop:fictitious_derivative_and_susceptibility_best}. To do so, we rerun the proof of \cref{lem:good_scales} but use our new knowledge that  the truncated susceptibility estimate $\E_{\beta_c,r}\min\{|K|,r^{2\alpha}\} \preceq r^\alpha$ holds on \emph{all} scales (under our fictitious assumption that the modified hydrodynamic condition does not hold); we previously knew this only on scales where the susceptibility was doubling.

\begin{proof}[Proof of \cref{prop:fictitious_derivative_and_susceptibility_best}]
We begin by proving that $\tilde \E_{r,\lambda}|K| \preceq \lambda^{-1}r^\alpha$ for all $r\geq 1$ and $0<\lambda\leq \lambda_0$.
Let $c_1$ be the constant from \cref{lem:truncated_susceptibility_sharp_all_scales} and, as in the proof of \cref{lem:good_scales_intermediate}, define $R(r)$ for each $r\geq 1$ by
\[
  R(r)=\left(\frac{2}{c_1} \sqrt{\E_r|K|\E_r|K|^2}\right)^{1/2\alpha},
\]
so that 
$\E_r|K| \leq 2 \E_r \min\{|K|,c_1 R(r)^{2\alpha}\}$
by the same reasoning as in the proof of \cref{lem:good_scales_intermediate}. 
If $R(2r) \geq 4r$ then, since $R(r)$ is increasing, we have that $R(s)\geq s$ for every $s\in [2r,4r]$. It follows by \cref{lem:truncated_susceptibility_sharp_all_scales} that if $R(2r)\geq 4r$ then
\[
  \E_s|K| \leq 2 \E_{R(s)} \min\{|K|,c_1 R(s)^{2\alpha}\} \preceq R(s)^{2\alpha}
\]
for every $s\in [2r,4r]$ and hence that $\E_s|K|^2 \succeq (\E_s|K|)^3$ for every $s\in [2r,4r]$. As we saw in the proof of \cref{lem:good_scales}, this implies that $\E_{2r}|K| \preceq \lambda^{-1}r^\alpha$ and hence that $\E_{r}|K| \preceq \lambda^{-1}r^\alpha$ also by monotonicity.
Otherwise, if $R(2r)\leq 4r$, then $R(s)\leq 2s$ for every $s\in [r,2r]$. Applying the generalized tree-graph inequality as before, this implies that $\E_s|K|^3 \preceq s^{4\alpha} \E_s|K|$ for all $s\in [r,2r]$, which again implies that $\tilde \E_{r,\lambda}|K| \preceq r^\alpha$ as we saw in the proof of \cref{lem:good_scales}.

\medskip

We next prove that the inequality $\tilde D_{r,\lambda} \succeq r^{2\alpha}$ holds on all scales when $\lambda$ is sufficiently small. It suffices by \cref{lem:derivative_stability} and \cref{lem:second_moment_differential_inequality} to prove that there exists an unbounded set of scales with $\tilde \E_{r,\lambda}|K|^2 \succeq (\tilde \E_{r,\lambda}|K|)^3$. If there exists an unbounded set of scales $r$ with $R(2r)\geq 4r$ then this follows directly from the discussion in the previous paragraph. Otherwise, we have by our analysis of the second case above that $\tilde \E_{r,\lambda}|K| \asymp r^\alpha$ for all sufficiently large $r$, and the claim follows since if $\lambda$ is sufficiently small then there must be infinitely many scales with $\tilde \E_{r,\lambda}|K|^2 \succeq r^{3\alpha}$ by \cref{I-lem:moments_bounded_below_by_M} (restated here as \eqref{eq:moments_bounded_below_by_M_restate}) and the assumption that the modified hydrodynamic condition does not hold. Finally, the claimed lower bound on $\tilde \E_{r,\lambda}|K|^2$ follows from the bounds $\tilde \E_{r,\lambda}|K| \preceq \lambda^{-1}r^\alpha$ and $\tilde D_{r,\lambda}\succeq r^{2\alpha}$ together with Durrett-Nguyen bound (\cref{lem:Durrett_Nguyen}).
\end{proof}

\subsection{Fictitious negligibility of mesoscopic clusters and mean-field asymptotics}
\label{sub:fictitious_negligibility_of_mesoscopic_clusters}

In this section we prove the following lemma, which is an analogue of Proposition~\ref{II-prop:negligibility_of_mesoscopic} holding in the case $d=3\alpha$ under the fictitious assumption that the modified hydrodynamic condition does not hold. (It follows from \cref{I-thm:hd_moments_main,I-thm:critical_dim_moments_main_slowly_varying} that the conclusion of this lemma is true when $d=3\alpha \geq 6$ under the perturbative assumptions needed to implement the lace expansion. When $d=3\alpha<6$ we will prove in \cref{sec:logarithmic_corrections_at_the_critical_dimension} that it is \emph{false}, with the main contribution to the susceptibility coming from $|K|$ of order $(\log r)^{-1/2} r^{2\alpha}$ in this case.)

\begin{lemma}
\label{lem:fictitious_negligibility}
Suppose that $d=3\alpha$ and that the modified hydrodynamic condition does \emph{not} hold. For each $\eps>0$ there exists $\delta>0$ such that
\[
  \E_{\beta_c,r} \min\{|K|,\delta r^{2\alpha}\} \leq \eps r^\alpha
\]
for every $r\geq 1$.
\end{lemma}

We also use \cref{lem:fictitious_negligibility} to deduce that the volume tail has \emph{exact mean-field asymptotics} under the fictitious assumption that the modified hydrodynamic condition does not hold. 

\begin{lemma}
\label{lem:fictitious_volume_tail_sharp}
Suppose that $d=3\alpha$ and that the modified hydrodynamic condition does \emph{not} hold. 
Then $\P_{\beta_c,r}(|K|\geq n) \asymp n^{-1/2}$  for every $n\geq 1$ and $r\geq n^{1/2\alpha}$.
\end{lemma}

We begin by proving the following lower bound on the volume tail  $\tilde \P_{r,\lambda}(|K|\geq n)$, again under the fictitious assumption that the modified hydrodynamic condition does not hold.

\begin{lemma}
\label{lem:fictitious_volume_tail_lower_bound_pointwise}
Suppose that $d=3\alpha$ and that the modified hydrodynamic condition does \emph{not} hold. There exists $0<\lambda_0\leq 1$ such that if $0<\lambda\leq\lambda_0$ then the inequality
$\tilde \P_{r,\lambda}(|K|\geq n) \succeq \lambda^{8} n^{-1/2}$ holds for all $n \leq \frac{1}{2}\E_{r,\lambda}|K|^2/\E_{r,\lambda}|K|$.
\end{lemma}

\begin{proof}[Proof of \cref{lem:fictitious_volume_tail_lower_bound_pointwise}]
Let $\lambda_0$ be as in \cref{prop:fictitious_derivative_and_susceptibility_best} and let $0<\lambda \leq \lambda_0$. We lighten notation by writing $\P_r=\tilde \P_{r,\lambda}$ and $\E_r=\tilde \E_{r,\lambda}$. Let $\hat \P_r$ and $\hat \E_r$ denote probabilities and expectations taken with respect to the $|K|$-biased measure, so that $\hat \E_r|K| = \E_r|K|^2/\E_r|K|$. We have by the Paley-Zygmund inequality that
\begin{align*}
\hat \P_r\Bigl(|K| \geq \frac{1}{2}\hat \E_r|K|\Bigr) \geq  \frac{(\hat \E_r|K|)^2}{4\hat \E_r|K|^2} = \frac{(\E_r|K|^2)^2}{4\E_r|K|^3\E_r|K|},
\end{align*}
and using the generalized tree-graph inequality $\E_r|K|^3 \preceq \E_r|K|^2 (\E_r|K|)^2$ (as proven in \cref{I-lem:generalized_tree_graph} and restated here in \eqref{eq:generalized_tree_graph}) we obtain that
\begin{align*}
\hat \P_r\Bigl(|K| \geq \frac{1}{2}\hat \E_r|K|\Bigr) \succeq \frac{\E_r|K|^2}{(\E_r|K|)^3} \succeq \lambda^{4},
\end{align*}
where we applied \cref{prop:fictitious_derivative_and_susceptibility_best} in the final inequality.
It follows by this and Markov's inequality that there exists a constant $C$ such that
\begin{align*}
\hat \P_r\Bigl(\frac{1}{2}\hat \E_r|K| \leq |K| \leq C\lambda^{-4}\hat \E_r|K|\Bigr) \succeq \lambda^{4}.
\end{align*}
Converting this back into an estimate for the unbiased measure, we have that
\begin{multline*}
\P_{r}\left( \frac{1}{2} \hat \E_r|K| \leq |K| \leq C \lambda^{-4} \hat \E_r|K|\right) = \E_r|K| \hat \E_r\left[\frac{1}{|K|}\mathbbm{1}\Bigl(\hat \E_r|K| \leq |K| \leq C\lambda^{-4}\hat \E_r|K|\Bigr)\right]
\\
\geq \frac{\lambda^{4}\E_r|K|}{C \hat \E_r|K|} \hat \P_r\left( \frac{1}{2} \hat \E_r|K| \leq |K| \leq C \lambda^{-4} \hat \E_r|K|\right) \succeq \lambda^{8} (\E_r|K|)^{-1/2},
\end{multline*}
where in the final inequality we again used that $\E_r|K| \geq (\hat \E_r|K|)^{1/2}$ by the tree-graph inequality.
Now, by the intermediate value theorem, for each $1\leq n\leq \frac{1}{2}\hat \E_r|K|$ there exists $0\leq s\leq r$ such that $\frac{1}{2}\hat \E_s|K|=n$, and we deduce that 
\[
  \P_{r}(|K|\geq n) \geq \P_s(|K|\geq n) \succeq \lambda^{8}n^{-1/2}
\]
as claimed.
\end{proof}

We next show that for $\lambda>0$ we can find a not-too-sparse set of $r$ for which the truncated susceptibility is smaller than $\frac{\alpha}{\beta_c}r^\alpha$ by a multiplicative constant.

\begin{lemma}
\label{lem:fictitious_negligibility_positive_lambda}
Suppose that $d=3\alpha$ and that the modified hydrodynamic condition does \emph{not} hold. There exists $0<\lambda_0\leq 1$ such that if $0<\lambda\leq \lambda_0$ then there exist positive constants $c_\lambda$ and $C_\lambda$ such that if $r\geq C_\lambda$ then
\[
 \limsup_{r\to \infty} \min_{s\in [r,C_\lambda r]} s^{-\alpha} \tilde\E_{r,\lambda}\min\{|K|,c_\lambda r^{2\alpha}\} \leq (1-c_\lambda) \frac{\alpha}{\beta_c}.
\]
\end{lemma}

\begin{proof}[Proof of \cref{lem:fictitious_negligibility_positive_lambda}]
Let $\lambda_0$ be the mininum of the constants appearing in \cref{prop:fictitious_derivative_and_susceptibility_best,lem:fictitious_volume_tail_lower_bound_pointwise} and fix $0<\lambda \leq \lambda_0$. We have by \cref{lem:truncated_expectation_diff_ineq_at_doubling_scales} that there exists a constant $C_1$ such that
\[
\frac{d}{dr}\log \tilde \E_{r,\lambda}|K| \geq  \beta_c |B_r| |J'(r)|  (\tilde\E_{r,\lambda}\min\{|K|,\delta r^{2\alpha}\}-C_1\delta r^\alpha)
\]
for every $r\geq 1$ and $\delta>0$. Integrating this estimate between two parameters $r$ and $R$, we obtain that
\begin{multline*}
  \log \frac{\tilde \E_{R,\lambda}|K|}{\tilde \E_{r,\lambda}|K|}  \geq \beta_c \int_r^R |B_s| |J'(s)|  (\tilde\E_{r,\lambda}\min\{|K|,\delta s^{2\alpha}\}-C_1\delta s^\alpha)\dif s
  \\
  \geq   \min_{s\in [r,R]} \left[\frac{\beta_c}{\alpha}s|B_s| |J'(s)|(\tilde\E_{s,\lambda}\min\{|K|,\delta s^{2\alpha}\}-C_1\delta s^\alpha)\right] \alpha \log \frac{R}{r}.
\end{multline*}
On the other hand, we have by \cref{prop:fictitious_derivative_and_susceptibility_best} that
that there exists a constant $C_2$ such that
\[
  \log \frac{\tilde \E_{R,\lambda}|K|}{\tilde \E_{r,\lambda}|K|} \leq \log \frac{C_2 R^\alpha}{\lambda r^\alpha} = \alpha \log \frac{R}{r} + \log \frac{C_2}{\lambda}.
\]
Rearranging, we obtain that
\begin{equation*}
\min_{s\in [r,R]} \left[\frac{\beta_c}{\alpha}s|B_s| |J'(s)|(\tilde\E_{s,\lambda}\min\{|K|,\delta s^{2\alpha}\}-C_1\delta s^\alpha)\right] -1 \preceq \frac{\log 2/\lambda}{\log R/r}.
\end{equation*}
Setting $R= \mu r$ and taking $r\to \infty$, it follows that 
\begin{equation}
\label{eq:R=mu_r_limit}
\limsup_{r\to \infty}\min_{s\in [r,\mu R]} \left[\frac{\beta_c}{\alpha} s^{-\alpha}(\tilde\E_{s,\lambda}\min\{|K|,\delta s^{2\alpha}\}-C_1\delta s^\alpha)\right] -1 \preceq \frac{\log 2/\lambda}{\log \mu}
\end{equation}
for each $\mu,\delta>0$ and $0 < \lambda\leq \lambda_0$.
On the other hand, by \cref{prop:fictitious_derivative_and_susceptibility_best,lem:fictitious_volume_tail_lower_bound_pointwise}, there exists a positive constant $c_1$ such that if $\delta \leq c \lambda^2$ and $s\geq (2 \delta^{-1})^{1/2\alpha}$ then
\begin{equation}
\label{eq:delta_to_2delta_truncation}
\tilde \E_{s,\lambda}\min\{|K|, 2\delta s^{2\alpha}\} - \tilde \E_{s,\lambda}\min\{|K|, \delta s^{2\alpha}\}\geq \sum_{n=\lceil \delta s^{2\alpha}\rceil+1}^{\lfloor 2\delta s^{2\alpha} \rfloor} \tilde \P_{s,\lambda}(|K|\geq n) \succeq \lambda^{8} \delta^{1/2} s^\alpha.
\end{equation}
Putting \eqref{eq:R=mu_r_limit} and \eqref{eq:delta_to_2delta_truncation} together, we obtain that there exists a positive constant $c_2$ such that if $\delta \leq c_1 \lambda^2$ then
\begin{multline*}
\limsup_{r\to \infty}\min_{s\in [r,\mu R]} \left[\frac{\beta_c}{\alpha} s^{-\alpha}(\tilde\E_{s,\lambda}\min\{|K|,\delta s^{2\alpha}\}+c_2 \lambda^{8}\delta^{1/2}s^\alpha-2C_1\delta s^\alpha)\right] -1 
\\\leq
\limsup_{r\to \infty}\min_{s\in [r,\mu R]} \left[\frac{\beta_c}{\alpha} s^{-\alpha}(\tilde\E_{s,\lambda}\min\{|K|,2\delta s^{2\alpha}\}-2C_1\delta s^\alpha)\right] -1 
\preceq \frac{\log 2/\lambda}{\log \mu}.
\end{multline*}
It follows that there exist positive constants $c_3$, $c_4$, and $C_2$ such that if $\delta \leq c_3 \lambda^{16}$ then
\begin{equation*}
\limsup_{r\to \infty}\min_{s\in [r,\mu R]} \left[\frac{\beta_c}{\alpha} s^{-\alpha}(\tilde\E_{s,\lambda}\min\{|K|,\delta s^{2\alpha}\}+c_4 \lambda^{8}\delta^{1/2}s^\alpha)\right]  
\leq 1+ C_2 \frac{\log 2/\lambda}{\log \mu}.
\end{equation*}
Consequently, there exists a constant $C_3$ such that if we define 
\[\mu=\mu(\delta,\lambda)=\exp\left[C_3 \frac{\log 2/\lambda}{\lambda^{8} \delta^{1/2}}\right]\]
then
\begin{equation*}
\limsup_{r\to \infty}\min_{s\in [r,\mu(\delta,\lambda) R]} \left[\frac{\beta_c}{\alpha} s^{-\alpha}\tilde\E_{s,\lambda}\min\{|K|,\delta s^{2\alpha}\}\right]  
\leq 1 - \frac{c_4}{2} \lambda^{8}\delta^{1/2}
\end{equation*}
whenever $\delta \leq c_3 \lambda^{16}$. This is easily seen to imply the claim with $c_\lambda = \min\{c_3 \lambda^{16},c_3^{3/2}\lambda^{16}/2\}$ and $C_\lambda = \mu(c_3\lambda^{16},\lambda)$.
(Fortunately, the extremely poor quantitative dependence of the constant $C_\lambda$ on $\lambda$ will not be a problem in the remainder of the argument.) 
\end{proof}

We are now ready to prove \cref{lem:fictitious_negligibility}. The proof will make crucial use of the improved bound on $\beta$ derivatives of truncated susceptibilities provided by \cref{lem:Durrett_Nguyen}.

\begin{proof}[Proof of \cref{lem:fictitious_negligibility}]
We have by \cref{lem:fictitious_negligibility_positive_lambda} that there exists $0<\lambda_0\leq 1$ such that if $0<\lambda\leq \lambda_0$ then there exist positive constants $c_\lambda$ and $C_\lambda$ such that if $r\geq C_\lambda$ then
\[
 \limsup_{r\to \infty} \min_{s\in [r,C_\lambda r]} s^{-\alpha} \tilde\E_{s,\lambda}\min\{|K|,c_\lambda r^{2\alpha}\} \leq (1-c_\lambda) \frac{\alpha}{\beta_c}.
\]
Fix $\lambda=\lambda_0$; all constants in the remainder of the proof will be permitted to depend on $\lambda$. Note by monotonicity that
\begin{equation}
\label{eq:almost_negligible}
 \limsup_{r\to \infty} \min_{s\in [r,C_\lambda r]} s^{-\alpha} \tilde\E_{s,\lambda}\min\{|K|,\delta r^{2\alpha}\} \leq (1-c_\lambda) \frac{\alpha}{\beta_c}
\end{equation}
for every $0<\delta\leq c_\lambda$. On the other hand, we have by \cref{lem:Durrett_Nguyen} that
\[
  \frac{\partial}{\partial \beta} \E_{\beta,r}\min\{|K|,\delta r^{2\alpha}\} \preceq \delta^{1/2} r^\alpha \E_{\beta,r}\min\{|K|,\delta r^{2\alpha}\}.
\]
Using Lemma~\ref{II-lem:truncated_expectation_differential_inequality}, which states that
\begin{equation}
\label{eq:truncated_expectation_differential_inequality}
\frac{\partial}{\partial r}\E_{\beta,r} [|K|\wedge m] \leq \beta |J'(r)||B_r|\E_{\beta,r} [|K|\wedge m]^2
\end{equation}
 for every $\beta,r>0$, 
 to bound the $r$-derivative and applying the chain rule we obtain that there exists a constant $C$ such that if $0<\delta\leq c_\lambda$ then
\[
 \frac{d}{dr}  \tilde\E_{r,\lambda}\min\{|K|,m\} \leq \beta_c |B_r| |J'(r)| (\tilde\E_{r,\lambda}\min\{|K|,m\})^2+C\delta^{1/2} r^{-1} \tilde \E_{r,\lambda}\min\{|K|,m\}
\]
for every $r\geq 1$ and $m\leq \delta r^{2\alpha}$. As such, we can choose $r_0<\infty$ and $\delta=\delta_0>0$ sufficiently small that if $r\geq r_0$ and $m\leq \delta_0 r^{2\alpha}$ then
\begin{multline}
\label{eq:ODE_lemma_with_f^2_and_f_terms_truncated_susceptibility}
 \frac{d}{dr}  \tilde\E_{r,\lambda}\min\{|K|,m\} \leq \beta_c \left(1+\frac{(\alpha \wedge 1)c_\lambda}{100}\right) r^{-\alpha-1}(\tilde\E_{r,\lambda}\min\{|K|,m\})^2
 \\+ \frac{(\alpha \wedge 1)c_\lambda}{100} r^{-1} \tilde \E_{r,\lambda}\min\{|K|,m\}.
\end{multline}
We will analyze this differential inequality using the following general lemma, whose proof is deferred to the end of the subsection.
\begin{lemma}\label{lem:ODE_lemma_with_f^2_and_f_terms}
Suppose that $f:(0,\infty)\to (0,\infty)$ is a differentiable function such that
\[f' \leq  A  r^{-\alpha-1}f^2 + \eps r^{-1} f\]
for some $A,\eps>0$ and every $r\geq r_0$. If there exists $r\geq r_0$ such that $f(r) \leq \frac{\alpha-\eps}{A} r^\alpha$ then
\[
  f(R) \leq \left(\frac{R}{r}\right)^\eps \left[\frac{1}{f(r)}-\frac{A}{(\alpha-\eps)r^\alpha}+ \left(\frac{R}{r}\right)^\eps\frac{AR^\eps}{(\alpha-\eps)R^\alpha}\right]^{-1}
\]
for every $R\geq r$.
\end{lemma}

Now,  we have by \eqref{eq:almost_negligible} that for each $r$ there exists $s\in [r,C_\lambda r]$ such that 
\begin{equation}
\label{eq:almost_negligible2}
\tilde \E_{s,\lambda} \min\{|K|,\delta_0 r^{2\alpha}\} \leq (1-c_\lambda) \frac{\alpha}{\beta_c} s^\alpha \leq (1-c_\lambda/100) \frac{(1-c_\lambda/100)}{(1+c_\lambda/100)} \frac{\alpha}{\beta_c}s^\alpha,
\end{equation}
where the second inequality follows by calculus.
Taking $m=\delta_0 r^{2\alpha}$ and applying \cref{lem:ODE_lemma_with_f^2_and_f_terms} to the differential inequality of \eqref{eq:ODE_lemma_with_f^2_and_f_terms_truncated_susceptibility} yields that
\[
  \tilde \E_{R,\lambda}\min\{|K|,\delta_0 r^{2\alpha}\} \preceq \left(\frac{R}{s}\right)^{\alpha/100} s^\alpha
\]
for every $s$ such that \eqref{eq:almost_negligible2} holds and every $R\geq s$. It follows that for every $\eps>0$ there exists a constant $C_\eps$ such that for each $r\geq 1$ there exists $R\leq C_\eps r$ with
\[
  \tilde \E_{R,\lambda}\min\{|K|,\delta_0 r^{2\alpha}\} \leq \eps \frac{\alpha}{\beta_c} R^\alpha.
\]
Taking $\eps$ small and integrating the differential inequality \eqref{eq:beta_derivative_simple_upper} as in the proof of \cref{lem:truncated_susceptibility_sharp_all_scales} yields that a similar estimate holds for $\E_{\beta_c,R}\min\{|K|,\delta_0 r^{2\alpha}\}$, which is easily seen to imply the claim. 
\end{proof}

We now owe the reader the proof of \cref{lem:ODE_lemma_with_f^2_and_f_terms}.

\begin{proof}[Proof of \cref{lem:ODE_lemma_with_f^2_and_f_terms}]
Using the integrating factor $r^{-\eps}$, we can rewrite the inequality as
\[
  (r^{-\eps} f)' = r^{-\eps}f' -\eps r^{-1-\eps} f \leq A  r^{-\alpha+\eps-1} (r^{-\eps} f)^2.
\]
We can integrate this differential inequality as usual to obtain that
\begin{equation*}
 \frac{r^\eps}{f(r)}-\frac{R^\eps}{f(R)}=\int_r^R \frac{(s^{-\eps} f(s))'}{(s^{-\eps}f(s))^2} \dif s 
 \\
 \leq \int_r^R A s^{-\alpha+\eps-1} \dif s = \frac{A }{\alpha-\eps}\left[r^{-\alpha+\eps}-R^{-\alpha+\eps}\right],
\end{equation*}
which can be rearranged to yield that
\[
  \frac{R^\eps}{f(R)} \geq  \frac{r^\eps}{f(r)} - \frac{A r^\eps}{(\alpha-\eps)r^\alpha} + \frac{A R^\eps}{(\alpha-\eps)R^\alpha} 
\]
If $f(r)\leq \frac{(\alpha-\eps)}{A }r^\alpha$ then the right hand side is non-negative and we can safely rearrange to obtain that
\[
  f(R) \leq \left(\frac{R}{r}\right)^\eps \left[\frac{1}{f(r)}-\frac{(\alpha-\eps)}{A r^\alpha}+ \left(\frac{R}{r}\right)^\eps\frac{A R^\eps}{(\alpha-\eps)R^\alpha}\right]^{-1}
\]
as claimed.
\end{proof}

It remains to prove \cref{lem:fictitious_volume_tail_sharp}. Our next step is to remove the $\lambda$-dependence from the relationship between the first and second moments established in \cref{prop:fictitious_derivative_and_susceptibility_best}. We do this over the course of two lemmas establishing successively stronger relationships between the two quantities.

\begin{lemma}
\label{lem:fictitious_large_second_moment}
Suppose that $d=3\alpha$ and that the modified hydrodynamic condition does \emph{not} hold.  Then there exists $0<\lambda \leq 1$ such that $\tilde \E_{r,\lambda}|K|^2 \succeq r^{2\alpha}\tilde \E_{r,\lambda}|K|$ for every $0\leq \lambda \leq \lambda_0$ and $r\geq 1$.
\end{lemma}

\begin{proof}[Proof of \cref{lem:fictitious_large_second_moment}]
It follows from Corollary \ref{I-cor:mean_lower_bound} and \cref{lem:fictitious_negligibility} that there exists positive constants $c$ and $0<\lambda_0 \leq 1$ such that
\[
 \tilde\E_{r,\lambda}\left[ |K|\mathbbm{1}(|K|\geq cr^{2\alpha})\right]\geq \tilde\E_{r,\lambda} \max\{|K|-cr^{2\alpha},0\}=\tilde\E_{r,\lambda} |K| - \tilde\E_{r,\lambda}\min\{|K|,cr^{2\alpha}\} \geq \frac{1}{2}\tilde\E_{r,\lambda} |K|,
\]
for every $0\leq \lambda \leq \lambda_0$ and $r\geq 1$, and we deduce that
\[
 \tilde\E_{r,\lambda}|K|^2 \geq cr^{2\alpha} \tilde\E_{r,\lambda}\left[ |K|\mathbbm{1}(|K|\geq cr^{2\alpha})\right]  \geq\frac{c}{2} r^{2\alpha}\tilde\E_{r,\lambda} |K|
\]
for every $0\leq \lambda \leq \lambda_0$ and $r\geq 1$ as claimed.
\end{proof}

We next improve this estimate by replacing $r^{2\alpha}\tilde \E_{r,\lambda}|K|$ with 
 $(\tilde \E_{r,\lambda}|K|)^3$.

\begin{lemma}
\label{lem:fictitious_large_second_moment2}
Suppose that $d=3\alpha$ and that the modified hydrodynamic condition does \emph{not} hold.  Then there exists $0<\lambda_0 \leq 1$ such that $\tilde\E_{r,\lambda}|K|^2 \succeq (\tilde \E_{r,\lambda}|K|)^3$ for every $0\leq \lambda \leq \lambda_0$ and $r\geq 1$.
\end{lemma}

\begin{proof}[Proof of \cref{lem:fictitious_large_second_moment2}]
Let $0<\lambda_0\leq 1$ be as in
\cref{lem:fictitious_large_second_moment} and let $0\leq \lambda \leq \lambda_0$.
It follows from \cref{lem:tilde_susceptibility_lower,lem:fictitious_large_second_moment} that there exists a constant $C\geq 2$ such that if we define $R_\lambda(r)$ by
\[
  R_\lambda(r) = \left(C\sqrt{\tilde\E_{r,\lambda}|K|\tilde\E_{r,\lambda}|K|^2}\right)^{1/2\alpha}
\]
then $R_\lambda(r)\geq r$ for every $r\geq 1$. It follows by the same argument as in the proof of \cref{lem:good_scales} that
\[
  \tilde\E_{r,\lambda} |K| \leq 2 \tilde\E_{r,\lambda}\min\{|K|,R_\lambda(r)^{2\alpha}\},
\]
and since $R_\lambda(r)\geq r$ we can use monotonicity and apply \cref{lem:truncated_susceptibility_sharp_all_scales} to bound
\[
  \tilde\E_{r,\lambda} |K| \leq 2 \tilde\E_{R_\lambda(r),r}\min\{|K|,R_\lambda(r)^{2\alpha}\} \preceq R_\lambda(r)^\alpha.
\]
This inequality rearranges to give the desired bound.
\end{proof}

\begin{proof}[Proof of \cref{lem:fictitious_volume_tail_sharp}]
The lower bound follows by the same argument used in the proof of \cref{lem:fictitious_volume_tail_lower_bound_pointwise} but where we now use \cref{lem:fictitious_large_second_moment2} to bound $\E_{\beta_c,r}|K|^2 \succeq (\E_{\beta_c,r}|K|)^3$, so that
\[
  \hat \P_{\beta_c,r}\left(|K|\geq \frac{1}{2}\hat \E_{\beta_c,r}|K|\right) \succeq 1.
\]
For the upper bound, we have by \cref{lem:fictitious_negligibility} and Markov's inequality that 
\[
  \P_{\beta_c,r}(|K|\geq  r^{2\alpha}) \leq \frac{1}{r^{2\alpha}} \E_{\beta_c,r}\min\{|K|,r^{2\alpha}\} \preceq r^{-\alpha}, 
\]
and, using \cref{lem:fictitious_negligibility} again, the same argument used to prove Proposition \ref{I-prop:up-to-constants_volume-tail} yields that
\[
  \P_{\beta_c,R}(|K|\geq r^{2\alpha}) \preceq r^{-\alpha}
\]
for every $R\geq r$ also. This is easily seen to imply the claimed upper bound.
\end{proof}

\section{The hydrodynamic condition II: Reaching a contradiction}
\label{sec:the_hydrodynamic_condition_ii_reaching_a_contradiction}

\subsection{The key fictitious gluability lemma}
\label{sub:the_key_fictitious_gluability_lemma}

Roughly speaking, the main goal of this section is to prove, under the fictitious assumption that the modified hydrodynamic condition does not hold, that large clusters have a good probability to be ``locally large'', that is, satisfy $|B_r \cap K| =\Omega(r^{2\alpha})$. For technical reasons we instead prove a slightly weaker statement involving balls centered at a random point $x$, chosen independently of the percolation configuration, rather than at the origin (which will unfortunately make the applications of the lemma more delicate later on).  

\medskip

Given a subset $W$ of $\Z^d$, a point $x\in \Z^d$, and a long-range percolation configuration $\omega$ on $\Z^d$, we write $K^W_x$ subgraph of the cluster $K_x$ induced by the set of vertices that can be reached from $x$ by an open path that does not visit any vertices of $W$ (including at its endpoints). We refer to $K^W_x$ as the cluster of $x$ \textbf{off~$W$}. As usual, we will also write $K^W=K^W_0$ for the cluster of the origin off~$W$. We also write $B_r(x)=\{y\in \Z^d: \|y-x\|\leq r\}=B_r+x$ for the ball of radius $r$ centered at $x$.

\begin{lemma}
\label{lem:key_gluability_lemma_mu}
Suppose that $d=3\alpha$ 
 and that the modified hydrodynamic condition does \emph{not} hold.
  There exist a family of probability measures $\mu_{r,\lambda}$ on $\Z^d$ such that the following holds.
\begin{enumerate}
\item If 
 $\liminf_{r\to\infty} r^{-\alpha}\E_{\beta_c,r}|K|<\infty$ then there exist constants
$\delta>0$ and $C$ and an unbounded set of scales $\mathscr{R}$ such that if $r\in \mathscr{R}$ then $\E_{\beta_c,r}|K| \leq C r^\alpha$ and
\begin{equation*}
 \liminf_{r\to \infty} 
 \sup_W r^{-\alpha} \E_{\beta_c,r}\left[ 
 |K^W| \mathbbm{1}\left( 
 \sum_{x\in \Z^d} |K^W \cap B_r(x)|\mu_{r,0}(x) \leq \delta r^{2\alpha}\right)\right]
  <  \frac{\alpha}{\beta_c}
\end{equation*}
 where  the supremum is taken over all subsets of $\Z^d$.
\item If $\liminf_{r\to\infty} r^{-\alpha}\E_{\beta_c,r}|K|=\infty$ then there exists $0<\lambda_0 \leq 1$ such that for every $0<\lambda\leq \lambda_0$ there exists $\delta>0$ such that 
\begin{equation*}
 \liminf_{r\to \infty} \sup_W r^{-\alpha}\tilde \E_{r,\lambda} \left[ |K^W| \mathbbm{1}\left(
 \sum_{x\in \Z^d}|K^W\cap B_r(x)|\mu_{r,\lambda}(x) \leq \delta r^{2\alpha}\right)\right] \\<  \frac{\alpha}{\beta_c} \bigl(1+ \lambda \alpha |J|\bigr)^{-1},
\end{equation*}
where the supremum is taken over all subsets of $\Z^d$.
\end{enumerate}
\end{lemma}

The significance of the specific bound appearing here is that
\begin{equation}
\label{eq:degree_defecit}
  \sum_{x \in \Z^d}|\beta_c J(0,x)-\beta(r,\lambda)J_r(0,x)|  \sim \bigl(1+ \lambda \alpha |J|\bigr)  \frac{\beta_c}{\alpha}r^{-\alpha}
\end{equation}
describes the asymptotics of the change in expected degree at the origin when passing from the measure $\tilde\P_{r,\lambda}$ to $\P_{\beta_c}$.
Including the supremum over the set $W$ in the statement of this lemma makes it stronger than if we only worked with the entire cluster, and will be important to our applications in the next section. 

\medskip

In \cref{subsec:critical_dim_initial_regularity,sub:fictitious_beta_derivative_and_susceptibility_estimates_on_all_scales,sub:fictitious_negligibility_of_mesoscopic_clusters}, we saw that if $\tilde\E_{r,\lambda} \min\{|K|,m\} \gg r^\alpha$ for some $m\ll r^{2\alpha}$, then the susceptibility $\tilde \E_{r,\lambda}|K|$ has a large derivative, a fact which eventually led to the upper bounds on $\tilde\E_{r,\lambda} \min\{|K|,m\}$ of \cref{lem:truncated_susceptibility_sharp_all_scales,lem:fictitious_negligibility}.
In order to prove \cref{lem:key_gluability_lemma_mu} we will prove a similar differential inequality stated in terms of clusters that are small \emph{when restricted to a ball of radius $r$}, rather than small globally. This is a subtle manner, particularly as we would also like to take suprema over the set $W$ of disallowed vertices.
Our first step in this direction will be a differential inequality stated in terms of \emph{geodesics} having small intersection with a ball of appropriately chosen random center.

\medskip

For each connected graph $H$ with vertex set contained in $\Z^d$ and pair of vertices $u,v$ in the vertex set of $H$, let $[u,v]_H$ be the set of vertices in a geodesic connecting $u$ to $v$ in $H$. We choose one such geodesic once-and-for-all for every triple $(u,v,H)$ in such a way that $[u+x,v+x]_{H+x}=[u,v]_H+x$ for any triple $(u,v,H)$ and vector $x\in \Z^d$. (The precise nature of the choice is not important. It is important important that the choice is made deterministically before any percolation configurations we consider are sampled. The translation-invariance property can be enforced by choosing the geodesic for every triple of the form $(0,v,H)$ and then defining the geodesic for all other triples via translation.)

\begin{lemma}
\label{lem:locally_large_measure_concrete}
 For each $\beta,r>0$ there exists a probability measure $\mu=\mu_{\beta,r}$ on $\Z^d$ such that
\begin{multline*}
\E_{\beta,r}\left[ \sum_{y\in B_r} \mathbbm{1}(0\nleftrightarrow y) |K_y| \sum_{z\in K^W} \mathbbm{1}\!\left(K^{W} \in \mathscr{A} \text{ and } \sum_{x\in \Z^d}|[0,z]_{K^W} \cap B_r(x)|\mu(x) \leq \ell \right)\right] 
\\
\geq\left(1-\frac{\ell\E_{\beta,r}|K|^2}{|B_r| \E_{\beta,r}|K|}\right) |B_r|\E_{\beta,r}|K| \E_{\beta,r}\left[\sum_{z\in K^W} \mathbbm{1}\!\left(K^{W} \in \mathscr{A} \text{ and } \sum_{x\in \Z^d}|[0,z]_{K^W} \cap B_r(x)|\mu(x) \leq \ell \right)\right] 
\end{multline*}
for every $\ell\in(0,\infty)$, subset $W$ of $\Z^d$, and set $\mathscr{A}$ of finite subsets of $\Z^d$.
\end{lemma}

We will apply the lemma with $\ell$ a small constant multiple of $r^{\alpha}$ and $\mathscr{A}=\{A\subseteq \Z^d : \sum_{x\in \Z^d} |A\cap B_r(x)|\mu(x)\leq \delta r^{2\alpha}\}$. The measure $\mu_{\beta,r}$ will arise concretely as the law of a uniform point chosen from the cluster of the origin under the measure biased by $|K|^2$.

\begin{proof}[Proof of \cref{lem:locally_large_measure_concrete}]
Fix $r$ and $\lambda$ and let $K$ and $\tilde K^W$ be independent copies of the cluster of the origin in the whole space and off $W$ respectively. We define $O(K)$ and $C(K)$ to be the set of edges in the subgraph $K$ and the set of edges incident to but not contained in $K$ respectively and define $O(\tilde K^W)$ and $C(\tilde K^W)$ to be the set of edges in the subgraph $\tilde K^W$ and the set of edges \emph{with neither endpoint in $W$} that are incident but not contained in $\tilde K^W$ respectively. These are precisely the sets of open and closed edges (in the respective configurations associated to $K$ and $\tilde K^W$) that are revealed when exploring the two clusters $K$ and $\tilde K^W$.  
 For each $y\in \Z^d$, define a percolation configuration $\omega_y$ by 
 \[
\omega_y(e) = \begin{cases}
1 & \text{ $e\in O(K)$}\\
0 & \text{ $e\in C(K)$}\\
1 & \text{ $e\in (y+O(\tilde K^W))\setminus O(K)\cup C(K)$}\\
0 & \text{ $e\in (y+C(\tilde K^W))\setminus O(K)\cup C(K)$}\\
\omega'(e) & \text{ $e\notin O(K)\cup C(K)\cup (y+O(\tilde K^W) \cup C(\tilde K^W))$}
\end{cases},
 \]
where $\omega'$ is a percolation configuration with law $\E_{\beta,r}$ independent of $K$ and $\tilde K^W$. In other words, we define $\omega_y$ by first exploring $K$ and writing down the status in $K$ of all edges we reveal, then exploring $y+\tilde K^W$ and writing down the status of all edges we reveal that were not revealed when exploring $K$, then assigning the status of all remaining edges independently at random. The resulting configuration $\omega_y$ has law $\P_{\beta,r}$ for every $y\in \Z^d$. The cluster of the origin in $\omega_y$ is always equal to $K$, while the off-$(y+W)$ cluster of $y$ in $\omega_y$, which we denote by $K_{y,y}^W$, is either contained in $K$ (if $K$ contains $y$) or else is contained in $y+\tilde K^W$. Finally, when $y\notin K$, each vertex $z$ in $y+\tilde K^W$ belongs to the smaller set $K_{y,y}^W$ if and only if there exists a simple path connecting $y$ to $z$ in $y+\tilde K^W$ that is disjoint from $K$.

\medskip

For the remainder of the proof we will abuse notation by writing $\E$ for expectations with respect to both our usual measure $\E_{\beta,r}$ and the collection of coupled configurations described above and write $[0,z]=[0,z]_{\tilde K^W}$.
 We have by the properties of $K$, $\tilde K$, and $K_{y,y}$ discussed above that if $z\in \tilde K^W$ then
\begin{align}
\sum_{y\in B_r} \mathbbm{1}(y\notin K, y+ z \in K_{y,y}^W)
&\geq \sum_{y\in B_r} \mathbbm{1}(K\cap(y+[0,z])=\emptyset)
=
|B_r|-  \sum_{y\in B_r} \mathbbm{1}(K\cap(y+[0,z])\neq \emptyset)\nonumber\\
&\geq |B_r| - \sum_{y\in B_r} |K \cap (y+[0,z])| 
= |B_r| - \sum_{x\in K} |[0,z] \cap (x+B_r)|.
\label{eq:K_yy_intersection_lower_initial}
\end{align}
 Let us stress that the inequality \eqref{eq:K_yy_intersection_lower_initial} holds deterministically regardless of what the clusters $K$ and $\tilde K^W$ and the point $z\in \tilde K^W$ are. 
It follows that if 
 $F$ is any non-negative function of a subgraph and a vertex then
 \begin{align}
&\E \left[\sum_{y\in B_r}|K| \sum_{w\in K_y^{W+y}} F(K_y^{W+y}-y,w-y) \mathbbm{1}(y \notin K)\right] 
\nonumber\\&\hspace{1cm}= \E \left[|K|\sum_{y\in B_r} \sum_{w\in K_{y,y}^W} F(K_{y,y}^W-y,w-y)\mathbbm{1}(y \notin K)\right]
\nonumber\\
&\hspace{1cm}\geq |B_r| \E |K| \E \left[\sum_{z\in K^W} F(K^W,z)\right] - \E \left[|K| \sum_{x\in K} \sum_{z\in \tilde K^W} F(\tilde K^W,z) |[0,z] \cap (x+B_r)|\right].
 \label{eq:Markov_intersection_restricted_to_tripod2}
 \end{align}
 Letting $\hathat \E$ be the law of the pair of random variables $(K,\tilde K^W)$ biased by $|K|^2$, we can rewrite the second term on the right hand side of \eqref{eq:Markov_intersection_restricted_to_tripod2} as
 \begin{multline*}
\E \left[|K| \sum_{x\in K} \sum_{z\in \tilde K^W} F(\tilde K^W,z) |[0,z] \cap (x+B_r)|\right] \\=  \E|K|^2 \hathat \E \left[ \frac{1}{|K|} \sum_{z\in \tilde K^W}\sum_{x\in K} F(\tilde K^W,z) |[0,z] \cap (x+B_r)|\right].
 \end{multline*}
Defining the probability measure $\mu=\mu_{\beta,r}$ by
 \[
\mu(x):= \hathat \E \frac{\mathbbm{1}(x\in K)}{|K|} = \frac{\E \left[|K|\mathbbm{1}(x\in K)\right]}{\E|K|^2}  ,
 \]
 we can use the independence of $K$ and $\tilde K^W$ to express this identity as
 \begin{multline*}
\E \left[|K| \sum_{x\in K} \sum_{z\in \tilde K^W} F(\tilde K^W,z) |[0,z] \cap (x+B_r)|\right] \\=  \E|K|^2 \E \left[ \sum_{x\in \Z^d}\sum_{z\in \tilde K^W} F(\tilde K^W,z) |[0,z] \cap (x+B_r)| \mu(x) \right].
 \end{multline*}
 The claim follows by substituting this identity into \eqref{eq:Markov_intersection_restricted_to_tripod2}, taking $F(H,z)= \mathbbm{1}(H\in \mathscr{A}$ and $\sum_{x\in \Z^d}|[0,z]_{H} \cap (x+B_r)|\mu(x) \leq \ell)$ and exchanging the roles of $0$ and $y$.
\end{proof}

The next lemma gives conditions under which we can compare the effect of restricting geodesics to have small intersection with our random ball to the effect of of restricting the total volume inside our random ball to be small.

\begin{lemma}
\label{lem:geodesic_to_bulk} Let $\mu$ be a probability measure on $\Z^d$. The inequality
\begin{multline*}\E_{\beta,r}\left[\sum_{z\in K^W} \mathbbm{1}\!\left(K^{W} \in \mathscr{A} \text{ and } \sum_{x\in \Z^d}|[0,z]_{K^W} \cap B_r(x)|\mu(x) \leq \ell \right)\right]
\\
\geq 
\E_{\beta,r}\left[|K^W| \mathbbm{1}\!\left(K^{W} \in \mathscr{A} \text{ and } \sum_{x\in \Z^d}|K^W \cap B_r(x)|\mu(x) \leq m  \right)\right] -
\frac{1}{\ell}\E_{\beta,r}\left[\min\{|K|,m\}\right] \E_{\beta,r}|K| 
\end{multline*}
holds for every $\beta,r,m,\ell\in (0,\infty)$ and every set $\mathscr{A}$ of finite subsets of $\Z^d$.
\end{lemma}

When we apply this lemma, we will take $\ell$ of order $r^\alpha$ and $m\ll r^{2\alpha}$ and use \cref{lem:fictitious_negligibility} to show that the negative term appearing in the lower bound is small.

\begin{proof}[Proof of \cref{lem:geodesic_to_bulk}]
We have trivially that
\begin{align*}
&\E_{\beta,r}\left[\sum_{z\in K^W} \mathbbm{1}\!\left(K^{W} \in \mathscr{A} \text{ and } \sum_{x\in \Z^d}|[0,z]_{K^W} \cap B_r(x)|\mu(x) \leq \ell \right)\right] 
\\
&\hspace{1cm}\geq
\E_{\beta,r}\left[|K^W| \mathbbm{1}\!\left(K^{W} \in \mathscr{A} \text{ and } \sum_{x\in \Z^d}|K^W \cap B_r(x)|\mu(x) \leq m  \right)\right]
\\&\hspace{2cm}-
\E_{\beta,r}\left[\sum_{z\in K^W} \mathbbm{1}\!\left(\sum_{x\in \Z^d}|K^W \cap B_r(x)|\mu(x) \leq m  \text{ and } \sum_{x\in \Z^d}|[0,z]_{K^W} \cap B_r(x)|\mu(x) >\ell  \right)\right].
\end{align*}
To bound the second term, we first use Markov's inequality to write
\begin{align*}
&\E_{\beta,r}\left[\sum_{z\in K^W} \mathbbm{1}\!\left( \sum_{x\in \Z^d}|[K^W \cap B_r(x)|\mu(x) \leq m,  \text{ and } \sum_{x\in \Z^d}|[0,z]_{K^W} \cap B_r(x)|\mu(x) >\ell \right)\right]
\\&\hspace{2cm}\leq 
\frac{1}{\ell} 
\E_{\beta,r}\left[\sum_{z\in K^W} \sum_{x\in \Z^d}|[0,z]_{K^W} \cap B_r(x)|\mu(x) \mathbbm{1}\!\left( \sum_{y\in \Z^d}|K^W \cap B_r(y)|\mu(y) \leq m \right)\right]
\\
&\hspace{2cm}=
\frac{1}{\ell} 
\E_{\beta,r}\left[\sum_{z\in K^W} \sum_{x\in \Z^d} \mu(x)\sum_{w\in B_r(x)} \mathbbm{1}\!\left(w\in [0,z]_{K^W} \text{ and } \sum_{y\in \Z^d}|K^W \cap B_r(y)|\mu(y) \leq m \right)\right].
\end{align*}
Now, observe that we have the inclusion of events
\begin{multline*}
\left\{w\in [0,z]_{K^W} \text{ and } \sum_{y\in \Z^d}|K^W \cap B_r(y)|\mu(y) \leq m \right\} 
\\
\subseteq \left\{w\in K^W \text{ and } \sum_{y\in \Z^d}|K^W \cap B_r(y)|\mu(y) \leq m \right\} \circ \{w \leftrightarrow z\}.
\end{multline*}
Indeed, the portion of the geodesic $[0,z]_{K^W}$ connecting $0$ to $w$ together with all the closed edges of the configuration is a witness for the first event on the right hand side which is disjoint from the remaining portion of the same geodesic, which is a witness for the second event on the right hand side. As such, we may apply Reimer's inequality \cite{MR1751301} to deduce that
\begin{align*}
&\E_{\beta,r}\left[\sum_{z\in K^W} \mathbbm{1}\!\left( \sum_{x\in \Z^d}|[K^W \cap B_r(x)|\mu(x) \leq m,  \text{ and } \sum_{x\in \Z^d}|[0,z]_{K^W} \cap B_r(x)|\mu(x) >\ell \right)\right]\\
&\hspace{1.8cm}\leq 
\frac{1}{\ell} 
\E_{\beta,r}\left[ \sum_{x\in \Z^d} \mu(x)\sum_{w\in B_r(x)} \mathbbm{1}\!\left(w\in K^W \text{ and } \sum_{y\in \Z^d}|K^W \cap B_r(y)|\mu(y) \leq m \right)\right] \E_{\beta,r}|K|
\\
&\hspace{3.6cm}=\frac{1}{\ell} 
\E_{\beta,r}\left[ \sum_{x\in \Z^d} |K^W \cap B_r(x)|\mu(x)  \mathbbm{1}\!\left(\sum_{y\in \Z^d}|K^W \cap B_r(y)|\mu(y) \leq m \right)\right] \E_{\beta,r}|K|.
\end{align*}
Finally, since $\sum_{x\in \Z^d} |K^W \cap B_r|\mu(x) \leq |K|$, we can use the trivial inequality
$x \mathbbm{1}(x\leq m) \leq \min\{x,m\} \leq \min\{y,m\}$, that holds whenever $y>x$, to obtain the desired estimate.
\end{proof}

We are now almost ready to prove \cref{lem:key_gluability_lemma_mu}. We begin with the first case, in which we assume that $\liminf_{r\to \infty}r^{-\alpha}\E_{\beta_c,r}|K| < \infty$ and that the hydrodynamic condition does not hold. The proof of this case will apply 
\cref{lem:fictitious_volume_tail_sharp}
 to lower bound the contribution to the $r$ derivative from two ``locally large'' clusters as recorded in the following lemma.

\begin{lemma}\label{lem:two_locally_large_clusters}
Suppose that $d=3\alpha$ and that the modified hydrodynamic condition does not hold. For each $\eta>0$ there exists a positive constant $c_1$ such that if $M_r\geq \eta r^{2\alpha}$ then
\[
  \E_{\beta,r}\left[ \sum_{y\in B_r} \mathbbm{1}(0\nleftrightarrow y) |K_y| \mathbbm{1}\!\left(|K \cap B_{2r}| \geq  c_1 r^{2\alpha} \right)\right] \geq c_1 r^{3\alpha}.
\]
\end{lemma}

\begin{proof}[Proof of \cref{lem:two_locally_large_clusters}]
 We first claim that for each $\eps>0$ and $\gamma \geq 1$ there exists $\delta>0$ such that
\begin{equation}
\label{eq:small_cluster_concentration_fictitious}
  \P_{\beta_c,r}\left( \sum_{y\in B_r} \mathbbm{1}(|K_y|\geq \delta r^{2\alpha}) \geq \gamma r^{2\alpha} \right) \geq 1-\eps.
\end{equation}
To prove this claim, we first apply \cref{lem:fictitious_volume_tail_sharp} to deduce that there exists a constant $c$ such that
\[
  \E_{\beta_c,r}\left[\sum_{y\in B_r} \mathbbm{1}(|K_y| \geq m)\right] = |B_r| \P_{\beta_c,r}(|K|\geq m) \asymp r^d m^{-1/2}
\]
for every $1\leq m \leq c r^{2\alpha}$.
On the other hand, we have by \eqref{eq:BK_disjoint_clusters_covariance3} that the variance of this number of points satisfies
\begin{multline*}
  \E_{\beta_c,r}\left[\sum_{y,z\in B_r} \mathbbm{1}(|K_y|,|K_z| \geq m)\right] - \E_{\beta_c,r}\left[\sum_{y\in B_r} \mathbbm{1}(|K_y|\geq m)\right]^2  
  \\\leq \E_{\beta_c,r}\left[ \sum_{y\in B_r} \mathbbm{1}(|K_y|\geq m) |K_y\cap B_r| \right]
  \leq |B_r| \E_{\beta_c,r} |K \cap B_{2r}| \preceq r^{d+\alpha},
\end{multline*}
where we ignored the constraint that $|K_y|\geq m$ in the penultimate inequality.
As such, if we take $m=\delta r^{2\alpha}$ for a small constant $\delta$, the variance of $\sum_{y\in B_r} \mathbbm{1}(|K_y| \geq m)$ is smaller than its squared mean by at least a factor of $\delta$. The claim follows by taking $\delta$ sufficiently small that the mean is larger than $2 \gamma r^{2\alpha}$ and the probability of a deviation of $\gamma r^{2\alpha}$ away from this mean is small by Chebyshev's inequality.

We now apply \eqref{eq:small_cluster_concentration_fictitious} to prove the lemma. We can cover the ball $B_{r}$ with a bounded number of balls $B^1_{r/2},\ldots,B^N_{r/2}$ of radius $r/2$, and applying \eqref{eq:small_cluster_concentration_fictitious} at the scale $r/2$ yields that for every $\eps>0$ and $\gamma \geq 1$ there exists $\delta>0$ such that
\[
  \P_{\beta_c,r}\left( 
  \sum_{y\in B_{r/2}^i}
   \mathbbm{1}(|K_y|\geq \delta r^{2\alpha}) 
   \geq \gamma r^{2\alpha} \text{ for every $1\leq i \leq N$} \right) \geq 1-\frac{\eps}{2}.
\]
On the other hand, the universal tightness theorem and the bound $M_r=O(r^{(d+\alpha)/2})$ of \eqref{eq:pre_Hydro} imply that if $M_r \geq \eta r^{2\alpha}$ then for each $\eps>0$ there exists a constant $C=C(\eta,\eps)$ such that
\[
\P( C^{-1} r^{2\alpha} \leq \max_z |K_z \cap B_r| \leq \max_z |K_z \cap B_{2r}|  \leq C r^{2\alpha}) \geq 1-\frac{\eps}{2},
\]
so that if we set $\gamma=2C(\eta,\eps)$ and take $\delta>0$ as above then
\begin{multline*}
  \P_{\beta_c,r}\Biggl( \sum_{y\in B_{r/2}^i} \mathbbm{1}(|K_y|\geq \delta r^{2\alpha}) \geq 2C r^{2\alpha} \text{ for every $1\leq i \leq N$}\\\text{and } C^{-1} r^{2\alpha} \leq  \max_z |K_z\cap B_r|\leq  \max_z |K_z\cap B_{2r}| \leq C r^{2\alpha} \Biggr) 
  \geq 1-\eps.
\end{multline*}
On the event whose probability we have just estimated, a uniform random vertex $x$ of $B_r$ has probability of order $r^{-\alpha}$ to have $|K_x\cap B_{2r}(x)|\geq |K_x \cap B_r| \geq C^{-1} r^{2\alpha}$ and for there to be at least $Cr^{2\alpha}$ points $y$ in the ball $B_r(x)$ that do not belong to $K_x$ but have $|K_y| \geq \delta r^{2\alpha}$, yielding the claimed inequality.
\end{proof}

\begin{proof}[Proof of \cref{lem:key_gluability_lemma_mu}, Case 1]
Suppose that $\liminf_{r\to \infty}r^{-\alpha}\E_{\beta_c,r}|K| = C_1 \frac{\alpha}{\beta_c}< \infty$ and let $\mathscr{R}$ denote the unbounded set $\{r\geq 1:\E_{\beta_c,r}|K| \leq 2 C_1 \frac{\alpha}{\beta_c}r^\alpha \}$.
It suffices to prove that if $\mu_r=\mu_{\beta_c,r}$ denotes the probability measure from \cref{lem:locally_large_measure_concrete} then
\[
\liminf_{\substack{r\to \infty \\ r\in \mathscr{R}}} r^{-\alpha} \sup_W \E_{\beta_c,r}\left[ |K^W| \mathbbm{1}\left(|K^W\cap B_r|\leq \delta r^{2\alpha} \text{ and } \sum_{x\in \Z^d}|K^W\cap B_r(x)|\mu_r(x) \leq \delta r^{2\alpha}\right)\right] < \frac{\alpha}{\beta_c}
\]
for some $\delta>0$, where we use the notation
\[
  \liminf_{\substack{ r\to\infty \\ r\in \mathscr{R}}} f(r):= \lim_{r\to \infty} \inf_{s \in \mathscr{R}, s\geq r} f(r).
\]
Indeed, the version of the claim in the statement of the lemma (which does not involve the constraint $|K^W\cap B_r|\leq \delta r^{2\alpha}$) follows from this version of the statement by considering the measure $\nu_r = \frac{1}{2}\mu_r + \frac{1}{2} \delta_0$ where $\delta_0$ is the Dirac delta measure at the origin.
  Suppose for contradiction that this claim does not hold.
By a standard compactness argument, there exists a decreasing function $\delta(r)$ with $\delta(r)\to 0$ as $r\to \infty$ such that
\begin{multline}
\label{eq:not_locally_large}
\hspace{-0.7em}\liminf_{\substack{r\to \infty \\ r\in \mathscr{R}}} r^{-\alpha} \sup_W \E_{\beta_c,r}\left[ |K^W| \mathbbm{1}\left(|K^W\cap B_r|\leq \delta(r) r^{2\alpha} \text{ and } \sum_{x\in \Z^d}|K^W\cap B_r(x)|\mu_r(x) \leq \delta(r) r^{2\alpha}\right)\right]\\ \geq \frac{\alpha}{\beta_c}.
\end{multline}
Write $\mathscr{A}_r$ for the set of subgraphs $H$ with $|H\cap B_r|\leq \delta(r) r^{2\alpha}$  and $\sum_{x\in \Z^d}|H\cap B_r(x)|\mu_r(x) \leq \delta(r) r^{2\alpha}$, so that \eqref{eq:not_locally_large} can be rewritten as
\begin{equation}
\label{eq:not_locally_large2}
\liminf_{\substack{r\to \infty \\ r\in \mathscr{R}}} r^{-\alpha}  \sup_W \E_{\beta_c,r}\left[ |K^W| \mathbbm{1}\left(K^W \in \mathscr{A}_r\right)\right]\\ \geq \frac{\alpha}{\beta_c}.
\end{equation}
It suffices to prove under these assumptions that
\begin{equation}
\label{eq:first_order_asymptotics_contradiction}
\limsup_{r\to\infty} r^{-\alpha} \E_{\beta_c,r}|K| \leq \frac{\alpha}{\beta_c}
\end{equation}
as $r\to \infty$. Indeed, the two estimates \eqref{eq:not_locally_large} and \eqref{eq:first_order_asymptotics_contradiction} are inconsistent with the failure of the hydrodynamic condition, which ensures that that there exists a constant $\delta>0$ such that
\[
\limsup_{r\to \infty} r^{-\alpha}\E\left[|K|\mathbbm{1}(|K\cap B_r|\right] > \delta r^{2\alpha}] \succeq \limsup_{r\to \infty}  r^{-d-\alpha} M_r^2 >0.
\]
To prove \eqref{eq:first_order_asymptotics_contradiction} it suffices in turn to prove (under our false assumptions) that for each $\gamma>1$ there exists $c=c(\gamma)>0$ and $r_0>0$ such that
\begin{align}
\left[\gamma \leq \frac{\beta_c}{\alpha} r^{-\alpha} \E_{\beta_c,r}|K| \leq 2C_1  \text{ and } r\geq r_0 \right] &\Rightarrow \left[\frac{\partial}{\partial r}\E_{\beta_c,r}|K| \geq (1+c) \alpha  r^{-1}\E_{\beta_c,r}|K|\right]
\label{eq:locally_small_implication_lambda=0}
\\
 &\Rightarrow \left[\frac{\partial}{\partial r} \log \frac{\E_{\beta_c,r}|K|}{r^\alpha} \geq c \alpha\right],\nonumber
\end{align}
where the second implication follows by calculus. Indeed, once this is proven, it will follow that for each $\gamma>1$ there exist constants $r_0$ and $C_2$ such that if $\E_{\beta_c,r}|K| \geq \gamma \frac{\alpha}{\beta_c} r^\alpha$ for some $r\geq r_0$ then $\E_{\beta_c,s}|K| \geq 2C_1 \frac{\alpha}{\beta_c}s^\alpha$ for all $s\geq C_2r$. This property is easily seen to imply the dichotomy that either $\limsup_{r\to \infty} \frac{\beta_c}{\alpha} r^{-\alpha}\E_{\beta_c,r}|K| \leq 1$ or $\liminf_{r\to \infty} \frac{\beta_c}{\alpha} r^{-\alpha}\E_{\beta_c,r}|K| \geq 2C_1$, and since we have assumed that the second possibility does not occur we must have that $\limsup_{r\to \infty} \frac{\beta_c}{\alpha} r^{-\alpha}\E_{\beta_c,r}|K| \leq 1$; this establishes \eqref{eq:first_order_asymptotics_contradiction} and contradicts our other assumptions as explained above. 

\medskip

It remains to prove the implication \eqref{eq:locally_small_implication_lambda=0}. Fix the constant $\gamma>1$.
If we define $\mathscr{R}_\gamma \subseteq \mathscr{R}$ to be the set of $r$ for which $\gamma \leq \frac{\beta_c}{\alpha} r^{-\alpha} \E_{\beta_c,r}|K| \leq 2C_1$ then the implication \eqref{eq:locally_small_implication_lambda=0} can be written equivalently as
\[
  \liminf_{\substack{ r\to\infty \\ r\in \mathscr{R}_\gamma}} \frac{r}{\E_{\beta_c,r}|K|}  \cdot \frac{\partial}{\partial r}\E_{\beta_c,r}|K| > \alpha.
\]
The same calculation used to bound the quantity $\mathcal{E}_{1,r}$ in \eqref{I-eq:E1_from_M} and \eqref{I-eq:critical_dim_E1rbeta} yields that there exists a constant $C_3$ such that
  \[
\frac{\partial}{\partial r} \E_{\beta_c,r}|K| \geq \beta_c |B_r| |J'(r)|\left(1-C_3 \frac{(\E_{\beta_c,r}|K|)^2 M_r^{1/2}}{r^{3\alpha}}\right)\left(\E_{\beta_c,r}|K|\right)^2.
 \]
Thus, there exist constants $\eta>0$, $c_1>0$ and $r_0<\infty$ (all depending on $\gamma$) such that if $r\geq r_0$, $\gamma \leq \frac{\beta_c}{\alpha}r^{-\alpha}\E_{\beta_c,r}|K| \leq 2 \gamma$, and $M_r \leq \eta r^{2\alpha}$ then
\[
\frac{\partial}{\partial r}\E_{\beta_c,r}|K| \geq (1+c_1)\alpha r^{-1}\E_{\beta_c,r}|K|
\]
as required. Defining $\mathscr{R}_\gamma' \subseteq \mathscr{R}_\gamma$ to be the set of $r$ for which $\gamma \leq \frac{\beta_c}{\alpha} r^{-\alpha} \E_{\beta_c,r}|K| \leq 2C_1$ and $M_r\geq \eta r^\alpha$, it therefore suffices to prove that
\[
  \liminf_{\substack{ r\to\infty \\ r\in \mathscr{R}_\gamma'}}  \frac{r}{\E_{\beta_c,r}|K|}\cdot  \frac{\partial}{\partial r}\E_{\beta_c,r}|K| > \alpha.
\]
By \cref{lem:fictitious_negligibility}, for each $\eps>0$ there exists a constant $\delta>0$ such that $\E_{\beta_c,r}\min\{|K|, \delta r^{2\alpha}\} \leq \eps r^\alpha$ for all $r\geq 1$. Since $\delta(r)\to 0$ as $r\to\infty$, it follows that 
\[
 \limsup_{\substack{ r\to\infty \\ r\in \mathscr{R}_\gamma'}} r^{-\alpha}\E_{\beta_c,r}\min\{|K|,\delta(r) r^{2\alpha}\} =0.
\]
Taking 
\[
  \ell(r):= \sqrt{ r^\alpha \E_{\beta_c,r}\min\{|K|,\delta(r) r^{2\alpha}\}},
\]
it follows that
\[
 \limsup_{\substack{ r\to\infty \\ r\in \mathscr{R}_\gamma'}} r^{-\alpha}\ell(r)=0 \qquad \text{ and } \qquad \limsup_{\substack{ r\to\infty \\ r\in \mathscr{R}_\gamma'}}\frac{r^\alpha}{\ell(r)} \E_{\beta_c,r}\min\{|K|,\delta(r)r^{2\alpha}\}=0
\]
so that if we apply \cref{lem:locally_large_measure_concrete,lem:geodesic_to_bulk} with $m=\delta(r)r^{2\alpha}$ and $\ell=\ell(r)$ then
\begin{multline}
 \liminf_{\substack{ r\to\infty \\ r\in \mathscr{R}_\gamma'}} \frac{\beta_c |J'(r)| r}{\E_{\beta_c,r}|K|}
\E_{\beta,r}\left[ \sum_{y\in B_r} \mathbbm{1}(0\nleftrightarrow y) |K_y| |K^W| \mathbbm{1}\!\left(K^{W} \in \mathscr{A}_r  \right)\right] 
 \\ \geq \liminf_{\substack{ r\to\infty \\ r\in \mathscr{R}_\gamma'}}\left(1-\frac{C \ell(r)}{r^{\alpha}} \right) r^{-\alpha} \beta_c \left(\E\left[|K^W| \mathbbm{1}(K^W\in \mathscr{A}_r)\right]-\frac{C r^\alpha}{\ell(r)} \E_{\beta_c,r}\min\{|K|,\delta(r)r^{2\alpha}\}\right)
 \\\geq \alpha.
 \label{eq:contradiction_small_cluster_contribution_to_derivative}
\end{multline}
Let $c_2=c_2(\gamma)>0$ be the constant from \cref{lem:two_locally_large_clusters} (called $c_1$ in the statement of that lemma) applied with $\eta=\eta(\gamma)$ as above and $A= 2 C_1 \frac{\alpha}{\beta_c}$. If $|K\cap B_r| \geq c_2 r^{2\alpha}$ and $r$ is sufficiently large that $\delta(r)\leq \frac{c_2}{2}$ then either $K^W \notin \mathscr{A}_r$ or $|K\setminus K^W|$ contains at least $\frac{c_2}{2} r^{2\alpha}$ points, so that if $r\in \mathscr{R}_\gamma'$ then
\begin{align*}
  \frac{1}{\beta_c|J'(r)|} \cdot \frac{\partial}{\partial r} \E_{\beta_c,r} |K| &= \E_{\beta_c,r}\left[\sum_{y\in B_r} \mathbbm{1}(0\nleftrightarrow y) |K||K_y|\right]\\
  & \geq \E_{\beta,r}\left[ \sum_{y\in B_r} \mathbbm{1}(0\nleftrightarrow y) |K_y| |K^W| \mathbbm{1}\!\left(K^{W} \in \mathscr{A}_r  \right)\right] 
  \\
  &\hspace{3cm}+ 
 \E_{\beta,r}\left[ \sum_{y\in B_r} \mathbbm{1}(0\nleftrightarrow y) |K_y| \mathbbm{1}\!\left(|K \cap B_r| \geq  c_2 r^{2\alpha} \right)\right] \frac{c_2}{2} r^{2\alpha}\\
 &\geq \E_{\beta,r}\left[ \sum_{y\in B_r} \mathbbm{1}(0\nleftrightarrow y) |K_y| |K^W| \mathbbm{1}\!\left(K^{W} \in \mathscr{A}_r  \right)\right] + c_4 r^{5\alpha},
\end{align*}
where the existence of a constant $c_4>0$ making the final inequality hold follows from \cref{lem:two_locally_large_clusters}. 
Putting this together with \eqref{eq:contradiction_small_cluster_contribution_to_derivative} we deduce that
\[
  \liminf_{\substack{ r\to\infty \\ r\in \mathscr{R}_\gamma'}} \frac{r}{\E_{\beta_c,r}|K|}\cdot \frac{\partial}{\partial r} \E_{\beta_c,r} |K| \geq \alpha + c_4 \beta_c > \alpha
\]
as desired.
\end{proof}

We now turn to the second claim, concerning the case in which the modified hydrodynamic condition does not hold and $\liminf_{r\to \infty}\E_{\beta_c,r}|K|=\infty$.

\begin{proof}[Proof of \cref{lem:key_gluability_lemma_mu}, Case 2]
Let $\mu_{r,\lambda}=\mu_{\beta(r,\lambda),r}$ be the probability measures from \cref{lem:locally_large_measure_concrete}. Suppose for contradiction that the
modified hydrodynamic condition does not hold, that the susceptibility satisfies $\liminf_{r\to\infty} r^{-\alpha}\E_{\beta_c,r}|K|=\infty$, and that for every $0<\lambda_0 \leq 1$ there exists $0<\lambda\leq \lambda_0$ such that 
\begin{equation*}
 \liminf_{r\to \infty} \sup_W r^{-\alpha}\tilde \E_{r,\lambda} \left[ |K^W| \mathbbm{1}\left(
 \sum_{x\in \Z^d}|K^W\cap B_r(x)|\mu_{r,\lambda}(x) \leq \delta r^{2\alpha}\right)\right] \\\geq  \frac{\alpha}{\beta_c} \Bigl[1+ \lambda \alpha |J|\Bigr]^{-1}
\end{equation*}
for every $\delta>0$. Call the parameter $\lambda$ \textbf{bad} if this estimate holds for all $\delta>0$. By a standard compactness argument, we have that if $\lambda$ is bad then there exists a function $\delta_\lambda(r)$ with $\delta_\lambda(r)\to 0$ as $r\to \infty$ such that if we define 
 $\mathscr{A}_{r,\lambda}$ to be the set of subgraphs with $|H\cap B_r|\leq \delta_\lambda(r)r^{2\alpha}$ and $\sum_{x\in\Z^d} |H\cap B_r(x)|\mu(x) \leq \delta_\lambda(r)r^{2\alpha}$ then
\begin{equation*}
 \liminf_{r\to \infty} \sup_W r^{-\alpha}\tilde \E_{r,\lambda} \left[ |K^W| \mathbbm{1}\left(K^W \in \mathscr{A}_{r,\lambda}\right)\right] \\\geq  \frac{\alpha}{\beta_c} \Bigl[1+ \lambda \alpha |J|\Bigr]^{-1}.
\end{equation*}
On the other hand, since we assumed that  $\liminf_{r\to\infty} r^{-\alpha}\E_{\beta_c,r}|K|=\infty$, we have by \cref{cor:bad_case_large_susceptibility_lambda} that there exist positive constants $c_1$ and $\lambda_0\leq 1$ such that
\[
\liminf_{r\to \infty}\tilde\E_{r,\lambda}|K| \geq \frac{c_1}{\lambda} 
\]
for all $0<\lambda \leq 1$ and the set 
\[
\mathscr{K}_\lambda:=\{k\in \N: \tilde \E_{r,\lambda}|K|^2 \geq c_1 (\tilde \E_{r,\lambda}|K|)^3 \text{ for all $r\in [2^k,2^{k+2}]$}\}
\]
has lower density at least $c_1$ for every $0<\lambda \leq \lambda_0$. Letting $\mathscr{R}_\lambda=\bigcup_{k\in \mathscr{K}_\lambda}[2^k,2^{k+1}]$, it follows from \cref{lem:second_moment_differential_inequality} and \cref{lem:tilde_susceptibility_lower} that there exists a positive constant $c_2$, not depending on the choice of $\lambda$, such that if $0<\lambda \leq \lambda_0$ then
\begin{multline}
\label{eq:extra_contribution_to_derivative_on_good_scales}
\frac{1}{\beta_c|J'(r)| |B_r| \tilde \E_{r,\lambda}|K|} \frac{\partial}{\partial r} \tilde \E_{r,\lambda}|K| 
\\\geq \left(1- \frac{\ell (\tilde\E_{r,\lambda}|K|)^2}{r^{3\alpha}}\right) \left(\tilde\E_{r,\lambda}\left[|K^W| \mathbbm{1}(K^W\in \mathscr{A}_r)\right]-\frac{C r^\alpha}{\ell} \tilde\E_{r,\lambda}\min\{|K|,\delta(r)r^{2\alpha}\}\right)
\\+ c_2 \mathbbm{1}(r \in \mathscr{R}_\lambda),
\end{multline}
for every $r,\ell\geq 1$, where the first term accounts for the $r$-derivative of $\E_{\beta,r}|K|$ and the second term accounts for the $\beta$-derivative of $\E_{\beta,r}$. (In the contribution to the right hand side from the $\beta$ derivative when $r\in \mathscr{R}_\lambda$, we get a $\lambda^{-1}$ in the lower bound on $\tilde \E_{r,\lambda}|K|^2/ (\tilde \E_{r,\lambda}|K|)^2 \succeq \tilde \E_{r,\lambda}|K|$ and a factor of $\lambda$ from $\frac{\partial}{\partial r} \beta(r,\lambda)$, so that the two factors of $\lambda$ cancel and we are left with a constant as written above.)
By \cref{lem:fictitious_negligibility}, for each $\eps>0$ there exists a constant $\delta>0$ such that $\tilde \E_{r,\lambda}\min\{|K|,\delta r^{2\alpha}\} \leq \E_{\beta_c,r}\min\{|K|, \delta r^{2\alpha}\} \leq \eps r^\alpha$ for all $r\geq 1$. Since $\delta_\lambda(r)\to 0$ as $r\to\infty$, it follows that 
\[
 \limsup_{r\to\infty} r^{-\alpha}\tilde\E_{r,\lambda}\min\{|K|,\delta_\lambda(r) r^{2\alpha}\} =0.
\]
On the other hand, we also have by \cref{prop:fictitious_derivative_and_susceptibility_best} that
$\limsup_{r\to \infty}r^{-\alpha}\tilde\E_{r,\lambda}|K|<\infty$, so that if we take
\[
  \ell_\lambda(r):= \sqrt{ r^\alpha \tilde\E_{r,\lambda}\min\{|K|,\delta_\lambda(r) r^{2\alpha}\}}
\]
then
\[
 \limsup_{r\to\infty} \frac{\ell_\lambda(r) (\tilde\E_{r,\lambda}|K|)^2}{r^{3\alpha}}=0 \qquad \text{ and } \qquad \limsup_{r\to\infty}\frac{r^\alpha}{\ell_\lambda(r)} \E_{\beta_c,r}\min\{|K|,\delta(r)r^{2\alpha}\}=0,
\]
and hence by \eqref{eq:extra_contribution_to_derivative_on_good_scales} that if $0<\lambda \leq \lambda_0$ is bad then
\[
 \frac{\partial}{\partial r} \log \tilde \E_{r,\lambda}|K| \geq (\alpha-o(1))\left[1+\lambda \alpha |J|\right]^{-1} r^{-1} + c_2 \mathbbm{1}(r\in \mathscr{R}_\lambda) r^{-1},
\]
where the rate of decay of the implicit error may depend on $\lambda$.
Integrating this inequality between the two parameters $1$ and $r$ yields that there exists a positive constant $c_3$ such that if $0<\lambda \leq \lambda_0$ is bad then
\begin{align*}
\limsup_{r \to \infty} \frac{1}{\log r} \log \tilde \E_{R,\lambda}|K| &\geq \limsup_{r\to \infty} \frac{1}{\log r} \int_1^r \left[ (\alpha-o(1))\left[1+\lambda \alpha |J|\right]^{-1} r^{-1} + c_2  \mathbbm{1}(r\in \mathscr{R}_\lambda) r^{-1}\right] \dif r\\
&\geq \alpha \left[1+\lambda \alpha |J|\right]^{-1} + c_3,
\end{align*}
where in the second line we used that $\mathscr{K}_\lambda$ has lower density at least $c_1$. If $\lambda$ is sufficiently small then the right hand side is strictly larger than $\alpha$, contradicting \cref{prop:fictitious_derivative_and_susceptibility_best}. Thus, there exists $0<\lambda_1 \leq \lambda_0$ such that every $0<\lambda \leq \lambda_1$ is not bad, completing the proof. \qedhere
\end{proof}

\subsection{Reaching a contradiction}
\label{sub:reaching_a_contradiction}

We now want to put together all the properties we have proven to hold under the fictitious assumption that the modified hydrodynamic condition does not hold with $d=3\alpha$ and deduce a contradiction. The rough idea is that the lower bound $\tilde D_{r,\lambda} \geq C_\lambda (\tilde \E_{r,\lambda}|K|)^2$ means that large clusters grown from two nearby vertices are only weakly interacting, whereas the key gluability lemma (\cref{lem:key_gluability_lemma_mu}) proven in the previous subsection should mean that there is an unbounded set of scales on which these two clusters have a good probability to merge, which should force $\tilde D_{r,\lambda} =o((\tilde \E_{r,\lambda}|K|)^2)$.

\medskip

To proceed, we will first need to transform the statement of \cref{lem:key_gluability_lemma_mu} into a statement about clusters being ``locally large'' on many scales with \emph{high} probability under the size-biased measure. The precise notion of ``locally large'' we use will be defined with respect to the standard monotone coupling of the measures $(\tilde \E_{r,\lambda})_{r\geq 0}$, and we denote probabilities and expectations taken with respect to these coupled configurations by $\tilde \P_\lambda$
 and $\tilde \E_\lambda$; in the case $\lambda=0$ we will sometimes write $\P=\tilde\P_0$ and $\E=\tilde\E_0$.
Let $\mu_{r,\lambda}$ be the probability measures from \cref{lem:key_gluability_lemma_mu}.
For each $\delta>0$ and $y\in \Z^d$, we say that a scale $r$ is \textbf{$(y,\delta)$-good} if
$\sum_{x} \mu_{r,\lambda}(x)|K^r\cap B_{4r}(x)|$ and $\sum_{x} \mu_{r,\lambda}(x)|K^r_y\cap B_{4r}(x)|$ are both at least $\delta r^{2\alpha}$, and that $r$ is $(y,\delta)$\textbf{-bad} otherwise. Note that 
the property of being $\delta$-good gets stronger as $\delta$ gets larger.

\begin{lemma}
\label{lem:gluing_good}
Suppose that $d=3\alpha$ and that the modified hydrodynamic condition does \emph{not} hold.
\begin{enumerate}
  \item If $\liminf_{r\to \infty} r^{-\alpha}\E_{\beta_c,r}|K|<\infty$, then 
 there exist constants $\delta>0$ and $C<\infty$ and an unbounded set $\mathscr{R} =\{r_1,r_2,\ldots\} \subset [1,\infty)$ 
with $r_{i+1} \geq 32r_i$ for every $i \geq 1$ such that  $\sup_{r\in \mathscr{R}} r^{-\alpha}\E_{\beta_c,r}|K|<\infty$ and 
\[
  \E \left[ \mathbbm{1}(0 \nxleftrightarrow{r} y)|K^r||K^r_y| \mathbbm{1}\left(\sum_{i=n+1}^{n+N} \mathbbm{1}\left(r_i \text{ \emph{is $(y,\delta)$-bad}}\right) \geq (1-\delta)N\right)\right] \leq C(1-\delta)^{N} (\E_{\beta_c,r}|K|)^2.
\]
for every $n,N \geq 1$, $r \geq r_{n+N}$, and $y\in \Z^d$ with $\|y\|\leq r_{n+1}$.
  \item If $\liminf_{r\to \infty} r^{-\alpha}\E_{\beta_c,r}|K|=\infty$, then there exists $0<\lambda_0\leq 1$ such that if $0<\lambda \leq \lambda_0$ then there exist constants $\delta=\delta_\lambda>0$ and $C=C_\lambda<\infty$ and an unbounded set $\mathscr{R}=\mathscr{R}_\lambda =\{r_1,r_2,\ldots\} \subset [1,\infty)$ 
with $r_{i+1} \geq 32 r_i$ for every $i \geq 1$ such that 
\[
  \tilde\E_{\lambda} \left[ \mathbbm{1}(0 \nxleftrightarrow{r} y)|K^r||K^r_y| \mathbbm{1}\left(\sum_{i=n+1}^{n+N} \mathbbm{1}\left(r_i \text{ \emph{is $(y,\delta)$-bad}}\right) \geq (1-\delta)N\right)\right] \leq C(1-\delta)^{N} (\tilde\E_{r,\lambda}|K|)^2.
\]
for every $n,N \geq 1$, $r \geq r_{n+N}$, and $y\in \Z^d$ with $\|y\|\leq r_{n+1}$.
\end{enumerate}
\end{lemma}

Note that the sequence $r_1 \leq r_2 \leq \cdots$ constructed in the proof of this lemma has extremely rapid growth, but this does not cause any problems for our applications.

\begin{proof}[Proof of \cref{lem:gluing_good}]
We give a single proof covering both cases of the lemma: 
The set $\mathscr{R}$ along with all constants in this proof will be permitted to depend on the choice of $\lambda$.
Suppose that the modified hydrodynamic condition does not hold, and let  $0<\lambda_0\leq 1$ be the minimum of the constants appearing in \cref{prop:fictitious_derivative_and_susceptibility_best} and \cref{lem:key_gluability_lemma_mu}. Suppose either that $\lambda=0$ and that $\limsup_{r\to \infty}r^{-\alpha} \E_{\beta_c,r}|K| < \infty$ or that $0<\lambda \leq \lambda_0$ and that $\limsup_{r\to \infty}r^{-\alpha} \E_{\beta_c,r}|K| = \infty$.
By \cref{prop:fictitious_derivative_and_susceptibility_best} and \cref{lem:key_gluability_lemma_mu}, there exists $\delta>0$ and an unbounded set $\mathscr{R}'$ such that 
\[\sup_{ r\in \mathscr{R}'}  r^{-\alpha}\tilde \E_{r,\lambda}|K| =: C_0 < \infty \]
and
\begin{multline}
\label{eq:R'_def_bad_scales}
\limsup_{\substack{r\to \infty\\ r\in \mathscr{R}'}} \sup_W r^{-\alpha}\tilde \E_{r,\lambda} \left[ |K^W| \mathbbm{1}\left(
 \sum_{x\in \Z^d}|K^W\cap B_r(x)|\mu_{r,\lambda}(x) \leq \delta r^{2\alpha}\right)\right] \\\leq (1-\delta)^3  \frac{\alpha}{\beta_c} \bigl(1+ \lambda \alpha |J|\bigr)^{-1}.
\end{multline}
We will fix this value of $\delta>0$ for the remainder of the proof, and  may take $\mathscr{R}'$ to be closed since taking the closure does not affect either of these properties.
Noting as in \eqref{eq:degree_defecit} that
\[
  \sum_{x \in \Z^d}|\beta_c J(0,x)-\beta(r,\lambda)J_r(0,x)|  \sim \bigl(1+ \lambda \alpha |J|\bigr)  \frac{\beta_c}{\alpha}r^{-\alpha},
\]
we define the set $\mathscr{R}=\{r_1,r_2,\ldots\}$ by taking $r_1 \in \mathscr{R}'$ to be minimal such that
\[
\sum_{x \in \Z^d}|\beta_c J(0,x)-\beta(r,\lambda)J_r(0,x)| \leq (1-\delta)^{-1} \bigl(1+ \lambda \alpha |J|\bigr)  \frac{\beta_c}{\alpha}r^{-\alpha}
\]
for every $r\geq r_1$
 and inductively setting $r_{i+1}\in \mathscr{R}'$ minimal such that $r_{i+1}\geq 32 r_i$,
\[
 \sum_{x \in \Z^d \setminus B_{r_{i+1}}}|\beta_c J(0,x)-\beta(r_i,\lambda)J_{r_i}(0,x)| \leq \frac{e^{-i}}{2C_0^2},
 \]
 and
\[\frac{\beta_c}{\alpha} r_i^{-\alpha} \tilde \E |K^{r_i}\setminus B_{r_{i+1}}|  \leq \frac{e^{-i}}{2C_0 }  (1-\delta) \bigl(1+ \lambda \alpha |J|\bigr)^{-1}.
\]
It will not be necessary for us to establish any bounds on the growth of the sequence $(r_i)_{i\geq 1}$.

\medskip

We begin by proving a version of the desired estimate involving only a single cluster.
Say that a scale $s$ is \textbf{half-good} if $\sum_{x\in \Z^d} \mu_{s,\lambda}(x)|B_{4s}(x)\cap K^s| \geq \delta s^{2\alpha}$ and that $s$ is \textbf{double-bad} otherwise. 
We claim that 
there exist constants $C_1<\infty$ and $\eta_1>0$ such that
\begin{equation}
\label{eq:double_bad}
  \tilde\E_{\lambda} \left[|K^r| \mathbbm{1}\left(\sum_{i=n}^{n+N} \mathbbm{1}\left(r_i \text{ is double-bad}\right) \geq (1-\eta_1)N\right)\right] \leq C_1(1-\eta_1)^{N} \tilde\E_{\lambda}|K^r|
\end{equation}
for every $n,N \geq 1$ and $r \geq r_{n+N}$.
To prove this, it suffices to prove that
 there exist constants $C_2<\infty$ and $\eta_2>0$ such that if $\mathscr{B}$ is any finite subset of $\mathscr{R}$ then
\[
  \tilde\E_{\lambda} \left[|K^r| \mathbbm{1}(\text{every scale in $\mathscr{B}$ is double-bad})\right] \leq C_2(1-\eta_2)^{|\mathscr{B}|} \tilde\E_{\lambda}|K^r|
\]
for every $r \geq \max \mathscr{B}$.
  Once this is proven, the claimed inequality \eqref{eq:double_bad} will follow by taking a union bound over all possible choices of subsets of $\{n,\ldots,n+N\}$ of size at least $(1-\eps)N$ for $\eps>0$ an appropriately small constant.

\medskip

Let $\mathscr{B} \subseteq \mathscr{R}$ be finite and enumerate $\mathscr{B}$ in increasing order as $\mathscr{B}=\{s_1,\ldots,s_{|\mathscr{B}|}\}$.
   Let $\mathcal{F}_0$ denote the trivial sigma-algebra, and for each $i\geq 1$ let $\mathcal{F}_i$ denote the sigma-algebra generated by the increasing sequence of clusters $(K^{s_1},\ldots,K^{s_i})$. We wish to estimate the conditional expectation $\tilde\E_{\lambda} \left[|K^{s_{i+1}}|\mathbbm{1}(s_{i+1} \text{ double-bad}) \mid \mathcal{F}_i \right]$. Applying a union bound and the BK inequality to sum over the locations of the first open edge not belonging to $K^{s_i}$ on a path in $K^{s_{i+1}}$ yields that
\begin{multline}
  \tilde\E_{\lambda} \left[|K^{s_{i+1}} \setminus K^{s_i}|\mathbbm{1}(s_{i+1} \text{ double-bad}) \mid \mathcal{F}_i \right] 
  \leq 
  \sum_{a\in K^{s_i}} \sum_{b\in \Z^d} |\beta(s_{i+1},\lambda)J_{s_{i+1}}(a,b)-\beta(s_{i},\lambda)J_{s_{i}}(a,b)| \\\cdot \sup_W \tilde\E_{s_{i+1},\lambda} \left[ |K^W_b| \mathbbm{1}\left(\sum_{x\in \Z^d} \mu_{s_{i+1},\lambda}(x)|B_{4s_{i+1}}(x)\cap K| < \delta s_{i+1}^{2\alpha}\right) \right],
  \label{eq:double_bad_pivotal_edge_expansion}
\end{multline}
where the supremum over sets $W$ accounts for the information that has already been revealed in the sigma-algebra $\cF_i$. If $b\in B_{2s_{i+1}}$ then we can bound 
\begin{multline*}
  \sup_W \tilde\E_{s_{i+1},\lambda} \left[ |K^W_b| \mathbbm{1}\left(\sum_{x\in \Z^d} \mu_{s_{i+1},\lambda}(x)|B_{4s_{i+1}}(x)\cap K| < \delta s_{i+1}^{2\alpha}\right) \right]\\\leq \sup_W \tilde\E_{s_{i+1},\lambda} \left[ |K^W| \mathbbm{1}\left(\sum_{x\in \Z^d} \mu_{s_{i+1},\lambda}(x)|B_{s_{i+1}}(x)\cap K| < \delta s_{i+1}^{2\alpha}\right) \right] \\
  \leq (1-\delta)^3 \frac{\alpha}{\beta_c} (1+\lambda \alpha |J|)^{-1} s_{i+1}^\alpha, 
\end{multline*}
so that
\begin{align*}
&\tilde\E_{\lambda} \left[|K^{s_{i+1}}|\mathbbm{1}(s_{i+1} \text{ double-bad}) \mid \mathcal{F}_i \right]
 \leq |K^{s_i}|+
|K^{s_i} \cap B_{s_{i+1}}| (1-\delta)^{2}\frac{s_{i+1}^{\alpha}}{s_i^{\alpha}} 
\\&\hspace{5cm}+ |K^{s_i}\setminus B_{s_{i+1}}| (1-\delta)^{-1} \bigl(1+ \lambda \alpha |J|\bigr)  \frac{\beta_c}{\alpha}s_i^{-\alpha} \tilde\E_\lambda |K^{s_{i+1}}| 
\\&\hspace{5cm}+ |K^{s_i} \cap B_{s_{i+1}}|  \sum_{x \in \Z^d \setminus B_{s_{i+1}}}|\beta_c J(0,x)-\beta(s_i,\lambda)J_{s_i}(0,x)| \tilde \E_\lambda|K^{s_{i+1}}|,
\end{align*}
where the term on the second line bounds the contributions to \eqref{eq:double_bad_pivotal_edge_expansion} from $a \notin B_{s_{i+1}}$ and the term on the third line bounds the contribution from pairs $(a,b)$ with $a\in B_{s_{i+1}}$ and $b\notin B_{2s_{i+1}}$.
Taking expectations and bounding $\tilde\E |K^{s_{i}}|  \leq C_0 s_{i}^\alpha$, it follows by induction on $i$ that there exists a constant $C_2$ such that
\begin{align*}
  &s_{i+1}^{-\alpha}\tilde \E_\lambda \left[|K^{s_{i+1}}|\mathbbm{1}(s_{1},\ldots,s_{i+1} \text{ all double-bad}) \right]\\
  &\hspace{4cm}\leq (1-\delta)s_i^{-\alpha}\tilde \E_\lambda \left[|K^{s_{i}}|\mathbbm{1}(s_{1},\ldots,s_{i} \text{ all double-bad}) \right] + e^{-i}\\
  &\hspace{4cm} \leq (1-\delta)^i s_1^{-\alpha}\tilde \E_\lambda \left[|K^{s_{1}}|\mathbbm{1}(s_{1} \text{ double-bad}) \right]
  + \sum (1-\delta)^{i-j}e^{-j} 
  \\&\hspace{4cm}\leq C_2 (1-\delta)^i 
\end{align*}
as claimed, where we used that $(s_i,s_{i+1})=(r_j,r_k)$ for some $k> j \geq i$ to bound the expression $\sum_{x \in \Z^d \setminus B_{s_{i+1}}}|\beta_c J(0,x)-\beta(s_i,\lambda)J_{s_i}(0,x)|$ and the expectation of $|K^{s_i}\setminus B_{s_{i+1}}|$ using the definition of the lacunary sequence $(r_i)_{i\geq 1}$. This completes the proof of \eqref{eq:double_bad}.

\medskip

We now apply \eqref{eq:double_bad} to prove the statement of the lemma. Let $r\geq r_N$, let $\cG_0$ be the sigma-algebra generated by $(K^{r_{n+1}},\ldots,K^{r_{n+N}})$ and $K^r$, 
let $N_g$ be the number of scales in $\{r_{n+1},\ldots,r_{n+N}\}$ that are not double-bad and, for each $1\leq i\leq N_g$, let $R_i$ be the $i$th scale in $\{r_{n+1},\ldots,r_N\}$ that is not double-bad. Note that $N_g$ and the sequence $R_1,\ldots,R_{N_g}$ are all measurable with respect to $\cG_0$. For each $i\geq 1$ let $\cG_i$ be the sigma-algebra generated by $\cG_0$ and $\{K^{r}_y:r\leq R_{\max \{i,N_g\}}\}$. We have by a similar argument to above that if $1\leq i \leq N_g$ and $\|y\|\leq R_i$ then
\begin{multline*}
  \tilde \E_\lambda \left[|K^{R_{i+1}}_y| \mathbbm{1}(y \notin K^{r}, R_{i+1} \text{ bad}) \mid \cG_i \right] \leq 
   (1-\delta) |K^{R_i}_y|\mathbbm{1}(y \notin K^{r})
   \\ + |K^{R_i}_y\setminus B_{R_{i+1}}|  \mathbbm{1}(y \notin K^{r}) (1-\delta)^{-1} \bigl(1+ \lambda \alpha |J|\bigr)  \frac{\beta_c}{\alpha}R_i^{-\alpha} \tilde\E |K^{R_{i+1}}|  
   \\+ |K^{R_{i+1}}_y \cap B_{R_i}|  \mathbbm{1}(y \notin K^{r}) \sum_{x \in \Z^d \setminus B_{R_{i+1}}}|\beta_c J(0,x)-\beta(R_i,\lambda)J_{R_i}(0,x)| \tilde \E|K^{R_{i+1}}|.
\end{multline*}
 It follows as before that there exists a constant $C_3$ such that if $\mathscr{B}$ is any subset of $\{R_1,\ldots,R_{N_g}\}$ and $\|y\|\leq R_1$ then
\begin{align*}
\tilde \E_\lambda \left[|K^{r}_y| \mathbbm{1}(y \notin K^{r},  \text{ every scale in $\mathscr{B}$ is bad}) \mid \cG_0 \right]
 \leq C_3(1-\delta)^{|\mathscr{B}|} \tilde \E|K^r|
\end{align*}
for every $r \geq R_{N_g}$. As before, it follows that there exist positive constants $\eta_1$ and $\eta_2$ such that
\begin{align*}
\tilde \E_\lambda \left[|K^{r}_y| \mathbbm{1}\left(y \notin K^{r}, \sum_{i=1}^{N_g} \mathbbm{1}(R_i \text{ is bad}) \geq \eta_1 N_g\right) \;\middle|\; \cG_0 \right]
 \leq C_3(1-\eta_2)^{N_g} \tilde \E|K^r|
\end{align*}
for every $r \geq R_{N_g}$. The claim follows easily from this estimate together with \eqref{eq:double_bad}. \qedhere
\end{proof}

Our next goal is to replace the randomly centered balls in the definition of a good scale with balls centered at the origin.
Given $y\in \Z^d$ and $\eta>0$, we say that a scale $r$ is \textbf{$(y,\eta)$-very good} if
$|K^r\cap B_{16r}|$ and $|K^r_y\cap B_{16r}|$ are both at least $\eta r^{2\alpha}$.

\begin{lemma}
\label{lem:gluing_very_good}
Suppose that $d=3\alpha$ and that the modified hydrodynamic condition does \emph{not} hold. 
\begin{enumerate}
  \item If $\liminf_{r\to \infty} r^{-\alpha}\E_{\beta_c,r}|K| < \infty$, then there exists an unbounded set $\mathscr{R} =\{r_1,r_2,\ldots\} \subset [1,\infty)$ 
with $r_{i+1} \geq 32 r_i$ for every $i \geq 1$ such that $\sup_{r\in \mathscr{R}} r^{-\alpha}\E_{\beta_c,r}|K|<\infty$ and for each $\eps>0$ there exists $\eta>0$ and $n_0<\infty$ such that
\[
  \E \left[ \mathbbm{1}(0 \nxleftrightarrow{r} y)|K^r||K^r_y| \mathbbm{1}\left(\sum_{i=n+1}^{n+N} \mathbbm{1}\left(r_i \text{ is $(y,\eta)$-very good}\right) \leq \eta N\right)\right] \leq \eps (\E|K^r|)^2.
\]
for every $n,N \geq n_0$, $r \geq r_{n+N}$, and $y\in \Z^d$ with $\|y\|\leq r_{n+1}$.
\item If $\liminf_{r\to \infty} r^{-\alpha}\E_{\beta_c,r}|K| = \infty$, then there exists $0<\lambda_0 \leq 1$ such if $0<\lambda \leq \lambda_0$ then there exists an unbounded set $\mathscr{R}=\mathscr{R}_\lambda =\{r_1,r_2,\ldots\} \subset [1,\infty)$ 
with $r_{i+1} \geq 32 r_i$ for every $i \geq 1$ such that for each $\eps>0$ there exists $\eta=\eta_\lambda>0$ and $n_0=n_{0,\lambda}<\infty$ such that
\[
  \tilde\E_{\lambda} \left[ \mathbbm{1}(0 \nxleftrightarrow{r} y)|K^r||K^r_y| \mathbbm{1}\left(\sum_{i=n+1}^{n+N} \mathbbm{1}\left(r_i \text{ is $(y,\eta)$-very good}\right) \leq \eta N\right)\right] \leq \eps (\tilde\E_{\lambda}|K^r|)^2.
\]
for every $n,N \geq n_0$, $r \geq r_{n+N}$, and $y\in \Z^d$ with $\|y\|\leq r_{n+1}$.
\end{enumerate}
\end{lemma}

\begin{proof}[Proof of \cref{lem:gluing_very_good}]
As with \cref{lem:gluing_good}, we give a single proof covering both cases of the lemma: 
The set $\mathscr{R}$ along with all constants in this proof will be permitted to depend on the choice of $\lambda$. When the choice of $y$ is unambiguous we drop it from the notation for scales being good, very good, or bad.
Suppose that the modified hydrodynamic condition does not hold, and let  $0<\lambda_0\leq 1$ be as in \cref{lem:gluing_good}. Suppose either that $\lambda=0$ and that $\limsup_{r\to \infty}r^{-\alpha} \E_{\beta_c,r}|K| < \infty$ or that $0<\lambda \leq \lambda_0$ and that $\limsup_{r\to \infty}r^{-\alpha} \E_{\beta_c,r}|K| = \infty$.
Let $\mathscr{R}=\mathscr{R}_\lambda$ and $\delta>0$ be as in \cref{lem:gluing_good} and let $0<\eta \leq \delta/2$. For each $n,N\geq 1$, $y\in \Z^d$ with $\|y\|\leq r_{n+1}$, and $r\geq r_{n+N}$ we have the union bound
\begin{align*}
  &\tilde\E_{\lambda} \left[ \mathbbm{1}(0 \nxleftrightarrow{r} y)|K^r||K^r_y| \mathbbm{1}\left(\sum_{i=n+1}^{n+N} \mathbbm{1}\left(r_i \text{ is $\eta$-very good}\right) \leq \eta N\right)\right]
  \\
  &\hspace{1.2cm}\leq 
   \tilde\E_{\lambda} \left[ \mathbbm{1}(0 \nxleftrightarrow{r} y)|K^r||K^r_y| \mathbbm{1}\left(\sum_{i=n+1}^{n+N} \mathbbm{1}\left(r_i \text{ is $\delta$-bad}\right) \geq \delta N\right)\right]
  \\
  &\hspace{2.4cm}+
  \tilde\E_{\lambda} \left[ \mathbbm{1}(0 \nxleftrightarrow{r} y)|K^r||K^r_y| \mathbbm{1}\left(\sum_{i=n+1}^{n+N} \mathbbm{1}\left(r_i \text{ is $\delta$-good but not $\eta$-very good}\right) \geq \frac{\delta}{2} N\right)\right]
\end{align*}
and hence by \cref{lem:gluing_good} and Markov's inequality that
\begin{align*}
  &\tilde\E_{\lambda} \left[ \mathbbm{1}(0 \nxleftrightarrow{r} y)|K^r||K^r_y| \mathbbm{1}\left(\sum_{i=n+1}^{n+N} \mathbbm{1}\left(r_i \text{ is $\eta$-very good}\right) \leq \eta N\right)\right] \leq C(1-\delta)^N (\tilde\E_{r,\lambda}|K|)^2 
  \\
  &\hspace{3.8cm}+
  \frac{2}{\delta N} \sum_{i=n+1}^{n+N} \tilde\E_{\lambda} \left[ \mathbbm{1}(0 \nxleftrightarrow{r} y)|K^r||K^r_y| \mathbbm{1}\left(r_i \text{ is $\delta$-good but not $\eta$-very good} \right)\right].
\end{align*}
Each of the terms in the sum on the right hand side can be bounded
\begin{align*}
  &\tilde\E_{\lambda} \left[ \mathbbm{1}(0 \nxleftrightarrow{r} y)|K^r||K^r_y| \mathbbm{1}\left(r_i \text{ is $\delta$-good but not $\eta$-very good} \right)\right] 
  \\&\hspace{0.25cm}\leq 2
  \tilde\E_{\lambda} \left[ \mathbbm{1}(0 \nxleftrightarrow{r} y)|K^r||K^r_y| \mathbbm{1}\left(\sum_{x\in \Z^d}|K^{r_i} \cap B_{4r_i}(x)|\mu_{r_i,\lambda}(x)> \delta r_i^{2\alpha} \text{ and } |K^{r_i} \cap B_{16r_i}(x)| \leq \eta r_i^{2\alpha}\right)\right]
  \\
  &\hspace{0.25cm}\leq 2 \E|K^r_y|
    \tilde\E_{\lambda} \left[|K^r| \mathbbm{1}\left(\sum_{x\in \Z^d}|K^{r_i} \cap B_{4r_i}(x)|\mu_{r_i,\lambda}(x)> \delta r^{2\alpha} \text{ and } |K^{r_i} \cap B_{8r_i}(x)| \leq \eta r_i^{2\alpha}\right)\right]
    \\
    &\hspace{0.25cm}\preceq r_i^{-\alpha} (\E|K^r|)^2
    \tilde\E_{\lambda} \left[|K^{r_i}| \mathbbm{1}\left(\sum_{x\in \Z^d}|K^{r_i} \cap B_{4r_i}(x)|\mu_{r_i,\lambda}(x)> \delta r_i^{2\alpha} \text{ and } |K^{r_i} \cap B_{8r_i}(x)| \leq \eta r_i^{2\alpha}\right)\right],
\end{align*}
where the last two inequalities follow by our usual BK/Reimer argument and the factor of two in the first inequality accounts for the fact that the cluster of either $0$ or $y$ could fail to have a large intersection with an appropriately sized ball centered at the origin. (We also decreased the radius of this ball from $16r_i$ to $8 r_i$ so that $B_{8r_i}(y) \subseteq B_{16r_i}$.)
As such, it suffices to prove that for each $\eps>0$ there exists $\eta>0$ such that 
\begin{equation}
\label{eq:good_but_not_very_good}
\limsup_{\substack{r\to \infty\\r\in \mathscr{R}}}  r^{-\alpha} \tilde \E_{r,\lambda}\left[|K| \mathbbm{1}\left(\sum_{x\in \Z^d}|K \cap B_{4r}(x)|\mu_{r,\lambda}(x)> \delta r^{2\alpha} \text{ and } |K \cap B_{8r}(x)| \leq \eta r^{2\alpha}\right) \right] \leq \eps
\end{equation}
for each $0\leq \lambda \leq 1$. The proof of this estimate will not use any properties of the probability measure $\mu_{r,\lambda}$, while  the only property of the set $\mathscr{R}$ it will use is that 
\begin{equation}
\label{eq:only_property}
\limsup_{\substack{r\to \infty\\r\in \mathscr{R}}} r^{-\alpha}\tilde \E_{r,\lambda}|K| <\infty.
\end{equation}
Suppose that \eqref{eq:good_but_not_very_good} does not hold, so that there exists $\eps>0$ 
 and a function $\eta(r):(0,\infty)\to(0,\infty)$ with $\eta(r)\to 0$ as $r\to \infty$ such that
\[
\limsup_{\substack{r\to \infty\\r\in \mathscr{R}}}  r^{-\alpha} \tilde \E_{r,\lambda}\left[|K| \mathbbm{1}\left(\sum_{x\in \Z^d}|K \cap B_{4r}(x)|\mu_{r,\lambda}(x)> \delta r^{2\alpha} \text{ and } |K \cap B_{8r}| \leq \eta(r) r^{2\alpha}\right) \right] \geq \eps
\]
By Markov's inequality, we must have that
\begin{equation}
\label{eq:very_good_contradiction}
\limsup_{\substack{r\to \infty\\r\in \mathscr{R}}}  r^{-3\alpha}  \tilde \E_{r,\lambda}\left[|K| \mathbbm{1}\left(|K \cap B_{8r}| \leq \eta(r) r^{2\alpha}\right) \sum_{x\in \Z^d}|K \cap B_{4r}(x)|\mu_{r,\lambda}(x)
 \right] \geq \frac{\eps}{\delta}.
\end{equation}
Consider a uniform random ordering $u_1,\ldots,u_{|B_{2r}|}$ of the points of $B_{2r}$, chosen independently of the percolation configuration, so that if $\mathscr{C}$ denotes the set of all clusters in the configuration then
\begin{multline*}
  \tilde \E_{r,\lambda}\left[ \sum_{C \in \mathscr{C}}|C| \sum_{x\in \Z^d}|C \cap B_{4r}(x)|\mu_{r,\lambda}(x)
 \right] \geq \tilde \E_{r,\lambda}\left[ \sum_{C \in \mathscr{C}} \mathbbm{1}(C\cap B_{4r}\neq \emptyset) |C| \sum_{x\in \Z^d}|C \cap B_{4r}(x)|\mu_{r,\lambda}(x)
 \right] \\
 = \sum_{i=1}^{|B_{2r}|} \tilde \E_{r,\lambda}\left[ \mathbbm{1}(K_{u_i} \cap \{u_j: j<i\} = \emptyset) |K_{u_i}| \sum_{x\in \Z^d}|K_{u_i} \cap B_{4r}(x)|\mu_{r,\lambda}(x)
 \right].
\end{multline*}
For each $1\leq i \leq |B_{2r}|$ we have that
\begin{multline*}
  \tilde \E_{r,\lambda}\left[ \mathbbm{1}(K_{u_i} \cap \{u_j: j<i\} = \emptyset) |K_{u_i}| \sum_{x\in \Z^d}|K_{u_i} \cap B_{4r}(x)|\mu_{r,\lambda}(x)
 \right]
  \\\geq \tilde \E_{r,\lambda}\left[ 
  \max\left\{0,
  1-i\frac{|K_{u_i} \cap B_{2r}|}{|B_{2r}|}
  \right\}
  |K_{u_i}| \sum_{x\in \Z^d}|K_{u_i} \cap B_{4r}(x)|\mu_{r,\lambda}(x) \right]
\end{multline*}
so that if $r$ is sufficiently large and $1\leq i \leq \eta(r)^{-1}r^\alpha/4$ then
\begin{align*}
  &\tilde \E_{r,\lambda}\left[ \mathbbm{1}(K_{u_i} \cap \{u_j: j<i\} = \emptyset) |K_{u_i}| \sum_{x\in \Z^d}|K_{u_i} \cap B_{4r}(x)|\mu_{r,\lambda}(x)
 \right]
  \\&\hspace{3cm} \geq \frac{1}{2}\tilde \E_{r,\lambda}\left[ 
\mathbbm{1}\left(|K_{u_i} \cap B_{2r}| \leq \eta(r) r^{2\alpha}\right)
  |K_{u_i}| \sum_{x\in \Z^d}|K_{u_i} \cap B_{4r}(x)|\mu_{r,\lambda}(x) \right]
  \\&\hspace{3cm}
  \geq \frac{1}{2}\tilde \E_{r,\lambda}\left[ 
\mathbbm{1}\left(|K_{u_i} \cap B_{8r}(u_i)| \leq \eta(r) r^{2\alpha}\right)
  |K_{u_i}| \sum_{x\in \Z^d}|K_{u_i} \cap B_{2r}(u_i+x)|\mu_{r,\lambda}(x) \right]
  \\&\hspace{3cm}
  = \frac{1}{2}\tilde \E_{r,\lambda}\left[ 
\mathbbm{1}\left(|K \cap B_{8r}| \leq \eta(r) r^{2\alpha}\right)
  |K| \sum_{x\in \Z^d}|K \cap B_{2r}(x)|\mu_{r,\lambda}(x) \right],
\end{align*}
where we used the fact that $B_{8r}(u_i) \supset B_{4r}$ and $B_{2r}(u_i+x) \subseteq B_{4r}(x)$ in the penultimate line and used transitivity in the last line. 
Applying \eqref{eq:very_good_contradiction} to estimate the quantity on the last line and
summing this estimate over $1\leq i \leq \eta(r)^{-1} r^\alpha/4$, we obtain that 
\begin{equation*}
\limsup_{\substack{r\to \infty\\r\in \mathscr{R}}} r^{-4\alpha}  \tilde \E_{r,\lambda}\left[ \sum_{C \in \mathscr{C}}|C| \sum_{x\in \Z^d}|C \cap B_{4r}(x)|\mu_{r,\lambda}(x)
 \right] \geq \limsup_{\substack{r\to \infty\\r\in \mathscr{R}}} \frac{1}{8\eta(r)} \cdot \frac{\eps}{\delta} = \infty.
\end{equation*}
On the other hand we can compute
\begin{multline*}
  \limsup_{\substack{r\to \infty\\r\in \mathscr{R}}} r^{-4\alpha} \tilde \E_{r,\lambda}\left[ \sum_{C \in \mathscr{C}}|C| \sum_{x\in \Z^d}|C \cap B_{4r}(x)|\mu(x)
 \right] = \limsup_{\substack{r\to \infty\\r\in \mathscr{R}}} r^{-4\alpha} \sum_{x\in \Z^d} \mu(x) \sum_{y\in B_{2r}(x)} \tilde \E_{r,\lambda}|K_y| 
 \\= \limsup_{\substack{r\to \infty\\r\in \mathscr{R}}} r^{-4\alpha}|B_{2r}| \tilde \E_{r,\lambda}|K| < \infty
\end{multline*}
by \eqref{eq:only_property}.
This establishes the desired contradiction and completes the proof. 
\end{proof}

We are finally ready to prove that the modified hydrodynamic condition holds.

\begin{proof}[Proof of \cref{prop:critical_dim_modified_hydro}]
Suppose for contradiction that the modified hydrodynamic condition does not hold. We will consider two cases according to whether or not $\liminf_{r\to \infty} r^{-\alpha}\E_{\beta_c,r}|K|$ is finite, taking $\lambda=0$ if it is finite and fixing some $0<\lambda \leq \lambda_0$ with $\lambda_0$ the minimum of the constants appearing in \cref{lem:gluing_very_good,prop:fictitious_derivative_and_susceptibility_best} if it is infinite. All constants in the remainder of the proof will be permitted to depend on the choice of $\lambda$. Let $\mathscr{R}=\{r_1,r_2,\ldots\}$ be the set of scales guaranteed to exist by \cref{lem:gluing_very_good}, which always satisfies $\sup_{r\in \mathscr{R}} r^{-\alpha} \tilde \E_{r,\lambda}|K|<\infty$ (by the statement of \cref{lem:gluing_very_good} when $\lambda=0$ and by \cref{prop:fictitious_derivative_and_susceptibility_best} when $\lambda>0$). By \cref{prop:fictitious_derivative_and_susceptibility_best}, there also exist positive constants $c_1$ and $C_1$ and an unbounded set of scales $\mathscr{S} \subseteq [1,\infty)$ such that
\[
  \tilde \E_{r,\lambda}|K| \leq C_1 r^{\alpha} \qquad \text{ and } \qquad \tilde D_{r,\lambda} \geq c_1 r^{2\alpha}
\]
for every $r\in \mathscr{S}$. (Indeed, for appropriate choices of $c_1$ and $C_1$, the second condition holds at \emph{all} scales by \cref{prop:fictitious_derivative_and_susceptibility_best}, while the first condition holds at all scales when $\lambda>0$ by \cref{prop:fictitious_derivative_and_susceptibility_best} and for an unbounded set of scales by assumption when $\lambda=0$.) To reach a contradiction, it suffices to prove that
\begin{equation}
\label{eq:small_derivative_contradiction}
  \limsup_{r\to \infty} \frac{\tilde \E_{r,\lambda} |K||K_y| \mathbbm{1}(0 \nleftrightarrow y)}{(\tilde \E_{r,\lambda} |K|)^2} = 0
\end{equation}
for each fixed $y\in \Z^d$; this is easily seen to imply that $\limsup_{r\to \infty}\tilde D_{r,\lambda}/ (\tilde \E_{r,\lambda}|K|)^2=0$ after using Russo's formula to expand the $\beta$ derivative in terms of pivotal edges. 

\medskip

Fix $y\in \Z^d$ and consider the standard monotone coupling $\tilde \P_\lambda$ of the measures $(\tilde \P_{r,\lambda})_{r\geq 0}$ as above.
Write $(\omega_r)_{r\geq 0}$ for the associated family of configurations so that $K^r$ and $K^r_y$ are the clusters of the origin and $y$ in $\omega_r$ respectively. 
 Let $i_0$ be minimal such that $r_{i_0}\geq \|y\|$ and for each $r \geq 1$ let $m(r)$ be maximal such that $32 r_{i_0 + m(r)}\leq r$. Let the process $(\tilde K^r,\tilde K_y^r)_{r\geq 0}$ be defined by adding edges to the clusters of the origin and of $y$ as they open \emph{unless this edge would connect the two clusters}, in which case it is ignored. Thus, the clusters $\tilde K^r$ and $\tilde K^r_y$ are always contained in the clusters $K^r$ and $K^r_y$, and on the event $\{0 \nxleftrightarrow{r} y\}$ the clusters $\tilde K^s$ and $\tilde K^s_y$ are \emph{equal} to the clusters $K^s$ and $K^s_y$ for all $s\leq r$. Letting 
$\mathcal{F}_r$
 denote the sigma-algebra generated by $(\tilde K^s,\tilde K^s_y)_{s=0}^r$, we observe that there exists a positive constant $c_2$ such that
 \begin{equation}
 \label{eq:tilde_K_gluing}
  \P(0 \nxleftrightarrow{r} y \mid \mathcal{F}_r) \leq \exp\left[-c_2 \eta^2 \sum_{i=i_0}^{i_0+m(r)} \mathbbm{1}\left(|\tilde K^{r_i}\cap B_{16 r_i}|, |\tilde K^{r_i}_y \cap B_{16r_i}| \geq \eta r_i^{2\alpha}\right) \right]. 
 \end{equation}
 Indeed, whenever the scale $r_i$ is such that $|\tilde K^{r_i}\cap B_{16 r_i}|, |\tilde K^{r_i}_y \cap B_{16r_i}| \geq \eta r_{i}^{2\alpha}$, the expected number of edges connecting $\tilde K^{r_i} \cap B_{16 r_i}$ to $\tilde K^{r_i}_y \cap B_{16 r_i}$ that get added  when we increase $r$ from $16 r_i$ to $32 r_i$ is of order at least $(\eta r_i^{2\alpha})^2 r_i^{-d-\alpha}=\eta^2$, so that the conditional probability that at least one such edge is added is at least of order $\eta^2$. The definition of the process $(\tilde K^r,\tilde K^r_y)_{r\geq 1}$ together with the assumed minimal spacing $r_{i+1}\geq 32 r_i$ between the scales in our sequence ensures that these gluing events are conditionally independent given $\mathcal{F}_r$, from which the claimed inequality follows.

\medskip

Fix $\eps>0$ and let $\eta>0$ be as in \cref{lem:gluing_very_good}, so that 
 \begin{multline*}
 \tilde \E_{r,\lambda} \left[ |K||K_y| \mathbbm{1}(0 \nleftrightarrow y)\right] \leq \eps (\tilde \E_{r,\lambda} |K|)^2
\\+
\tilde \E_\lambda \left[ \mathbbm{1}(0 \nxleftrightarrow{r} y) |K^r||K^r_y| \mathbbm{1}\left( \sum_{i=i_0}^{i_0+m(r)} \mathbbm{1}(r_i \text{ is $(y,\eta)$-very good}) > \eta m(r)\right)\right].
 \end{multline*}
 Using \eqref{eq:tilde_K_gluing}, we can bound
 \begin{multline*}
\tilde \E_\lambda \left[ \mathbbm{1}(0 \nxleftrightarrow{r} y) |K^r||K^r_y| \mathbbm{1}\left( \sum_{i=i_0}^{i_0+m(r)} \mathbbm{1}(r_i \text{ is $(y,\eta)$-very good}) > \eta m(r)\right)\right] 
\\\leq e^{-c_3 \eta^3 m(r)} \tilde \E_\lambda \left[|\tilde K^r||\tilde K^r_y|\right]
 \end{multline*}
 for some constant $c_3>0$. 
 Since $\eps>0$ was arbitrary and $e^{-c_3 \eta^3 m(r)}\to 0$ as $r\to \infty$, to prove \eqref{eq:small_derivative_contradiction} it suffices to prove that this expectation can be bounded $\tilde \E_\lambda [|\tilde K^r||\tilde K^r_y|] \preceq (\tilde \E_\lambda |K^r|)^2=(\tilde \E_{r,\lambda} |K|)^2$. To prove this, note that for each $a,b>0$ we have the inclusion of events
 \[
   \{|\tilde K^r| \geq a\} \cap \{|\tilde K^r_y|\geq b \} \subseteq \{|K^r|\geq a\} \circ \{|K^r_y|\geq b\}
 \]
 and hence by the BK inequality that
 \begin{multline*}
   \tilde \E_\lambda \left[|\tilde K^r||\tilde K^r_y|\right] = \sum_{a,b \geq 1} \tilde \P_\lambda(\{|\tilde K^r| \geq a\} \cap \{|\tilde K^r_y|\geq b \}) \leq \sum_{a,b \geq 1} \tilde \P_\lambda(|K^r| \geq a)\tilde \P_\lambda(|K^r_y|\geq b) =  (\tilde \E_\lambda |K^r|)^2
 \end{multline*}
 as claimed. This completes the proof of the contradictory estimate \eqref{eq:tilde_K_gluing} and hence the proof that the modified hydrodynamic condition holds.
\end{proof}

It remains to deduce \cref{thm:critical_dim_hydro} from \cref{prop:critical_dim_modified_hydro}, i.e., to show that the modified hydrodynamic condition implies the hydrodynamic condition. This will be done with the aid of the following pair of differential inequalities.

\begin{lemma} The estimates
\label{lem:second_moment_in_box_derivative_upper}
\[
  \frac{\partial}{\partial r}\E_{\beta,r}|K \cap B_R(x)|^2 \preceq \beta r^{-\alpha-1} \E_{\beta,r}|K| \sup_w \E_{\beta,r}|K \cap B_R(w)|^2
\]
and
\[\frac{\partial }{\partial \beta}\E_{\beta,r}|K \cap B_R(x)|^2 \preceq \E_{\beta,r}|K| \sup_w \E_{\beta,r}|K \cap B_R(w)|^2
\]
hold for all $r,R\geq 1$, $\beta>0$, and $x\in \Z^d$.
\end{lemma}

\begin{proof}[Proof of \cref{lem:second_moment_in_box_derivative_upper}] We will prove the claim concerning the $r$ derivative, the proof for the $\beta$ derivative being very similar. We can use Russo's formula as usual to write
\begin{multline*}
 \hspace{-0.6em} \frac{\partial}{\partial r}\E_{\beta,r}|K \cap B_R(x)|^2 = \beta |J'(r)|\E_{\beta,r}\left[ \sum_{y\in K} \sum_{z\in B_r(y)} \mathbbm{1}(y\nleftrightarrow z)\bigl(|(K_z \cup K) \cap B_R(x)|^2-|K \cap B_R(x)|^2\bigr) \right]
  \\
  =\beta |J'(r)|\E_{\beta,r}\left[ \sum_{y\in K} \sum_{z\in B_r(y)} \mathbbm{1}(y \nleftrightarrow z)\bigl(|K_z \cap B_R(x)|^2+2|K\cap B_R(x)| |K_z\cap B_R(x)|\bigr) \right].
\end{multline*}
Using the BK inequality, we can bound the first term by
\[
  \E_{\beta,r}\left[ \sum_{y\in K} \sum_{z\in B_r(y)} \mathbbm{1}(y \nleftrightarrow z)|K_z \cap B_R(x)|^2\right] \leq |B_r|\E_{\beta,r}|K| \sup_w \E_{\beta,r}|K \cap B_R(w)|^2
\]
as required.
For the second term, we write
\begin{align*}
  &\E_{\beta,r}\left[ \sum_{y\in K} \sum_{z\in B_r(y)} \mathbbm{1}(y \nleftrightarrow z)|K\cap B_R(x)| |K_z\cap B_R(x)|\right]
  \\
  &\hspace{3cm}=
    \sum_{y,z,a,b \in \Z^d} \P(0 \leftrightarrow y \leftrightarrow a, y \nleftrightarrow z, z \leftrightarrow b) \mathbbm{1}(a,b\in B_R(x),z\in B_r(y))
      \\&\hspace{3cm}=
  \sum_{y,z,a,b \in \Z^d} \P(b \leftrightarrow y \leftrightarrow a, y \nleftrightarrow z, z \leftrightarrow 0) \mathbbm{1}(a,0\in B_R(x-b),z\in B_r(y))
  \\
  &\hspace{3cm}\leq  \sum_{y,z,a,b \in \Z^d} \P(b \leftrightarrow y \leftrightarrow a, y \nleftrightarrow z, z \leftrightarrow 0) \mathbbm{1}(a,b\in B_{2R}(x), z\in B_r(y))
  \\ &\hspace{3cm}\leq |B_r|\E_{\beta,r}|K| \sup_w \E_{\beta,r}|K \cap B_{2R}(w)|^2 \preceq |B_r|\E_{\beta,r}|K| \sup_w \E_{\beta,r}|K \cap B_{R}(w)|^2,
\end{align*}
where we used the mass-transport principle in the third line and used that balls of radius $2R$ can be covered by boundedly many balls of radius $R$ in the last line.
\end{proof}

\begin{proof}[Proof of \cref{thm:critical_dim_hydro}]
Since the modified hydrodynamic condition holds, we can run the same argument as in the proof of Proposition \ref{I-prop:first_moment} with the measures $\tilde \E_{r,\lambda}$ instead of $\E_{\beta_c,r}$ to obtain that
\[
\tilde \E_{r,\lambda}|K| \leq (1+o(1)) \frac{\alpha}{\beta_c} r^\alpha
\]
for each fixed $0<\lambda \leq 1$ as $r\to \infty$, where the rate of convergence in the error term may depend on the choice of $\lambda$. If we take a sufficiently small value of $\lambda>0$ and integrate the differential inequality \eqref{eq:beta_derivative_simpler_upper} as in the proof of \cref{cor:bad_case_large_susceptibility_lambda} we deduce that in fact
\[
  \E_{\beta_c,r}|K| \sim \frac{\alpha}{\beta_c} r^\alpha,
\]
as $r\to \infty$, where the lower bound follows from Corollary \ref{I-cor:mean_lower_bound}. Since the modified hydrodynamic condition holds, we also have by the universal tightness theorem applied as in \cref{I-cor:universal_tightness_moments} that
\[
  \sup_x \tilde \E_{r,\lambda}|K\cap B_r(x)|^2 \preceq r^\alpha \tilde M_{r,\lambda} = o(r^{3\alpha})
\]
as $r\to \infty$ for each fixed $0<\lambda \leq 1$. Integrating the second differential inequality of \cref{lem:second_moment_in_box_derivative_upper} between $\beta(r,\lambda)$ and $\beta_c$, we obtain that
\[
  \log \frac{\sup_x \E_{\beta_c,r}|K\cap B_r(x)|^2}{\sup_x \tilde \E_{r,\lambda}|K\cap B_r(x)|^2} \preceq r^{\alpha} |\beta_c-\beta(r,\lambda)| = \lambda \beta_c
\]
and hence that
\[
  \sup_x \E_{\beta_c,r}|K\cap B_r(x)|^2 \preceq \sup_x \tilde \E_{r,1/2}|K\cap B_r(x)|^2 = o(r^{3\alpha}).
\]
Comparing this estimate with \cref{I-lem:moments_bounded_below_by_M} (restated here as \eqref{eq:moments_bounded_below_by_M_restate}) shows that the hydrodynamic condition holds.
\end{proof}

\section{The three-point function}
\label{sec:the_three_point_function}

In this section we prove versions of the three-point function estimate of \cref{thm:pointwise_three_point} stated in terms of the not-yet-computed \textbf{vertex factor}
\[
  V_r := \frac{\E_{\beta_c,r}|K|^2}{(\E_{\beta_c,r}|K|)^3} 
\]
which we prove to be asymptotic to a constant multiple of $(\log r)^{-1/2}$ in \cref{prop:second_moment_critical_dim_asymptotics}. (The \emph{two-point} function estimate of \cref{thm:pointwise_three_point} has already been established as a corollary of \cref{II-thm:CL_Sak}, which applies when $d=3\alpha<6$ by \cref{thm:critical_dim_hydro} as noted in \cref{II-cor:HD_CL}.)
The results proven in this section apply not just at the critical dimension but also in the case $d>3\alpha$, $\alpha<2$, in which case the model is effectively long-range and high-dimensional; although the same estimates can be proven more easily in this case as direct consequences of \cref{II-cor:HD_CL},  \cref{II-thm:CL_Sak}, and the tree-graph inequalities, we work at this level of generality to clarify the logical structure of the proof.
\medskip

At this stage of the proof, we already know by \cref{I-thm:hd_moments_main}, \cref{I-thm:critical_dim_moments_main_slowly_varying}, and \cref{thm:critical_dim_hydro} that $V_r$ is slowly varying when $d=3\alpha<2$ and converges to a positive constant when $d>3\alpha$ and $\alpha<2$. As with our scaling limit theorems, it is interesting to note that this level of knowledge is already sufficient to compute the asymptotics of the three-point function (in terms of $V_r$), \emph{before} we compute the asymptotics of $V_r$ itself, and indeed we will apply the three-point estimate of \cref{prop:three_point_upper} in our proof that $V_r \sim \mathrm{const.} (\log r)^{-1/2}$ when $d=3\alpha<6$ (see the proof of \cref{lem:triple_interaction_negligible2}).

\subsection{Upper bounds}

The goal of this subsection is to prove the following proposition. Given three distinct points $x,y,z\in \Z^d$, we write $\tau_{\beta_c,r}$ for the probability under $\P_{\beta_c,r}$ that $x,y,z$ all belong to the same cluster and write $d_\mathrm{max}(x,y,z)$ and $d_\mathrm{min}(x,y,z)$ for the maximum and mnimum distances between $x$, $y$, and $z$ respectively, setting $\tau_{\beta_c,\infty}=\tau_{\beta_c}$.

\begin{prop}
If $d\geq 3\alpha$ and $\alpha<2$ then there exists a decreasing function $h_p:(0,\infty)\to (0,1]$ decaying faster than any power such that
\label{prop:three_point_upper}
\[
  \tau_{\beta_c,r}(x,y,z) \preceq V_{d_\mathrm{min}(x,y,z)}  d_\mathrm{min}(x,y,z)^{-d+2\alpha} d_\mathrm{max}(x,y,z)^{-d+\alpha}  h\left(\frac{d_\mathrm{max}(x,y,z)}{r}\right)
\]
for every $r\in [1,\infty]$ and distinct triple of points $x,y,z\in \Z^d$, where we set $h(d_\mathrm{max}(x,y,z)/r)\equiv 1$ when $r=\infty$.
\end{prop}

The proof of this proposition is an elaboration of the proof of \cref{II-thm:CL_Sak} (specifically \cref{II-lem:CL_Sak_upper}), which concerns the two-point function and holds throughout the entire effectively long-range regime.
When $d=3\alpha$ we have that $d-2\alpha=\alpha=(d-\alpha)/2$ and hence that
\[
  d_\mathrm{min}(x,y,z)^{-d+2\alpha} d_\mathrm{max}(x,y,z)^{-d+\alpha} \asymp \sqrt{\|x-y\|^{-d+\alpha}\|y-z\|^{-d+\alpha}\|z-x\|^{-d+\alpha}}
\]
so that \cref{prop:three_point_upper} will imply the three-point upper bound of \cref{thm:pointwise_three_point}
once we compute the order of $V_r$ in \cref{sec:logarithmic_corrections_at_the_critical_dimension}.

\medskip

To prove \cref{prop:three_point_upper}, we begin by proving an analogous estimate with spatial averaging over one of the three points.

\begin{lemma}
\label{lem:three_point_moments}
If $d\geq 3 \alpha$ and $\alpha<2$ then for each integer $p\geq 0$ there exists a decreasing function $h_p:(0,\infty)\to (0,1]$ decaying faster than any power such that
\[
  \E_{\beta_c,r}\mathbbm{1}(0\leftrightarrow x) \sum_{y\in K}\|y\|^p \preceq_p r^{p+\alpha} V_{\|x\|} \|x\|^{-d+2\alpha} h_p(\|x\|/r)
\]
for every $x\in \Z^d\setminus \{0\}$ and $r\geq 1$.
\end{lemma}

The proof of this lemma will apply \cref{I-lem:mixed_moments_order_estimates}, which when $d\geq 3\alpha$ and $\alpha<2$ implies that
\begin{equation}
\label{eq:mixed_moments_order_estimates}
  \E_{\beta_c,r}\left[ \sum_{x_1,\ldots,x_n \in K} \prod_{i=1}^n \|x_i\|^{p_i}\right] \preceq_{n,p}  V_r^{n-1} r^{(2n-1)\alpha+p}
\end{equation}
for every $n\geq 1$, $p\geq 0$, and sequence of non-negative integers $p_1,\ldots,p_n$ with $\sum_{i=1}^n = p$. (The precise statement of \cref{I-lem:mixed_moments_order_estimates} is stronger and more general than this.) It is an immediate consequence of this estimate  that if $n\geq 1$, $p\geq 0$ and $p_2,\ldots,p_n$ are non-negative integers summing to $p$ then
\begin{multline*}
  \E_{\beta_c,r}\left[ |K\setminus B_{\lambda r}| \sum_{x_2,\ldots,x_n \in K} \prod_{i=2}^n \|x_i\|^{p_i}\right] \\\leq
  (\lambda r)^{-q}\E_{\beta_c,r}\left[ \sum_{x_1,x_2,\ldots,x_n \in K} \|x_1\|^q \prod_{i=2}^n \|x_i\|^{p_i}\right]
  \preceq_{n,p,q}
   \lambda^{-q} V_r^{n-1} r^{(2n-1)\alpha+p}
\end{multline*}
for every $q\geq 0$ and $r\geq 1$ and hence that for each $n\geq 1$ and $p\geq 0$ there exists a decreasing function $h_{n,p}:(0,\infty)\to(0,1]$ decaying faster than any power such that
\begin{equation}
\label{eq:mixed_moments_order_estimates2}
  \E_{\beta_c,r}\left[ |K\setminus B_{\lambda r}| \sum_{x_2,\ldots,x_n \in K} \prod_{i=1}^n \|x_i\|^{p_i}\right] \preceq_{n,p}  V_r^{n-1} r^{(2n-1)\alpha+p} h(\lambda)
\end{equation}
for every for every $\lambda,r\geq1$ and sequence of non-negative integers $p_2,\ldots,p_n$ with $\sum_{i=2}^n = p$.

\begin{proof}[Proof of \cref{lem:three_point_moments}]
Write $\E$ for expectations taken with respect to the standard monotone coupling of the measures $(\P_{\beta_c,r})_{r\geq 0}$, writing $\omega_r$ for the configuration associated to the measure $\P_{\beta_c,r}$.  Fix $x\in \Z^d$, let $R(x)=\inf\{r\geq 0: 0\leftrightarrow x$ in $\omega_r\}$, and consider the quantities
\[
 \mathfrak{R}_r(x;p):= \E\left[\mathbbm{1}(R(x)\leq r) R(x)^{-\alpha} \sum_{y\in \Z^d}\|y\|^p \mathbbm{1}\Bigl(0\leftrightarrow y \text{ in }\omega_{R(x)}\Bigr)\right]
\]
and
\[
  \mathfrak{S}_r(x;p):= \E\left[\mathbbm{1}(R(x)\leq r)  \sum_{y\in \Z^d}\|y\|^p \mathbbm{1}\Bigl(0\leftrightarrow y \text{ in }\omega_{R(x)}\Bigr)\right].
\]
We first claim that for each $p\geq 0$ there exists a decreasing function $h_p:(0,\infty)\to (0,1]$ decaying faster than any power such that
\begin{equation}
\label{eq:mathfrakR_integrated}
 \mathfrak{R}_r(x;p) \preceq_p \begin{cases}
   V_{r} r^{-d+2\alpha}\|x\|^p  h_p(\|x\|/r) & r \leq \|x\|/8\\
 \int_{\|x\|/16}^r V_s s^{p-d+2\alpha} \frac{ds}{s} & r \geq \|x\|/8
  \end{cases}
\end{equation}
and
\begin{equation}
\label{eq:mathfrakS_integrated}
 \mathfrak{S}_r(x;p) \preceq_p \begin{cases}
   V_{r} \|x\|^{-d+3\alpha+p}  h_p(\|x\|/r) & r \leq \|x\|/8\\
 \int_{\|x\|/16}^r V_s s^{p-d+3\alpha} \frac{ds}{s} & r \geq \|x\|/8.
  \end{cases}
\end{equation}
We will prove the claim regarding $\mathfrak{R}_r(x;p)$, the proof for $\mathfrak{S}_r(x;p)$ being similar. First note that the same argument used to prove Russo's formula yields that $ \mathfrak{R}_r(x;p)$  is  differentiable with derivative
\[
 \frac{d}{dr}\mathfrak{R}_r(x;p)= \beta_c |J'(r)| r^{-\alpha} \E_r\left[\mathbbm{1}(0\nleftrightarrow x)  \sum_{y\in K \cup K_x}\|y\|^p \sum_{a \in K,b\in K_x}\mathbbm{1}(\|a-b\|\leq r) \right].
\]
%
Using \eqref{eq:mixed_moments_order_estimates}, this derivative can always be bounded by
\begin{align}
 \frac{d}{dr}\mathfrak{R}_r(x;p)&\preceq r^{-d-2\alpha-1} \E_r\left[\mathbbm{1}(0\nleftrightarrow x)  \sum_{y\in K \cup K_x}\|y\|^p |K||K_x| \right]
\nonumber\\&\leq r^{-d-2\alpha-1} \E_r\left[|K|  \sum_{y\in K}\|y\|^p  \right] \E_r|K|
 +
 r^{-d-2\alpha-1} \E_r\left[|K|  \sum_{y\in K }\|x+y\|^p  \right] \E_r|K|
 \nonumber\\&\preceq_p V_r r^{-d+2\alpha-1}(r^p + \|x\|^p)
 \label{eq:mathfrakR_large_r}
\end{align}
which will be a reasonable bound for $r$ of order at least $\|x\|$. If $r\leq \|x\|/8$, then for each pair $a,b\in \Z^d$ with $\|a-b\|\leq r$ we must have that $\|a\|\geq \|x\|/4$, $\|x-b\|\geq \|x\|/4$, or both, so that we can bound 
\begin{align*}
 &r^{d+2\alpha+1}\frac{d}{dr}\mathfrak{R}_r(x;p)\preceq  \E_r\left[\mathbbm{1}(0\nleftrightarrow x)  \sum_{y\in K \cup K_x}\|y\|^p \left(|K\setminus B_{\|x\|/4}|\cdot|K_x|+|K|\cdot|K_x\setminus B_{\|x\|/4}(x)|\right) \right]
\\&\hspace{2.5cm}\leq  \E_r\left[|K|  \sum_{y\in K}\|y\|^p  \right] \E_r|K\setminus B_{\|x\|/4}|
+ \E_r\left[|K\setminus B_{\|x\|/4}|  \sum_{y\in K}\|y\|^p  \right] \E_r|K|
 \\&\hspace{2.5cm}+
  \E_r\left[|K|  \sum_{y\in K }\|x+y\|^p  \right] \E_r|K\setminus B_{\|x\|/4}|
 + \E_r\left[|K\setminus B_{\|x\|/4}|  \sum_{y\in K }\|x+y\|^p  \right] \E_r|K|,
\end{align*}
when $r \leq \|x\|/8$. Bounding $\|x+y\|^p\preceq_p\|x\|^p+\|y\|^p$ and using \eqref{eq:mixed_moments_order_estimates2}, we obtain that there exists a decreasing function $h_p:(0,\infty)\to(0,1]$ decaying faster than any polynomial such that
\begin{equation}
  \frac{d}{dr}\mathfrak{R}_r(x;p) \preceq V_r r^{-d+2\alpha-1} \|x\|^p h_p(\|x\|/r)
\label{eq:mathfrakR_small_r}
\end{equation}
for every $r\leq \|x\|/8$. The claimed inequality \eqref{eq:mathfrakR_integrated} follows easily by integrating the differential inequalities \eqref{eq:mathfrakR_large_r} and \eqref{eq:mathfrakR_small_r} from $1$ to $r$.

\medskip
We now apply \eqref{eq:mathfrakR_integrated} and \eqref{eq:mathfrakS_integrated} to prove the claimed estimate on $\E_{\beta_c,r}\mathbbm{1}(0\leftrightarrow x) \sum_{y\in K}\|y\|^p$. To do so, we first observe that
\begin{align}
\label{eq:mathfrakR_BK}
  \E_{\beta_c,r}\mathbbm{1}(0\leftrightarrow x) \sum_{y\in K}\|y\|^p \preceq_p 
\mathfrak{S}_r(x;p) + 
  \sum_{k=0}^p \mathfrak{R}_r(x;k) \E_{\beta_c,r} \sum_{y\in K} (r+\|y\|)^{p-k}.
\end{align}
Indeed, suppose that we continuously increase $r$ from $0$, revealing all edges incident to the cluster of the origin in the relevant configuration, until we find an open path from $0$ to $x$, and let $\mathcal{F}_{R(x)}$ denote the sigma-algebra generated by this stopped revealment process. 
The term $\mathfrak{S}_r(x;p)$ accounts for the contribution to $\E_{\beta_c,r}\mathbbm{1}(0\leftrightarrow x) \sum_{y\in K}\|y\|^p$ from the points that are already revealed to belong to the cluster of the origin at level $R(x)$. If a point $y$ belongs to the cluster of the origin at level $r$ but not at 
level $R(x)$, then there must exist an edge in the boundary of the $R(x)$-cluster of the origin that becomes open when passing to level $r$ such that $y$ is connected to the other endpoint of this edge off of the $R(x)$-cluster of the origin. The claimed inequality \eqref{eq:mathfrakR_BK} follows as usual from the fact that the cluster of any point off of the $R(x)$-cluster of the origin is stochastically smaller than the cluster of that point in the unconditioned model and that the expected number of edges of the origin that become open when passing from level $R(x)$ to level $r$ is $O(R(x)^{-\alpha})$, with the sum over $k$ coming from the binomial expansion $\|y\|^p\leq(\|y-b\|+r+\|a\|)^p$ where $\{a,b\}$ is the edge that becomes open in the union bound sketched above. The factor $R(x)^{-\alpha}$ in the definition of $\mathfrak{R}_r(x;p)$ was included precisely to account for the $O(R(x)^{-\alpha})$ edges per vertex that become open when passing from level $R(x)$ to level $r$. (We leave the proof of this bound as a sketch since it is similar to arguments we have used repeatedly throughout the series.)

\medskip

Substituting \eqref{eq:mixed_moments_order_estimates}, \eqref{eq:mathfrakR_integrated}, and \eqref{eq:mathfrakS_integrated} into the estimate \eqref{eq:mathfrakR_BK} yields that 
if $r\leq \|x\|/8$ then
\[\E_{\beta_c,r}\mathbbm{1}(0\leftrightarrow x) \sum_{y\in K}\|y\|^p \preceq_p V_r r^{-d+3\alpha}\|x\|^p h_p(\|x\|/r), \]
which is equivalent to the claim in this case.
Similarly, if $r\geq \|x\|/8$ then
\begin{align*}
  \E_{\beta_c,r}\mathbbm{1}(0\leftrightarrow x) \sum_{y\in K}\|y\|^p &\preceq_p \sum_{k=0}^p r^{p-k+\alpha}
  \int_{\|x\|/16}^r V_s s^{k-d+2\alpha} \frac{ds}{s} \\
  &=
  \|x\|^{-d+2\alpha}\sum_{k=0}^p r^{p-k+\alpha}
  \int_{\|x\|/16}^r V_s s^{k} \left(\frac{s}{\|x\|}\right)^{-d+2\alpha} \frac{ds}{s}\\
  &\preceq_p
  V_{\|x\|} r^{p+\alpha}\|x\|^{-d+2\alpha} 
\end{align*}
as required, where in the last line we used that 
\[
  V_s \left(\frac{s}{\|x\|}\right)^{-d+2\alpha} \preceq V_{\|x\|}
\]
for every $s\geq \|x\|/16$ since $V_r$ is slowly varying while the other factor on the left hand side decays as a power of $s/\|x\|$. \qedhere

\end{proof}

We now apply \cref{lem:three_point_moments} to prove \cref{prop:three_point_upper}.

\begin{proof}[Proof of \cref{prop:three_point_upper}]
We may assume without loss of generality that $\|x-y\|=d_\mathrm{min}(x,y,z)$ and $\|x-z\|=d_\mathrm{max}(x,y,z)$. 
To lighten notation we write $\E_r=\E_{\beta_c,r}$ and $\tau_r=\tau_{\beta_c,r}$. We have by Russo's formula that
\begin{align}
  \frac{1}{\beta_c|J'(r)|}\frac{d}{dr}\tau_{r}(x,y,z) &= 
  \E_r \mathbbm{1}(x\leftrightarrow y \nleftrightarrow z) \sum_{a\in K_x,b\in K_z} \mathbbm{1}(\|a-b\|\leq r)
  \nonumber\\&\hspace{1.5cm}+
  \E_r \mathbbm{1}(x\leftrightarrow z \nleftrightarrow y) \sum_{a\in K_x,b\in K_y} \mathbbm{1}(\|a-b\|\leq r)
  \nonumber\\&\hspace{3cm}+
  \E_r \mathbbm{1}(x\nleftrightarrow y \leftrightarrow z) \sum_{a\in K_x,b\in K_z} \mathbbm{1}(\|a-b\|\leq r).
  \label{eq:Russo_three_point}
\end{align}
As before, we can apply the $p=0$ case of \cref{lem:three_point_moments} to bound
\begin{multline}
\E_r \mathbbm{1}(x\leftrightarrow y \nleftrightarrow z) \sum_{a\in K_x,b\in K_z} \mathbbm{1}(\|a-b\|\leq r)
\leq 
\E_r \mathbbm{1}(x\leftrightarrow y \nleftrightarrow z) |K_x||K_z| \\
\leq 
\E_r \mathbbm{1}(x\leftrightarrow y) |K_x| \E_r|K| \preceq r^{2\alpha} V_{\|x-y\|} \|x-y\|^{-d+2\alpha},
\label{eq:three_point_derivative1}
\end{multline}
which is a good bound when $r$ is of order at least $\|x-y\|$, and with analogous bounds holding for the other two terms also. Moreover, if $r\leq \|x-z\|/8$ then for each pair $a,b\in \Z^d$ with $\|a-b\|\leq r$ we must have that $\|x-a\|\geq \|x-z\|/4$, $\|z-b\|\geq \|x-z\|/4$, or both, so that we can use \cref{lem:three_point_moments} to bound 
\begin{multline}
  \E_r \mathbbm{1}(x\leftrightarrow y \nleftrightarrow z) \sum_{a\in K_x,b\in K_z} \mathbbm{1}(\|a-b\|\leq r)
  \\\leq \E_r \mathbbm{1}(x\leftrightarrow y \nleftrightarrow z) (|K_x|\cdot|K_z\setminus B_{\|x-z\|/8}(z)| +|K_x\setminus B_{\|x-z\|/8}(x)|\cdot|K_z|)
  \\\preceq r^{2\alpha} V_{\|x-y\|} \|x-y\|^{-d+2\alpha} h(\|x-z\|/r)
  \label{eq:three_point_derivative2}
\end{multline}
for some decreasing function $h:(0,\infty)\to(0,1]$ decaying faster than any power. The claim follows easily by substituting \eqref{eq:three_point_derivative1} and \eqref{eq:three_point_derivative2} (and their analogues for the other two terms) into \eqref{eq:Russo_three_point}
and integrating.
\end{proof}

\begin{remark}
\label{remark:tree_graph_vs_Gladkov}
We now explain the relationship between \cref{prop:three_point_upper} and the tree-graph inequality \cite{MR762034} by proving that
\begin{equation}
  \sum_{w\in \Z^d} \|x-w\|^{-d+\alpha}\|y-w\|^{-d+\alpha}\|z-w\|^{-d+\alpha} \asymp d_\mathrm{min}(x,y,z)^{-d+2\alpha} d_\mathrm{max}(x,y,z)^{-d+\alpha}
\end{equation}
whenever $d>2\alpha$ and hence that the upper bound of \cref{prop:three_point_upper} is of the same order as the tree-graph bound
\[
  \tau_{\beta_c}(x,y,z) \preceq \sum_{w\in \Z^d} \|x-w\|^{-d+\alpha}\|y-w\|^{-d+\alpha}\|z-w\|^{-d+\alpha}
\]
when $d>3\alpha$ and $\alpha<2$. 
In the case $d=3\alpha$, this shows that
\begin{multline*}
  \sum_{w\in \Z^d} \|x-w\|^{-d+\alpha}\|y-w\|^{-d+\alpha}\|z-w\|^{-d+\alpha} \asymp d_\mathrm{min}(x,y,z)^{-\alpha} d_\mathrm{max}(x,y,z)^{-2\alpha} \\ \asymp \|x-y\|^{-\alpha}\|y-z\|^{-\alpha}\|z-x\|^{-\alpha} \asymp \sqrt{\|x-y\|^{-d+\alpha}\|y-z\|^{-d+\alpha}\|z-x\|^{-d+\alpha}},
\end{multline*}
so that the tree-graph bound and the Gladkov inequality are of the same order when $d=3\alpha<6$ (and hence both wasteful by a $\sqrt{\log}$ factor by \cref{thm:pointwise_three_point}).
We may assume without loss of generality that $\|x-y\|=d_\mathrm{min}(x,y,z)$ and $\|x-z\|=d_\mathrm{max}(x,y,z)$.
 We break the sum over $w$ into three pieces according to whether $w$ is closest to $x$, $y$, or $z$. The sum over points $w$ that are closest to $x$ is of order
\begin{multline*}
  \sum_{r=1}^{\|x-y\|} r^{d-1}r^{-d+\alpha} \|x-y\|^{-d+\alpha}\|x-z\|^{-d+\alpha}
  + 
  \sum_{r=\|x-y\|}^{\|x-z\|} r^{d-1}r^{-2d+2\alpha}\|x-z\|^{-d+\alpha}
  +\sum_{r=\|x-z\|}^{\infty} r^{d-1}r^{-3d+3\alpha}\\
  \asymp \|x-y\|^{-d+2\alpha}\|x-z\|^{-d+\alpha} + \|x-y\|^{-d+2\alpha}\|x-z\|^{-d+\alpha}+\|x-z\|^{-2d+3\alpha}
  \\
  \asymp  \|x-y\|^{-d+2\alpha}\|x-z\|^{-d+\alpha},
\end{multline*}
where we used that $d>2\alpha$ in the second line.
The sum over points closest to $y$ is of the same order, while the sum over points closest to $z$ is of order $\|x-z\|^{-2d+3\alpha}$ which is of smaller order when $\|x-y\| \ll \|x-z\|$. 
\end{remark}

\subsection{Lower bounds}

We now prove the following lower bound, complementing the upper bound of \cref{prop:three_point_upper}. The proof is a straightforward variation on the proof of \cref{II-prop:k-point_hyperscaling_lower}.

\begin{prop}
\label{prop:three_point_lower}
If $d\geq 3\alpha$ and $\alpha<2$ then 
\[
  \tau_{\beta_c}(x,y,z) \succeq V_{d_\mathrm{min}(x,y,z)}  d_\mathrm{min}(x,y,z)^{-d+2\alpha} d_\mathrm{max}(x,y,z)^{-d+\alpha}  
\]
for every distinct triple of points $x,y,z\in \Z^d$.
\end{prop}

To prove this proposition, we will use the fact that if $d\geq 3\alpha$ and $\alpha<2$ then
\[
V_r r^\alpha \P_{\beta_c,r}(|K \cap B_r| \geq \lambda V_r r^{2\alpha}) \to f(\lambda)
\]
as $r\to \infty$, where $f(\lambda)$ is a positive function determined by the superprocess scaling limit of the model,
which is an immediate consequence of \cref{I-thm:scaling_limit_diagrams,I-cor:scaling_limit_cut_off}. In fact for our purposes in this section it will suffice to use the weaker estimate 
\begin{equation}
\label{eq:locally_large_cluster_order_estimate}
  \P_{\beta_c,r}(|K \cap B_r| \geq \zeta(r)) \asymp \frac{r^\alpha}{\zeta(r)} \qquad \text{ where  } \qquad \zeta(r):= \frac{\E_r|K|^2}{4\E_r|K|} \asymp V_r r^{2\alpha}
\end{equation}
which holds for every $r\geq 1$ as a consequence of \cref{I-thm:scaling_limit_diagrams,I-cor:scaling_limit_cut_off}.

\medskip 

The next lemma shows that if the origin belongs to a ``typical large cluster'' on scale $R$ then this cluster is also likely to be ``typically large'' on each smaller scale $r$. 

\begin{lemma}
\label{lem:large_clusters_are_large_on_meso_scales_CD}
If $d\geq 3\alpha$ and $\alpha<2$ then for each $\eps>0$ there exists $\delta>0$ and $\lambda<\infty$ such that if $R\geq r$ then
\[
  \P\left(|K^r\cap B_{\lambda r}| \leq \delta \zeta(r) \;\middle|\; |K^R \cap B_R| \geq \zeta(R)\right) \leq \eps,
\]
where $K^r$ and $K^R$ denote the two clusters of the origin in the standard monotone coupling of the measures $\P_{\beta_c,r}$ and $\P_{\beta_c,R}$.
\end{lemma}

\begin{proof}[Proof of \cref{lem:large_clusters_are_large_on_meso_scales_CD}]
The proof is identical to that of \cref{II-lem:large_clusters_are_large_on_meso_scales}, which applies in low effective dimension. (In our context, the fact that $\E|K^r| \mathbbm{1}(|K^r \cap B_{\delta^{-1}r}| \leq \delta \zeta(r))$ is much smaller than $\E|K^r| \asymp r^\alpha$ when $\delta$ is small is an immediate consequence of \cref{I-cor:scaling_limit_cut_off}.)
\end{proof}

\begin{proof}[Proof of \cref{prop:three_point_lower}]
Fix a triple of distinct points $x,y,z\in \Z^d$.
We may assume without loss of generality that $\|x-y\|=d_\mathrm{min}(x,y,z)$ and $\|x-z\|=d_\mathrm{max}(x,y,z)$. 
Let $\delta>0$ and $\lambda<\infty$ be as in \cref{lem:large_clusters_are_large_on_meso_scales_CD} applied with $\eps=1/2$ and consider the standard monotone coupling of the measures $(\P_{\beta_c,r})_{r\geq 0}$; we write $\P$ and $\E$ for probabilities and expectations taken with respect to these coupled configurations and write $K^r_x$ for the cluster of $x$ in the level $r$ configuration.  Let $r=d_\mathrm{min}(x,y,z)$ and $R=4 \lambda d_\mathrm{max}(x,y,z)$. 
Let 
\[\mathscr{A}:=\Biggl\{|K_y^r\cap B_{\lambda r}(y)|, |K_x^r \cap B_{\lambda r}(x)| \geq \delta \zeta(r) \text{ and } |K_x^R \cap B_R(x)|, |K_z^R \cap B_R(z)| \geq \zeta(R)\Biggr\}.\]
We have by the Harris-FKG inequality, \eqref{eq:locally_large_cluster_order_estimate}, and \cref{lem:large_clusters_are_large_on_meso_scales_CD} that
\begin{multline}
\label{eq:all_scales_good_merging}
  \P(\mathscr{A}) 
  \geq 
  \P\Bigl(|K_y^r\cap B_{\lambda r}(y)|\geq \delta \zeta(r)\Bigr)\P\Bigl(|K_x^r \cap B_{\lambda r}(x)| \geq \delta \zeta(r) \text{ and } |K_x^R \cap B_R(x)| \geq \zeta(R)\Bigr) \\\cdot\P\Bigl(|K_z^R \cap B_R(z)| \geq \zeta(R)\Bigr)
   \succeq \frac{r^\alpha R^{2\alpha}}{\zeta(r)\zeta(R)^2}.
\end{multline}
Letting $\mathcal{F}_r$ denote the sigma-algebra generated by $(\omega_s:s\leq r)$, we have that
\begin{equation}
\label{eq:merge_scale_r}
  \P(y \in K^{4 \lambda r}_x | \mathcal{F}_r) \succeq  \mathbbm{1}\left(|K_y^r\cap B_{\lambda r}(y)|, |K_x^r \cap B_{\lambda r}(x)| \geq \delta \zeta(r)\right) \zeta(r)^2 r^{-d-\alpha}
\end{equation}
and
\begin{equation}
\label{eq:merge_scale_R}
  \P(x \in K^{4 R}_z | \mathcal{F}_R) \succeq  \mathbbm{1}\left(|K_x^R\cap B_{R}(x)|, |K_z^R \cap B_{R}(z)| \geq  \zeta(R)\right) \zeta(R)^2 R^{-d-\alpha},
\end{equation}
where in each case the bound arises from considering the possibility that the two clusters merge via the addition of a single edge that opens when passing from the configuration at level $2\lambda r$ to $4\lambda r$ or from $2R$ to $4R$ respectively (if the two clusters were not already equal); here we have used that $\zeta(r)^2 r^{-d-\alpha} =O(1)$, so that the probability such an edge becomes open is of the same order as the expected number of such edges. Applying Harris-FKG to the conditional measure given $\mathcal{F}_{r}$, we obtain from \eqref{eq:merge_scale_r} that
\begin{multline*}
  \P\left(\mathscr{A}, y\in K^{4\lambda r}_x \;\middle|\; \mathcal{F}_r\right)
    \succeq \zeta(r)^2 r^{-d-\alpha} \mathbbm{1}\left(|K_y^r\cap B_{\lambda r}(y)|, |K_x^r \cap B_{\lambda r}(x)| \geq \delta \zeta(r)\right)
    \\ \cdot \P\left( |K_x^R \cap B_R(x)|, |K_z^R \cap B_R(z)| \geq \zeta(R), y\in K^{4\lambda r}_x \;\middle|\; \mathcal{F}_r\right)
\end{multline*}
and hence, taking expectations, that
\[
  \P\left(\mathscr{A}, y\in K^{4\lambda r}_x\right)
    \succeq \zeta(r)^2 r^{-d-\alpha} \P(\mathscr{A}).
\]
Using this together with \eqref{eq:all_scales_good_merging} and \eqref{eq:merge_scale_R} we obtain that
\begin{multline*}
 \tau_{\beta_c}(x,y,z) \geq \P\left(\mathscr{A}, y\in K^{4\lambda r}_x, x \in K^{4R}_z\right)
    \succeq \zeta(r)^2 r^{-d-\alpha} \zeta(R)^2 R^{-d-\alpha} \P(\mathscr{A}) \\\succeq  \zeta(r) r^{-d} R^{-d+\alpha} \asymp V_r r^{-d+2\alpha} R^{-d+\alpha}
\end{multline*}
as claimed.
\end{proof}

\begin{remark}
The above proof can also be used to prove analogous pointwise lower bounds on general $k$-point functions as was done for the effectively low-dimensional case in \cref{II-prop:k-point_hyperscaling_lower}. Specifically, it can be shown that if $d\geq 3\alpha$ and $\alpha<2$ then
\begin{equation}
  \tau_{\beta_c}(A) \succeq_{|A|}  \frac{\operatorname{diam}(A)^{d+\alpha}}{\zeta(\operatorname{diam}(A))^2 } \prod_{x\in A} R_x^{-d} \zeta(R_x) \asymp 
  \frac{\operatorname{diam}(A)^{d-3\alpha}}{V_{\operatorname{diam}(A)}^2 } \prod_{x\in A} V_{R_x} R_x^{-d+2\alpha} 
\end{equation}
for every finite set of vertices $A$, where the radii $(R_x)_{x\in A}$ are as in the definition of $\operatorname{sweep}(A)$ (see \cref{II-def:sweep_and_spread}). We believe that a matching upper bound should hold via an elaboration of the proof of \cref{prop:three_point_upper} but do not pursue this here.
\end{remark}

\section{Logarithmic corrections to scaling}
\label{sec:logarithmic_corrections_at_the_critical_dimension}

The goal of this section is to prove \cref{thm:critical_dim_moments_main,cor:superprocess_main_CD}. As explained in the introduction, it suffices by \cref{I-thm:critical_dim_moments_main_slowly_varying,I-thm:superprocess_main_regularly_varying,thm:critical_dim_hydro} to prove the following proposition.

\begin{prop}
\label{prop:second_moment_critical_dim_asymptotics}
If $d=3\alpha<6$ then there exists a positive constant $A$ such that
\begin{equation}
\label{eq:second_moment_critical_dim_asymptotics}
  \E_{\beta_c,r}|K|^2 \sim \frac{\alpha}{\beta_c}\frac{A r^{3\alpha}}{\sqrt{\log r}}.
\end{equation}
as $r\to \infty$.
\end{prop}

 We will establish the asymptotic formula \eqref{eq:second_moment_critical_dim_asymptotics} by computing the logarithmic derivative of $\E_{\beta_c,r}|K|^2$ to second order as encapsulated in the following proposition. Recall that a measurable function $\delta:(0,\infty)\to \R$ is said to be a \textbf{logarithmically integrable error function} if satisfies $\delta(r)\to 0$ as $r\to\infty$ and $\int_{r_0}^\infty \frac{\delta_r}{r} \dif r <\infty$ for some $r_0<\infty$.

\begin{prop}
\label{prop:critical_dim_second_order}
Suppose that $d\geq 3\alpha$ and $\alpha<2$. There exists a positive constant $C$ and a logarithmically integrable error function $\delta:(0,\infty)\to \R$  such that
\[
  \frac{d}{dr}\E_{\beta_c,r}|K|^2 = \frac{3\alpha}{r} \left(1-(C\pm o(1)) \frac{(\E_{\beta_c,r}|K|^2)^2}{r^{d}(\E_{\beta_c,r}|K|)^3}+\delta_r\right) \E_{\beta_c,r}|K|^2
\]
as $r\to \infty$.
\end{prop}

The constant $C$ appearing here satisfies $0<C<2$ and is determined by the superprocess scaling limit of the model in a simple explicit way.
This explicit expression and related explicit expressions for the constant $A$ and the constant in the volume-tail estimate of \cref{thm:critical_dim_moments_main} are given in \cref{remark:explicit_constants}.

\medskip

We now explain how \cref{prop:critical_dim_second_order} allows us to complete the proof of \cref{thm:critical_dim_moments_main,cor:superprocess_main_CD}, beginning with a relevant ODE lemma. This lemma shows in particular that we do not need any quantitative guarantees on the implicit $\pm o(1)$ error term appearing in \cref{prop:critical_dim_second_order} to arrive at our desired conclusions.

\begin{lemma}
\label{lem:ODE_log_correction}
Suppose that $a,\gamma>0$ and that $f:(0,\infty)\to (0,\infty)$ satisfies the ODE
\[
  f'= 
  \frac{a}{r}(1-C_r (r^{-a}f)^\gamma+\delta_{r})f
\]
where $r\mapsto C_r$ is continuous a continuous function converging to some positive limit $C$ and where $r\mapsto \delta_{r}$
 is a
 logarithmically integrable error function. Then
\[
  f(r) \sim (a \gamma C \log r)^{-1/\gamma} r^{a}
\]
as $r\to \infty$.
\end{lemma}

\begin{proof}[Proof of \cref{lem:ODE_log_correction}]
The function $g=\exp[-a \int_1^r \frac{\delta_{s}}{s} \dif s]f$ satisfies the ODE
\begin{multline*}
  g' = ar^{-1}(1-C_r (r^{-a} f)^\gamma) g + ar^{-1}\delta_{r} g- a r^{-1}\delta_{r} g 
  \\= a r^{-1}(1 - e^{a \gamma \int_1^r \frac{\delta_{s}}{s}\dif s}  C_r r^{-a\gamma}g^\gamma) g
  = a r^{-1}(1 -\tilde C_r r^{-a\gamma}g^\gamma) g,
\end{multline*}
where we define $\tilde C_r =  \exp[\gamma a \int_1^r \frac{\delta_{s}}{s}\dif s]  C_r$.
Letting $h=r^{-a}g$, we have that
\[
  h' = - a r^{-1} h + a r^{-1} (1-\tilde C_r r^{-a\gamma }g^\gamma) h = - a \tilde C_r r^{-1} h^{\gamma+1}.
\]
This equality can be rewritten as
\[
  \left(\frac{1}{h^\gamma}\right)' = - \gamma \frac{h'}{h^{\gamma+1}} = \frac{a \gamma \tilde C_r}{r},
\]
which can be integrated to obtain that
\[
  \frac{1}{h^\gamma(r)} = \frac{1}{h^\gamma(1)} + \int_1^r \frac{a \gamma \tilde C_s}{s} \dif s.
\]
Since $\delta_{r}$ is a logarithmically integrable error function, $\tilde C_r$ converges to $e^{ a \gamma \int_1^\infty \frac{\delta_{s}}{s}\dif s}C>0$ as $r\to \infty$. It follows that
\[
  g(r) \sim r^{a} (a \gamma e^{a \gamma \int_1^\infty \frac{\delta_{s}}{s}\dif s}C \log r)^{-1/\gamma}
\]
as $r\to \infty$
and hence that
\[
  f(r) \sim e^{a \int_1^\infty \frac{\delta_{s}}{s} \dif s} g(r) \sim r^{a} (a \gamma C \log r)^{-1/\gamma}
\]
as $r\to \infty$ as claimed.
\end{proof}

\begin{proof}[Proof of \cref{prop:second_moment_critical_dim_asymptotics} given \cref{prop:critical_dim_second_order}]
Since $\E_{\beta_c,r}|K|\sim \frac{\alpha}{\beta_c} r^\alpha$ as $r\to \infty$ by \cref{thm:critical_dim_hydro} and Theorem \ref{I-thm:critical_dim_moments_main_slowly_varying}, it follows from \cref{prop:critical_dim_second_order} that there exists a positive constant $C>0$ such that
\[
  \frac{d}{dr}\E_{\beta_c,r}|K|^2 = \frac{3\alpha}{r} \left(1-(C\pm o(1)) \frac{\beta_c^3}{\alpha^3} \left(\frac{\E_{\beta_c,r}|K|^2}{r^{3\alpha}}\right)^2+\delta_{r}\right) \E_{\beta_c,r}|K|^2
\]
for some logarithmically integrable error function $\delta_{r}$, and hence by \cref{lem:ODE_log_correction} (applied with $a=3\alpha$ and $\gamma=2$) that 
\begin{equation}
\label{eq:C_to_A_constants}
  \E_{\beta_c,r}|K|^2 \sim \sqrt{\frac{\alpha^2}{6\beta_c^3 C}} \cdot \frac{r^{3\alpha}}{\sqrt{\log r}} = \frac{\alpha}{\beta_c} \cdot \frac{1}{\sqrt{6\beta_c C}} \cdot \frac{r^{3\alpha}}{\sqrt{\log r}}
\end{equation}
as $r\to \infty$.
\end{proof}

\begin{proof}[Proof of \cref{thm:critical_dim_moments_main,cor:superprocess_main_CD} given \cref{prop:critical_dim_second_order}]
The claims follow immediately from \cref{prop:second_moment_critical_dim_asymptotics} together with \cref{I-thm:critical_dim_moments_main_slowly_varying,I-thm:superprocess_main_regularly_varying}, which apply when $d=3\alpha$ by \cref{thm:critical_dim_hydro}.
\end{proof}

\begin{remark}
In the case that $\alpha<2$ and $d>3\alpha$, the calculations performed in this section still have content in that they can be used  to determine the second order corrections to scaling for $\E_{\beta_c,r}|K|$ and $\E_{\beta_c,r}|K|^2$ provided that $d-3\alpha$ is sufficiently small and the error in the asymptotic formula $|J'(r)||B_r| \sim r^{-\alpha-1}$ decays quickly enough for the associated logarithmically integrable error functions to be negligible with respect to the  error of order $r^{3\alpha-d}$ whose asymptotics are determined here.
\end{remark}

\begin{remark}
It is possible to get a \emph{lower bound} of the correct order on $\E_{\beta_c,r}|K|^2$ by a much simpler argument than that used to prove the upper bound (which we do not know how to do without obtaining the first-order asymptotic estimate of \eqref{eq:second_moment_critical_dim_asymptotics}).
Suppose that $d=3\alpha$.
Given \cref{thm:critical_dim_hydro} and the resulting relationships between moments established in Theorem \ref{I-thm:critical_dim_moments_main_slowly_varying}, it follows easily from Lemma \ref{I-lem:E1_1} and \eqref{I-eq:E2_bound} that there exists a positive constant $\tilde C$ such that the one-sided estimate
\[
  \frac{d}{dr}\E_{\beta_c,r}|K|^2 \geq \frac{3\alpha}{r} \left(1-(1\pm o(1))\tilde C \frac{(\E_{\beta_c,r}|K|^2)^2}{r^{d}(\E_{\beta_c,r}|K|)^3}+\delta_r\right) \E_{\beta_c,r}|K|^2
\]
holds as $r\to \infty$, 
where $\delta_r$ is a logarithmically integrable error function. (Indeed one may take $\tilde C=15=5!!$ by bounding the error $\cE_{1,r}$ below by zero and the error $\cE_{2,r}$ above in terms of $\E_{\beta_c,r}|K|^4 \sim 5!! \E_r|K|(\hat \E_r|K|)^3$. The true limiting constant $C$ is at most $2$.) This implies by a similar analysis to above that
\[
  \E_{\beta_c,r}|K|^2 \succeq \frac{r^{3\alpha}}{\sqrt{\log r}}
\]
for every $r\geq 1$. 
\end{remark}

\subsection{A tale of two vertex factors}
\label{subsec:a_tale_of_two_vertex_factors}

 We now explain the big-picture strategy behind the proof of \cref{prop:critical_dim_second_order} including heuristic interpretations of various important intermediate steps. (Throughout this discussion we will omit the hypothesis that the hydrodynamic condition holds from the $d=3\alpha$ case of the results of \cite{LRPpaper1} since this has now been proven to hold in \cref{thm:critical_dim_hydro}.) We will lighten notation by writing $\E_r=\E_{\beta_c,r}$ for the remainder of the paper.
 In Section \ref{I-sec:analysis_of_moments} we used Russo's formula to derive the derivative formulae
\begin{equation}
\label{eq:first_moment_ODE_TTVF}
\frac{1}{\beta_c |J'(r)|} \frac{d}{dr}\E_{r} |K| = |B_r| (\E_{r} |K|)^2 -
\sum_{y\in B_r}((\E_r|K|)^2-\E_r|K||K_y|\mathbbm{1}(0\nleftrightarrow y))
\end{equation}
and
\begin{equation}
\frac{1}{\beta_c |J'(r)|} \frac{d}{dr}\E_{r} |K|^2 = 3|B_r| \E_{r} |K|\E_{r} |K|^2 -3 \sum_{y\in B_r}(\E_r|K|\E_r|K|^2-\E_r|K||K_y|^2\mathbbm{1}(0\nleftrightarrow y)),
 \label{eq:second_moment_ODE_TTVF}
\end{equation}
which we showed in the proofs of \cref{I-prop:first_moment,I-prop:second_moment} to yield the simple asymptotic ODEs
\[
  \frac{d}{dr}\E_r |K| \sim \beta_c r^{-\alpha-1}(\E_r|K|)^2
\qquad \text{ and } \qquad
  \frac{d}{dr}\E_r |K|^2 \sim 3\beta_c r^{-\alpha-1}\E_r|K|\E_r|K|^2
\]
when $d\geq 3\alpha$. To prove \cref{prop:critical_dim_second_order}, we need to compute the asymptotics of the subleading order term in both \eqref{eq:first_moment_ODE_TTVF} and \eqref{eq:second_moment_ODE_TTVF}.
To do this, it is convenient to introduce the quantities
\begin{align*}
  \mathscr{D}_r^{(1)}(x,y)&:=\E_r|K_x|\E_r|K_y|-\E_r|K_x||K_y|\mathbbm{1}(x\nleftrightarrow y)=(\E_r|K|)^2-\E_r|K_x||K_y|\mathbbm{1}(x\nleftrightarrow y)\intertext{and}
  \mathscr{D}_r^{(2)}(x,y)&:=\E_r|K_x|\E_r|K_y|^2-\E_r|K_x||K_y|^2\mathbbm{1}(x\nleftrightarrow y)=\E_r|K|\E_r|K|^2-\E_r|K_x||K_y|^2\mathbbm{1}(x\nleftrightarrow y),
\end{align*}
so that if we define
\[
    \cE_{1,r} = 
    \frac{\sum_{y\in B_r}\mathscr{D}_r^{(1)}(0,y)}{|B_r|(\E_{r}|K|)^2}
\qquad\text{ and }\qquad
    \cE_{2,r} = \frac{\sum_{y\in B_r}\mathscr{D}_r^{(2)}(0,y)}{|B_r|\E_{r}|K|\E_{r}|K|^2}
\]
then
\begin{equation}
\label{eq:first_moment_ODE_TTVF2}
\frac{d}{dr}\E_{r} |K| = \beta_c |J'(r)||B_r| (\E_{r} |K|)^2 (1-\cE_{1,r})
\end{equation}
and
\begin{equation}
 \frac{d}{dr}\E_{r} |K|^2 = 3 \beta_c |J'(r)||B_r| \E_{r} |K|\E_{r} |K|^2 (1-\cE_{2,r}).
 \label{eq:second_moment_ODE_TTVF2}
\end{equation}
To proceed, we will prove a scaling limit theorem for the two error functions $\mathscr{D}_r^{(1)}(x,y)$ and $\mathscr{D}_r^{(2)}(x,y)$ analogous to that established for the size-biased law of a single cluster in \cref{I-thm:scaling_limit_diagrams}; this will enable us to compute the asymptotics of the two error terms $\cE_{1,r}$ and $\cE_{2,r}$ and hence to prove \cref{prop:critical_dim_second_order} as explained below.

\medskip

Before stating the relevant scaling limit theorem for the error terms, let us first recall our scaling limit theorem for the law of a single cluster under the size-biased cut-off measure $\hat \E_r$. We will restrict to the case $\alpha<2$ for the purposes of this discussion since this is the only case relevant for \cref{prop:critical_dim_second_order}.
 In \cref{I-thm:scaling_limit_diagrams}, it is shown that if $d\geq 3\alpha$ and $\alpha<2$ then
\begin{multline}
  \E_{r} \left[ \sum_{x_1,\ldots,x_n \in K} \prod_{i=1}^n P_i\left(\frac{x_i}{r}\right)\right] \\= 
\left(\sum_{T\in \mathbb{T}_n} \idotsint \prod_{\substack{i<j \\ i\sim j}} \kappa(x_i-x_j) \prod_{i=1}^n P_i(x_i)\dif x_1 \cdots \dif x_{2n-1} \pm o(1)\right) \left(\frac{\E_{r}|K|^2}{\E_{r}|K|}\right)^{n-1}\E_r|K|
\label{eq:scaling_limit_diagrams}
\end{multline}
as $r\to \infty$ 
for each sequence of polynomials $P_1,\ldots,P_n:\R^d \to \R$, where the sum is taken over isomorphism classes of trees with leaves labelled $0,1,\ldots,n$ and unlabelled non-leaf vertices all of degree $3$ and where $\kappa:\R^d\to [0,\infty)$ is a measurable, symmetric kernel with $\int_{\R^d}\kappa(x)\dif x=1$ making all integrals on the right hand side finite.
(For notational purposes we take our trees to have vertex set $\{1,\ldots,2n-1\}$ even though we regard the internal vertices $\{n+1,\ldots,2n-1\}$ as unlabelled; isomorphisms are not required to preserve the labels of the internal vertices.)
The kernel
$\kappa$ arises explicitly as the density of the 
law of the symmetric L\'evy jump process with L\'evy measure 
\begin{equation}
\label{eq:Levy_measure}
  \frac{d\Pi_1(x)}{dx} = \alpha \mathbbm{1}(\|x\|\leq 1)\int_{\|x\|}^1 s^{-d-\alpha-1}\dif s 
\end{equation}
 stopped at an independent Exp$(1)$-distributed random time. 
 We will not need to work with this definition directly in this paper; all its properties we need are encapsulated in the recurrence relation \eqref{eq:recurrence_from_derivative_LR3}.
  The asymptotic formula \eqref{eq:scaling_limit_diagrams} implies that the size-biased cut-off measure $\hat \E_r$ has a super-L\'evy scaling limit as $r\to \infty$ as explained in \cref{I-cor:scaling_limit_cut_off}.

\medskip

Diagrammatically, \eqref{eq:scaling_limit_diagrams} can be thought of as giving asymptotic expressions for $k$-point functions under the cut-off measure $\P_r$ of the form
\begin{align*}
  \P_r(x\leftrightarrow y\leftrightarrow z) &\approx 
  \begin{array}{l}
\includegraphics{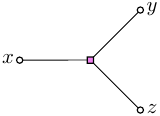}
  \end{array} = V_r \sum_{w\in \Z^d} \P_r(x\leftrightarrow w)\P_r(w\leftrightarrow y)\P_r(w\leftrightarrow z),
\\
   \P_r(x\leftrightarrow y\leftrightarrow z\leftrightarrow w) &\approx 
  \begin{array}{l}
\includegraphics{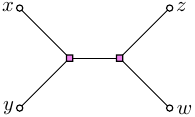}
  \end{array}
  +
    \begin{array}{l}
\includegraphics{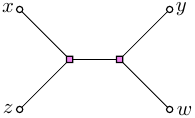}
  \end{array}
  +
    \begin{array}{l}
\includegraphics{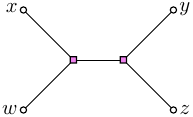}
  \end{array},
\end{align*}
and so on, where each line represents a copy of the two-point function
\[
  \P_r(x\leftrightarrow y) \approx r^{-d}\E_r|K| \kappa\Bigl(\frac{x-y}{r}\Bigr) \approx \frac{\alpha}{\beta_c}r^{-d+\alpha} \kappa\Bigl(\frac{x-y}{r}\Bigr),
\]
internal vertices are summed over, and the violet box at each internal vertex denotes the fact that we must include a copy of the \textbf{vertex factor}
\begin{equation}
\label{eq:vertex_factor_def}
  V_r := \frac{\E_r|K|^2}{(\E_r|K|)^3} \sim \frac{\beta_c^2}{\alpha^2}A_r,
\end{equation}
where $r\mapsto A_r$ is the slowly-varying function from \eqref{eq:critical_dim_moments_main}.
  (We keep the precise meaning of the asymptotic equivalence $\approx$ imprecise, using this notation only in heuristic arguments. The main caveat here is that \eqref{eq:scaling_limit_diagrams} establishes these asymptotic estimates only in the sense of moments, rather than pointwise, and indeed these estimates will not be accurate for points with distance much smaller than $r$.)
The vertex factor $V_r$ accounts for the interactions between the different parts of the cluster that prevent the tree-graph upper bound
$\P_r(x\leftrightarrow y\leftrightarrow z) \leq \sum_{w\in \Z^d} \P_r(x\leftrightarrow w)\P_r(w\leftrightarrow y)\P_r(w\leftrightarrow z)$
 from being an equality, including the fact that the point where the two paths meet is not unique and that the three ``arms'' of the cluster are not really independent.
When $d>3\alpha$ the vertex factor converges to a constant  whereas for $d=3\alpha<6$ 
it follows from \cref{thm:critical_dim_hydro} and \cref{I-lem:mixed_moments_order_estimates} (restated here as \eqref{eq:mixed_moments_order_estimates}) that 
 the vertex factor is divergently small:
\begin{equation}
  V_r \preceq r^{-3\alpha}\E_r|K|^2 \asymp r^{-3\alpha}\E_r|K \cap B_{Cr}|^2 \preceq r^{-3\alpha}M_{Cr} \E_r|K \cap B_{Cr}| = o(1)
\label{eq:vertex_factor_divergently_small},
\end{equation}
where we applied the universal tightness theorem \cite{hutchcroft2020power} as in \eqref{I-eq:E1_from_M} in the penultimate estimate.
   Intuitively, the fact that the vertex factor is converging to zero in a slowly varying fashion means that the interactions between different parts of the cluster are spread out over a large range of mesoscopic scales in contrast to high effective dimensions where they occur primarily on microscopic scales (i.e., on scales of the same order as the lattice spacing). (In low effective dimensions there is a constant interaction strength on every scale, leading to \emph{power-law} decay of the vertex factor.)

\medskip 

What should a similar theorem look like for the error terms $\mathscr{D}_r^{(1)}(x,y)$ and $\mathscr{D}_r^{(2)}(x,y)$? Recall that we can sample the two clusters $K_x$ and $K_y$ by first sampling a pair of \emph{independent} clusters $K_x$ and $\tilde K_y$ and then setting the status of any edge explored by both clusters using its status as revealed by the cluster $K_x$. (See e.g.\ the proof of \cref{lem:locally_large_measure_concrete} for details.) In this coupling, the cluster $K_y$ of $y$ is either equal to $K_x$ (if $y\in K_x$) or is contained in the independent cluster $\tilde K_y$, with a vertex $z$ of $\tilde K_y$ contained in the cluster $K_y$ if and only if there exists a path from $y$ to $z$ in $\tilde K_y$ that is disjoint from the cluster $K_x$. As such, the quantity $\mathscr{D}_r^{(1)}(x,y)$ can be thought of as the expected number of points of $\tilde K_y$ that are ``cut off'' from $y$ by the cluster $K_x$ in this way, weighted by the size of $K_x$.
As such, it seems reasonable to expect a diagrammatic estimate for $\mathscr{D}_r^{(1)}(x,y)$ of the form
\begin{equation}
  \mathscr{D}_r^{(1)}(x,y) \approx \begin{array}{l}\includegraphics{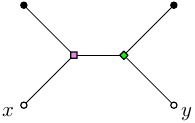}\end{array} \approx r^{-d} (\E_r|K|)^5 V_r \tilde V_r \cdot [\kappa*\kappa*\kappa]\left(\frac{x-y}{r}\right)
  \label{eq:overlap_vertex1}
\end{equation}
where the two black leaf vertices and the two internal vertices are summed over, the violet square on the left internal vertex represents a copy of the vertex factor $V_r$ as before, and the green diamond on the right internal vertex represents another kind of vertex factor, denoted by $\tilde V_r$,  accounting for the precise way in which the cluster $K_x$ intersects with the path connecting $y$ to to the point of $\tilde K_y$ being summed over; this second vertex factor $\tilde V_r$ should account for the ratio between the probability that a point $w$ lies on an open simple path $y\leftrightarrow z$ with the upper bound $\P_r(y\leftrightarrow w)\P_r(w\leftrightarrow z)$, the number of intersections of $K_x$ with the path given that such an intersection exists, and so on. 
(If desired, one could define $\tilde V_r$ precisely via the equality $\sum_{y\in \Z^d} \mathscr{D}_r^{(1)}(0,y)=V_r\tilde V_r (\E_r|K|)^5$ making the integrated version of \eqref{eq:overlap_vertex1} true by definition; for our purposes one can keep the precise definition unspecified.)
Similarly, it is reasonable to expect that $\mathscr{D}_r^{(2)}(x,y)$ admits a diagrammatic estimate of the form
\begin{multline}
  \mathscr{D}_r^{(2)}(x,y) \approx \begin{array}{l}\includegraphics{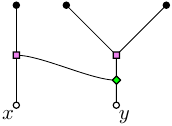}\end{array}+\begin{array}{l}\includegraphics{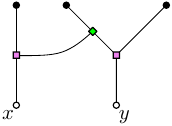}\end{array}
  +
  \begin{array}{l}\includegraphics{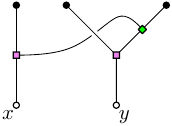}\end{array}
  \label{eq:overlap_vertex2}
  \\
  \approx r^{-d}(\E_r|K|)^7 V_r^2 \tilde V_r \cdot [\kappa^{*3}+2\kappa^{*4}]\left(\frac{x-y}{r}\right)
\end{multline}
with two copies of the vertex factor $V_r$ and one copy of the other vertex factor $\tilde V_r$, where $\kappa^{*n}$
 denotes the $n$th convolution power of $\kappa$. (The expression $\kappa^{*3}+2\kappa^{*4}$ arises since there is one diagram with $x$ and $y$ connected by a path of length three and two diagrams with $x$ and $y$ connected by a path of length four.)

\medskip

Given sufficiently strong precise versions of \eqref{eq:overlap_vertex1} and \eqref{eq:overlap_vertex2}, one could sum over the ball $B_r$ to deduce that
\[
  \cE_{1,r} \sim \frac{V_r \tilde V_r \int_{B}\kappa^{*3}(y)\dif y (\E_r|K|)^5}{|B_r|(\E_r|K|)^2} \sim \left(\frac{\alpha}{\beta_c}\right)^3 V_r \tilde V_r \left[\int_{B}\kappa^{*3}(y)\dif y\right] r^{3\alpha-d}
\]
and
\[
  \cE_{2,r} \sim \frac{V_r^2 \tilde V_r \int_{B}(\kappa^{*3}(y)+2\kappa^{*4}(y))\dif y (\E_r|K|)^7}{|B_r|\E_r|K|\E_r|K|^2} \sim \left(\frac{\alpha}{\beta_c}\right)^3 V_r \tilde V_r \left[\int_{B}(\kappa^{*3}(y)+2\kappa^{*4}(y))\dif y\right] r^{3\alpha-d}
\]
as $r\to\infty$, 
where $B$ denotes the unit ball of $\|\cdot\|$.
These estimates may appear to be of limited use to compute the model's logarithmic corrections to scaling when $d=3\alpha<6$ since they involve the newly introduced vertex factor $\tilde V_r$ whose relation to the original vertex factor $V_r$ is unclear; to use these estimates to deduce an estimate of the form claimed by \cref{prop:critical_dim_second_order} one would need  that $\tilde V_r$ is asymptotic to a constant multiple of $V_r$, and the above heuristic discussion gives no obvious indication that this should be the case. 

\medskip

It is here that we encounter a miracle: the two vertex factors $V_r$ and $\tilde V_r$ are not only of the same order but in fact \emph{asymptotically equal} to each other, both in high dimensions and at the critical dimension. This is made precise in the following theorem, which establishes the validity of the diagrammatic estimates \eqref{eq:overlap_vertex1} and \eqref{eq:overlap_vertex2} in the sense of moments together with the asymptotic equality of the two different vertex factors. 

\begin{theorem}[Scaling limits of second-order corrections]
\label{thm:correction_moments}
If $d\geq 3\alpha$ and $\alpha<2$  then
\begin{equation}
\label{eq:correction_moment1}
\sum_{y\in \Z^d} \mathscr{D}^{(1)}_r(0,y)P\Bigl(\frac{y}{r}\Bigr) =\left[\int_{\R^d} \kappa^{*3}(y)P(y) \dif y \pm o(1)\right]V_r^2(\E_r|K|)^4
 \end{equation}
and
\begin{equation}
\label{eq:correction_moment2}
\sum_{y\in \Z^d} \mathscr{D}^{(2)}_r(0,y)P\Bigl(\frac{y}{r}\Bigr) =\left[\int_{\R^d} (\kappa^{*3}(y)+2\kappa^{*4}(y))P(y) \dif y \pm o(1)\right] V_r^3(\E_r|K|)^7
 \end{equation}
 as $r\to \infty$ for each polynomial $P:\R^d\to \R$.
\end{theorem}

The proof of this theorem relies on the same techniques used to prove our scaling limit theorems in \cref{I-sec:superprocesses}: we establishment asymptotic ODEs for the moments of $\mathscr{D}_r^{(1)}(0,y)$ and $\mathscr{D}_r^{(2)}(0,y)$, show that these ODEs lead to recurrence relations for the constant prefactors on the right hand side of \eqref{eq:correction_moment1} and \eqref{eq:correction_moment2}, and finally show that these recurrence relations are solved by the moments of $\kappa^{*3}$ or $\kappa^{*3}+2\kappa^{*4}$ as appropriate (this last fact being a special case of \cref{I-prop:recurrence_from_derivative}). The key point is that we can write these asymptotic ODEs in such a way that some of the leading-order terms are expressed solely in terms of quantities we have already computed the asymptotics of in \cref{I-thm:scaling_limit_diagrams}, which leads to the asymptotic equivalence between the two (\emph{a priori}) different vertex factors $V_r$ and $\tilde V_r$ discussed above. Beyond the methods and results of \cref{I-sec:superprocesses}, the main additional technical ingredient required to prove \cref{thm:correction_moments} is \cref{lem:triple_interaction}, which expresses the ``triple interaction'' between three distinct clusters as a sum of ``pair interactions'' modulo negligible correction terms.

\medskip
Although these calculations are easy enough to follow,
it would still be very interesting to have a more combinatorial or conceptual explanation of why the two vertex factors $V_r$ and $\tilde V_r$, which arise from the need to account for two apparently very different interactions between percolation clusters, are in fact asymptotically equal. The fact that this identity holds even in high-dimensions, where both vertex factors should be determined by the small-scale behaviour of the model and hence highly dependent on the fine details of the kernel $J$, suggests that some version of this equality may hold at a very high level of generality.

\medskip

Let us now see how \cref{thm:correction_moments} can be used to conclude the proof of \cref{prop:critical_dim_second_order}. We refer the reader to \cref{I-subsec:regular_variation_and_logarithmic_integrability} for relevant background on functions of regular variation.

\begin{proof}[Proof of \cref{prop:critical_dim_second_order} given \cref{thm:correction_moments}]
It follows from \cref{thm:correction_moments} and Carleman's criterion (which applies to the moments of $\kappa$ and its convolutions by the generating function computations of \cref{I-prop:displacement_moments}) that
\[
\sum_{y\in B_r} \mathscr{D}^{(1)}_r(0,y) \sim \left[\int_{B} \kappa^{*3}(y) \dif y \right]V_r^2(\E_r|K|)^5 = \left[\int_{B} \kappa^{*3}(y) \dif y \right] \frac{(\E_r|K|^2)^2}{\E_r|K|}
 \]
and
\[
\sum_{y\in B_r} \mathscr{D}^{(2)}_r(0,y) \sim \left[\int_{B} (\kappa^{*3}(y)+2\kappa^{*4}(y)) \dif y \right] V_r^3(\E_r|K|)^7 = \left[\int_{B} (\kappa^{*3}(y)+2\kappa^{*4}(y)) \dif y \right] \frac{(\E_r|K|^2)^3}{(\E_r|K|)^2}
 \]
 as $r\to \infty$ and hence that
 \begin{equation}
 \label{eq:E1r_asymptotics_final}
  \cE_{1,r} 
   \sim  \left[\int_{B}\kappa^{*3}(y)\dif y\right] \frac{(\E_r|K|^2)^2}{|B_r|(\E_r|K|)^3}
\end{equation}
and
\begin{equation}
\label{eq:E2r_asymptotics_final}
  \cE_{2,r} 
   \sim  \left[\int_{B}(\kappa^{*3}(y)+2\kappa^{*4}(y))\dif y\right] \frac{(\E_r|K|^2)^2}{|B_r|(\E_r|K|)^3}
\end{equation}
as $r\to \infty$. In particular, both quantities $\cE_{1,r}$ and $\cE_{2,r}$ are slowly varying by Theorems \ref{I-thm:critical_dim_moments_main_slowly_varying} and \ref{thm:critical_dim_hydro} and converge to zero as $r\to \infty$ by \eqref{eq:vertex_factor_divergently_small}.
As in \cref{I-sec:analysis_of_moments}, we define the errors $\cE_{0,r}$ and $\overline{\cE}_{1,r}$ by
\[
(1-\cE_{0,r}):=\frac{|J'(r)||B_r|}{r^{-\alpha-1}},\quad (1-\overline{\cE}_{1,r}) = (1-\cE_{1,r})(1-\cE_{0,r}), \;\;\text{ and }\;\;
(1-\overline{\cE}_{2,r}) = (1-\cE_{2,r})(1-\cE_{0,r})
\]
and define the error function $\cH_r$ by
\[
  \E_r|K|=(1+\cH_r) \frac{\alpha}{\beta_c} r^\alpha
\]
so that \eqref{eq:first_moment_ODE_TTVF2} and \eqref{eq:second_moment_ODE_TTVF2} can be written more simply as
\begin{equation}
\label{eq:first_moment_ODE_TTVF3}
\frac{d}{dr}\E_{r} |K| =  (1-\overline{\cE}_{1,r}) \beta_c r^{-\alpha-1} (\E_{r} |K|)^2 
\end{equation}
and
\begin{equation}
 \frac{d}{dr}\E_{r} |K|^2 = (1-\overline{\cE}_{2,r})(1+\cH_r) 3 \alpha  r^{-1} \E_{r} |K|^2 .
 \label{eq:second_moment_ODE_TTVF3}
\end{equation}
Our assumptions \eqref{eq:normalization_conventions} on the kernel $J$ together with \cref{I-lem:ball_regularity} ensure that $\cE_{0,r}$ is a logarithmically integrable error function.
Since $\cE_{1,r}$ converges to zero as $r\to \infty$, it follows from \cref{I-lem:f'=f^2} (see \eqref{I-eq:first_moment_E1_formula}) that
\[
  1+\cH_r = \left(1-\alpha r^\alpha \int_r^\infty \overline{\mathcal{E}}_{1,s} s^{-\alpha-1} \dif s\right)^{-1}.
\]
As such, if we define 
\[\tilde \cH_r:= \left(1-\alpha r^\alpha \int_r^\infty \mathcal{E}_{1,s} s^{-\alpha-1} \dif s\right)^{-1}-1
\]
then the difference $\cH_r-\tilde \cH_r$ is a logarithmically integrable error function and 
\begin{equation}
\label{eq:tildeHr_asymptotics_final}
  \tilde \cH_r \sim \alpha r^\alpha \int_r^\infty \mathcal{E}_{1,s} s^{-\alpha-1} \dif s \sim  \cE_{1,r},
\end{equation}
where we used that $r^{-\alpha-1}\cE_{1,r}$ is regularly varying of index $-\alpha-1<-1$ in the final estimate (see \cref{I-lem:regular_variation_integration}). As such, it follows from \eqref{eq:E1r_asymptotics_final}, \eqref{eq:E2r_asymptotics_final}, \eqref{eq:second_moment_ODE_TTVF3}, and \eqref{eq:tildeHr_asymptotics_final} that
\begin{equation}
  \frac{d}{dr}\E_r|K|^2 = 3\alpha r^{-1} \left(1 -  \left[2\int_B \kappa^{*4}(y)\dif y\pm o(1)\right] \frac{(\E_r|K|^2)^2}{|B_r|(\E_r|K|)^3}+\delta_r \right) \E_r|K|^2
\label{eq:final_answer_C}
\end{equation}
for some logarithmically integrable error function $\delta_r$, where the constant appearing here is the difference of the two constants in \eqref{eq:E1r_asymptotics_final} and \eqref{eq:E2r_asymptotics_final}. This concludes the proof. 
\end{proof}

\begin{remark}
\label{remark:explicit_constants}
Using the explicit formula 
  $C=2\int_B \kappa^{*4}(y)\dif y$
for the constant appearing in \cref{prop:critical_dim_second_order},
 we can apply \eqref{eq:C_to_A_constants} to obtain the explicit asymptotic formula
\begin{equation}
\label{eq:Ar_asymptotics_explicit}
A_r \sim 
\frac{1}{\sqrt{12 \beta_c \int_B \kappa^{*4}(y)\dif y}} \cdot \frac{1}{\sqrt{\log r}} 
\end{equation}
as $r\to\infty$. Using \eqref{eq:volume_tail_exact_constant_intro}, we obtain the explicit asymptotic  volume-tail formula
\begin{equation}
\label{eq:volume_tail_explicit}
\P_{\beta_c}(|K|\geq n) \sim 
\frac{\alpha}{\beta_c}\sqrt{\frac{2}{\pi}} \left(\frac{6\beta_c}{\alpha}\int_B \kappa^{*4}(y)\dif y\right)^{1/4}
 \cdot\frac{(\log n)^{1/4}}{n^{1/2}}
\end{equation}
as $n\to\infty$. The integral $\int_B \kappa^{*4}(y) \dif y$ appearing here
 can be interpreted as the probability that the symmetric L\'evy jump process with L\'evy measure $\Pi_1$ (defined in \eqref{eq:Levy_measure}) belongs to the unit ball $B$ when stopped at an independent $\operatorname{Gamma}(4,1)$ random time. 
\end{remark}

\begin{remark}
The framework we have discussed here can also be used to explain the numerology appearing in the hierarchical case \cite[Proposition 6.5 and Corollary 6.6]{hutchcroft2022critical}, which we left somewhat mysterious in that earlier work, in terms of diagram counting. (This remark is intended only for the reader that is familiar with the computation of logarithmic corrections to scaling in \cite{hutchcroft2022critical} and can safely be ignored by others.) In the hierarchical case, it is impossible for clusters in the cut-off model to leave the ball, so that the appropriate analogues of $\int_B \kappa^{*3}(y) \dif y$ and $\int_B \kappa^{*4}(y) \dif y$ are equal to~$1$. The covariances studied in \cite[Proposition 6.5]{hutchcroft2022critical} can be expressed as the difference of the cluster moments with the sums of the analogues of the quantities $\mathscr{D}_r^{(1)}(0,y)$
and 
$\mathscr{D}_r^{(2)}(0,y)$ considered here, and the ratios $2/3$ and $4/5$ appearing in the same proposition arise as 
$1-(1/3)$ and $1-(3/15)$: in the first quantity there is one diagram contributing to $\sum_y\mathscr{D}_r^{(1)}(0,y)$ as in \eqref{eq:overlap_vertex1} and three contributing to the third moment, while in the second there are three diagrams contributing to $\sum_y\mathscr{D}_r^{(2)}(0,y)$ as in \eqref{eq:overlap_vertex2} and fifteen contributing to the fourth moment.
\end{remark}

We note that we have now also concluded the proof of \cref{thm:pointwise_three_point} given \cref{thm:correction_moments}.

\begin{proof}[Proof of \cref{thm:pointwise_three_point} given \cref{thm:correction_moments}]
The two-point estimate was already established in \cref{II-thm:CL_Sak}, which applies when $d=3\alpha<6$ as a consequence of \cref{thm:critical_dim_hydro} (and \cref{I-prop:displacement_moments}) as noted in \cref{II-cor:HD_CL}. The three-point estimate follows immediately from \cref{prop:three_point_upper,prop:three_point_lower} and the asymptotic estimate
\begin{equation}
  V_r \sim \frac{\mathrm{const.}}{\sqrt{\log r}}
\end{equation}
established in \cref{prop:second_moment_critical_dim_asymptotics}.
\end{proof}

\subsection{Correction asymptotics}

In this section we prove \cref{thm:correction_moments} subject to a technical estimate on ``triple interactions'' (\cref{lem:triple_interaction}) whose proof is deferred to \cref{subsec:triple_interactions}. 
%
We say that a polynomial $P:\R^d\to \R$ is a \textbf{monomial} if it can be written $P(x)=\prod_{i=1}^n\langle x,u_i\rangle$ for some sequence of unit vectors $u_1,\ldots,u_n$, where we allow our monomial to be given by the empty product $P(x)=1$. The number of terms in this sequence is the \textbf{degree} of $P$. It suffices to prove \cref{thm:correction_moments} for monomials, in which case we can write $P(y/r)=r^{-\deg(P)}P(y)$ and move the power of $r$ to the other side of the asymptotic equality. Throughout this section we will allow all implicit constants to depend on the choice of monomial $P$, and omit this dependence from our notation.
Beyond the fact that $\kappa$ has $\int \kappa(x)\dif x=1$ and has well-defined moments of all orders, the only specific feature of the kernel $\kappa$ we will need in this paper is the identity \eqref{I-eq:recurrence_from_derivative_LR3}, a special case of \cref{I-prop:recurrence_from_derivative} which (when $\alpha<2$ as we are assuming throughout this section) states that
  \begin{equation}
 \label{eq:recurrence_from_derivative_LR3}
 (\deg(P)+\alpha n) \int_{\R^d} \kappa^{*n}(y)P(y) \dif y =
  \alpha n \int_{\R^d} [\kappa^{*(n+1)} * \mathbbm{1}_B](y) P(y) \dif y
\end{equation}
for each $n\geq 1$ and each monomial $P$.

\medskip

To prove \cref{thm:correction_moments}, we begin by noting that the absolute values of the moments of $\mathscr{D}_r^{(1)}(0,y)$ and $\mathscr{D}_r^{(2)}(0,y)$ satisfy upper bounds of the correct order as an immediate consequence of the analogous bounds \cref{I-lem:mixed_moments_order_estimates}. 

\begin{lemma}
\label{lem:correction_order_estimates}
If $d\geq 3\alpha$ and $\alpha<2$ then the estimates
\begin{equation*}
  \sum_{y\in \Z^d} \left|\mathscr{D}_r^{(1)}(0,y) P(y)\right| \preceq r^{\deg(P)} V_r^2 (\E_r|K|)^5
\quad\text{ and }\quad
 \sum_{y\in \Z^d}  \left|\mathscr{D}_r^{(2)}(0,y) P(y)\right| \preceq r^{\deg(P)} V_r^3 (\E_r|K|)^7
\end{equation*}
hold for each monomial $P$.
\end{lemma}

\begin{proof}[Proof of \cref{lem:correction_order_estimates}]
It follows from \cref{I-lem:disjoint_connections} that 
\[
  0\leq \mathscr{D}_r^{(1)}(0,y) \leq \E_r |K|^2 \mathbbm{1}(0\leftrightarrow y)
\qquad \text{and} \qquad
  0\leq \mathscr{D}_r^{(2)}(0,y) \leq \E_r |K|^3 \mathbbm{1}(0\leftrightarrow y)
\]
for every $r\geq 1$ and $y\in \Z^d$, so that
\begin{equation*}
  \sum_{y\in \Z^d} \left|\mathscr{D}_r^{(1)}(0,y) P(y)\right| \leq \E_r|K|^2 \sum_{y\in K}|P(y)|
\quad\text{ and }\quad
 \sum_{y\in \Z^d}  \left|\mathscr{D}_r^{(2)}(0,y) P(y)\right| \leq \E_r|K|^3 \sum_{y\in K}|P(y)|
\end{equation*}
and the claim follows from \cref{I-lem:mixed_moments_order_estimates} (a special case of which is restated here as \eqref{eq:mixed_moments_order_estimates}).
\end{proof}


 We next write down exact expressions for the derivatives of $\mathscr{D}_r^{(1)}(0,y)$ and $\mathscr{D}_r^{(2)}(0,y)$, which hold regardless of the value of $d$ and $\alpha$.

\begin{lemma}
\label{lem:correction_exact_derivative}
If $y\in \Z^d$ and $r\geq 1$ then 
  \begin{multline*}
\frac{1}{\beta_c|J'(r)|} \frac{d}{dr} \mathscr{D}_r^{(1)}(0,y)=
2\E_r \sum_{x\in K} \sum_{z \in B_r(x)} \mathbbm{1}(0\nleftrightarrow z) |K_z| (\E|K_y|-\mathbbm{1}(y\nleftrightarrow 0,z)|K_y| )
\\+
\E_r \sum_{x\in K} \sum_{z\in K_y} \mathbbm{1}(0\nleftrightarrow y, \|x-z\| \leq r) |K||K_y|
\end{multline*}
and
\begin{align*}
  \frac{1}{\beta_c|J'(r)|}\frac{d}{dr} \mathscr{D}_r^{(2)}(0,y) &= 
  \E_r \sum_{x\in K}\sum_{z\in B_r(x)} \mathbbm{1}(0\nleftrightarrow z)(|K_z|^2+2|K||K_z|) ( \E_r|K_y|-\mathbbm{1}(y\nleftrightarrow 0,z)|K_y|)\\ &\hspace{2.7cm}+ \E_r \sum_{x\in K} \sum_{z\in B_r(x)} \mathbbm{1}(0\nleftrightarrow z)|K_z|  (\E_r|K_y|^2-\mathbbm{1}(y\nleftrightarrow 0,z)|K_y|^2)\\
  &\hspace{2.7cm}+\E_r \sum_{x\in K} \sum_{z\in K_y} \mathbbm{1}(0\nleftrightarrow y, \|x-z\| \leq r) |K||K_y|^2.
\end{align*}
\end{lemma}

We refer to the last term on the right hand side of each expression as the ``merging term'' and the other terms on the right hand side as the ``triple interaction'' terms.

\begin{proof}[Proof of \cref{lem:correction_exact_derivative}]
We can use Russo's formula to write
\begin{align*}
\frac{1}{\beta_c|J'(r)|}\frac{d}{dr}\E_r |K||K_y|\mathbbm{1}(0\nleftrightarrow y) &= 
2\E_r \sum_{x\in K} \sum_{z \in B_r(x)} \mathbbm{1}(0,y,z\text{ all different clusters}) |K_z| |K_y|\\
&-\E_r \sum_{x\in K} \sum_{z\in K_y} \mathbbm{1}(0\nleftrightarrow y, \|x-z\| \leq r) |K||K_y|, 
\end{align*}
where  the first term describes the effect of increasing the size of either $K$ or $K_y$ on the event that $0$ and $y$ are not connected (we used the symmetry between $0$ and $y$ to combine the effects of growing both clusters into a single term, giving the prefactor $2$), and the second term describes the effect of merging the two clusters. 
Meanwhile, the derivative of the product is given by
\begin{align*}
\frac{1}{\beta_c|J'(r)|}\frac{d}{dr}(\E_r |K|)^2 &= 2 \E_r|K| \cdot \frac{1}{\beta_c|J'(r)|}\frac{d}{dr}\E_r |K| = 2 \E_r|K| \E_r \sum_{x\in K} \sum_{z\in B_r(x)} |K_z| \mathbbm{1}(0\nleftrightarrow z).
\end{align*}
Taking the difference of these two expressions yields the first claimed identity.
Similarly, we can compute that
\begin{align*}
&\frac{1}{\beta_c|J'(r)|}\frac{d}{dr}\E_r |K||K_y|^2\mathbbm{1}(0\nleftrightarrow y) = 
\E_r \sum_{x\in K} \sum_{z \in B_r(x)} \mathbbm{1}(0,y,z\text{ all different clusters}) |K_z| |K_y|^2
\\&\hspace{4.5cm}+
\E_r \sum_{x\in K_y} \sum_{z \in B_r(x)} \mathbbm{1}(0,y,z\text{ all different clusters}) |K| (|K_z|^2+2|K_y||K_z|)
\\
&\hspace{4.5cm}-\E_r \sum_{x\in K} \sum_{z\in K_y} \mathbbm{1}(0\nleftrightarrow y, \|x-z\| \leq r) |K||K_y|^2, 
\end{align*}
and
\begin{align*}
\frac{1}{\beta_c|J'(r)|}\frac{d}{dr}(\E_r |K|)(\E_r|K|^2) &= \E_r|K| \cdot \frac{1}{\beta_c|J'(r)|}\frac{d}{dr}\E_r|K|^2+\E_r|K|^2 \cdot \frac{1}{\beta_c|J'(r)|}\frac{d}{dr}\E_r|K|\\
&= 
\E_r|K| \E \sum_{x\in K}\sum_{z\in B_r(x)} (|K_z|^2+2|K||K_z|) \mathbbm{1}(0\nleftrightarrow z)\\ &\hspace{5cm}+\E_r|K|^2 \E_r \sum_{x\in K} \sum_{z\in B_r(x)} |K_z| \mathbbm{1}(0\nleftrightarrow z),
\end{align*}
and taking the difference of these two terms yields the second claimed identity.
\end{proof}

The following lemma, whose proof is deferred to \cref{subsec:triple_interactions}, should be thought of as writing the (moments of the) ``triple interaction'' terms appearing in \cref{lem:correction_exact_derivative} as sums of ``pair interactions'' plus a negligible error term.

\begin{lemma}[Triple interactions]
\label{lem:triple_interaction}
 If $d\geq 3\alpha$ and $\alpha<2$ then
\begin{align*}&\E_r \left[\sum_{y\in \Z^d}\sum_{x\in K} \sum_{z \in B_r(x)} \mathbbm{1}(0\nleftrightarrow z) |K|^{p_1}|K_z|^{p_2} (\E|K_y|^{p_3}-\mathbbm{1}(y\nleftrightarrow 0,z)|K_y|^{p_3} ) P(y)\right]
\\&\hspace{2cm}=
|B_r|\E_r|K|^{p_2} \E_r\left[  |K|^{p_1+1} (\E|K_y|^{p_2}-\mathbbm{1}(y\nleftrightarrow 0)|K_y|^{p_3} )P(y)\right]
\\&\hspace{2cm}+
|B_r| \sum_{z\in \Z^d} \E_r \left[\sum_{x\in K} |K|^{p_1} |K \cap B_r(z)|\right] \E_r \left[\sum_{y\in \Z^d} |K_z|^{p_2} (\E|K_y|^{p_3}-\mathbbm{1}(y\nleftrightarrow z)|K_y|^{p_3} )P(y)\right]
\\
&\hspace{2cm}\pm o\left(|B_r|V_r^{p_1+p_2+p_3} r^{\deg(P)}(\E_r|K|)^{2p_1+2p_2+2p_3+2}\right)
\end{align*}
as $r\to\infty$ 
for each triple of integers $p_1\geq 0$ and $p_2,p_3 \geq 1$ and monomial $P:\R^d\to \R$, where the implicit constants may depend on  $p_1,p_2,p_3$ and $P$.
\end{lemma}

We next leverage our knowledge of the scaling limit to compute the asymptotics of the (moments of the) ``merging terms" in \cref{lem:correction_exact_derivative}. The fact that the asymptotics of these merging terms can be expressed purely in terms of quantities already computed in \cref{I-thm:scaling_limit_diagrams} leads to the asymptotic relationship between the two vertex factors discussed in \cref{subsec:a_tale_of_two_vertex_factors}.

\begin{lemma}
\label{lem:merge_term_asymptotics}
 If $d\geq 3\alpha$ and $\alpha<2$ then 
\begin{multline*}
\E_r \sum_{y\in \Z^d}\sum_{x\in K} \sum_{z\in K_y} \mathbbm{1}(0\nleftrightarrow y, \|x-z\| \leq r) |K||K_y| P(y)
\\= \left[\int_{\R^d} [\kappa^{*4} * \mathbbm{1}_B](y) P(y) \dif y \pm o(1)\right] |B_r| V_r^2 r^{\deg(P)}(\E_r|K|)^6
\end{multline*}
and
\begin{multline*}
\E_r\sum_{y\in \Z^d} \sum_{x\in K} \sum_{z\in K_y} \mathbbm{1}(0\nleftrightarrow y, \|x-z\| \leq r) |K||K_y|^2 P(y)
\\= \left[\int_{\R^d} [(2\kappa^{*5}+\kappa^{*4}) * \mathbbm{1}_B](y) P(y) \dif y \pm o(1)\right] |B_r| V_r^3 r^{\deg(P)}(\E_r|K|)^8
\end{multline*}
as $r\to \infty$ for each monomial $P:\R^d\to\R$.
\end{lemma}

Before proving this lemma, we first introduce some relevant terminology that will be used in the proof. 
 We say that a monomial $Q$ \textbf{divides} a monomial $P=\prod_{i=1}^n \langle x,u_i\rangle$, and write $Q \mid P$, if $Q$ can be written $Q(x)=\prod_{i\in I} \langle x,u_i\rangle$ for some subset of $\{1,\ldots,\deg(P)\}$. We define $N(Q\mid P)$ to be the number of subsets $I$ giving rise to the same monomial $Q$, so that every monomial has $N(1\mid P)=N(P\mid P)=1$. The relevance of these notions to us comes from the addition formula 
\begin{align}
  P(x+y)  &= \prod_{i=1}^{\deg(P)} \langle x +y,u_i\rangle = \sum_{I \subseteq \{1,\ldots,\deg(P)\}} \prod_{i\in I} \langle x,u_i\rangle\prod_{i\notin I} \langle y,u_i\rangle
  \nonumber\\&=
   \sum_{Q\mid P} N(Q \mid P) \frac{P(x)}{Q(x)} Q(y)
\label{eq:monomial_addition}
\end{align}
which holds for every monomial $P:\R^d\to \R$ and every $x,y\in \R^d$; this leads in particular to the convolution identity
\begin{equation}
  \int_{\R^d}[f*g](x)P(x)\dif x  = \sum_{Q\mid P} N(Q \mid P) \int f(x) \frac{P(x)}{Q(x)} \dif x \int g(x) Q(x)\dif x
\label{eq:monomial_convolution}
\end{equation}
holding whenever all relevant integrals converge absolutely.

\begin{proof}[Proof of \cref{lem:merge_term_asymptotics}]
We prove the second asymptotic identity, the proof of the first being similar and slightly simpler. We begin by writing
\begin{multline*}
  \E_r\sum_{y\in \Z^d} \sum_{x\in K} \sum_{z\in K_y} \mathbbm{1}(0\nleftrightarrow y, \|x-z\| \leq r) |K||K_y|^2 P(y) 
  \\=
  \E_r\sum_{x\in K} \sum_{z\in B_r(x)} \mathbbm{1}(x\nleftrightarrow z) |K||K_z|^2 \sum_{y\in K_z}P(y) 
  \\
  =\E_r \sum_{z\in B_r} \mathbbm{1}(0\nleftrightarrow z) |K||K_z|^2 \sum_{x\in K}\sum_{y\in K_z}P(y-x),
\end{multline*}
where we used Fubini's theorem in the first identity and the mass-transport principle in the second.
Using the monomial addition formula \eqref{eq:monomial_addition} we can write
\begin{multline}
\label{eq:monomial_convolution_merge_term}
 \E_r\sum_{y\in \Z^d} \sum_{x\in K} \sum_{z\in K_y} \mathbbm{1}(0\nleftrightarrow y, \|x-z\| \leq r) |K||K_y|^2 P(y) 
 \\
  =\sum_{Q\mid P} N(Q\mid P)\E_r \left[ \sum_{z\in B_r} \mathbbm{1}(0\nleftrightarrow z)\left(|K| \sum_{x\in K} \frac{P(-x)}{Q(-x)}\right)\left(|K_z|^2 \sum_{y\in K_z}Q(y)\right) \right],
\end{multline}
and using \cref{I-lem:BK_disjoint_clusters_covariance} we obtain the approximate factorization
\begin{multline*}
  \E_r \left[ \sum_{z\in B_r} \mathbbm{1}(0\nleftrightarrow z)\left(|K| \sum_{x\in K} \frac{P(-x)}{Q(-x)}\right)\left(|K_z|^2 \sum_{y\in K_z}Q(y)\right) \right] 
  \\= 
  \E_r \left[|K| \sum_{x\in K} \frac{P(-x)}{Q(-x)}\right] \E_r \left[\sum_{z\in K}|K_z|^2 \sum_{y\in K_z}Q(y) \right]
  \pm O\left(\E_r|K|^4 \sum_{x\in K} \left|\frac{P(-x)}{Q(-x)}\right| \sum_{y\in K}|Q(y)| \right)
\end{multline*}
for each $Q \mid P$. It follows from \cref{I-lem:mixed_moments_order_estimates} that the error term appearing here has order at most $V_r^4 r^{\deg(P)} (\E_r|K|)^{11} =o(|B_r| V_r^3 r^{\deg(P)} (\E_r|K|)^{8})$ as required. Meanwhile, it follows from \cref{I-thm:scaling_limit_diagrams} that
\[
  \E_r \left[|K| \sum_{x\in K} \frac{P(-x)}{Q(-x)}\right] = \left[\int_{\R^d} \kappa^{*2}(x) \frac{P(-x)}{Q(-x)} \dif x \pm o(1)\right] V_r (\E_r|K|)^3
\]
and
\begin{align*}
& \sum_{z\in B_r} \E_r \left[|K_z|^2 \sum_{y\in K_z} Q(y) \right] 
= \sum_{S \mid Q}N(S\mid Q) \sum_{z\in B_r}  \frac{Q(z)}{S(z)} \E_r \left[|K_z|^2 \sum_{y\in K_z} S(y-z) \right]
\\
&\hspace{2.5cm}=  \left[\sum_{S \mid Q}N(S\mid Q) \int \mathbbm{1}_B(z) \frac{Q(z)}{S(z)} \dif z \int [\kappa^{*2}+2\kappa^{*3}](y) S(y) \dif y  \pm o(1) \right]|B_r|V_r^2 (\E_r|K|)^5
 \\&\hspace{2.5cm}= \left[\int_{\R^d} [(\kappa^{*2}+2\kappa^{*3})*\mathbbm{1}_B](y) Q(y) \dif y \pm o(1)\right] |B_r|V_r^2 (\E_r|K|)^5
\end{align*}
as $r\to \infty$, where in the second estimate we used \eqref{eq:monomial_addition} in the first line and \eqref{eq:monomial_convolution} in the last line.
Multiplying these estimates together, summing over $Q\mid P$ as in \eqref{eq:monomial_convolution_merge_term}, and applying the convolution formula \eqref{eq:monomial_convolution} yields the claim.
\end{proof}


We now combine \cref{lem:correction_exact_derivative,lem:merge_term_asymptotics,lem:triple_interaction} to obtain a solvable system of asymptotic ODEs for the moments of $\mathscr{D}_r^{(1)}(0,y)$ and $\mathscr{D}_r^{(2)}(0,y)$.

\begin{lemma}
\label{lem:correction_ODEs}
I $d\geq 3\alpha$ and $\alpha<2$
then 
\begin{multline*}
  \frac{d}{dr} \sum_y \mathscr{D}_r^{(1)}(0,y)P(y) \sim \frac{4\alpha}{r} \sum_y \mathscr{D}_r^{(1)}(0,y)P(y) 
  \\ + 
  \frac{2\alpha}{r} \sum_{\substack{Q \mid P\\Q\neq P}} N(Q\mid P) 
   \left[\int_{\R^d} [\kappa*\mathbbm{1}_B](x) \frac{P(x)}{Q(x)} \dif x\right] r^{\deg(P)-\deg(Q)} \sum_{y\in \Z^d} \mathscr{D}_r^{(1)}(0,y) Q(-y) 
  \\+
  \frac{\alpha}{r}\left[\int_{\R^d} [\kappa^{*3} * \mathbbm{1}_B](y) P(y) \dif y \pm o(1)\right]  V_r^2 r^{\deg(P)}(\E_r|K|)^5
\end{multline*}
and
\begin{align*}
  &\frac{d}{dr} \sum_y \mathscr{D}_r^{(2)}(0,y)P(y)
   \sim 
   \frac{5\alpha}{r}  \sum_y \mathscr{D}_r^{(2)}(0,y)P(y)
 + \frac{3\alpha}{r}V_r (\E_r|K|)^2  \sum_y \mathscr{D}_r^{(1)}(0,y)P(y)
 \\
 &\hspace{1cm}+\frac{2\alpha}{r} \sum_{\substack{Q \mid P\\Q\neq P}}  N(Q\mid P) \left[ \int_{\R^d}[\kappa*\mathbbm{1}_B](x) \frac{P(x)}{Q(x)} \dif x\right] r^{\deg(P)-\deg(Q)}
 \sum_y \mathscr{D}_r^{(2)}(0,y) Q(-y)
 \\&\hspace{1cm}+\frac{2\alpha}{r} V_r (\E_r|K|)^2  \sum_{\substack{Q \mid P\\Q\neq P}}  N(Q\mid P) \left[ \int_{\R^d}[\kappa^{*2}*\mathbbm{1}_B](x) \frac{P(x)}{Q(x)} \dif x\right] r^{\deg(P)-\deg(Q)}
 \sum_y \mathscr{D}_r^{(1)}(0,y) Q(-y)
 \\
 &\hspace{1cm}+
  \frac{3\alpha}{r}\left[\int_{\R^d} [\kappa^{*4} * \mathbbm{1}_B](y) P(y) \dif y \pm o(1)\right]  V_r^3 r^{\deg(P)}(\E_r|K|)^7
\end{align*}
as $r\to \infty$
 for each monomial $P:\R^d\to \R$.
\end{lemma}

\begin{proof}[Proof of \cref{lem:correction_ODEs}]
We begin by simplifying the second term in \cref{lem:triple_interaction}. We can use the mass-transport principle to write
\begin{multline*}
  \sum_{z\in \Z^d} \E_r \left[\sum_{x\in K} |K|^{p_1} |K \cap B_r(z)|\right] \E_r \left[\sum_{y\in \Z^d} |K_z|^{p_2} (\E|K_y|^{p_3}-\mathbbm{1}(y\nleftrightarrow z)|K_y|^{p_3} )P(y)\right]
  \\=
   \sum_{y\in \Z^d}\E_r \left[ |K|^{p_2} (\E|K_y|^{p_3}-\mathbbm{1}(y\nleftrightarrow z)|K_y|^{p_3} )\right] \E_r \left[\sum_{x\in B_r} |K|^{p_1} \sum_{z\in K_x} P(y-z)\right]
\end{multline*}
and use the monomial addition formula \eqref{eq:monomial_addition} to obtain that
\begin{align}
   &\sum_{y\in \Z^d}\E_r \left[ |K|^{p_2} (\E|K_y|^{p_3}-\mathbbm{1}(y\nleftrightarrow z)|K_y|^{p_3} )\right] \E_r \left[\sum_{x\in B_r} |K|^{p_1} \sum_{z\in K_x} P(y-z)\right]
    \label{eq:monomial_convolution_application}
   \\
   \nonumber
   &\hspace{1cm}=\sum_{Q\mid P}N(Q\mid P) 
   \E_r \left[\sum_{y\in \Z^d} |K|^{p_2} (\E|K_y|^{p_3}-\mathbbm{1}(y\nleftrightarrow z)|K_y|^{p_3} ) \frac{P(y)}{Q(y)}\right] \E_r \left[\sum_{x\in B_r} |K|^{p_1} \sum_{z\in K_x} Q(-z)\right]
\end{align}
for every $p_1\geq 0$, $p_2,p_3\geq1$, and monomial $P$, 
where we used the $x\mapsto -x$ symmetry of the model to replace $Q(-z)$ by $Q(z)$ in the last line. Moreover, it follows from \cref{I-thm:scaling_limit_diagrams} as in the proof of \cref{lem:merge_term_asymptotics} that
\[
  \E_r \left[\sum_{x\in B_r} \sum_{z\in K_x} Q(-z)\right]=(1\pm o(1)) |B_r|\E_r|K| \int_{\R^d} \kappa*\mathbbm{1}_B(z) Q(-z)\dif z
\]
and
\[
  \E_r \left[\sum_{x\in B_r} |K_x|\sum_{z\in K_x} Q(-z)\right]=(1\pm o(1)) |B_r|V_r(\E_r|K|)^3 \int_{\R^d} \kappa^{*2}*\mathbbm{1}_B(z) Q(-z)\dif z
\]
as $r\to\infty$ for each monomial $Q:\R^d\to\R$.
Given the identities \eqref{eq:monomial_convolution} and \eqref{eq:monomial_convolution_application} and these two asymptotic equalities, the claim follows immediately from \cref{lem:correction_exact_derivative,lem:merge_term_asymptotics,lem:triple_interaction}, where we make repeated use of the asymptotic equalities
\[
  \beta_c |J'(r)| |B_r| \E_r|K| \sim \frac{\alpha}{r}
\qquad\text{ and }\qquad
  \beta_c |J'(r)| |B_r| \E_r|K|^p \sim (2p-3)!! \frac{\alpha}{r} V_r^{p-1} (\E_r|K|)^{2p-2}
\]
to simplify the various expressions that appear.
\end{proof}

To conclude the proof we will analyze the asymptotic ODEs of \cref{lem:correction_ODEs} with the aid of \cref{I-lem:signed_ODE_analysis}, which we restate here for the reader's convenience:

\begin{lemma}[ODE analysis for signed functions]
\label{lem:signed_ODE_analysis}
Let $a,b>0$, let $h:(0,\infty)\to (0,\infty)$ be a measurable, regularly varying function of index $b$, and suppose that 
$f:(0,\infty)\to \R$ is a (not necessarily positive) differentiable function such that 
\begin{equation}
\label{eq:signed_ODE_f}
f'(r) = \frac{a \pm o(1)}{r}f(r) + \frac{A \pm o(1)}{r} h 
\end{equation}
as $r\to \infty$ for some constant $A\in \R$. If $b>a$ and $|f(r)|=O(h)$ as $r\to \infty$ then
\[
  f(r) = \frac{A \pm o(1)}{b-a} h(r)
\]
as $r\to \infty$.
\end{lemma}

When analyzing asymptotic ODEs of the form \eqref{eq:signed_ODE_f}, we will often refer to the function $h$ on the right hand side as the \textbf{driving term} of the asymptotic ODE.

\begin{proof}[Proof of \cref{thm:correction_moments}]
We first analyze the moments of $\mathscr{D}_r^{(1)}(0,y)$. We begin with the case $P\equiv 1$ as a warm up (although it is not really necessarily to perform the analysis of this base case separately from the general case). In this case, the asymptotic ODE of \cref{lem:correction_ODEs} reads simply that
\[
  \frac{d}{dr} \sum_y \mathscr{D}_r^{(1)}(0,y) \sim \frac{4\alpha}{r} \sum_y \mathscr{D}_r^{(1)}(0,y)
  +\frac{\alpha}{r} V_r^2 (\E_r|K|)^5,
\]
and since the driving term is regularly varying of index $5\alpha>4\alpha$ we deduce from \cref{lem:correction_order_estimates,lem:signed_ODE_analysis} that
\[
\sum_y \mathscr{D}_r^{(1)}(0,y) \sim V_r^2 (\E_r|K|)^5.
\]
(This asymptotic formula already reveals the asymptotic equivalence of the two \emph{a priori} distinct vertex factors discussed in \cref{subsec:a_tale_of_two_vertex_factors}.)
We next prove by induction that if we define the constants $C_1(P)$ for each monomial $P$
by the recurrence relation
\begin{multline}
\label{eq:C_1(P)_recurrence}
  C_1(P) = \frac{2\alpha}{\alpha+\deg(P)} 
\sum_{\substack{Q \mid P\\Q\neq P}} N(Q\mid P) 
   \left[\int_{\R^d} [\kappa*\mathbbm{1}_B](x) \frac{P(x)}{Q(x)} \dif x\right] C_1(Q)
 \\ +
  \frac{\alpha}{\alpha+\deg(P)} \int_{\R^d} [\kappa^{*4} * \mathbbm{1}_B](y) P(y) \dif y.
\end{multline}
then
\begin{equation}
\label{eq:C_1(P)_asymptotics}
  \sum_{y\in\Z^d} \mathscr{D}_r^{(1)}(0,y)P(y) = (C_1(P)\pm o(1))V_r^2 r^{\deg(P)}(\E_r|K|)^5
\end{equation}
as $r\to \infty$ for each monomial $P$. (Note that the recurrence \eqref{eq:C_1(P)_recurrence} applies even to determine the base case $C_1(1)=1$, in which case the first sum is empty and the integral in the second term is equal to~$1$.) To prove this, suppose that $P$ is a monomial and that the claim has been proven for all strict divisors of $P$. It follows from \cref{lem:correction_ODEs} and the induction hypothesis that
\begin{multline*}
  \frac{d}{dr} \sum_{y\in\Z^d} \mathscr{D}_r^{(1)}(0,y)P(y) \sim \frac{4\alpha}{r} \sum_y \mathscr{D}_r^{(1)}(0,y)P(y) 
  \\ + 
  \frac{\alpha}{r}\left[2 \sum_{\substack{Q \mid P\\Q\neq P}} N(Q\mid P) 
   \left[\int_{\R^d} [\kappa*\mathbbm{1}_B](x) \frac{P(x)}{Q(x)} \dif x\right] C(Q) +
  \int_{\R^d} [\kappa^{*4} * \mathbbm{1}_B](y) P(y) \dif y \pm o(1)\right] \\\cdot V_r^2 r^{\deg(P)}(\E_r|K|)^5.
\end{multline*}
Since the driving term appearing here has regular variation of index $5\alpha+\deg(P)$, it follows from \cref{lem:correction_order_estimates}, \cref{lem:signed_ODE_analysis}, and the definition of $C_1(P)$ that the claimed asymptotic formula \eqref{eq:C_1(P)_asymptotics} holds for this choice of $P$, completing the induction step. 
To complete the proof of the claim regarding the moments of $\mathscr{D}_r^{(1)}(0,y)$, it suffices to verify that the assignment $P\mapsto \int_{\R^d} \kappa^{*3}(y)P(y) \dif y$ also satisfies the recurrence relation \eqref{eq:C_1(P)_recurrence} (the solution to this recurrence clearly being unique), so that
\begin{multline}
\label{eq:C_1(P)_recurrence2}
  \int_{\R^d} \kappa^{*3}(y)P(y) \dif y = \frac{2\alpha}{\alpha+\deg(P)} 
\sum_{\substack{Q \mid P\\Q\neq P}} N(Q\mid P) 
   \left[\int_{\R^d} [\kappa*\mathbbm{1}_B](x) \frac{P(x)}{Q(x)} \dif x\right] \int_{\R^d} \kappa^{*3}(y)Q(y) \dif y
 \\ +
  \frac{\alpha}{\alpha+\deg(P)} \int_{\R^d} [\kappa^{*4} * \mathbbm{1}_B](y) P(y) \dif y.
\end{multline}
for every monomial $P:\R^d\to\R$. Using that
\begin{align*}
  &\sum_{\substack{Q \mid P\\Q\neq P}} N(Q\mid P) 
   \left[\int_{\R^d} [\kappa*\mathbbm{1}_B](x) \frac{P(x)}{Q(x)} \dif x\right] \int_{\R^d} \kappa^{*3}(y)Q(y) \dif y
   \\ &\hspace{4cm}= 
   \int_{\R^d}\int_{\R^d}  [\kappa*\mathbbm{1}_B](x) \kappa^{*3}(y) \sum_{Q \mid P} N(Q\mid P)  \frac{P(x)}{Q(x)}  Q(y) \dif y
   - \int_{\R^d} \kappa^{*3}(y)P(y)\dif x \dif y
   \\
    &\hspace{4cm}= \int_{\R^d}\int_{\R^d}  [\kappa*\mathbbm{1}_B](x) \kappa^{*3}(y) P(x+y) \dif x \dif y
   - \int_{\R^d} \kappa^{*3}(y)P(y)\dif y
   \\
       &\hspace{4cm}= \int_{\R^d}  [\kappa^{*4}*\mathbbm{1}_B](y) P(y) \dif y
   - \int_{\R^d} \kappa^{*3}(y)P(y)\dif y,
\end{align*}
 One can rewrite the recurrence relation \eqref{eq:C_1(P)_recurrence2} as
\begin{equation}
\label{eq:C_1(P)_recurrence3}
 (\deg(P)+3\alpha) \int_{\R^d} \kappa^{*3}(y)P(y) \dif y =
  3\alpha \int_{\R^d} [\kappa^{*4} * \mathbbm{1}_B](y) P(y) \dif y,
\end{equation}
which is the $n=3$ case of the relation \eqref{eq:recurrence_from_derivative_LR3}.

\medskip

We now extend this analysis to the moments of $\mathscr{D}_r^{(2)}(0,y)$. Using our analysis of the moments of $\mathscr{D}_r^{(2)}(0,y)$ above, we can simplify the relevant asymptotic ODE from \cref{lem:correction_ODEs} to state that
\begin{align*}
  &\frac{d}{dr} \sum_{y\in\Z^d} \mathscr{D}_r^{(2)}(0,y)P(y)
   \sim 
   \frac{5\alpha}{r}  \sum_{y\in\Z^d} \mathscr{D}_r^{(2)}(0,y)P(y)
 \\
 &\hspace{1.2cm}+\frac{2\alpha}{r} \sum_{\substack{Q \mid P\\Q\neq P}}  N(Q\mid P) \left[ \int_{\R^d}[\kappa*\mathbbm{1}_B](x) \frac{P(x)}{Q(x)} \dif x\right] r^{\deg(P)-\deg(Q)}
 \sum_y \mathscr{D}_r^{(2)}(0,y) Q(y)
 \\
 &\hspace{1.2cm}+
  \frac{\alpha \tilde C(P) \pm o(1)}{r} V_r^3 r^{\deg(P)}(\E_r|K|)^7
\end{align*}
for each monomial $P$ as $r\to \infty$, where the constant $\tilde C(P)$ is defined by
\begin{align*} 
  \tilde C(P) &= 3 \int \kappa^{*3}(y)P(y)\dif y + 2 \sum_{\substack{Q \mid P\\Q\neq P}}  N(Q\mid P) \left( \int_{\R^d}[\kappa^{*2}*\mathbbm{1}_B](x) \frac{P(x)}{Q(x)} \dif x\right) \left(\int \kappa^{*3}(y)Q(y)\dif y\right)
  \\&\hspace{5cm}+\int_{\R^d} [(2\kappa^{*5}+\kappa^{*4}) * \mathbbm{1}_B](y) P(y) \dif y
  \\&= \int_{\R^d} [\kappa^{*3}+(4 \kappa^{*5}+\kappa^{*4}) * \mathbbm{1}_B](y)P(y)\dif y.
\end{align*}
Repeating the same argument as above, we obtain that
\[
  \sum_{y\in \Z^d} \mathscr{D}_r^{(2)}(0,y)P(y) = (C_2(P) \pm o(1)) V_r^3 r^{\deg(P)} (\E_r|K|)^7
\]
where the constants $C_2(P)$ are determined recursively by
\begin{align*}
C_2(P) = \frac{2\alpha}{\deg(P)+2\alpha} 
 \sum_{\substack{Q \mid P\\Q\neq P}}  N(Q\mid P) \left[ \int_{\R^d}[\kappa*\mathbbm{1}_B](x) \frac{P(x)}{Q(x)} \dif x\right] C_2(Q)
 + \frac{\alpha}{\deg(P)+2\alpha} \tilde C(P).
\end{align*}
As before, we need to verify that this recurrence is satisfied by
\[
  C_2(P) = \int_{\R^d} (\kappa^{*3}(y)+2\kappa^{*4}(y))P(y) \dif y.
\]
Using the monomial convolution identity \eqref{eq:monomial_convolution}, this is equivalent to the statement that
\begin{multline*}
  (\deg(P)+4\alpha) 
  \int_{\R^d} (\kappa^{*3}(y)+2\kappa^{*4}(y))P(y) \dif y = 2\alpha \int [(\kappa^{*4}+2\kappa^{*5})*\mathbbm{1}_B](y)P(y) \dif y
  \\+ \alpha 
  \int_{\R^d} [\kappa^{*3}+(4 \kappa^{*5}+\kappa^{*4}) * \mathbbm{1}_B](y)P(y)\dif y
\end{multline*}
for every monomial $P$.
Rearranging, this identity holds if and only if the equivalent identity
\begin{multline*}
  (\deg(P)+3\alpha) \int_{\R^d} \kappa^{*3}(y)P(y)\dif y +
  2(\deg(P)+4\alpha) \int_{\R^d} \kappa^{*4}(y)P(y)\dif y 
  \\=
  3 \alpha \int [\kappa^{*4}*\mathbbm{1}_B](y)P(y)\dif y
  +
  8 \alpha \int [\kappa^{*5}*\mathbbm{1}_B](y)P(y)\dif y
\end{multline*}
holds for every monomial $P$. This final form of the identity is  an immediate consequence of \eqref{eq:recurrence_from_derivative_LR3}, completing the proof.
\end{proof}

\subsection{Triple interactions}
\label{subsec:triple_interactions}

In this section we prove \cref{lem:triple_interaction}. To keep the proof readable, we will restrict attention to the concrete case $(p_1,p_2,p_3)=(1,1,1)$. (The only cases used in the proof of the main theorems are this together with $(0,1,1)$, $(0,2,1)$ and $(0,1,2)$.) The general case is similar but messier. We need to prove that
\begin{align*}&\E_r \sum_{y\in \Z^d}\sum_{x\in K} \sum_{z \in B_r(x)} \mathbbm{1}(0\nleftrightarrow z) |K| |K_z| (\E|K_y|-\mathbbm{1}(y\nleftrightarrow 0,z)|K_y| ) P(y)
\\&\hspace{1.5cm}=
|B_r|\E_r|K| \E_r \left[|K|^2 (\E|K_y|-\mathbbm{1}(y\nleftrightarrow 0)|K_y|)P(y)\right]
\\&\hspace{3cm}+
|B_r| \sum_{z\in \Z^d} \E_r \left[|K|\sum_{x\in K} |K \cap B_r(z)|\right] \E_r \left[\sum_{y\in \Z^d} |K_z| (\E|K_y|-\mathbbm{1}(y\nleftrightarrow z)|K_y| )P(y)\right]
\\
&\hspace{3cm}\pm o\left(|B_r|V_r^{3} r^{\deg(P)}(\E_r|K|)^{8}\right)
\end{align*}
as $r\to \infty$ for each monomial $P:\R^d\to \R$. 
To prove this, it will be convenient to use the mass-transport principle to exchange the roles of $0$ and $z$, yielding the the equivalent identity
\begin{align*}&\E_r \sum_{y\in \Z^d} \sum_{x \in B_r} \mathbbm{1}(0\nleftrightarrow x) |K| |K_x| (\E|K_y|-\mathbbm{1}(y\nleftrightarrow 0,x)|K_y| ) \sum_{z\in K_x} P(y-z)
\\&\hspace{1.5cm}=
\E_r|K| \E_r \left[ \sum_{x\in B_r} |K_x| \sum_{y\in\Z^d}(\E|K_y|-\mathbbm{1}(y\nleftrightarrow x)|K_y|) \sum_{z\in K_x}P(y-z)\right]
\\&\hspace{3cm}+
\sum_{z\in \Z^d} \E_r \left[\sum_{x\in B_r} |K_x| \mathbbm{1}(x\leftrightarrow z)\right] \E_r \left[|K|\sum_{y\in \Z^d}  (\E|K_y|-\mathbbm{1}(y\nleftrightarrow 0)|K_y| )P(y-z)\right]
\\
&\hspace{3cm}\pm o\left(|B_r|V_r^{3} r^{\deg(P)}(\E_r|K|)^{8}\right).
\end{align*}
Using the monomial convolution identities \eqref{eq:monomial_addition}, it suffices to prove that
\begin{align}&\E_r \left[ |K|  \sum_{y\in \Z^d} (\E|K_y|-\mathbbm{1}(y\nleftrightarrow 0,x)|K_y| )  P(y) \sum_{x \in B_r} \mathbbm{1}(0\nleftrightarrow x) |K_x| \sum_{z\in K_x} Q(z)\right]
\nonumber\\&\hspace{2cm}=
\E_r|K| \E_r \left[ \sum_{x\in B_r} |K_x|  \sum_{z\in K_x}Q(z) \sum_{y\in\Z^d} (\E|K_y|-\mathbbm{1}(y\nleftrightarrow x)|K_y|)P(y)\right]
\nonumber\\&\hspace{4cm}+
 \E_r \left[\sum_{x\in B_r} |K_x| \sum_{z\in K_x}Q(z)\right] \E_r \left[\sum_{y\in \Z^d} |K| (\E|K_y|-\mathbbm{1}(y\nleftrightarrow 0)|K_y| )P(y)\right]
\nonumber\\
&\hspace{6cm}\pm o\left(|B_r|V_r^{3} r^{\deg(P)+\deg(Q)}(\E_r|K|)^{8}\right)
\label{eq:triple_interaction_goal}
\end{align}
as $r\to \infty$ for every pair of monomials $P,Q:\R^d\to \R$.
We will prove this by studying a coupling of the three clusters $K_0$, $K_z$, and $K_y$ with three \emph{independent} clusters $K_0^0$, $K_z^z$, and $K_y^y$ analogous to the coupling between two clusters used in the proofs of e.g.\ \cref{I-lem:disjoint_connections_triangle} and \cref{lem:locally_large_measure_concrete}.

\medskip

We now define this coupling.
Let $\omega_0$ have law $\P_{\beta_c,r}$ and for each $x,y\in \Z^d$ let $\omega_x$ and $\omega_y$ be independent percolation configurations with law $\P_{\beta_c,r}$.
Let $K^0_0$, $K^x_x$, and $K^y_y$ be the cluster of $0$ in $\omega_0$, the cluster of $x$ in $\omega_x$, and the cluster of $y$ in $\omega_y$ respectively.
We define a configuration $\omega_{0y}$ as in the proof of \cref{I-lem:disjoint_connections_triangle}, where 
$\omega_{0y}(e)=\omega_0(e)$ if $e$ has an endpoint in $K^0_0$, $\omega_{0y}(e)=\omega_y(e)$ if $e$ has an endpoint in $K^y_y$ but not in $K^0_0$, and letting $\omega_{0y}$ be given by an independent coin flip otherwise. 
We denote the clusters of $0$ and $y$ in $\omega_{0y}$ by $K_0^{0y}$ and $K^{0y}_y$ respectively, noting that $K_0^{0y}=K_0^0$.
 We define configurations $\omega_{0x}$ and $\omega_{xy}$ and clusters $K_0^{0x}$, $K_x^{0x}$, $K_x^{xy}$, and $K_y^{xy}$ similarly, where the cluster of the first point in the subscript has priority over the cluster of the second point in the subscript when determining the full configuration.  Finally, we define $\omega_{0xy}$ by setting  
$\omega_{0xy}(e)=\omega_0(e)$ if $e$ has an endpoint in $K^0_0$, $\omega_{0xy}(e)=\omega_x(e)$ if $e$ has an endpoint in $K^x_{0x}$ but not in $K^0_0$, and 
$\omega_{0xy}(e)=\omega_z(e)$ if $e$ has an endpoint in $K^y_y$ but not in $K^0_0$ or $K^{0x}_x$ (determining the status of all remaining undetermined edges by a coin flip as before).
Each of these  configurations has law $\P_{\beta_c,r}$, and the cluster $K_0$ of $0$ is equal to $K^0_0$ in all of these configurations. We write $K_0=K^{0xy}_0$, $K_x=K_x^{0xy}$, and $K_y=K_y^{0xy}$ for the clusters of $0$, $x$, and $y$ in the configuration $\omega_{0xy}$, so that $K_0=K_0^0$ and $K_x=K_x^{0x}$.
We write $\mathbf{E}_r$ for expectations taken with respect to the joint law of these coupled configurations.




\noindent We can write the quantity we are interested in as
\begin{align}
&\E_r \left[ |K|  \sum_{x \in B_r} \mathbbm{1}(0\nleftrightarrow x) |K_x| \sum_{z\in K_x} Q(z) \sum_{y\in \Z^d} (\E|K_y|-\mathbbm{1}(y\nleftrightarrow 0,x)|K_y| )  P(y)\right]
\nonumber\\
&\hspace{2cm}=
\mathbf{E}_r \left[ |K| \sum_{x \in B_r} \mathbbm{1}(0\nleftrightarrow x) |K_x| \sum_{z\in K_x} Q(z) \sum_{y\in \Z^d} (|K_y^y|-\mathbbm{1}(y\notin K_0,K_x)|K_y| )  P(y) \right]
\nonumber\\
&\hspace{2cm}=
\mathbf{E}_r \left[ |K| \sum_{x \in B_r} \mathbbm{1}(0\nleftrightarrow x) |K_x| \sum_{z\in K_x} Q(z) \sum_{y\in \Z^d} \sum_{a\in \Z^d}\mathbbm{1}(\mathscr{C}_{0xy}(a))  P(y) \right]
\label{eq:getting_cozy}
\end{align}
where we write $\mathscr{C}_{0xy}(a)$ for the event
\begin{align*}
  \mathscr{C}_{0xy}(a) &= \{a\in K_y^y\} \setminus \{a \in K_y, y\notin K_0 \cup K_x\}.
\end{align*}
We also define
\begin{equation*}
  \mathscr{C}_{0y}(a) = \{a\in K_y^y\} \setminus \{a \in K_y^{0y}, y\notin K_0\} \qquad \text{ and } \qquad
  \mathscr{C}_{xy}(a) = \{a\in K_y^y\} \setminus \{a \in K_y^{xy}, y\notin K_x^x\}
\end{equation*}
so that
\begin{equation}
 \E_r \left[\sum_{y\in \Z^d} |K| (\E|K_y|-\mathbbm{1}(y\nleftrightarrow 0)|K_y| )P(y)\right]
  =
  \mathbf{E}_r
\left[\sum_{y\in \Z^d} |K| \sum_{a\in \Z^d} \mathbbm{1}(\mathscr{C}_{0y}(a))P(y)\right]
\label{eq:triple_intersection_target1}
\end{equation}
and
\begin{multline}
  \E_r \left[ \sum_{x\in B_r} |K_x|  \sum_{z\in K_x}Q(z) \sum_{y\in\Z^d} (\E|K_y|-\mathbbm{1}(y\nleftrightarrow x)|K_y|)P(y)\right]
  \\=
  \mathbf{E}_r
   \left[ \sum_{x\in B_r} |K_x|  \sum_{z\in K_x}Q(z) \sum_{y\in\Z^d} \sum_{a\in \Z^d} \mathbbm{1}(\mathscr{C}_{xy}(a))P(y)\right].
   \label{eq:triple_intersection_target2}
\end{multline}
The asymptotic equality \eqref{eq:triple_interaction_goal}
will be deduced from the inclusion-exclusion expansion
\begin{multline}
  \mathbbm{1}(\mathscr{C}_{0xy}(a) \cap \{x\notin K_0\})
=
\mathbbm{1}(\mathscr{C}_{0y}(a)\cap \{x\notin K_0\})
+
\mathbbm{1}(\mathscr{C}_{xy}(a)\cap \{x\notin K_0\})
\\
-
\mathbbm{1}(\mathscr{C}_{0y}(a) \cap \mathscr{C}_{xy}(a)\cap \{x\notin K_0\})
+
\mathbbm{1}([\mathscr{C}_{0xy}(a) \setminus (\mathscr{C}_{0y}(a) \cup \mathscr{C}_{xy}(a))] \cap \{x\notin K_0\})\\
-
\mathbbm{1}([(\mathscr{C}_{0y}(a) \cup \mathscr{C}_{xy}(a)) \setminus \mathscr{C}_{0xy}(a)] \cap \{x\notin K_0\}).
\label{eq:cozy_inclusion_exclusion}
\end{multline}
To proceed, we will first prove that the contribution from each of the last three terms on the right hand side is negligible. This is done over the course of the following three lemmas.

\begin{lemma}[Redundancy]
\label{lem:triple_interaction_negligible1} If $d\geq 3\alpha$ and $\alpha<2$ then
\begin{multline*}
\mathbf{E}_r \left[ |K| \sum_{x \in B_r} \mathbbm{1}(0\nleftrightarrow x) |K_x| \sum_{z\in K_x} \|z\|^q \sum_{y\in \Z^d} \sum_{a\in \Z^d}\mathbbm{1}(\mathscr{C}_{0y}(a) \cap \mathscr{C}_{xy}(a)\cap \{x\notin K_0\})  \|y\|^p \right]
\\=o\left(
  |B_r| V_r^3 r^{p+q} (\E_r|K|)^8
\right)
\end{multline*}
as $r\to \infty$ for each pair of integers $p,q\geq 0$.
\end{lemma}

\begin{lemma}[Tag team]
\label{lem:triple_interaction_negligible2}
 If $d\geq 3\alpha$ and $\alpha<2$ then
\begin{multline*}
\mathbf{E}_r \left[ |K| \sum_{x \in B_r} \mathbbm{1}(0\nleftrightarrow x) |K_x| \sum_{z\in K_x} \|z\|^q \sum_{y\in \Z^d} \sum_{a\in \Z^d}\mathbbm{1}([\mathscr{C}_{0xy}(a) \setminus (\mathscr{C}_{0y}(a) \cup \mathscr{C}_{xy}(a))]\cap \{x\notin K_0\})  \|y\|^p \right]
\\=o\left(
  |B_r| V_r^3 r^{p+q} (\E_r|K|)^8
\right)
\end{multline*}
as $r\to \infty$ for each pair of integers $p,q\geq 0$.
\end{lemma}

\begin{lemma}[Interception]
\label{lem:triple_interaction_negligible3}
 If $d\geq 3\alpha$ and $\alpha<2$ then
\begin{multline*}
\mathbf{E}_r \left[ |K| \sum_{x \in B_r} \mathbbm{1}(0\nleftrightarrow x) |K_x| \sum_{z\in K_x} \|z\|^q \sum_{y\in \Z^d} \sum_{a\in \Z^d}\mathbbm{1}([(\mathscr{C}_{0y}(a) \cup \mathscr{C}_{xy}(a)) \setminus \mathscr{C}_{0xy}(a)] \cap \{x\notin K_0\})  \|y\|^p \right]
\\=o\left(
  |B_r| V_r^3 r^{p+q} (\E_r|K|)^8
\right)
\end{multline*}
as $r\to \infty$ for each pair of integers $p,q\geq 0$.
\end{lemma}

\begin{figure}
\centering
\includegraphics[scale=1.2]{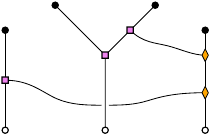} \hspace{1.25cm}  \includegraphics[scale=1.2]{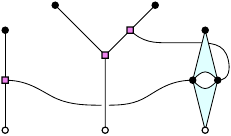} \hspace{1.25cm} \includegraphics[scale=1.2]{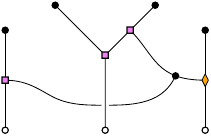}
\caption{Schematic illustrations of the events bounded by \cref{lem:triple_interaction_negligible1,lem:triple_interaction_negligible2,lem:triple_interaction_negligible3}. Each lemma controls one way in which three clusters can ``all interact with each other'' in the coupling described at the beginning of the subsection. Left: \cref{lem:triple_interaction_negligible1} bounds the probability that the clusters $K_0$ and $K_x^x$ \emph{both} cut the connection from $y$ to $a$. Each violet box in the diagram denotes a copy of the vertex factor, while the orange diamond represents a different kind of vertex factor accounting for the fact that the point must lie on the \emph{geodesic} from $y$ to $a$. Intuitively, \cref{lem:Durrett_Nguyen,lem:chemical_distances_sublinear} show that this ``orange diamond'' yields a divergently small correction in the critical dimension. Middle: \cref{eq:almost_negligible2} bounds the probability that neither cluster $K_0$ or $K_x^x$ cuts the connection from $y$ to $a$ by itself, but the two clusters $K_0$ and $K_x$ together successfully produce such a cut. This is bounded using the three-point function estimates of \cref{prop:three_point_upper,lem:three_point_moments} (the blue triangular regions representing copies of the three-point function). Right: \cref{lem:triple_interaction_negligible3} bounds the probability that the connection from $y$ to $a$ is cut by $K_x^x$ but not by $K_x$, meaning that the connection from $x$ to the geodesic $[y,a]_{K_y^y}$ is itself cut by the cluster $K_0$. This is shown to be negligible using similar techniques to those of \cref{lem:triple_interaction_negligible1}.}
\label{fig:triple_intersection}
\end{figure}

The proofs of \cref{lem:triple_interaction_negligible1,lem:triple_interaction_negligible3} will both rely on the following consequence of \cref{lem:chemical_distances_sublinear}. As in \cref{sub:the_key_fictitious_gluability_lemma}, for each finite connected graph $H$ with vertex set contained in $\Z^d$ and pair of vertices $x,y$ in $H$ we choose a geodesic $[x,y]_H$ in a translation-covariant way, so that $[x+z,y+z]_{H+z}=[x,y]_H$ for every $z\in \Z^d$.  

\begin{lemma}
\label{lem:spatial_geodesic}
If $d\geq 3\alpha$ and $\alpha<2$ then
\[\E_r \sum_{a\in K} d_\mathrm{chem}(0,x)^q \sum_{b\in [0,a]_K} \|b\|^p = o\left(r^{d+(q-1)\alpha+p}\right)\]
as $r\to \infty$ for each pair of integers $p,q\geq 0$.
\end{lemma}

The proof of this lemma gives a (very weak) quantitative estimate, with the left hand side smaller than the right hand side by a factor of order at least $|\log V_r|^{\Omega(1)}$ when $d=3\alpha<6$. Fortunately, for our applications in this section this poor quantitative dependence on $r$ will not cause any problems.

\begin{proof}[Proof of \cref{lem:spatial_geodesic}]
If $d> 3\alpha$ and $\alpha<2$ then we have by an easy application of the BK inequality (similar to that of \cref{I-lem:spatial_tree_graph}) that
\begin{multline*}
\E_r \sum_{a\in K} d_\mathrm{chem}(0,x)^q \sum_{b\in [0,a]_K} \|b\|^p = o\left(r^{d+(q-1)\alpha+p}\right)
\preceq_{p,q} (\E_r|K|)^{q+1}\E_r\sum_{x\in K}\|x\|^p\\
\preceq_{p,q} r^{(q+2)\alpha+p} =o(r^{d+(q-1)\alpha+p}),
\end{multline*}
where we applied \cref{I-lem:mixed_moments_order_estimates} (restated here as \cref{eq:mixed_moments_order_estimates}) to bound $\E_r\sum_{x\in K}\|x\|^p$ in the first inequality of the second line. Now suppose that $d=3\alpha<6$. Since $\E_r|K|^2=o((\E_r|K|)^3)$ as $r\to \infty$, it follows from \cref{lem:Durrett_Nguyen,lem:chemical_distances_sublinear} that
\[
  \E_r \sum_{x\in K} d_\mathrm{chem}(0,x) =o\left((\E_r|K|)^2\right)=o\left(r^{2\alpha}\right).
\]
Since we also have by BK that
\[
  \E_r \sum_{x\in K} d_\mathrm{chem}(0,x)^q \preceq_q (\E_r|K|)^{q+1}
\]
for every $q\geq 1$, it follows by H\"older's inequality that
\begin{equation}
  \E_r \sum_{x\in K} d_\mathrm{chem}(0,x)^q =o(r^{(q+1)\alpha})
\label{eq:chemical_Holder}
\end{equation}
as $r\to \infty$ for each integer $q\geq 1$. For each $x\in K$, let $Z_x$ denote a uniform random element of $[0,x]_K$. We have by the BK inequality as above that
\[
  \E_r \sum_{x\in K} d_\mathrm{chem}(0,x) \|Z_x\|^{p} = \E_r \sum_{x\in K} \sum_{b\in [0,x]_K}\|b\|^p \preceq \E_r|K|\E_r\sum_{x\in K}\|x\|^p \preceq_p r^{2\alpha+p}
\]
where we applied \eqref{eq:mixed_moments_order_estimates} again in the final inequality. As such, we can use H\"older's inequality again to obtain that
\begin{align*}
\E_r \sum_{a\in K} d_\mathrm{chem}(0,x)^q \sum_{b\in [0,a]_K} \|b\|^p 
&=
  \E_r \sum_{x\in K} d_\mathrm{chem}(0,x)^{q+1} \|Z_x\|^p 
  \\&\leq \sqrt{\E_r\left[ \sum_{x\in K} d_\mathrm{chem}(0,x)^{2q+1}\right]\E_r\left[ \sum_{x\in K} d_\mathrm{chem}(0,x) \|Z_x\|^{2p}\right]} 
  \\&=o\left(r^{(q+2)\alpha+p}\right)
  =o\left(r^{d+(q-1)\alpha+p}\right)
\end{align*}
as $r\to \infty$ as claimed.
\end{proof}

\begin{proof}[Proof of \cref{lem:triple_interaction_negligible1}]
In order for the event $\mathscr{C}_{0y}(a) \cap \mathscr{C}_{xy}(a)$ to occur, we must have that $a \in K_y^y$ and that every geodesic from $y$ to $a$ in $K_y^y$ is intersected by both $K_0$ and $K_x^x$. Using that $K_x$ is contained in $K_x^x$ on the event that $x\notin K_0$, we can therefore bound 
the quantity of interest by
\begin{equation*}
  \mathbf{E}_r \sum_{x,y\in \Z^d} |K_0| |K_x^x| \sum_{z\in K_x^x} \|z\|^q \sum_{a,b,c\in \Z^d}  \mathbbm{1}(b \in K_0, c \in K_z^z, \text{ and $b,c\in [y,a]_{K_y^y}$}) \|y\|^p,
\end{equation*}
where we have ignored the constraint that $x\in B_r$.
Since the three clusters $K_0$, $K_z^z$, and $K_y^y$ are independent we can write this quantity as
\begin{equation*}
 \sum_{x,y,a,b,c\in \Z^d} \E_r \left[|K_0|\mathbbm{1}(b\in K_0)\right]
 \E_r\left[ |K_x| \mathbbm{1}(c\in K_x) \sum_{z\in K_x} \|z\|^q \right]\P_r(\text{$b,c\in [y,a]_{K_y}$}) \|y\|^p.
\end{equation*}
As in the proof of \cref{I-lem:spatial_tree_graph}, we can use the triangle inequality and the AM-GM inequality to bound this expression by a sum of four expressions each featuring one of the factors $\|b\|^{p+q}$, $\|y-b\|^{p+q}$, $\|y-c\|^{p+q}$, or $\|z-c\|^{p+q}$, and no other distance factors.
For example, the sum involving $\|y-b\|^{p+q}$ is
\begin{equation*}
 \Sigma:=\sum_{x,y,a,b,c\in \Z^d} \E_r \left[|K_0|\mathbbm{1}(b\in K_0)\right]
 \E_r\left[ |K_x|^2 \mathbbm{1}(c\in K_x) \right]\P_r(\text{$b,c\in [y,a]_{K_y}$}) \|y-b\|^{p+q}.
\end{equation*}
We will analyze this sum in detail, the other sums being similar. Using the mass-transport principle to exchange the roles of $0$ and $y$, we have that
 \begin{align*}
\Sigma &= 
\sum_{a,b,c\in \Z^d} 
\P_r(\text{$b,c\in [0,a]_{K}$}) \|b\|^{p+q}  \sum_{x\in \Z^d} \E_r\left[ |K_x|^2 \mathbbm{1}(c\in K_x) \right] \sum_{y\in \Z^d} \E_r \left[|K_y|\mathbbm{1}(b\in K_y)\right]
\\
&= \E_r \left[\sum_{a\in K}d_\mathrm{chem}(0,a) \sum_{b\in [0,a]_K} \|b\|^{p+q}\right] \E_r|K|^2\E_r|K|^3=o(|B_r|r^{p+q}V_r^3 r^{8\alpha}) 
 \end{align*}
as required, where we used \cref{lem:spatial_geodesic} in the final estimate.
\end{proof}

\begin{proof}[Proof of \cref{lem:triple_interaction_negligible2}]
In order for the event $[\mathscr{C}_{0xy}(a) \setminus (\mathscr{C}_{0y}(a) \cup \mathscr{C}_{xy}(a))] \cap \{x\notin K_0\}$ to occur,
there must exist a simple path from $y$ to $a$ in $K_y^y$ that intersects $K_0$ but not $K_x^x$ and another simple path from $y$ to $a$ in $K_y^y$ that intersects $K_0$ but not $K_x^x$. In particular,
 there must exist vertices $b \in K_0 \cap K_y^y$ and $c \in K_x^x \cap K_y^y$ and disjoint subgraphs $H_1$ and $H_2$ of $K_y^y$ such that $y$ is connected to both $b$ and $c$ in $H_1$ and $a$ is connected to both $b$ and $c$ in $H_2$.  As such, using the BK inequality, we can bound the quantity of interest by
 \[
   \sum_{x,y,a,b,c\in \Z^d} \E_r[|K_0|\mathbbm{1}(b\in K_0)] \E_r\left[|K_x| \mathbbm{1}(c\in K_x) \sum_{z\in K_x}\|z\|^q\right] \P_r(y\leftrightarrow b \leftrightarrow c) \P_r(b \leftrightarrow c\leftrightarrow a) \|y\|^p.
 \]
As in the proof of \cref{lem:triple_interaction_negligible1}, we can bound this expression by a sum of four expressions each featuring one of the factors $\|b\|^{p+q}$, $\|y-b\|^{p+q}$, $\|y-c\|^{p+q}$, or $\|z-c\|^{p+q}$, and no other distance factors. We will analyze in detail the sum involving $\|y-b\|^{p+q}$, namely
 \[
  \Sigma:= \sum_{x,y,a,b,c\in \Z^d} \E_r[|K_0|\mathbbm{1}(b\in K_0)] \E_r\left[|K_x|^2 \mathbbm{1}(c\in K_x) \right] \P_r(y\leftrightarrow b \leftrightarrow c) \P_r(b \leftrightarrow c\leftrightarrow a) \|y-b\|^{p+q};
 \]
 the analysis of the other sums is similar.
 Using the mass-transport principle to exchange the roles of $0$ and $b$ and using \cref{lem:three_point_moments} to bound the resulting sums over $a$ and $y$, we obtain that
 \begin{align*}
   \Sigma&=\E_r|K|^2 \E_r|K|^3 \sum_{c \in \Z^d} \sum_{a\in \Z^d}
    \P_r(0 \leftrightarrow c\leftrightarrow a)\sum_{y\in \Z^d}\P_r(0\leftrightarrow c \leftrightarrow y)\|y\|^{p+q} 
   \\
   &\preceq_{p,q}
   \E_r|K|^2 \E_r|K|^3 \sum_{c \in \Z^d} V_{\|c\|}^2 r^{p+q+2\alpha}  \|c\|^{-2d+4\alpha} h_{p+q}(\|c\|/r)
 \end{align*}
 for some decreasing function $h{p+q}:(0,\infty)\to (0,1]$ decaying faster than any power.
To conclude a bound of the desired order on this term, it suffices to verify that
\[
  \sum_{c \in \Z^d} V_{\|c\|}^2  \|c\|^{-2d+4\alpha} h_{p+q}(\|c\|/r) =o\left(r^{d-2\alpha}\right);
\]
when $d>3\alpha$ the left hand side is smaller than the right by a power of $r$, while if $d=3\alpha$ then the sum is of order $V_r^2 r^\alpha$ which is of the required order by \eqref{eq:vertex_factor_divergently_small}.
\end{proof}

\begin{proof}[Proof of \cref{lem:triple_interaction_negligible3}]
First note that $\mathscr{C}_{0y}$ is contained in $\mathscr{C}_{0xy}$, so the event of interest is equal to $[\mathscr{C}_{xy}(a) \setminus \mathscr{C}_{0xy}(a)] \cap \{x\notin K_0\}$.
In order for this event to occur, we must have that $y$ is connected to $a$ in $K_y^y$, that $x$ is connected to some point $b\in [0,a]_{K_y^y}$ by a path in $K_x^x$, and that every path connecting $x$ to $b$ in $K_x^x$ is intersected by $K$. As such, we can bound the quantity of interest by
\[
  \sum_{x,y,a,b,c \in \Z^d}\E_r\left[|K|\mathbbm{1}(c\in K)\right] \E_r\left[|K_x|\mathbbm{1}(b\in K_x) \sum_{z\in K_x}\|z\|^q \right]\P_{r}(a\in K_y, b\in [y,a]_{K_y})\|y\|^p.
\]
This sum can be bounded using a very similar argument to that of \cref{lem:triple_interaction_negligible1} and we omit the details.
\end{proof}

\begin{proof}[Proof of \cref{lem:triple_interaction}]
As explained at the beginning of the section, we restrict our proof to the case $(p_1,p_2,p_3)=(1,1,1)$ for clarity of exposition.
It follows from \eqref{eq:getting_cozy},
\eqref{eq:cozy_inclusion_exclusion}, and \cref{lem:triple_interaction_negligible1,lem:triple_interaction_negligible2,lem:triple_interaction_negligible3} that
\begin{align*}
  &\E_r \left[ |K|  \sum_{x \in B_r} \mathbbm{1}(0\nleftrightarrow x) |K_x| \sum_{z\in K_x} Q(z) \sum_{y\in \Z^d} (\E|K_y|-\mathbbm{1}(y\nleftrightarrow 0,x)|K_y| )  P(y)\right]
  \\&\hspace{3cm}=
\mathbf{E}_r \left[ |K| \sum_{x \in B_r} \mathbbm{1}(0\nleftrightarrow x) |K_x| \sum_{z\in K_x} Q(z) \sum_{y\in \Z^d} \sum_{a\in \Z^d}\mathbbm{1}(\mathscr{C}_{0xy}(a))  P(y) \right]
\\&\hspace{3cm}=
\mathbf{E}_r \left[ |K| \sum_{x \in B_r} \mathbbm{1}(0\nleftrightarrow x) |K_x| \sum_{z\in K_x} Q(z) \sum_{y\in \Z^d} \sum_{a\in \Z^d}\mathbbm{1}(\mathscr{C}_{0y}(a))  P(y) \right]
\\&\hspace{3cm}\hspace{2.5cm}+
\mathbf{E}_r \left[ |K| \sum_{x \in B_r} \mathbbm{1}(0\nleftrightarrow x) |K_x| \sum_{z\in K_x} Q(z) \sum_{y\in \Z^d} \sum_{a\in \Z^d}\mathbbm{1}(\mathscr{C}_{xy}(a))  P(y) \right]
\\&\hspace{3cm}\hspace{7.35cm}\pm o\left(|B_r|V_r^3 r^{\deg(P)+\deg(Q)} (\E_r|K|)^8\right)
\end{align*}
as $r\to \infty$. In light of this together with \eqref{eq:triple_intersection_target1} and \eqref{eq:triple_intersection_target2},  it suffices to prove that
\begin{multline*}
  \mathbf{E}_r \left[ |K| \sum_{x \in B_r} \mathbbm{1}(0\nleftrightarrow x) |K_x| \sum_{z\in K_x} Q(z) \sum_{y\in \Z^d} \sum_{a\in \Z^d}\mathbbm{1}(\mathscr{C}_{0y}(a))  P(y) \right]
  \\=
  \mathbf{E}_r \left[ |K|  \sum_{y\in \Z^d} \sum_{a\in \Z^d}\mathbbm{1}(\mathscr{C}_{0y}(a))  P(y) \right]\E_r \left[\sum_{x \in B_r} \mathbbm{1}(0\nleftrightarrow x) |K_x| \sum_{z\in K_x} Q(z)\right]
  \\\pm o\left(|B_r|V_r^3 r^{\deg(P)+\deg(Q)} (\E_r|K|)^8\right)
\end{multline*}
and
\begin{multline*}
  \mathbf{E}_r \left[ |K| \sum_{x \in B_r} \mathbbm{1}(0\nleftrightarrow x) |K_x| \sum_{z\in K_x} Q(z) \sum_{y\in \Z^d} \sum_{a\in \Z^d}\mathbbm{1}(\mathscr{C}_{xy}(a))  P(y) \right]
  \\=
  \E_r|K|\mathbf{E}_r \left[ \sum_{x \in B_r} |K_x| \sum_{z\in K_x} Q(z) \sum_{y\in \Z^d} \sum_{a\in \Z^d}\mathbbm{1}(\mathscr{C}_{xy}(a))  P(y) \right] 
  \\\pm o\left(|B_r|V_r^3 r^{\deg(P)+\deg(Q)} (\E_r|K|)^8\right).
\end{multline*}
Both estimates follow easily by similar considerations to those used throughout this subsection, and we omit the details.
\end{proof}

\subsection*{Acknowledgements}
This work was supported by NSF grant DMS-1928930 and a Packard Fellowship for Science and Engineering. We thank Gordon Slade for helpful comments on an earlier draft.

\addcontentsline{toc}{section}{References}

 \setstretch{1}
 \footnotesize{
  \bibliographystyle{abbrv}
  \bibliography{unimodularthesis.bib}
  }

\end{document}